\documentclass{amsart}
\usepackage{amssymb}
\usepackage[matrix,arrow]{xy}
\usepackage{verbatim}
\usepackage{a4wide}
\usepackage[applemac]{inputenc}
\newtheorem{theorem}{Theorem}[section]
\newtheorem{lemma}[theorem]{Lemma}

\newtheorem{corollary}[theorem]{Corollary}
\newtheorem{remarq}[theorem]{\it Remark\/}

\newtheorem{proposition}[theorem]{Proposition}

\usepackage{tikz}
\newcounter{braid}
\newcounter{strands}
\pgfkeyssetvalue{/tikz/braid height}{1cm}
\pgfkeyssetvalue{/tikz/braid width}{1cm}
\pgfkeyssetvalue{/tikz/braid start}{(0,0)}
\pgfkeyssetvalue{/tikz/braid colour}{black}
\pgfkeys{/tikz/strands/.code={\setcounter{strands}{#1}}}

\makeatletter
\def\cross{%
  \@ifnextchar^{\message{Got sup}\cross@sup}{\cross@sub}}

\def\cross@sup^#1_#2{\render@cross{#2}{#1}}

\def\cross@sub_#1{\@ifnextchar^{\cross@@sub{#1}}{\render@cross{#1}{1}}}

\def\cross@@sub#1^#2{\render@cross{#1}{#2}}

\def\render@cross#1#2{
  \def\strand{#1}
  \def\crossing{#2}
  \pgfmathsetmacro{\cross@y}{-\value{braid}*\braid@h}
  \pgfmathtruncatemacro{\nextstrand}{#1+1}
  \foreach \thread in {1,...,\value{strands}}
  {
    \pgfmathsetmacro{\strand@x}{\thread * \braid@w}
    \ifnum\thread=\strand
    \pgfmathsetmacro{\over@x}{\strand * \braid@w + .5*(1 - \crossing) * \braid@w}
    \pgfmathsetmacro{\under@x}{\strand * \braid@w + .5*(1 + \crossing) * \braid@w}
    \draw[braid] \pgfkeysvalueof{/tikz/braid start} +(\under@x pt,\cross@y pt) to[out=-90,in=90] +(\over@x pt,\cross@y pt -\braid@h);
    \draw[braid] \pgfkeysvalueof{/tikz/braid start} +(\over@x pt,\cross@y pt) to[out=-90,in=90] +(\under@x pt,\cross@y pt -\braid@h);
    \else
    \ifnum\thread=\nextstrand
    \else
     \draw[braid] \pgfkeysvalueof{/tikz/braid start} ++(\strand@x pt,\cross@y pt) -- ++(0,-\braid@h);
    \fi
   \fi
  }
  \stepcounter{braid}
}

\tikzset{braid/.style={double=\pgfkeysvalueof{/tikz/braid colour},double distance=1pt,line width=2pt,white}}

\newcommand{\braid}[2][]{%
  \begingroup
  \pgfkeys{/tikz/strands=2}
  \tikzset{#1}
  \pgfkeysgetvalue{/tikz/braid width}{\braid@w}
  \pgfkeysgetvalue{/tikz/braid height}{\braid@h}
  \setcounter{braid}{0}
  \let\dsigma=\cross
  #2
  \endgroup
}
\makeatother
\def\eps{\varepsilon}
\def\GL{\mathrm{GL}}

\def\C{\mathbf{C}}
\def\Z{\mathbf{Z}}
\def\kk{\mathbf{k}}
\def\F{\mathbf{F}}
\def\Q{\mathbf{Q}}
\def\C{\mathbf{C}}

\def\into{\hookrightarrow}

\def\Ad{\mathrm{Ad}\,}

\date{October 30, 2011.}

\begin{document}
\centerline{}

\title{The cubic Hecke algebra on at most 5 strands}
\author[I.~Marin]{Ivan Marin}
\address{Institut de Math\'ematiques de Jussieu, Universit\'e Paris 7}
\email{marin@math.jussieu.fr}
\subjclass[2010]{}%Primary 20F36; Secondary 20F55}
\dedicatory{\medskip \parbox{8cm}{To the memory of Johann Gustav Hermes, who worked 10 years on completing
the construction of the 65537-gon and on producing the corresponding beautiful artwork of drawings and numbers, these days known as `Der Koffer' in G\"ottingen's library.} \medskip }
\medskip

\begin{abstract} We prove that the quotient of the group algebra of the braid group on 5 strands by a generic
cubic relation has finite rank. This was conjectured in 1998 by Brou\'e, Malle and Rouquier and has for consequence that this algebra is a flat deformation of the group algebra of the complex reflection group $G_{32}$,
of order 155,520. 
\end{abstract}

\maketitle

%\tableofcontents

\section{Introduction}

In 1957 H.S.M. Coxeter proved (see \cite{COXETER}) that the quotient of the braid group $B_n$ on $n \geq 2$ strands by the relations $s_i^k = 1$, where $s_1,\dots,s_{n-1}$
denote the usual Artin generators, is a finite group if and only if $\frac{1}{k} + \frac{1}{n} > \frac{1}{2}$. This means that, besides the obvious case $k=2$ , which leads
to the symmetric group, and the case $n = 2$, there is only a finite number of such groups. They all turn out to be irreducible complex reflection groups, namely finite subgroups of $\GL_n(\C)$
generated by endomorphisms which fix an hyperplane (so-called pseudo-reflections), and which leave no proper subspace invariant. In the classical classification of such objects,
due to Shephard and Todd, they are nicknamed as $G_4,G_8,G_{16}$ for $n=3$ and $k = 3,4,5$, $G_{25}, G_{32}$ for $n=4,5$ and $k = 3$.

In 1998, M. Brou\'e, G. Malle and R. Rouquier conjectured (see \cite{BMR}) that the group algebra of complex reflection groups admit flat deformations similar to
the Hecke algebra of a Weyl or Coxeter group. They actually introduced natural deformations of such group algebras, called them
the (generic) Hecke algebra associated to such a group,
and they conjectured  that these were flat deformations, and in particular that they have finite rank. For the groups we are interested in, this conjecture actually amounts to saying
that the quotients of the group algebra $R B_n$ by the relations $s_i^k + a_{k-1} s_i^{k-1} + \dots + a_1 s_i + a_0 = 0$,
where $R = \Z[a_{k-1},\dots,a_1,a_0,a_0^{-1}]$, is a flat deformation of the group algebra $R W$, where $W = B_n/s_i^k$ (note that we actually use
a slightly smaller ring than the one used in \cite{BMR} and \cite{BM}). This conjecture was proved in \cite{BM} for all the five groups above but the largest case $G_{32}$
(the proof for $G_{25}$ is however only sketched there).

According to \cite{BMR} (see the proof of theorem 4.24 there) only the following needs to be proved~: that the algebra is spanned over $R$ by $|W|$ elements. This is what we prove here.

\begin{theorem} The generic Hecke algebra associated to $W = G_{32}$ is spanned by $|W|$ elements,
and is thus a free $R$-module of rank $|W|$ which becomes isomorphic to the group algebra of $W$
after a suitable extension of scalars.
\end{theorem}

More precisely, according to \cite{MALLE} corollary 7.2, a convenient extension
of scalars would be $\Q(\zeta_3, (\zeta_3^{-r} u_r)^{\frac{1}{6}}, r= 0,1,2)$ where $\zeta_3$ is a
primitive 3rd root of $1$ and $X^3 + a_2 X^2 + a_1 X + a_0 = (X - u_0)(X-u_1)(X-u_2)$ or,
better, the algebraic extension of $\Q(\zeta_3)(u_0,u_1,u_2)$ generated by
$\sqrt{u_0 u_1}$ and $\sqrt[3]{u_0 u_1 u_2}$ (see \cite{MALLE} table 8.2 and proposition 5.1).

In the general setting of complex reflection groups, it is known that this conjecture is true
\begin{itemize}
\item for the general
series (usually denoted $G(de,e,r)$) of complex reflection groups (by work of Ariki and Ariki-Koike),
\item for most of the exceptional groups of rank 2 by \cite{BM} and \cite{JMUL}, which are numbered $G_4$ to $G_{22}$,
and by \cite{ETINGOFRAINS} for all exceptional groups of rank 2 over a larger ring than expected,
\item  for the Coxeter groups.
\end{itemize}
 The remaining cases are
in rank 4 the groups $G_{29}$ (\cite{JMUL} however proves it over the field of fraction by computer means), $G_{31}$, $G_{32}$, in rank 5 the group $G_{33}$ and in rank $6$ the group $G_{34}$. All
but $G_{32}$, whose case we settled here, have all their pseudo-reflections of order $2$.

In the case studied here, we actually prove more. Here and in the sequel we denote $A_n$ the quotient of $RB_n$
by the generic cubic relation $s_i^3 - a s_i^2 -b s_i -c = 0$. The usual embedding $B_n \into B_{n+1}$ induces
a natural morphism $A_n \to A_{n+1}$, hence a $A_n$-bimodule structure on $A_{n+1}$. For $n \leq 4$, we give a decomposition
of $A_{n+1}$ as $A_n$-bimodule. This immediately provides an explicit $R$-basis of $A_n$ for $n \leq 5$, made of images
of braids in $B_n$. Recall that the orders of $G_4$, $G_{25}$ and $G_{32}$ are $24$, $648$, $155520$.

The following theorem is a recollection of the main results of this article : 
see in particular theorems \ref{theodecA3}, \ref{theodecA4}, \ref{theodecA5raf} and \ref{theofinal} as well as corollary \ref{corA4B72}, and recall that the argument of \cite{BMR} theorem 4.24
(which involves a transcendantal monodromy construction) shows that proving that the Hecke algebra of type $W$ is $R$-generated by $|W|$ elements
ensures that this Hecke algebra is free as a $R$-module, with basis the given $|W|$ elements. Moreover, notice
that, if we have an inclusion of parabolic subgroups $W_0 \subset W$ with corresponding Hecke algebras
$H_0 \subset H$, knowing the conjecture for $H_0$ and that $H$ is generated by $|W/W_0|$ elements as
a $H_0$-module proves (1) the conjecture for $H$ and (2) that $H$ is free as a $H_0$-module, with basis these elements.
Indeed, letting $N = |W/W_0|$ the assumption provides a $H_0$-module morphism $H_0^N \to H$ ;
composing with $(R^{|W_0|})^N \simeq H_0^N$ this yields a surjective morphism $R^{|W|} \to H$ which
is an isomorphism by the argument of \cite{BMR}. This proves that the original morphism $H_0^N \to H$
has no kernel either, and so is an isomorphism.  

\begin{theorem}{\ } %Let $S_3 = S_2 \sqcup \{ s_1^{\eps_1} s_2^{\eps_2} s_1^
\begin{itemize}
\item Let $S_2 = \{ 1,s_1,s_1^{-1} \} \subset B_2$. One has $|S_2| = 3$ and $S_2$ provides an $R$-basis of $A_2$.
\item Let $S_3 = S_2 \sqcup  S_2  s_2^{\pm} S_2  \sqcup
S_2 s_2^{-1} s_1 s_2^{-1} \subset B_3$. One has $|S_3| = 24$ and  $S_3$ provides a $R$-basis of $A_3$.
\item 
%$A_4$ as a $A_3$-module is generated by $27$ elements which are images of elements in the
%braid group and which map to a system of representatives of $G_4 \backslash G_{25} / G_4 $.
$A_4$ is a free $A_3$-module of rank 27. A basis of this $A_3$-module is provided by elements of
the braid group (including $1$) which map to a system of representatives of $ G_{25} / G_4 $.
\item $A_4$ is a free $R$-module of rank 648. A basis of this $R$-module is provided by elements of
the braid group including $1$ which map to all $G_{25}$.
\item 
%$A_4$ is generated as a $\langle s_1,s_3 \rangle$-module by $72$ elements
%which are images of elements in the
%braid group and which map to a system of representatives of $(\Z/3\Z)^2 \backslash G_{25} / (\Z/3\Z)^2 $.
$A_4$ is a free $A_2 \otimes_R  A_2 \simeq \langle s_1,s_3 \rangle$-module of rank $72$.
 A basis of this $\langle s_1,s_3 \rangle$-module is provided by elements of
the braid group including $1$ which map to a system of representatives of $ G_{25} / (\Z/3\Z)^2 $.
\item 
$A_5$ is a free $A_4$-module of rank $240$. A basis is provided by 
elements of the
%$A_5$ is generated as a $A_4$-module by $240$ elements
%which are images of elements in the
braid group including $1$ which map to a system of representatives of $G_{32} / G_{25} $.
\item $A_5$ is a free $R$-module of rank $155,520$. A basis of this $R$-module is provided by elements of
the braid group which include $1$ and which map to all $G_{32}$.% which map to a system of representatives of $(\Z/3\Z)^2 \backslash G_{25} / (\Z/3\Z)^2 |$.
%A basis of this $A_4$-module is provided by elements of
%the braid group.
\end{itemize}
\end{theorem}

\begin{corollary} The natural map $A_n \to A_{n+1}$ is injective for $2 \leq n \leq 4$.
\end{corollary}

We describe the plan of the proof.
Our method is inductive. We find generators of $A_{n+1}$ as a $A_n$-bimodule, and only
then as a $A_n$-module.  After some preliminaries in section \ref{sectprelim} we do the case of $A_3$ in section 
\ref{sectA3}. The structure of $A_4$ as a $A_3$-module is obtained in section \ref{sectA4}.
Before considering $A_5$, we provide in section \ref{sectA4B} an alternative description of $A_4$, this time
as a $\langle s_1, s_3 \rangle$-module. In addition to providing an alternative proof of the
conjecture for $A_4$, this is used in the decomposition of $A_5$ as a $A_4$-module.
This decomposition is obtained in section \ref{sectA5}. We first obtain a decomposition of
$A_5$ as a $A_4$-bimodule, and introduce a filtration of $A_5$ by simpler
$A_4$-bimodules. The latest step of the filtration has original generators originating
from the center of the braid group, and this turns out to be the crucial reason why
this filtration terminates, thus proving that $A_5$ is a $R$-module of finite rank.
For proving this crucial property one needs a lengthy calculation which is postponed in section
\ref{sectcrucial}. We conclude the section \ref{sectA5} and the proof of the main theorem
by studying the structure as $A_4$-modules of the $A_4$-bimodules involved there.

\subsection{Perspectives}

It seems likely that our methods can be used to attack the conjecture for other complex
reflection groups of higher rank. One indeed has the following standard inclusions
of parabolic subgroups (except for the dotted line, which is not a parabolic inclusion).
The number associated to the inclusion is the number of double classes. Note again that
the groups of rank at least $3$ for which the conjecture remains open have all
their reflections of order $2$.
$$
\xymatrix{
G(3,3,2) \ar@{^{(}->}[r]_4 & G(3,3,3) \ar@{^{(}->}[r]_4  &G(3,3,4)  \ar@{^{(}->}[r]_6 & G_{33} \ar@{^{(}->}[r]_{13} &G_{34} \\
 G(4,4,2) \ar@{^{(}->}[r]_5 & G(4,4,3) \ar@{^{(}->}[r]_9 & G_{29} \ar@{^{(}.>}[d]^2 \\
     & G(2,1,3) \ar@{^{(}->}[ur]_{16}  & G_{31} \\
}
$$
For instance, 8 of the 9 double classes of $W = G_{29} = \langle g_1,g_2,g_3,g_4\rangle$ with respect to $W_0 = G(4,4,3) = \langle g_2,g_3,g_4\rangle$
have for representatives $g_1^{\eps} z$ for $z \in Z(W)$ and $\eps \in \{0,1 \}$. If we had
a practical knowledge of the braid groups of type
$G_{29}$ and $G(4,4,3)$ of the same level than the one we have for the usual braid group, the methods used here would then probably
yield a proof of the conjecture for $G_{29}$ in the same way we managed to get one for $G_{32}$, as this kind of phenomenon
(that the most complicated double classes are mainly represented by central elements) is crucial in our proof. Similarly, if $G_{34} = \langle s_1,\dots,s_6\rangle$
with $G_{33} = \langle s_1,\dots,s_5 \rangle$, one can check that $12$ of the $13$ double classes have for
representative a term of the form $z s_6^{\eps}$ for $\eps \in \{0,1 \}$ and $z$ a central element of $G_{34}$.

Another natural question is whether similar deformations exist for a higher number of strands. Indeed, although it is known
that the groups $\Gamma_n = B_n / s_i^3$ are infinite for $n \geq 6$, it was proved in \cite{ASS}
(see also \cite{CABANESMARIN}) that $\Gamma_n^{(1)} = \Gamma_n/z_5^2$ and $\Gamma_n^{(2)} = \Gamma_n/z_5^3$ are finite for arbitrary $n \geq 5$,
and are related to symplectic group over $\F_3$ and to unitary groups over $\F_2$, respectively. Here $z_5$ denotes the image of
the generator $(s_1 s_2 s_3 s_4)^5$ of the braid group on 5 strands into $\Gamma_n, n \geq 5$, which has order $6$ in $\Gamma_5$.
It is thus tempting to look for deformations of the group algebras of $\Gamma_n^{(1)}$ and $\Gamma_n^{(2)}$ for arbitrary $n$
that would be quotients of the group algebra of the braid group
by a generic cubic relations \emph{and} other relations probably
involving $z_5$.

\subsection{Applications}
We mention the following consequences. A first one concerns the study of linear representations of the (usual) braid groups. A consequence of the
proof in \cite{BM} for the cases $G_4, G_8$ and $G_{16}$ was a classification of the linear representations of the braid group $B_3$
in which the image of $s_1$ (and thus of all $s_i$) is killed by a polynomial of degree at most $5$ : indeed, such a representation has to factorize through the
corresponding Hecke algebra. This proves that such representations have a very rigid structure, a result rediscovered in \cite{TW}. 
A similar consequence of this new result is a classification of the linear representations of the braid group $B_n$ for $n$ at most $5$
in which the image of $s_1$ is killed by a cubic polynomial.

A second one is about the cubic invariants of knots and links. 
The algebras connected to cubic invariants,
including the Kauffman polynomial and the Links-Gould polynomial, are quotients of $A_n$. Our result
gives the structure of $A_5$ ; in order to prove it, we actually establish its decomposition as a $A_4$-bimodule, which may be
useful in order to understand the possible Markov traces factorizing through $A_n$.

Specifically, in \cite{CABANESMARIN}, we used the representation theory of the
group $G_{32}$ to prove that an algebra $K_n(\gamma)$ introduced by L. Funar in \cite{FUNAR}
for studying knot invariants collapsed for large $n$ over a field of characteristic distinct from $2$,
and in characteristic $0$ for $n \geq 5$. An immediate consequence of the present
result is that our argument in characteristic $0$ applies verbatim to prove
that the deformation $K_n(\alpha,\beta)$ introduced by P. Bellingeri and L. Funar in \cite{BELFUN} also collapses for $n \geq 5$.
We provide the details below.

\begin{theorem} The generic algebra $K_n(\alpha,\beta)$ introduced in \cite{BELFUN} is zero for $n \geq 5$.
\end{theorem}
\begin{proof}
We use the notations of \cite{BELFUN}. Let $\kk$ be an algebraically closed extension
of $\Q(\alpha,\beta)$. The $\Z[\alpha,\beta]$-algebra $K_n(\alpha,\beta)$ is defined as the
quotient of the group algebra $\Z[\alpha,\beta] B_n$ by the two-sided ideal generated by the
elements $s_i^3 - \alpha s_i^2 - \beta s_i - 1$ and another element $q \in \Z[\alpha,\beta] B_3 \subset \Z[\alpha,\beta] B_n$. We let $\varphi : \Z[a,b,c,c^{-1} ] \to
\Z[\alpha,\beta]$ be the specialization $a \mapsto \alpha$, $b \mapsto \beta$, $c \mapsto 1$,
and let $A_n^0$ denote $A_n \otimes_{\varphi} \Z[\alpha,\beta]$. Obviously $K_n(\alpha,\beta)$
is a quotient of $A_n^0$, more precisely the quotient of $A_n^0$ by the two-sided ideal
generated by (the canonical image of) $q$. Let $\kk$ denote an algebraically closed extension of $\Q(\alpha,\beta)$. We have $A_3^0 \otimes_{\Z[\alpha,\beta]} \kk \simeq \kk G_4 \simeq \kk^3 \oplus Mat_2(\kk)^3 \oplus
Mat_3(\kk)$, and the ideal generated by $q$ is by definition the factor $\kk^3$ in this decomposition (see remark 1.3 in \cite{BELFUN}). As a consequence,
the $\kk$-algebra $\kk K_5(\alpha,\beta)$ is the quotient of the semisimple algebra $\kk A_5^0 \simeq \kk G_{32}$ by the following two-sided ideal : make the direct sum of all the direct factors $Mat_N(\kk)$
whose corresponding irreducible representations have at least one 1-dimensional component in their
restriction to $\kk A_3^0$. Now, to the expense of possibly enlarging $\kk$,
the isomorphisms between the algebras $A_n^0$ and the corresponding group algebras
can be chosen in such a way that the following diagram commutes (e.g. by theorem 2.9 of \cite{KRAMCRG}
 -- see also remark 2.11 there).
$$
\xymatrix{
\kk A_3^0 \ar[r] \ar[d] & \kk A_4^0 \ar[r]  \ar[d] & \kk A_5^0  \ar[d] \\
\kk G_4 \ar[r] & \kk G_{25} \ar[r] & \kk G_{32}
}
$$
As in \cite{CABANESMARIN}, the induction table between the (ordinary) characters of $G_4$
of $G_{32}$ then shows that \emph{all} direct factors $Mat_N(\kk)$ satisfy this property,
and thus the two-sided ideal is all $A_5^0$. It follows that $K_5(\alpha,\beta) = 0$,
whence $K_n(\alpha,\beta) = 0$ for $n \geq 5$, as $K_n(\alpha,\beta)$ is generated by conjugates of the
image of $K_5(\alpha,\beta)$.

\end{proof}

\section{Preliminaries and notations}
\label{sectprelim}

We let $R = \Z[a,b,c,c^{-1}]$
and let $B_n$ denote the braid group on $n$ strands,
generated by the braids $s_1,\dots,s_{n-1}$ with relations $s_i s_{i+1} s_i = s_{i+1} s_i s_{i+1}$
and $s_i s_j = s_j s_i$ for $|j-i| \geq 2$. The cubic Hecke algebra $A_n$ for $n \geq 2$ is the quotient
of the group algebra $R B_n$ by the relations $s_i^3 = a s_i^2 + b s_i + c$. We identify $s_i$ to their images in $A_n$. Notice that, since
$c$ is invertible in $R$, $s_i$ is still invertible, and we have the equivalent relations
$s_i^2 = a s_i + b + c s_i^{-1}$, etc.

The group algebra $R B_n$ admits the automorphism  $s_i \mapsto s_{n-i}$, which
induces an automorphism of $A_n$, as a $R$-algebra. The automorphism
$s_i \mapsto s_i^{-1}$ of $B_n$ induces an automorphism $\Phi$ of $A_n$ as a $\Z$-algebra,
defined by $s_i \mapsto s_i^{-1}$, $a \mapsto -b c^{-1}$, $b \mapsto -a c^{-1}$, $c \mapsto c^{-1}$,
and similarly the skew-automorphism $\Psi$ of $B_n$ defined by $s_i \mapsto s_i^{-1}$ induces a
skew-automorphism of $A_n$ as a $\Z$-algebra.
\begin{comment}
$s_i \mapsto s_i^{-1}$, the
antiautomorphisms $s_i \mapsto s_i^{-1}$, and the automorphism $s_i \mapsto s_{n-i}$.
We remark that all of them induce (anti)automorphisms of $A_n$.
\end{comment}

In the sequel we will denote $u_i$ the $R$-subalgebra of $A_n$ generated by $s_i$ (or
equivalently by $s_i^{-1}$).

The following equalities hold in the braid group, and thus also in $A_n$. We state
them as a lemma because of their importance in the sequel. Notice that they
transform an element of the form $s_{i+1}^{\pm } s_i^{\epsilon} s_{i+1}^{\mp}$ into
an element of $u_i u_{i+1} u_i$. 
\begin{lemma} \label{lemsplusmoins} For $\alpha \in \{ -1,1\}$, we have
$s_{i+1}^{\alpha} s_{i}^{\alpha} s_{i+1}^{-\alpha} =s_{i}^{-\alpha} s_{i+1}^{\alpha} s_{i}^{\alpha}$
and
$s_{i+1}^{\alpha} s_{i}^{-\alpha} s_{i+1}^{-\alpha} =s_{i}^{-\alpha} s_{i+1}^{-\alpha} s_{i}^{\alpha}$, that is

$$\begin{array}{lcl}
s_{i+1} s_{i} s_{i+1}^{-1} &=& s_{i}^{-1} s_{i+1} s_{i} \\
s_{i+1} s_{i}^{-1} s_{i+1}^{-1} &=& s_{i}^{-1} s_{i+1}^{-1} s_{i} \\
s_{i+1}^{-1} s_{i} s_{i+1} &=& s_{i} s_{i+1} s_{i}^{-1} \\
s_{i+1}^{-1} s_{i}^{-1} s_{i+1} &=& s_{i} s_{i+1}^{-1} s_{i}^{-1} 
\end{array}$$
\end{lemma}

\begin{lemma} \label{lemspmgross}{\ }
\begin{enumerate}
\item $s_{i+1}^{\pm } s_i^{\epsilon} s_{i+1}^{\mp} \in u_i u_{i+1} u_i$
\item $s_{i+1}^{\pm } s_i^{\pm} s_{i+1}^{\epsilon} \in u_i u_{i+1} u_i$
\item $s_{i+1}^{\epsilon } s_i^{\pm} s_{i+1}^{\pm} \in u_i u_{i+1} u_i$
\end{enumerate}
\end{lemma}
\begin{proof}
The first item is a direct consequence of lemma \ref{lemsplusmoins},
and the latter two items are consequences of (1) and of the
braid relations $s_i^{\pm} s_{i+1}^{\pm} s_i^{\pm}=
s_i^{\pm} s_{i+1}^{\pm} s_i^{\pm}$.
\end{proof}
\begin{lemma} \label{lemquasicom} {\ }
\begin{enumerate}
\item For all $x \in u_i$, $(s_{i+1}^{-1} s_i s_{i+1}^{-1}) x \in x (s_{i+1}^{-1} s_i s_{i+1}^{-1}) + u_i u_{i+1} u_i$.
\item For all $x \in u_i$, $(s_{i+1} s_i^{-1} s_{i+1}) x \in x (s_{i+1} s_i^{-1} s_{i+1}) + u_i u_{i+1} u_i$.
\end{enumerate}
\end{lemma}
\begin{proof}
(2) is a consequence of (1) up to applying an automorphism of $A_n$, so we restrict ourselves to proving (1).
Since $u_i$ is generated as a $R$-algebra by $s_i^{-1}$, we only need to prove
$(s_{i+1}^{-1} s_i s_{i+1}^{-1}) s_i^{-1} \in s_i^{-1} (s_{i+1}^{-1} s_i s_{i+1}^{-1}) + u_i u_{i+1} u_i$.
We use $s_i = c s_i^{-2} + b s_i^{-1} + a$, $s_i^{-2} = c^{-1} s_i - a c^{-1} - bc^{-1} s_i^{-1}$
and the braid relations,
and get
$$
\begin{array}{lcl}
%(s_{i+1}^{-1} s_i s_{i+1}^{-1}) s_i^{-1} &=&  s_{i+1}^{-1} \mathbf{s_i} s_{i+1}^{-1} s_i^{-1} \\
(s_{i+1}^{-1} s_i s_{i+1}^{-1}) s_i^{-1} &=&  s_{i+1}^{-1} s_i s_{i+1}^{-1} s_i^{-1} \\
&=&  s_{i+1}^{-1} ( c s_i^{-2} + b s_i^{-1} + a) s_{i+1}^{-1} s_i^{-1} \\
&=&  c s_{i+1}^{-1}  s_i^{-2}s_{i+1}^{-1} s_i^{-1} + b s_{i+1}^{-1}s_i^{-1}s_{i+1}^{-1} s_i^{-1} + as_{i+1}^{-1} s_{i+1}^{-1} s_i^{-1} \\
%&=& c s_{i+1}^{-1}  s_i^{-2}s_{i+1}^{-1} s_i^{-1} + b \mathbf{s_{i+1}^{-1}s_i^{-1}s_{i+1}^{-1}} s_i^{-1} + as_{i+1}^{-1} s_{i+1}^{-1} s_i^{-1} \\
&=& c s_{i+1}^{-1}  s_i^{-2}s_{i+1}^{-1} s_i^{-1} + b s_{i}^{-1}s_{i+1}^{-1}s_{i}^{-1} s_i^{-1} + as_{i+1}^{-1} s_{i+1}^{-1} s_i^{-1} \\
%&\in & c s_{i+1}^{-1}  s_i^{-1}\mathbf{s_i^{-1} s_{i+1}^{-1} s_i^{-1}} + u_i u_{i+1} u_i  \\
&\in & c s_{i+1}^{-1}  s_i^{-1}(s_i^{-1} s_{i+1}^{-1} s_i^{-1}) + u_i u_{i+1} u_i  \\
%&\in & c \mathbf{s_{i+1}^{-1}  s_i^{-1}s_{i+1}^{-1}} s_{i}^{-1} s_{i+1}^{-1} + u_i u_{i+1} u_i  \\
&\in & c (s_{i+1}^{-1}  s_i^{-1}s_{i+1}^{-1}) s_{i}^{-1} s_{i+1}^{-1} + u_i u_{i+1} u_i  \\
&\in & c s_{i}^{-1}  s_{i+1}^{-1}s_{i}^{-2}  s_{i+1}^{-1} + u_i u_{i+1} u_i  \\
&\in & c s_{i}^{-1}  s_{i+1}^{-1}( c^{-1} s_i - a c^{-1} - bc^{-1} s_i^{-1}) s_{i+1}^{-1} + u_i u_{i+1} u_i  \\
&\in &  s_{i}^{-1}  s_{i+1}^{-1}(  s_i - a  - b s_i^{-1}) s_{i+1}^{-1} + u_i u_{i+1} u_i  \\
%&\in &  s_{i}^{-1}  s_{i+1}^{-1} s_i s_{i+1}^{-1}- as_{i}^{-1}  s_{i+1}^{-1}  s_{i+1}^{-1} -b s_{i}^{-1}  \mathbf{s_{i+1}^{-1}  s_i^{-1}s_{i+1}^{-1}} + u_i u_{i+1} u_i  \\
&\in &  s_{i}^{-1}  s_{i+1}^{-1} s_i s_{i+1}^{-1}- as_{i}^{-1}  s_{i+1}^{-1}  s_{i+1}^{-1} -b s_{i}^{-1}  (s_{i+1}^{-1}  s_i^{-1}s_{i+1}^{-1}) + u_i u_{i+1} u_i  \\
&\in &  s_{i}^{-1}  s_{i+1}^{-1} s_i s_{i+1}^{-1}- as_{i}^{-1}  s_{i+1}^{-1}  s_{i+1}^{-1} -b s_{i}^{-1}  s_{i}^{-1}  s_{i+1}^{-1}s_{i}^{-1} + u_i u_{i+1} u_i  \\
&\in &  s_{i}^{-1} ( s_{i+1}^{-1} s_i s_{i+1}^{-1}) + u_i u_{i+1} u_i  \\
\end{array}
$$
\end{proof}

\begin{lemma} \label{leminverse}
$s_{i+1}^{-1} s_i s_{i+1}^{-1} \in c^{-1} (s_{i+1} s_i^{-1} s_{i+1}) s_i + u_i u_{i+1} u_i$
\end{lemma}
\begin{proof}
We have $(s_{i+1} s_i^{-1} s_{i+1}) s_i = s_{i+1} (s_i^{-1} s_{i+1} s_i)  = s_{i+1}  s_{i+1} s_i s_{i+1}^{-1} =  s_{i+1}^2 s_i s_{i+1}^{-1} $
by lemma \ref{lemsplusmoins}. Since $s_{i+1}^2 = a s_{i+1} + b + c s_{i+1}^{-1}$ we get
$(s_{i+1} s_i^{-1} s_{i+1}) s_i =(a s_{i+1} + b + c s_{i+1}^{-1})s_i s_{i+1}^{-1} 
= a s_{i+1} s_i s_{i+1}^{-1}+ b s_i s_{i+1}^{-1}+ c s_{i+1}^{-1}s_i s_{i+1}^{-1} \in c s_{i+1}^{-1}s_i s_{i+1}^{-1}  + u_i u_{i+1} u_i$
since $s_{i+1} s_i s_{i+1}^{-1} \in u_i u_{i+1} u_i$ by lemma \ref{lemsplusmoins}. 
\end{proof}

\section{The algebra $A_3$}
\label{sectA3}

We identify $A_2$ with its image in $A_3$ under $s_i \mapsto s_i$,
that is with the subalgebra of $A_3$ generated by $s_1$.
Lemma \ref{lemsplusmoins} provides the following equalities
$$
\begin{array}{lcl}
s_2 s_1 s_2^{-1} &=& s_1^{-1} s_2 s_1 \\
s_2 s_1^{-1} s_2^{-1} &=& s_1^{-1} s_2^{-1} s_1 \\
s_2^{-1} s_1 s_2 &=& s_1 s_2 s_1^{-1} \\
s_2^{-1} s_1^{-1} s_2 &=& s_1 s_2^{-1} s_1^{-1} 
\end{array}
$$

\begin{comment}
The following equalities hold in the braid group, and thus also in $A_3$. We state
them as a lemma because of their importance in the sequel. Notice that they
transform an element of the form $s_{i+1}^{\pm } s_i^{\epsilon} s_{i+1}^{\mp}$ into
an element of $u_i u_{i+1} u_i$. 
\begin{lemma} 
$$\begin{array}{lcl}
s_2 s_1 s_2^{-1} &=& s_1^{-1} s_2 s_1 \\
s_2 s_1^{-1} s_2^{-1} &=& s_1^{-1} s_2^{-1} s_1 \\
s_2^{-1} s_1 s_2 &=& s_1 s_2 s_1^{-1} \\
s_2^{-1} s_1^{-1} s_2 &=& s_1 s_2^{-1} s_1^{-1} 
\end{array}$$
\end{lemma}
\end{comment}
%\begin{proof}

\begin{lemma}
$s_2^{-1} s_1 s_2^{-1} A_2 \subset A_2 s_2^{-1} s_1 s_2^{-1} + u_2 u_1 u_2$ and
$s_2 s_1^{-1} s_2 A_2 \subset A_2 s_2 s_1^{-1} s_2 + u_2 u_1 u_2$
\end{lemma}
\begin{proof}
Straightforward consequences of lemma \ref{lemquasicom}
\end{proof}

\begin{theorem} {\ } \label{theodecA3}
\begin{enumerate}
\item $A_3 = A_2 + A_2 s_2 A_2 + A_2 s_2^{-1} A_2 + A_2 s_2^{-1} s_1 s_2^{-1} A_2$
\item $A_3 = A_2 + A_2 s_2 A_2 + A_2 s_2^{-1} A_2 + A_2 s_2 s_1^{-1} s_2 A_2$
\item $A_3 = A_2 + A_2 s_2 A_2 + A_2 s_2^{-1} A_2 + A_2 s_2 s_1^{-1} s_2 = A_2 + A_2 s_2 A_2 + A_2 s_2^{-1} A_2 + s_2 s_1^{-1} s_2A_2$
\item $A_3 = A_2 + A_2 s_2 A_2 + A_2 s_2^{-1} A_2 + A_2 s_2^{-1} s_1 s_2^{-1} = A_2 + A_2 s_2 A_2 + A_2 s_2^{-1} A_2 + s_2^{-1} s_1 s_2^{-1} A_2$
\end{enumerate}
\end{theorem}
\begin{proof}
Up to applying $\Phi$,
% the ($\Z$-algebra) automorphism of $A_n$ induced by $s_i \mapsto s_i^{-1}$, 
 (2) is a consequence of (1). Then (3) and (4) are consequences
of (1) and (2) by the above lemma. We now prove (1), and let $U$ denote its RHS. It is clearly a $A_2$-submodule of $A_3$ which contains $1$,
so we only need to prove $s_2 U \subset U$. Note that, clearly, $u_1 u_2 u_1 \subset U$.
We first prove $u_2 u_1 u_2 \subset U$. Since we know $u_1 u_2 \subset U$, $u_2 u_1 \subset U$, 
this means that $w= s_2^{\alpha} s_1^{\beta} s_2^{\gamma} \in U$ for all $\alpha,\beta,\gamma \in \{ -1 , 1 \}$. If $\alpha$
and $\beta$ have opposite signs this element belongs to $u_1 u_2 u_1 \subset U$ by lemma \ref{lemsplusmoins},
so we can assume $\alpha = \beta$. If $\alpha = \beta = \gamma$, then the braid relations imply $w \in u_2 u_1 u_2 \subset U$.
Thus only remains $w \in \{ s_2^{-1} s_1 s_2{-1},s_2 s_1^{-1} s_2 \}$. Clearly  $s_2^{-1} s_1 s_2^{-1} \in U$,
and $s_2 s_1^{-1} s_2 \in c(s_2^{-1} s_1 s_2^{-1})s_1^{-1} + u_1 u_2 u_1 \subset u_1 U u_1 = U$
by lemma \ref{leminverse}. We thus proved $u_2 u_1 u_2 \subset U$. We now prove $s_2 U \subset U$.
Clearly $s_2(A_2 + A_2 s_2 A_2 + A_2 s_2^{-1} A_2) \subset u_2 u_1 u_2 u_1 \subset U u_1 \subset U$,
so we need to prove $s_2 u_1  s_2^{-1} s_1 s_2^{-1} \subset U$. But $s_2 u_1  s_2^{-1} \subset u_1 u_2 u_1$ by
 lemma \ref{lemsplusmoins} hence $s_2 u_1  s_2^{-1} s_1 s_2^{-1} \subset u_1 u_2 u_1  u_2 \subset u_1 U \subset U$.  This proves
 the claim.
%, that is
%$s_2 s_1^{\pm}  s_2^{-1} s_1 s_2^{-1} \in U$. But $s_2 s_1^{\pm} s_2^{-1} \in u_1 u_2 u_1$ by lemma \ref{lemsplusmoins},
%hence 

\end{proof}

\begin{corollary} We have  $A_3 = u_1 u_2 u_1 u_2 = u_2 u_1 u_2 u_1$. Moreover,
$$
\begin{array}{lclcl}
A_3 
&=& u_1 u_2 u_1 + u_2 u_1 u_2 + R s_1^{-1} s_2 s_1^{-1} s_2
&=& u_1 u_2 u_1 + u_2 u_1 u_2 + R s_2^{-1} s_1 s_2^{-1} s_1 \\
&=& u_1 u_2 u_1 + u_2 u_1 u_2 + R s_1 s_2^{-1} s_1 s_2
&=& u_1 u_2 u_1 + u_2 u_1 u_2 + R s_2 s_1^{-1} s_2 s_1
\end{array}
$$
\end{corollary}

\begin{corollary} Let $n \geq 3$. For all $1 \leq i,j \leq n-1$, we have in $A_n$ the equality $u_i u_j u_i u_j = u_j u_i u_j u_i$.  
\end{corollary}

This theorem implies that $A_3$ is a free $R$-module of finite rank, consequently that $A_3 \subset A_3 \otimes_R K \simeq Mat_3(K) \oplus Mat_2(K)^3 \oplus K^3$
where $K$ is a sufficiently large extension of the quotient field of $R$, and the isomorphism is explicitely given by the matrix models of the irreducible
representations of $A_3$. From this it is simply a linear algebra matter to check equalities in $A_3$,
or to express a given element in a given basis.
We used this approach to get the following identities in $A_3$. 

\begin{comment}
\begin{lemma} \label{lemdecomp212121old}
$$
\begin{array}{lcl}
s_2^{-1} s_1 s_2^{-1} s_1 s_2^{-1} s_1 &=&  
{\frac { \left( b+c \right)  \left( a+c \right)  \left( a+b \right) 
 \left( b+c+a \right) }{{c}^{2}{a}^{2}{b}^{2}}} s_1 + {\frac {b+c+a}{cab}}  s_2^{-1} s_1 s_2 
 + {\frac {b+c+a}{cab}}s_1^{-1} s_2 s_1 \\ & & + {\frac { \left( b+c+a \right)  \left( bc+ab+ac \right) }{cab}}s_2^{-1} s_1 s_2^{-1}
 + {\frac { \left( b+c+a \right)  \left( bc+ab+ac \right) }{cab}}s_1^{-1}
 -{\frac { \left( b+c+a \right)  \left( bc+ab+ac \right) }{{c}^{2}{a}^{
2}{b}^{2}}} s_2 s_1 \\ & & + s_1^{-1} s_2 s_1^{-1} + {\frac {bc+ab+ac}{cab}} s_2^{-1} s_1 s_2^{-1} s_1 
-{\frac { \left( b+c+a \right)  \left( bc+ab+ac \right) ^{2}}{{c}^{2}{
a}^{2}{b}^{2}}}s_2^{-1} s_1
-{\frac {bc+ab+ac}{cab}} s_1^{-1} s_2 \\ & & 
-{\frac {b+c+a}{cab}}s_1 s_2^{-1} s_1
 -{\frac {bc+ab+ac}{cab}} s_2 s_1^{-1}
 -{\frac { \left( bc+ab+ac \right) ^{2}}{cab}} s_2^{-1} s_1^{-1}
 + (bc+ab+ac)s_1^{-1} s_2^{-1} s_1^{-1}
 \\
%(2a^{-2}+2b^{-2}+2c^{-2}+a^{-2}b^{-1}c+a^{-2}bc^{-1}+a^{-1}b^{-2}c+4a^{-1}b^{-1}+4a^{-1}c^{-1}+4b^{-1}c^{-1}+
%a^{-1}bc^{-2}+ab^{-2}c^{-1}+ab^{-1}c^{-2} )s_1\\
%& & + \\
\end{array}
$$
\end{lemma}
\end{comment}

\begin{lemma} \label{lemdecomp212121}
$$
\begin{array}{lcl}
s_2^{-1} s_1 s_2^{-1} s_1 s_2^{-1} s_1 &=&  
%{\frac { \left( b+c \right)  \left( a+c \right)  \left( a+b \right) 
 %\left( b+c+a \right) }{{c}^{2}{a}^{2}{b}^{2}}} 
 \frac{-(c+ab)a}{c^2}
 s_1 + 
% {\frac {b+c+a}{cab}}  s_2^{-1} 
 
 \frac{a}{c} s_1 s_2 
 + \frac{a}{c} s_1^{-1} s_2 s_1% \\ & & +
%  {\frac { \left( b+c+a \right)  \left( bc+ab+ac \right) }{cab}}
  
  \frac{-ab}{c}s_2^{-1} s_1 s_2^{-1}
  
 + %{\frac { \left( b+c+a \right)  \left( bc+ab+ac \right) }{cab}}
 
 \frac{-ab}{c} s_1^{-1}
% -{\frac { \left( b+c+a \right)  \left( bc+ab+ac \right) }{{c}^{2}{a}^{2}{b}^{2}}} 
+ \frac{ab}{c^2}
s_2 s_1 \\ & & + s_1^{-1} s_2 s_1^{-1} 
%+ {\frac {bc+ab+ac}{cab}} 

- \frac{b}{c}
s_2^{-1} s_1 s_2^{-1} s_1 
%-{\frac { \left( b+c+a \right)  \left( bc+ab+ac \right) ^{2}}{{c}^{2}{a}^{2}{b}^{2}}}
- \frac{ab^2}{c^2} s_2^{-1} s_1
%-{\frac {bc+ab+ac}{cab}} 
+ \frac{b}{c}
s_1^{-1} s_2 \\ & & 
%-{\frac {b+c+a}{cab}}
-\frac{a}{c}s_1 s_2^{-1} s_1
 %-{\frac {bc+ab+ac}{cab}} 
 + \frac{b}{c} s_2 s_1^{-1}
 %-{\frac { \left( bc+ab+ac \right) ^{2}}{cab}} 
 - \frac{b^2}{c} s_2^{-1} s_1^{-1}
 %+ (bc+ab+ac)
 - b s_1^{-1} s_2^{-1} s_1^{-1}
 \\
%(2a^{-2}+2b^{-2}+2c^{-2}+a^{-2}b^{-1}c+a^{-2}bc^{-1}+a^{-1}b^{-2}c+4a^{-1}b^{-1}+4a^{-1}c^{-1}+4b^{-1}c^{-1}+
%a^{-1}bc^{-2}+ab^{-2}c^{-1}+ab^{-1}c^{-2} )s_1\\
%& & + \\
\end{array}
$$
\end{lemma}

\begin{comment} 
\begin{lemma} \label{lemdecomp1212}
$$
\begin{array}{lcl}
s_1 s_2^{-1} s_1 s_2^{-1}  &=&  s_2^{-1} s_1 s_2^{-1} s_1 + 
{\frac {b+c+a}{cab}} s_1 s_2 -{\frac {b+c+a}{cab}} s_2 s_1 +
{\frac { \left( b+c+a \right)  \left( bc+ab+ac \right) }{cab}} s_1 s_2^{-1}
-{\frac { \left( b+c+a \right)  \left( bc+ab+ac \right) }{cab}} s_2^{-1} s_1 \\
& & 
-(bc+ab+ac)s_2^{-1} s_1^{-1} + (bc+ab+ac) s_1^{-1} s_2^{-1} \\
s_2 s_1^{-1} s_2 s_1^{-1} &= & s_2^{-1} s_1 s_2^{-1} s_1   + (b+c+a) (s_1^{-1} s_2 s_1^{-1} - s_2^{-1} s_1 s_2^{-1})
+ {\frac { \left( b+c+a \right)  \left( bc+ab+ac \right) }{cab}} s_1 s_2^{-1} \\ & & 
 - {\frac { \left( b+c+a \right)  \left( bc+ab+ac \right) }{cab}}s_1^{-1} s_2
 -{\frac {bc+ab+ac}{cab}} s_1 s_2^{-1} s_1 
 + {\frac {bc+ab+ac}{cab}} s_2 s_1^{-1} s_2 \\
s_1^{-1} s_2 s_1^{-1} s_2 &=& s_2^{-1} s_1 s_2^{-1} s_1 + {\frac {b+c+a}{cab}} s_1 s_2  - (a+b+c) s_2^{-1} s_1 s_2^{-1}-{\frac {b+c+a}{cab}} s_2 s_1
+ (a+b+c) s_1^{-1} s_2 s_1^{-1}  \\ & & + {\frac { \left( b+c+a \right)  \left( bc+ab+ac \right) }{cab}} s_1 s_2^{-1} -{\frac {bc+ab+ac}{cab}} s_1 s_2^{-1} s_1
-{\frac { \left( b+c+a \right)  \left( bc+ab+ac \right) }{cab}} s_2 s_1^{-1} + {\frac {bc+ab+ac}{cab}} s_2 s_1^{-1} s_2 \\ & & - (bc+ab+ac) s_2^{-1} s_1^{-1} + (bc+ab+ac) s_1^{-1} s_2^{-1}
 \end{array} 
$$ 
\end{lemma}
\end{comment} 

\begin{lemma} \label{lemdecomp1212}
$$
\begin{array}{lcl}
s_1 s_2^{-1} s_1 s_2^{-1}  &=&  s_2^{-1} s_1 s_2^{-1} s_1 + 
%{\frac {b+c+a}{cab}} 
\frac{a}{c} s_1 s_2 -
%{\frac {b+c+a}{cab}} 
\frac{a}{c} s_2 s_1 
%+ {\frac { \left( b+c+a \right)  \left( bc+ab+ac \right) }{cab}} 
- \frac{ab}{c} s_1 s_2^{-1}
%-{\frac { \left( b+c+a \right)  \left( bc+ab+ac \right) }{cab}} 
+ \frac{ab}{c} s_2^{-1} s_1 
%\\& & 
%-(bc+ab+ac)
+b s_2^{-1} s_1^{-1} 
%+ (bc+ab+ac) 
-b s_1^{-1} s_2^{-1} \\
s_2 s_1^{-1} s_2 s_1^{-1} &= & s_2^{-1} s_1 s_2^{-1} s_1  
% + (b+c+a)
 +a (s_1^{-1} s_2 s_1^{-1} - s_2^{-1} s_1 s_2^{-1})
%+ {\frac { \left( b+c+a \right)  \left( bc+ab+ac \right) }{cab}} 
- \frac{ab}{c} s_1 s_2^{-1} %\\ & & 
% - {\frac { \left( b+c+a \right)  \left( bc+ab+ac \right) }{cab}}
 + \frac{ab}{c} s_1^{-1} s_2
 %-{\frac {bc+ab+ac}{cab}} 
 + \frac{b}{c} s_1 s_2^{-1} s_1 \\ & & 
% + {\frac {bc+ab+ac}{cab}}
 - \frac{b}{c} s_2 s_1^{-1} s_2 \\
s_1^{-1} s_2 s_1^{-1} s_2 &=& s_2^{-1} s_1 s_2^{-1} s_1 
%+ {\frac {b+c+a}{cab}} 
+ \frac{a}{c} s_1 s_2  
%- (a+b+c) 
- a s_2^{-1} s_1 s_2^{-1}
%-{\frac {b+c+a}{cab}} 
- \frac{a}{c} s_2 s_1
%+ (a+b+c) 
+as_1^{-1} s_2 s_1^{-1}  \\ & & 
%+ {\frac { \left( b+c+a \right)  \left( bc+ab+ac \right) }{cab}} 
- \frac{ab}{c} s_1 s_2^{-1} 
%-{\frac {bc+ab+ac}{cab}} 
+ \frac{b}{c} s_1 s_2^{-1} s_1
%-{\frac { \left( b+c+a \right)  \left( bc+ab+ac \right) }{cab}} 
+ \frac{ab}{c} s_2 s_1^{-1} 
%+ {\frac {bc+ab+ac}{cab}}
-\frac{b}{c}  s_2 s_1^{-1} s_2 %\\ & &
 %- (bc+ab+ac) 
 +b s_2^{-1} s_1^{-1} 
 %+ (bc+ab+ac) 
 - b s_1^{-1} s_2^{-1}
 \end{array} 
$$ 
\end{lemma}

As a consequence, we get
\begin{lemma}
$$
\begin{array}{lcl}
s_1 s_2^{-1} s_1 s_2^{-1}  -  s_2^{-1} s_1 s_2^{-1} s_1 &=&  
%{\frac {b+c+a}{cab}}
\frac{a}{c} s_1 s_2 
%-{\frac {b+c+a}{cab}} 
- \frac{a}{c} s_2 s_1 +
%{\frac { \left( b+c+a \right)  \left( bc+ab+ac \right) }{cab}} 
-\frac{ab}{c} s_1 s_2^{-1}
%-{\frac { \left( b+c+a \right)  \left( bc+ab+ac \right) }{cab}} 
+\frac{ab}{c} s_2^{-1} s_1% \\ & &
 %-(bc+ab+ac)
 +b s_2^{-1} s_1^{-1} 
 %+ (bc+ab+ac) 
 -b s_1^{-1} s_2^{-1} \\

s_2 s_1^{-1} s_2 s_1^{-1} - s_1^{-1} s_2 s_1^{-1} s_2 &=& 
 %- {\frac { \left( b+c+a \right)  \left( bc+ab+ac \right) }{cab}}
\frac{ab}{c} 
 s_1^{-1} s_2 
 %- {\frac {b+c+a}{cab}}
 -\frac{a}{c}  s_1 s_2  
 %+{\frac {b+c+a}{cab}}
 + \frac{a}{c}  s_2 s_1
%+{\frac { \left( b+c+a \right)  \left( bc+ab+ac \right) }{cab}} 
-\frac{ab}{c}
s_2 s_1^{-1} %\\ & & 
%+ (bc+ab+ac) 
-b s_2^{-1} s_1^{-1} 
%- (bc+ab+ac) 
+b s_1^{-1} s_2^{-1} \\

\end{array}
$$

\end{lemma}

\section{The algebra $A_4$ as a $A_3$ (bi)module}
\label{sectA4}

We identify $A_3$ with its image in $A_4$, and denote $sh(A_3)$ the $R$-subalgebra of
$A_4$ generated by $s_2,s_3,s_4$. It is the image of $A_3$ under the `shift' morphism $s_i \mapsto s_{i+1}$. The goal
of this section is to prove the following theorem.

\begin{theorem} {\ } \label{theodecA4}
\begin{enumerate}
\item $A_4  =A_3 + A_3 s_3 A_3 + A_3 s_3^{-1} A_3 + A_3 s_3 s_2^{-1} s_3 A_3 + A_3 s_3^{-1} s_2 s_1^{-1} s_2 s_3^{-1} A_3
+ A_3 s_3 s_2^{-1} s_1 s_2^{-1} s_3 A_3$
\item $A_4  =A_3 + A_3 s_3 A_3 + A_3 s_3^{-1} A_3 + A_3 s_3 s_2^{-1} s_3 A_3 + A_3 s_3^{-1} s_2 s_1^{-1} s_2 s_3^{-1}
+ A_3 s_3 s_2^{-1} s_1 s_2^{-1} s_3 $
\item $A_4  =A_3 + A_3 s_3 A_3 + A_3 s_3^{-1} A_3 + A_3 s_3 s_2^{-1} s_3 A_3 + s_3^{-1} s_2 s_1^{-1} s_2 s_3^{-1} A_3
+ s_3 s_2^{-1} s_1 s_2^{-1} s_3 A_3$
\end{enumerate}
\end{theorem}

We denote $U$ the right-hand side of (1). We notice that $sh(A_3) \subset U$, because 
of theorem \ref{theodecA3} (2). Also notice that $\Psi(U) = U$ and $\Phi(U) = U$ because of lemma \ref{leminverse}.

%$U$ is stable under the automorphism
%and the antiautomorphism of $A_3$ induced by $s_i \mapsto s_i^{-1}$, because of lemma \ref{leminverse}.

\begin{lemma} \label{lemA4uAu} $u_3 A_3 u_3 \subset U$.
\end{lemma}
\begin{proof}
By theorem \ref{theodecA3} we have $A_3 = u_1 u_2 u_1 + u_1 s_2^{-1} s_1 s_2^{-1}$
hence $u_3 A_3 u_3 \subset  u_3 u_1 u_2 u_1u_3 +u_3  u_1 s_2^{-1} s_1 s_2^{-1} u_3$.
But $ u_3 u_1 u_2 u_1u_3 =  u_1 u_3  u_2 u_3 u_1 \subset u_1 sh(A_3) u_1 \subset u_1 U u_1 \subset U$,
and $u_3  u_1 s_2^{-1} s_1 s_2^{-1} u_3= u_1 u_3   s_2^{-1} s_1 s_2^{-1} u_3$
so we need to prove $s_3^{\alpha}   s_2^{-1} s_1 s_2^{-1} s_3^{\beta} \in U$
for $\alpha, \beta \in \{ -1 ,1 \}$. The case $(\alpha,\beta) = (1,1)$ is clear by definition of $U$.
When $(\alpha,\beta) = (-1,-1)$, we have
$s_3^{-1}   (s_2^{-1} s_1 s_2^{-1}) s_3^{-1} \in c^{-1} s_3^{-1} s_2 s_1^{-1} s_2 s_1 s_3^{-1} + 
s_3^{-1} u_1 u_2 u_1 s_3^{-1}$
that is 
$s_3^{-1}   (s_2^{-1} s_1 s_2^{-1}) s_3^{-1}  \in 
 s_3^{-1} s_2 s_1^{-1} s_2 s_3^{-1} u_1 + 
u_1 s_3^{-1}  u_2  s_3^{-1}u_1 \subset U + u_1 sh(A_3) u_1 \subset U$.
When $(\alpha,\beta) = (1,-1)$ we get
$s_3  s_2^{-1} s_1 s_2^{-1} s_3^{-1}=s_3  s_2^{-1} s_1 (s_2^{-1} s_3^{-1}s_2^{-1} )s_2
= s_3  s_2^{-1} s_1 s_3^{-1} s_2^{-1}s_3^{-1}  s_2
= (s_3  s_2^{-1} s_3^{-1})  s_1  s_2^{-1}s_3^{-1}  s_2 = s_2^{-1}  s_3^{-1}( s_2  s_1  s_2^{-1})s_3^{-1}  s_2
\in  s_2^{-1}  s_3^{-1}u_1 u_2 u_1 s_3^{-1}  s_2 \subset s_2^{-1} u_1 s_3^{-1} u_2 s_3^{-1}  u_1 s_2
\subset A_3 sh(A_3) A_3 \subset U$. The case $(-1,1)$ is similar.

%s_2 s_1^{-1} s_2 \in c(s_2^{-1} s_1 s_2^{-1})s_1^{-1} + u_1 u_2 u_1
\end{proof}

\begin{lemma} \label{lemA4troisvers2}
$u_3 A_3 u_3 A_3 u_3 \subset A_3 u_3 A_3 u_3 A_3$
\end{lemma}
\begin{proof}
For $x \in A_3$, we say that $x$ has at most $p$ factors if it belongs to
as $u_{\sigma(1)} \dots u_{\sigma(p)}$ for some $\sigma : [1,p] \to \{1,2 \}$. By
theorem \ref{theodecA3} the minimal number of factors for such an $x$ is at most 4.
We let $x,y \in A_3$, with minimal number of factors $p$ and $q$,
and prove that $u_3 x u_3 y u_3 \subset A_3 u_3 A_3 u_3 A_3$
by induction on $(p,q)$ in lexicographic order. Note that, since $\Psi(U) = U$,
% is invariant under
%the skew-automorphism of $A_4$ induced by $s_i \mapsto s_i^{-1}$,
we may assume $p \geq q$. Moreover, since $A_3 = u_1 u_2 u_1 u_2$ and
$u_3 (u_1 u_2 u_1 u_2) u_3 y u_3 = u_1 u_3 u_2 u_1 u_2 u_3 y u_3$,
we can assume $p \leq 3$ (hence $q \leq 3$).

 The case $q = 0$ is trivial. If $x \in u_1u_{\sigma(2)}\dots u_{\sigma(p)}$,
we have $u_3 x u_3 y u_3 \in u_3  u_1u_{\sigma(2)}\dots u_{\sigma(p)} u_3 y u_3
= u_1 u_3u_{\sigma(2)}\dots u_{\sigma(p)} u_3 y u_3$ and we are reduced to the case $(p-1,q)$.
Similarly, if $y \in u_{\tau(1)}\dots u_{\tau(q-1)} u_1$, we are reduced to $(p,q-1)$. As a consequence,
the only non-trivial case for $p \leq 1$ is $u_3 u_2 u_3 u_2 u_3 \subset sh(A_3) \subset A_3 u_3 A_3 u_3 A_3$
because of theorem \ref{theodecA3}.

We consider the case $(p,q) = (2,1)$. The only nontrivial case is $u_3 u_2 u_1 u_3 u_2 u_3$.
We need to prove $s_3^{\alpha} u_2 u_1 s_3^{\beta} u_2 s_3^{\gamma} \subset A_3 u_3 A_3 u_3 A_3$
for all $\alpha,\beta,\gamma \in \{ -1,1 \}$. Because $s_3^{\alpha} u_2 u_1 (s_3^{\beta} u_2 s_3^{\gamma} )
= (s_3^{\alpha} u_2 s_3^{\beta})u_1  u_2 s_3^{\gamma} $ this is clear by lemma \ref{lemsplusmoins}
unless $\alpha,\beta$ and $\gamma$ are all the same. We thus need to prove
$s_3^{\alpha} s_2^{\beta} u_1 s_3^{\alpha} s_2^{\gamma} s_3^{\alpha} \subset A_3 u_3 A_3 u_3 A_3$
for all $\beta, \gamma \in \{ -1,1 \}$. Since $s_3^{\alpha} s_2^{\alpha} s_3^{\alpha} =
s_2^{\alpha} s_3^{\alpha} s_2^{\alpha}$ we can assume $\beta = - \alpha$ and $\gamma = - \alpha$
and consider 
$s_3^{\alpha} s_2^{-\alpha} u_1 s_3^{\alpha} s_2^{-\alpha} s_3^{\alpha}$. By lemma \ref{leminverse} we
have $s_3^{\alpha} s_2^{-\alpha} s_3^{\alpha} \in s_3^{-\alpha} s_2^{\alpha} s_3^{-\alpha} u_2 + u_2 u_3 u_2$
hence 
$s_3^{\alpha} s_2^{-\alpha} u_1 s_3^{\alpha} s_2^{-\alpha} s_3^{\alpha} \subset
s_3^{\alpha} s_2^{-\alpha} u_1 s_3^{-\alpha} s_2^{\alpha} s_3^{-\alpha} u_2 + 
s_3^{\alpha} s_2^{-\alpha} u_1u_2 u_3 u_2 \subset 
s_3^{\alpha} s_2^{-\alpha} u_1 s_3^{-\alpha} s_2^{\alpha} s_3^{-\alpha} u_2 + 
u_3A_3 u_3 A_3$ and we already noticed 
$$s_3^{\alpha} s_2^{-\alpha} u_1 s_3^{-\alpha} s_2^{\alpha} s_3^{-\alpha} u_2
=( s_3^{\alpha} s_2^{-\alpha}  s_3^{-\alpha})u_1 s_2^{\alpha} s_3^{-\alpha} u_2
\subset u_2 u_3 u_2 u_1 u_2 u_3 u_2 \subset A_3 u_3 A_3 u_3 A_3.
$$

All cases for $(p,q) = (2,2)$ can be easily reduced to smaller values by commutation relations.
The only a priori irreducible case for $(p,q) = (3,1)$ is $u_3 u_2 u_1 u_2 u_3 u_2 u_3$.
Since $u_2 u_3 u_2 u_3 \subset u_2 u_3 u_2 + u_3 u_2 u_3$ by theorem \ref{theodecA3},
we are reduced to case $(2,1)$.

For the case $(p,q) = (3,2)$, we can use a similar argument : the only nontrivial case
is $u_3 u_2 u_1 u_2 u_3 u_1 u_2 u_3 = u_3 u_2 u_1 u_2 u_1 u_3  u_2 u_3$
and $u_2 u_1 u_2 u_1 \subset u_2 u_1 u_2 + u_1 u_2 u_1$, and we are reduced
to smaller cases.

The only remaining case is thus $(p,q) = (3,3)$. Since $x \in A_3= u_1 u_2 u_1 + u_1 s_2^{-1}s_1 s_2^{-1}$
and $y \in A_3 = u_1 u_2 u_1 + s_2^{-1}s_1 s_2^{-1} u_1$
we are reduced to considering $s_3^{\alpha} s_2^{-1}s_1 s_2^{-1} s_3^{\beta} s_2^{-1}s_1 s_2^{-1} s_3^{\gamma}$
for $\alpha,\beta,\gamma \in \{ -1,1 \}$. Up to applying $\Phi$ if necessary, we can assume $\beta = -1$.
% If $\beta = -1$,
Then $s_3^{\alpha} s_2^{-1}s_1 s_2^{-1} s_3^{-1} s_2^{-1}s_1 s_2^{-1} s_3^{\gamma}
= s_3^{\alpha} s_2^{-1}s_1 s_3^{-1} s_2^{-1} s_3^{-1}s_1 s_2^{-1} s_3^{\gamma}
= (s_3^{\alpha} s_2^{-1} s_3^{-1})s_1 s_2^{-1} s_1  (s_3^{-1}s_2^{-1} s_3^{\gamma})
\subset u_2 u_3 u_2 A_3 u_2 u_3 u_2 \subset A_3 u_3 A_3  u_3A_3$
by lemma \ref{lemspmgross}, 
\begin{comment}
We thus can assume $\beta = 1$. We then use $s_2^{-1} s_1 s_2^{-1}
\in u_1 s_2 s_1^{-1} s_2 + u_1 u_2 u_1$ and
$s_2^{-1} s_1 s_2^{-1}
\in  s_2 s_1^{-1} s_2u_1 + u_1 u_2 u_1$ (see lemmas \ref{leminverse} and \ref{lemquasicom})
and the induction assumption
to get 
$$
\begin{array}{lcl}
s_3^{\alpha} (s_2^{-1}s_1 s_2^{-1}) s_3 (s_2^{-1}s_1 s_2^{-1}) s_3^{\gamma}
&\subset& u_1 s_3^{\alpha} s_2s_1^{-1} (s_2 s_3 s_2)s_1^{-1} s_2 s_3^{\gamma} u_1 + A_3 u_3 A_3 u_3 A_3 \\
&\subset& u_1 s_3^{\alpha} s_2s_1^{-1} s_3 s_2 s_3s_1^{-1} s_2 s_3^{\gamma} u_1 + A_3 u_3 A_3 u_3 A_3 \\
&\subset& u_1 (s_3^{\alpha} s_2s_3) s_1^{-1}  s_2 s_1^{-1}(s_3 s_2 s_3^{\gamma}) u_1 + A_3 u_3 A_3 u_3 A_3 \\
&\subset& u_1 u_2 u_3 u_2 s_1^{-1}  s_2 s_1^{-1}u_2u_3u_2 u_1 + A_3 u_3 A_3 u_3 A_3 \\
&\subset& A_3 u_3 A_3 u_3 A_3 \\
\end{array}
$$
\end{comment}
and this concludes the proof.
\end{proof}

We let $U_0 = A_3 u_3 A_3 + A_3 s_3 s_2^{-1} s_3 A_3 =A_3 u_3 A_3 + A_3 s_3^{-1} s_2 s_3^{-1} A_3 = A_3 sh(A_3) A_3 \subset U$.

\begin{lemma} \label{lemA4droitegauche} {\ }
\begin{enumerate}
\item $s_3^{-1} s_2 s_1^{-1} s_2 s_3^{-1} A_3 \subset A_3 s_3^{-1} s_2 s_1^{-1} s_2 s_3^{-1}  + U_0$
\item $s_3^{-1} s_2 s_1^{-1} s_2 s_3^{-1} A_2 \subset A_2 s_3^{-1} s_2 s_1^{-1} s_2 s_3^{-1}  + U_0$ 
\item $s_3 s_2^{-1} s_1  s_2^{-1} s_3 A_3 \subset A_3 s_3 s_2^{-1} s_1  s_2^{-1} s_3  + U_0$
\item $s_3 s_2^{-1} s_1  s_2^{-1} s_3 A_2 \subset A_2 s_3 s_2^{-1} s_1  s_2^{-1} s_3  + U_0$ 
\end{enumerate}
\end{lemma}

Statements (3) and (4) are consequences of (1) and (2)
by application of $\Phi$,
%the automorphism of $A_4$ induced by $s_i \mapsto s_i^{-1}$,
and (1) and (2) are immediate consequences of the more detailed
lemma below.
 
\begin{comment}
\begin{proof}
We first prove (2). Since $A_2$ is generated by $s_1^{-1}$,
we only need to prove
$(s_3^{-1} s_2 s_1^{-1} s_2 s_3^{-1}) s_1^{-1} = 
s_3^{-1} s_2 s_1^{-1} s_2  s_1^{-1}s_3^{-1} \in U_0 + A_2s_3^{-1} 
s_2 s_1^{-1} s_2 s_3^{-1} $. Since
$s_2 s_1^{-1} s_2  s_1^{-1} \in u_1 u_2 u_1 + u_1 s_2 s_1^{-1} s_2$
we get $s_3^{-1} s_2 s_1^{-1} s_2  s_1^{-1}s_3^{-1} \in u_1 s_3^{-1} u_2 s_3^{-1} u_1 + u_1 s_3^{-1} 
s_2 s_1^{-1} s_2 s_3^{-1} \subset U_0 + A_2s_3^{-1} 
s_2 s_1^{-1} s_2 s_3^{-1} $.

To prove (1), it is then sufficient to prove 
$(s_3^{-1} s_2 s_1^{-1} s_2 s_3^{-1}) s_2^{-1}  
\in U_0 + A_3s_3^{-1} 
s_2 s_1^{-1} s_2 s_3^{-1} $.
We have $s_3^{-1} (s_2 s_1^{-1} s_2) s_3^{-1} s_2^{-1} 
\in s_3^{-1} u_1s_2^{-1} s_1 s_2^{-1} s_3^{-1} s_2^{-1} + s_3^{-1} u_1 u_2 u_1 s_3^{-1} s_2^{-1}
\subset s_3^{-1} u_1s_2^{-1} s_1 s_2^{-1} s_3^{-1} s_2^{-1} + u_1 s_3^{-1}  u_2  s_3^{-1}u_1 s_2^{-1}
\subset u_1 s_3^{-1} s_2^{-1} s_1 (s_2^{-1} s_3^{-1} s_2^{-1}) + A_3 sh(A_3) A_3 
\subset u_1 s_3^{-1} s_2^{-1} s_1 s_3^{-1} s_2^{-1} s_3^{-1} + U_0
\subset u_1 (s_3^{-1} s_2^{-1}  s_3^{-1})s_1 s_2^{-1} s_3^{-1} + U_0
\subset u_1 s_2^{-1} s_3^{-1}  s_2^{-1}s_1 s_2^{-1} s_3^{-1} + U_0
\subset A_3 s_3^{-1}  s_2^{-1}s_1 s_2^{-1} s_3^{-1} + U_0
$\end{proof}
\end{comment}
\begin{comment}
IL EST BIEN POSSIBLE QUE L'ON AIT BIEN PLUS. Genre $s_3^{-1} s_2 s_1^{-1} s_2 s_3^{-1} x - x s_3^{-1} s_2 s_1^{-1} s_2 s_3^{-1}
\in U_0$ pour tout $x \in A_3$ ? C'est le genre de truc qui pourrait etre utile pour la suite
\end{comment}

\begin{lemma} {\ }
\begin{enumerate}
\item For all $x \in A_2$,  $s_3^{-1} s_2 s_1^{-1} s_2 s_3^{-1} x \in x s_3^{-1} s_2 s_1^{-1} s_2 s_3^{-1} + U_0$.
\item $(s_3^{-1} s_2 s_1^{-1} s_2 s_3^{-1}) s_2^{-1}  \in  s_1^{-1}s_2^{-1} s_1 (s_3^{-1} s_2 s_1^{-1} s_2 s_3^{-1}) + U_0$
\item For all $x \in A_3$, $(s_3^{-1} s_2 s_1^{-1} s_2 s_3^{-1})x \in x^{s_1}(s_3^{-1} s_2 s_1^{-1} s_2 s_3^{-1}) + U_0$ (where
$x^{s_1} = s_1^{-1} x s_1$).
\end{enumerate}
\end{lemma}
\begin{proof}

We first prove (1).
%First consider $x = s_1^{-1}$. 
We have
$$
\begin{array}{lcl}
(s_3^{-1} s_2 s_1^{-1} s_2 s_3^{-1}) s_1^{-1} &=& 
s_3^{-1} (s_2 s_1^{-1} s_2)  s_1^{-1} s_3^{-1}\\
& \in &s_3^{-1}  s_1^{-1} (s_2 s_1^{-1} s_2)  s_3^{-1} + s_3^{-1} u_1 u_2 u_1 s_3^{-1} \\
&\subset & s_3^{-1}  s_1^{-1} (s_2 s_1^{-1} s_2)  s_3^{-1} + u_1 s_3^{-1}  u_2 s_3^{-1} u_1 \\
&\subset & s_1^{-1}(s_3^{-1}   s_2 s_1^{-1} s_2  s_3^{-1}) + A_3 sh(A_3) A_3 \\
&\subset & s_1^{-1}(s_3^{-1}   s_2 s_1^{-1} s_2  s_3^{-1}) + U_0 \\
\end{array}
$$
by lemma \ref{lemquasicom}. Since $s_1^{-1}$ generates $A_2$ this proves (1).

We now prove (2).
%Now assume $x = s_2^{-1}$.  
We have 
$$
\begin{array}{lcl} 
(s_3^{-1} s_2 s_1^{-1} s_2 s_3^{-1}) s_2^{-1} &=&
 s_3^{-1} (s_2 s_1^{-1} s_2) s_3^{-1} s_2^{-1} \\
 &\in & c s_3^{-1} (s_2^{-1} s_1 s_2^{-1})s_1^{-1} s_3^{-1} s_2^{-1}
 + s_3^{-1} u_1 u_2 u_1  s_3^{-1} s_2^{-1}\\
 & \subset&
 c s_3^{-1} (s_2^{-1} s_1 s_2^{-1})s_1^{-1} s_3^{-1} s_2^{-1} + U_0
 \end{array}
 $$ by lemma \ref{leminverse}.
 By lemma \ref{lemquasicom} it follows that 
$$
\begin{array}{clcl} & (s_3^{-1} s_2 s_1^{-1} s_2 s_3^{-1}) s_2^{-1} &
 \in &
c s_3^{-1} s_1^{-1} (s_2^{-1} s_1 s_2^{-1}) s_3^{-1} s_2^{-1} + U_0 \\
= & c s_1^{-1}s_3^{-1}  s_2^{-1} s_1 (s_2^{-1} s_3^{-1} s_2^{-1}) + U_0 &
= & c s_1^{-1}s_3^{-1}  s_2^{-1} s_1 s_3^{-1} s_2^{-1} s_3^{-1} + U_0 \\
= &c s_1^{-1}(s_3^{-1}  s_2^{-1} s_3^{-1}s_1  s_2^{-1} s_3^{-1} + U_0
&=& c s_1^{-1}s_2^{-1}  s_3^{-1} (s_2^{-1}s_1  s_2^{-1}) s_3^{-1} + U_0\\
\subset  & c s_1^{-1}s_2^{-1}  s_3^{-1} c^{-1}(s_2s_1^{-1}  s_2)s_1 s_3^{-1} + U_0
&=&   s_1^{-1}s_2^{-1}  s_3^{-1} (s_2s_1^{-1}  s_2)s_1 s_3^{-1} + U_0\\
\end{array}$$
again by lemma \ref{leminverse}.  Since 
$$s_1^{-1}s_2^{-1}  s_3^{-1} (s_2s_1^{-1}  s_2)s_1 s_3^{-1}
\in s_1^{-1}s_2^{-1}  s_3^{-1} s_1 (s_2s_1^{-1}  s_2) s_3^{-1} + U_0
\subset s_1^{-1}s_2^{-1} s_1 s_3^{-1}  (s_2s_1^{-1}  s_2) s_3^{-1} + U_0$$
by lemma \ref{lemquasicom}, this proves (2)

Since $A_3$ is generated by $s_1^{-1}$ and $s_2^{-1}$ and $U_0$ is a $A_3-A_3$ submodule of $A_4$,
we need to check (3) only for $x = s_1^{-1}$ and $x = s_2^{-1}$, and we just did.
\end{proof}

\emph{Proof of theorem \ref{theodecA4}.}

Since $1 \in U$ and $U$ is a $A_3$-submodule of $A_4$, in order to prove
(1) one need to prove $s_3 U \subset U$. Clearly
$U \subset A_3 u_3 A_3 u_3 A_3$ hence
$s_3 U \subset u_3 A_3 u_3 A_3 u_3 A_3 \subset A_3 u_3 A_3 u_3 A_3$
by lemma \ref{lemA4troisvers2}, and $A_3 (u_3 A_3 u_3) A_3 \subset A_3 U A_3 = U$
by lemma \ref{lemA4uAu} which proves the claim. Then (2) and (3) are consequences
of (1) by lemma \ref{lemA4droitegauche}.  This concludes the proof of the theorem.

\begin{comment}
\begin{remarq} Il faudra faire la remarque que ces deux modules exceptionnels
proviennent du centre de $B_4$ (d'image d'ordre $3$ dans $G_{25}$, et que
cela explique le phénomène. Aussi que l'on doit pouvoir remplacer
ces éléments par $s_3 s_2 s_1^2 s_2 s_3$ et son inverse.
\end{remarq}
\end{comment}

\begin{comment}

\begin{center}
!!!! N'IMPORTE QUOI. $c_4$ EST ANNULE PAR UN
POLYNOME DE DEHRE 24, PAS 3 !!!
\end{center}

\begin{lemma} \label{lemcentrecubiqueA4} Letting $c_4 = (s_1 s_2 s_3)^4$,
we have in $A_4$ the identity $(c_4 - \dots) (c_4- \dots) (c_4-\dots) = 0$,
hence $c_4^2 \in R^{\times} c_4^{-1} + R c_4 + R$.
\end{lemma}

\begin{proof} This lemma can be proved by monodromy methods as in e.g. ???, as $c_4$ is the
generator of the center of the braid group originating from the center of $G_{25}$, which has order $3$ ;
or, one can get this identity by direct (computer assisted) computation as in the proof of lemma \ref{lemdecomp212121}, using
explicit matrix models of the representations of $A_4$.
\end{proof}

\end{comment}

%For subsequent use, we mention the following identity in $A_4$.  
\medskip

We now let $w^+ = s_3 s_2^{-1} s_1 s_2^{-1} s_3$, and
$w^- = s_3^{-1} s_2 s_1^{-1} s_2 s_3^{-1} \in A_4$.
We recall that $U_0 = A_3 u_3 A_3 + A_3 u_3 u_2 u_3 A_3 \subset A_4$ is a sub-bimodule,
and let $U^+ = A_3 w^+ + U_0$.

Let $w_0 = s_3 s_2 s_1^2 s_2 s_3$. It is classical that, already in the braid group $B_4$, $w_0$ commutes
with $s_1$ and $s_2$. Thus clearly $A_3 w_0 A_3 = A_3 w_0$ and $A_3 w_0^{-1} A_3 = A_3 w_0^{-1}$.
The lemma below thus provides another explanation
of lemma \ref{lemA4droitegauche} above.

\begin{lemma} {\ } \label{lemauxA4w0}
\begin{enumerate}
\item $w_0 \in A_3^{\times} w^+ + U_0$, $w_0^{-1} \in A_3^{\times} w^- + U_0$
\item $U^+ = A_3 w_0 + U_0$
\item $s_3 A_3 s_3^{-1} \subset U_0$, $s_3^{-1} A_3 s_3 \subset U_0$
\item $s_3 A_3 s_3 \subset U^+$
\end{enumerate}
\end{lemma}
\begin{proof}
We have
$w_0 = s_3 (s_2 s_1^2 s_2) s_3 \in c s_3 s_2 s_1^{-1} s_2 s_3
+ R s_3 s_2 s_1 s_2 s_3 + R s_3 s_2^2 s_3$. Clearly $s_3 s_2^2 s_3 \in U_0$
and $s_3 (s_2 s_1 s_2) s_3 = s_3 (s_1 s_2 s_1) s_3 = s_1s_3  s_2  s_3s_1 \in U_0$.
Moreover, by lemmas \ref{leminverse} and \ref{lemquasicom},
 $s_3 (s_2 s_1^{-1} s_2) s_3 \in c s_3 s_1^{-1} (s_2^{-1} s_1 s_2^{-1}) s_3 
+ s_3 u_1 u_2 u_1 s_3 \subset c s_1^{-1} w^+ + U_0$ and
thus $w_0 \in A_3^{\times} w^+ + U_0$.
As a consequence, $w_0^{-1} =  s_3^{-1} s_2^{-1} s_1^{-2} s_2^{-1} s_3^{-1}
= \Phi(w_0) \in \Phi(A_3^{\times}) \Phi(w^+) + \Phi(U_0) = A_3^{\times} w^- + U_0$,
and
this proves (1). By definition we have $U^+ = A_3 w^+ + U_0 \subset A_3 (A_3^{\times}w_0 + U_0) + U_0
\subset A_3 w_0 + U_0$, and conversely $A_3 w_0 + U_0 \subset A_3(A_3^{\times} w^+ + U_0) + U_0
\subset U^+$ ; this proves (2). (3) and (4) are given by the proof of lemma \ref{lemA4uAu}.

\end{proof}

An immediate consequence is the following variation on theorem \ref{theodecA4}.

\begin{theorem} {\ } \label{theodecA4variante}
\begin{enumerate}
\item $A_4  =A_3 + A_3 s_3 A_3 + A_3 s_3^{-1} A_3 + A_3 s_3 s_2^{-1} s_3 A_3 + A_3 w_0 A_3
+ A_3 w_0^{-1} A_3$
\item $A_4  =A_3 + A_3 s_3 A_3 + A_3 s_3^{-1} A_3 + A_3 s_3 s_2^{-1} s_3 A_3 + A_3 w_0
+ A_3 w_0^{-1} $
\item $A_4  =A_3 + A_3 s_3 A_3 + A_3 s_3^{-1} A_3 + A_3 s_3 s_2^{-1} s_3 A_3 + w_0 A_3
+ w_0^{-1} A_3$
\end{enumerate}
\end{theorem}

From this one easily gets the following generating set of $A_4$ as $A_3$-module. Another generating
set can be found in \cite{BM} \S 4B.

\begin{proposition} \label{propA4A327} As a left $A_3$-module, 
%\begin{enumerate}
%\item $A_3 
%\end{enumerate}
$A_4$ is generated by the 27 elements
 $$\{ 1,s_3^{-1} s_2 s_1^{-1} s_2 s_3^{-1}, s_3 s_2^{-1} s_1 s_2^{-1} s_3, s_3, s_3^{-1}, s_3^{\pm} s_2^{\pm}, s_3^{\pm} s_2^{\pm} s_1^{\pm}, $$
 {} 
 $$ s_3^{\pm} s_2^{-1} s_1 s_2^{-1}, s_3 s_2^{-1} s_3 ,s_3 s_2^{-1} s_3 s_1^{\pm}, s_3 s_2^{-1} s_3 s_1 s_2^{-1} s_1, s_3 s_2^{-1} s_3 s_1^{\pm} s_2^{\pm} \}.
 $$
\end{proposition}
\begin{proof}
We denote $S$ the set of 27 elements of the statement and $L$ its span as a $A_3$-module. We have $A_4 = A_3 + A_3s_3^{-1} s_2 s_1^{-1} s_2 s_3^{-1} + A_3 s_3 s_2^{-1} s_1 s_2^{-1} s_3 + A_3 s_3 A_3
+ A_3 s_3^{-1} A_3 + A_3 s_3 s_2^{-1} s_3 A_3$
by theorem \ref{theodecA4}, and clearly $A_3 + A_3s_3^{-1} s_2 s_1^{-1} s_2 s_3^{-1} + A_3 s_3 s_2^{-1} s_1 s_2^{-1} s_3 \subset L$. Moreover, since $A_3 = A_2 + A_2 s_2 A_2 + A_2 s_2^{-1} A_2 + A_2 s_2^{-1} s_1 s_2^{-1}$
we have 
$$A_3 s_3^{\alpha} A_3 = A_3 s_3^{\alpha}  + \sum_{\stackrel{\eps \in \{-1,0,1 \}}{\beta \in \{ -1,1 \}}}A_3 s_3^{\alpha}  s_2^{\beta} s_1^{\eps} + A_3 s_3^{\alpha}  s_2^{-1} s_1 s_2^{-1}\subset L
$$
for any $\alpha \in \{ -1,1 \}$.
It remains to prove $A_3 s_3 s_2^{-1} s_3 A_3 \subset L$.
Since $A_3 = A_2 + A_2 s_2 A_2 + A_2 s_2^{-1} A_2 +  s_2^{-1} s_1 s_2^{-1}A_2$, we have
$A_3 s_3 s_2^{-1} s_3 A_3
=A_3 s_3 s_2^{-1} s_3 A_2 +A_3 s_3 s_2^{-1} s_3 A_2 s_2^{-1} A_2 +A_3 s_3 s_2^{-1} s_3 A_2 s_2 A_2 +A_3 s_3 s_2^{-1} s_3  s_2^{-1} s_1 s_2^{-1}A_2$.
Clearly $A_3 s_3 s_2^{-1} s_3 A_2$ is $A_3$-spanned by the $s_3 s_2^{-1} s_3 s_1^{\eps}$ for $\eps \in \{ 0, 1, -1 \}$
hence $A_3 s_3 s_2^{-1} s_3 A_2\subset L$.
Now $s_3 s_2^{-1} s_3  s_2^{-1} \in s_2^{-1} s_3 s_2^{-1} s_3 + u_2 u_3 + u_3 u_2$ by lemma \ref{lemdecomp1212}, hence
$A_3 s_3 s_2^{-1} s_3  s_2^{-1} s_1 s_2^{-1}A_2 \subset A_3 s_3 s_2^{-1} s_3s_1 s_2^{-1}A_2 + A_3 u_3 u_2 s_1 s_2^{-1}A_2 + A_3 u_3 s_1 s_2^{-1}A_2$,
that is 
$$A_3 s_3 s_2^{-1} s_3  s_2^{-1} s_1 s_2^{-1}A_2 \subset A_3 s_3 s_2^{-1} s_3s_1 s_2^{-1}A_2 + A_3 u_3 A_3.
$$ 
We thus only need to show $A_3 s_3 s_2^{-1} s_3 A_2 s_2^{\beta} A_2 \subset L$ for $\beta \in \{ -1 ,1 \}$. This module is $A_3$-spanned by 
the $s_3 s_2^{-1} s_3 s_1^{\alpha} s_2^{\beta} s_1^{\gamma}$ for $\alpha, \gamma \in \{ 0, 1, -1 \}$. The elements belong to $S$
when $\gamma = 0$, so we can assume $\gamma \in \{-1,1 \}$. When $\alpha = 0$, in
case $\beta = 1$ we have $s_3 (s_2^{-1} s_3  s_2) s_1^{\gamma} = s_3 s_3 s_2 s_3^{-1} s_1^{\gamma} \in u_3 s_2 s_3^{-1} s_1^{\gamma}$.
This latter module is spanned by the $s_3^{\eps} s_2 s_3^{-1} s_1^{\gamma}$ for $\eps \in \{ -1,0,1 \}$. In case $\eps = 0$
such an element belongs to $A_3 u_3 A_3 \subset L$ ; when $\eps = 1$ we can use $(s_3 s_2 s_3^{-1}) s_1^{\gamma} =
s_2^{-1} s_3 s_2  s_1^{\gamma} \in A_3 s_3 s_2 s_1^{\gamma} \in L$ ; when $\eps = -1$ we have $s_3^{-1} s_2 s_3^{-1} s_1^{\gamma} \in 
A_3 s_3 s_2^{-1} s_3 s_1^{\gamma}$ by lemmas \ref{leminverse} and \ref{lemquasicom}, and $s_3 s_2^{-1} s_3 s_1^{\gamma} \in S$.
We can thus assume $\alpha \neq 0$.

We consider first the case $\gamma = - \alpha$. We have
$s_3 s_2^{-1} s_3 s_1^{\alpha} s_2^{\beta} s_1^{-\alpha}  = 
s_3 s_2^{-1} s_3 s_2^{-\alpha} s_1^{\beta} s_2^{\alpha}$. Then, either
$\alpha = 1$ and, by lemma \ref{lemdecomp1212},
$$(s_3 s_2^{-1} s_3 s_2^{-1}) s_1^{\beta} s_2 
\in s_2^{-1} s_3 s_2^{-1} s_3 s_1^{\beta} s_2 + 
u_2u_3s_1^{\beta} s_2
+
u_3u_2s_1^{\beta} s_2  \subset L,
$$
or $\alpha = -1$ and 
$s_3 (s_2^{-1} s_3 s_2) s_1^{\beta} s_2^{-1}=
s_3 s_3 s_2 s_3^{-1} s_1^{\beta} s_2^{-1}\in u_3 s_2 s_3^{-1} s_1^{\beta} s_2^{-1}$.
This latter module is spanned by the $s_3^{\eps} s_2 s_3^{-1} s_1^{\beta} s_2^{-1}$
which clearly belong to $L$ for $\eps = 0$ and, because of $s_3 s_2 s_3^{-1} = s_2^{-1} s_3 s_2$,
for $\eps = 1$ ; in case $\eps = -1$ it is readily shown to belong to $L$ by lemmas \ref{leminverse} and \ref{lemquasicom}
applied to $s_3^{-1} s_2 s_3^{-1}$.
%by lemma \ref{lemdecomp1212}

We can now assume $\gamma = \alpha$. In case $\beta = \alpha = \gamma$, we have
$s_3 s_2^{-1} s_3 (s_1^{\alpha} s_2^{\alpha} s_1^{\alpha} ) 
=s_3s_2^{-1} s_3 s_2^{\alpha} s_1^{\alpha} s_2^{\alpha}  $
and, when $\alpha = 1$ we have $s_3(s_2^{-1} s_3 s_2) s_1 s_2 
= s_3s_3 s_2 s_3^{-1} s_1 s_2  \in u_3 s_2 s_3^{-1} s_1 s_2 \subset L$ by similar arguments as for 
$u_3 s_2 s_3^{-1} s_1^{\beta} s_2^{-1}$ ; when $\alpha = -1$,
we have, by lemma \ref{lemdecomp1212},
$$(s_3s_2^{-1} s_3 s_2^{-1}) s_1^{-1} s_2^{-1} 
\in s_2^{-1} s_3s_2^{-1} s_3  s_1^{-1} s_2^{-1} 
+ u_3u_2 s_1^{-1} s_2^{-1} 
+ u_2u_3 s_1^{-1} s_2^{-1} \subset L. $$

\begin{comment}
and we already know  
elements of
$A_3 s_3 s_2^{-1} s_3 A_2$, and in addition by the $s_3 s_2^{-1} s_3 s_2^{\pm} s_1^{\eps}$, the $s_3 s_2^{-1} s_3 s_1^{\pm} s_2^{\eps}$,
the $s_3 s_2^{-1} s_3 s_1^{\pm} s_2^{\pm} s_1^{\pm}$ for $\eps \in \{0,-1,1 \}$. From $s_3 s_2^{-1} s_3 s_2 = s_3^2 s_2 s_3^{-1}$
one is readily reduced to $s_3 s_2^{-1} s_3 s_2^{-1} s_1^{\eps}$ for the first terms ; among the $s_3 s_2^{-1} s_3 s_1^{\pm} s_2^{\eps}$,
the case $\eps = 0$ has already taken care of, and the other ones belong to our list ; finally,
among the $s_3 s_2^{-1} s_3 s_1^{\alpha} s_2^{\beta} s_1^{\gamma}$ for $\alpha,\beta,\gamma \in \{ -1,1 \}$,
one can assume $\gamma = \alpha$ otherwise using $s_1^{\alpha} s_2^{\beta} s_1^{-\alpha} =
 s_2^{-\alpha} s_1^{\beta} s_2^{\alpha}$ one could reduce to older terms. 
 Again because of the braid relation
 between $s_1$ and $s_2$ one can then assume $\beta = - \alpha$, and 
  \end{comment}
We thus
 only need to consider the $s_3 s_2^{-1} s_3 s_1^{\alpha} s_2^{-\alpha} s_1^{\alpha}$.
By lemmas \ref{leminverse} and \ref{lemquasicom}, we have
$ s_1^{-1} s_2 s_1^{-1} \in s_2 (s_1 s_2^{-1} s_1)  + u_2 u_1 u_2$, and $s_3 s_2^{-1} s_3 u_2 u_1 u_2$
belongs to the $A_3$-span of our list by our previous arguments. It follows that it only remains to consider
$s_3 s_2^{-1} s_3s_1 s_2^{-1} s_1$, which belongs to our list,
and $s_3 (s_2^{-1} s_3 s_2) s_1 s_2^{-1} s_1 = s_3^2 s_2 s_3^{-1} s_1 s_2^{-1} s_1$, which lies in the linear
span of the $s_3^{\eps} s_2 s_3^{-1} s_1 s_2^{-1} s_1$ for $\eps \in \{ -1,0,1 \}$.
Clearly this element belongs to $L$ in case $\eps = 0$, when $\eps = 1$
it also belongs to $L$ because of $(s_3 s_2 s_3^{-1}) s_1 s_2^{-1} s_1  
=s_2^{-1} s_3 s_2 s_1 s_2^{-1} s_1  \in A_3 u_3 A_3 \subset L$,
and when $\eps = -1$ lemmas \ref{leminverse} and \ref{lemquasicom} applied to
$s_3^{-1} s_2 s_3^{-1}$ show that 
$$
s_3^{-1} s_2 s_3^{-1} s_1 s_2^{-1} s_1 \in
A_3 s_3 s_2^{-1} s_3 s_1 s_2^{-1} s_1 + A_3 u_3 A_3 \subset L,
$$
and this concludes the proof.
%=s_2^{-1} s_3 s_2 s_1 s_2^{-1} s_1  
% $ \in u_3 u_2 u_3s_1 s_2^{-1} s_1$.
%By the same arguments as before, it is easily shown that $u_3 u_2 u_3s_1 s_2^{-1} s_1 \subset L$,
%and 
%which has already been proved to be included in the $A_3$-span of our list.
%,  Since 

%is spanned by 
\end{proof}

For subsequent use we prove here the following lemma.

\begin{lemma} \label{lemw0carre} $w_0^2 \in A_3^{\times} w_0^{-1} + U^+$.
\end{lemma}
\begin{proof}
We have $w_0^2 = s_3 s_2 (s_1^2) s_2 s_3^2 s_2 s_1^2 s_2 s_3
\in R^{\times} s_3 s_2 s_1^{-1} s_2 s_3^2 s_2 s_1^2 s_2 s_3
+ R s_3 s_2 s_1 s_2 s_3^2 s_2 s_1^2 s_2 s_3
+ R s_3 s_2^2 s_3^2 s_2 s_1^2 s_2 s_3$, and $s_3 (s_2 s_1 s_2) s_3^2 s_2 s_1^2 s_2 s_3
= s_3 s_1 s_2 s_1 s_3^2 s_2 s_1^2 s_2 s_3
= s_1 (s_3  s_2  s_3) s_3 s_1 s_2 s_1^2 s_2 s_3
= s_1 s_2  (s_3  s_2 s_3) s_1 s_2 s_1^2 s_2 s_3 
= s_1 s_2  s_2  s_3 s_2 s_1 s_2 s_1^2 s_2 s_3 
\in U_0^+$ by lemma \ref{lemauxA4w0},
while $s_3 s_2^2 (s_3^2) s_2 s_1^2 s_2 s_3 \in
R s_3 s_2^2 s_3^{-1} s_2 s_1^2 s_2 s_3
+ R s_3 s_2^2 s_3 s_2 s_1^2 s_2 s_3
+R s_3 s_2^3 s_2 s_1^2 s_2 s_3$,
clearly  $s_3 s_2^3 s_2 s_1^2 s_2 s_3 \in U^+$ by lemma \ref{lemauxA4w0},
$$
\begin{array}{clcl} & s_3 s_2^2 s_3 s_2 s_1^2 s_2 s_3
&=& s_3 s_2 (s_2  s_3 s_2) s_1^2 s_2 s_3 \\
= &s_3 s_2 s_3  s_2 s_3 s_1^2 s_2 s_3
&=& (s_3 s_2 s_3)  s_2  s_1^2 (s_3 s_2 s_3) \\
= & s_2 s_3 s_2  s_2  s_1^2 s_2 s_3 s_2 \in U^+
\end{array}$$ by lemma \ref{lemauxA4w0},
and finally 
$$(s_3 s_2^2 s_3^{-1}) s_2 s_1^2 s_2 s_3 = 
s_2^{-1} (s_3^2) s_2^2 s_1^2 s_2 s_3 \in 
R s_2^{-1} s_3^{-1} s_2^2 s_1^2 s_2 s_3 + R s_2^{-1} s_3 s_2^2 s_1^2 s_2 s_3 +
R  s_2 s_1^2 s_2 s_3 \subset U^+$$ by lemma \ref{lemauxA4w0}.

Thus, $w_0^2 \in R^{\times} s_3 s_2 s_1^{-1} s_2 s_3^2 s_2 s_1^2 s_2 s_3 + U^+$.
Now, $s_3 s_2 s_1^{-1} s_2 s_3^2 s_2 (s_1^2) s_2 s_3
\in R^{\times} 
s_3 s_2 s_1^{-1} s_2 s_3^2 s_2 s_1^{-1} s_2 s_3
+ R s_3 s_2 s_1^{-1} s_2 s_3^2 s_2 s_1 s_2 s_3
+ R s_3 s_2 s_1^{-1} s_2 s_3^2 s_2^2 s_3$.
We have $s_3 s_2 s_1^{-1} s_2 s_3^2 (s_2 s_1 s_2) s_3 =
s_3 s_2 s_1^{-1} s_2 s_3^2 s_1 s_2 s_1 s_3
= s_3 s_2 (s_1^{-1} s_2 s_1) s_3^2  s_2  s_3 s_1
= s_3 s_2^2  s_1 (s_2^{-1} s_3^2  s_2)  s_3 s_1
= s_3 s_2^2  s_1 s_3 s_2^2  s_3^{-1}  s_3 s_1
= s_3 s_2^2  s_1 s_3 s_2^2   s_1 \in U^+$   by lemma \ref{lemauxA4w0}.
Now $s_3 s_2 s_1^{-1} s_2 (s_3^2) s_2^2 s_3 
\in R s_3 s_2 s_1^{-1} s_2 s_3^{-1} s_2^2 s_3 
+ R s_3 s_2 s_1^{-1} s_2 s_3 s_2^2 s_3 
+ R s_3 s_2 s_1^{-1} s_2^3 s_3 $ ;
we have $s_3 s_2 s_1^{-1} s_2^3 s_3 \in U^+$ by lemma \ref{lemauxA4w0},
$s_3 s_2 s_1^{-1} s_2 s_3 s_2^2 s_3 
=s_3 s_2 s_1^{-1} (s_2 s_3 s_2) s_2 s_3  
=s_3 s_2 s_1^{-1} s_3 s_2 s_3 s_2 s_3  
=(s_3 s_2 s_3)s_1^{-1}  s_2 (s_3 s_2 s_3)  
=s_2 s_3 s_2 s_1^{-1}  s_2 s_2 s_3 s_2 \in U^+ $  by lemma \ref{lemauxA4w0},
and $s_3 s_2 s_1^{-1} s_2 (s_3^{-1} s_2^2 s_3)
= s_3 s_2 s_1^{-1} s_2 s_2 s_3^2 s_2^{-1}
\in s_3 s_2 s_1^{-1} s_2^2 s_2 (R + R s_3 + R s_3^{-1}) s_2^{-1} \subset U^+$
by lemma \ref{lemauxA4w0}.

Thus $w_0^2 \in R^{\times} 
s_3 s_2 s_1^{-1} s_2 s_3^2 s_2 s_1^{-1} s_2 s_3 + U^+$. Now,
$$
s_3 s_2 s_1^{-1} s_2 (s_3^2) s_2 s_1^{-1} s_2 s_3
\in R^{\times} s_3 s_2 s_1^{-1} s_2 s_3^{-1} s_2 s_1^{-1} s_2 s_3
+ R s_3 s_2 s_1^{-1} s_2 s_3 s_2 s_1^{-1} s_2 s_3
+ R s_3 s_2 s_1^{-1} s_2^2 s_1^{-1} s_2 s_3.$$
 We have 
$$
\begin{array}{lclcl}
s_3 s_2 s_1^{-1} (s_2 s_3 s_2) s_1^{-1} s_2 s_3 
&=& s_3 s_2 s_1^{-1} s_3 s_2 s_3 s_1^{-1} s_2 s_3
&=& (s_3 s_2  s_3) s_1^{-1} s_2  s_1^{-1}(s_3 s_2 s_3)\\ & 
=& s_2 s_3  s_2 s_1^{-1} s_2  s_1^{-1}s_2 s_3 s_2 & \in & U^+ \\
\end{array}
$$  by lemma \ref{lemauxA4w0},
$s_3 s_2 s_1^{-1} s_2^2 s_1^{-1} s_2 s_3 \in U^+$ by lemma \ref{lemauxA4w0},
hence $w_0^2 \in R^{\times} s_3 s_2 s_1^{-1} s_2 s_3^{-1} s_2 s_1^{-1} s_2 s_3 + U^+$.
Using $s_2 s_1^{-1} s_2 \in A_2^{\times} s_2^{-1} s_1 s_2^{-1} + u_1 u_2 u_1$
(see lemmas \ref{leminverse} and \ref{lemquasicom}),
we get 
$ s_3 (s_2 s_1^{-1} s_2) s_3^{-1} s_2 s_1^{-1} s_2 s_3
\in A_2^{\times} s_3 s_2^{-1} s_1 s_2^{-1} s_3^{-1} s_2 s_1^{-1} s_2 s_3 
+ s_3 u_1 u_2 u_1 s_3^{-1} s_2 s_1^{-1} s_2 s_3$.
Since 
$$s_3 u_1 u_2 u_1 s_3^{-1} s_2 s_1^{-1} s_2 s_3
= u_1 (s_3  u_2  s_3^{-1})u_1 s_2 s_1^{-1} s_2 s_3
\subset u_1 s_2^{-1} u_3 s_2u_1 s_2 s_1^{-1} s_2 s_3 \subset U^+$$ by  lemma \ref{lemauxA4w0},
we have $w_0^2 \in A_2^{\times} s_3 s_2^{-1} s_1 s_2^{-1} s_3^{-1} s_2 s_1^{-1} s_2 s_3  + U^+$.
Now 
$$
\begin{array}{llclcl}
& s_3 s_2^{-1} s_1 (s_2^{-1} s_3^{-1} s_2) s_1^{-1} s_2 s_3
&=& s_3 s_2^{-1} s_1 s_3 s_2^{-1} s_3^{-1} s_1^{-1} s_2 s_3
&=& s_3 s_2^{-1} s_1 s_3 s_2^{-1}  s_1^{-1}(s_3^{-1} s_2 s_3)\\
= &s_3 s_2^{-1} s_1 s_3 (s_2^{-1}  s_1^{-1}s_2) s_3 s_2^{-1}
&=& s_3 s_2^{-1} s_1 s_3 s_1  s_2^{-1}s_1^{-1} s_3 s_2^{-1}
&=& s_3 s_2^{-1} s_3  s_1^2   s_2^{-1}s_1^{-1} s_3 s_2^{-1}
\end{array}
$$
and, using $s_3 s_2^{-1} s_3 \in u_2^{\times} s_3^{-1} s_2 s_3^{-1} + u_2 u_3 u_2$,
we get 
$$s_3 s_2^{-1} s_3  s_1^2   s_2^{-1}s_1^{-1} s_3 s_2^{-1}
\in u_2^{\times} s_3^{-1} s_2 s_3^{-1}  s_1^2   s_2^{-1}s_1^{-1} s_3 s_2^{-1}
+ u_2 u_3 u_2  s_1^2   s_2^{-1}s_1^{-1} s_3 s_2^{-1}.$$
We have $u_2 u_3 u_2  s_1^2   s_2^{-1}s_1^{-1} s_3 s_2^{-1} \in U^+$ by lemma
 \ref{lemauxA4w0}, and 
 $$s_3^{-1} s_2 s_3^{-1}  s_1^2   s_2^{-1}s_1^{-1} s_3 s_2^{-1}
= s_3^{-1} s_2   s_1^2 (s_3^{-1}  s_2^{-1}s_3) s_1^{-1}  s_2^{-1}
= s_3^{-1} s_2   s_1^2 s_2  s_3^{-1}s_2^{-1} s_1^{-1}  s_2^{-1}.$$

Thus $w_0^2 \in A_3^{\times} s_3^{-1} s_2   s_1^2 s_2  s_3^{-1}(s_2^{-1} s_1^{-1}  s_2^{-1})
+ U^+$. Since $s_3^{-1} s_2   (s_1^2) s_2  s_3^{-1}
\in R^{\times} s_3^{-1} s_2   s_1^{-1} s_2  s_3^{-1}
+ R s_3^{-1} s_2   s_1 s_2  s_3^{-1} + R s_3^{-1} s_2^2  s_3^{-1}$
and clearly $s_3^{-1} s_2^2  s_3^{-1}\in U_0$,
$s_3^{-1} (s_2   s_1 s_2)  s_3^{-1} = s_3^{-1} s_1   s_2 s_1  s_3^{-1}
= s_1   s_3^{-1}  s_2   s_3^{-1}s_1 \in U_0$,
we have $w_0^2 \in A_3^{\times} w^- (s_2^{-1} s_1^{-1}  s_2^{-1}) + U^+$,
hence $w_0^2 \in A_3^{\times} w_0^{-1} (s_2^{-1} s_1^{-1}  s_2^{-1}) + U^+$
by lemma \ref{lemauxA4w0} (1). Since $w_0$ commutes with $s_1$ and $s_2$
this yields  $w_0^2 \in A_3^{\times} w_0^{-1}  + U^+$.
%\subset 

\end{proof}

\section{The algebra $A_4$ as a $\langle s_1,s_3 \rangle$ (bi)module}
\label{sectA4B}

Let $B = \langle s_1,s_3 \rangle$ denote the subalgebra (with $1$) of $A_4$
generated by $s_1$ and $s_3$. In order to describe $A_5$ as a $A_4$-module
we will need the description of $A_4$ as a $B$-module, that we do in this
section. Note that this will provide another proof of the conjecture
of \cite{BMR} for $A_4$.

First note that there are \emph{three} automorphisms or skew-automorphisms
of the pair $(A_4,B)$ : in addition to the automorphism $\Phi$ and the
skew-automorphism $\Psi$, there is the automorphism $\Ad \Delta : x \mapsto \Delta x \Delta^{-1}$,
where $\Delta$ is the image in $A_4$ of Garside's $\Delta$ in the braid group on 4 strands,
that is $\Delta = s_1 s_2 s_3 s_1 s_2 s_1 = s_1 (s_2 s_3 s_1 s_2) s_1$ ;  this
automorphism exchanges $s_1$ and $s_3$ and fixes $s_2$.

We denote $A_4^{[0]} = B$, $A_4^{[n+1]} = A_4^{[n]} u_2 B = A_4^{[n]} + A_4^{[n]} s_2 B + A_4^{[n]} s_2^{-1} B$,
and in particular $A_4^{[1]} = B + B s_2 B + B s_2^{-1} B$.

We first prove several lemmas.

\begin{lemma} {\ } \label{lemB1}
\begin{enumerate}
\item For $i,j \in \{ 1, 3 \}$ we have $u_2 u_i u_2 u_j u_2 \subset A_4^{[2]}$.
\item For $i,j,k \in \{ 1, 3 \}$ we have $u_2 u_i u_2 u_j u_k u_2 \subset A_4^{[2]}$
and $u_2 u_i u_j u_2  u_k u_2 \subset A_4^{[2]}$
\end{enumerate}

\end{lemma}
\begin{proof} We prove (1). If $i=j$, up to applying $\Ad \Delta$
we can assume $i=j=1$ and the statement is a consequence of the
study of $A_3$, as $u_2 u_1 u_2 u_1 u_2 \subset A_3 \subset u_1 u_2 u_1 u_2 + u_1 u_2 u_1$.
Thus we can assume $i \neq j$, and by using $\Ad \Delta$ and $\Psi$
we only need to consider $X = s_2^{\alpha} s_1^{\beta} s_2^{\gamma} s_3^{\delta} s_2^{\eps}$
with $\alpha,\dots,\eps \in \{ -1, 1 \}$. If $\alpha = -\gamma$ or $\gamma = - \eps$,
then we get $X \in A_4^{[2]}$ by using $s_2^{\alpha} s_1^{\beta} s_2^{-\alpha} = s_1^{-\alpha} s_2^{\beta} s_1^{\alpha}$
and $s_2^{\gamma} s_3^{\delta} s_2^{-\gamma} = s_3^{-\gamma} s_2^{\delta} s_3^{\gamma}$. Up to applying $\Phi$
we can thus assume $\alpha = \gamma = \eps = 1$, that is $X = s_2 s_1^{\beta} s_2 s_3^{\delta} s_2$.
If $\beta = 1$ or $\delta = 1$ we get $X \in A_4^{[2]}$ by
$s_2 s_1 s_2 = s_1 s_2 s_1$ and $s_2 s_3 s_2 = s_3 s_2 s_3$. One can thus assume $X = s_2 s_1^{-1} s_2 s_3^{-1} s_2$.
By lemmas \ref{leminverse} and \ref{lemquasicom} we have $s_2 s_1^{-1} s_2 \in u_1^{\times} s_2^{-1} s_1 s_2^{-1} + u_1 u_2 u_1$
hence $X \in u_1^{\times} s_2^{-1} s_1 s_2^{-1} s_3 s_2 + A_4^{[2]}$ and $s_2^{-1} s_1 (s_2^{-1} s_3 s_2) =s_2^{-1} s_1 s_3 s_2 s_3^{-1}
\in A_4^{[2]}$, and this concludes the proof of (1).

We prove (2). Up to applying $\Psi$ we can confine ourselves to prove $u_2 u_i u_2 u_j u_k u_2 \subset A_4^{[2]}$. By (1)
and $u_j^2 = u_j$, $u_k^2 = u_k$ we can assume $j \neq k$, that is $\{j,k \} = \{ 1, 3 \}$. Up to applying $\Ad \Delta$ we can assume $i=1$,
hence we want to prove $u_2 u_1 u_2 u_1 u_3 u_2 \subset A_4^{[2]}$. We have
$u_2 u_1 u_2 u_1 \subset A_3 = u_1 u_2 u_1 u_2 + u_1 u_2 u_1$ hence
$u_2 u_1 u_2 u_1 u_3 u_2 \subset u_1 u_2 u_1 u_2 u_3 u_2 + u_1 u_2 u_1 u_3 u_2 \subset A_4^{[2]}$
by (1).
\end{proof}

\begin{lemma} \label{lemB2}
$$
A_4^{[3]} \subset A_4^{[2]} + \sum_{\alpha,\beta \in \{ -1,1 \}} B s_2^{\alpha} (s_1 s_3^{-1})^{\beta} s_2^{\alpha} (s_1 s_3^{-1})^{\beta} s_2^{\alpha} B
+ \sum_{\alpha,\beta \in \{ -1,1 \}} B s_2^{\alpha} (s_1 s_3)^{\beta} s_2^{-\beta} (s_1 s_3)^{\beta} s_2^{\eps} B
$$
\end{lemma}
\begin{proof}
We only need to prove that all the terms of the form $s_2^{\alpha} s_1^{\beta_1} s_3^{\beta_3} s_2^{\gamma} s_1^{\delta_1} s_3^{\delta_3} s_2^{\eps}$ belong
to the RHS, as all the over natural linear generators of $A_4^{[3]}$ belong to $A_4^{[2]}$ by lemma \ref{lemB1}. We remark
that the RHS is stable under $\Phi$, $\Psi$ and $\Ad \Delta$.

We first assume $\beta_1 = - \delta_1$. Then 
$$s_2^{\alpha} s_1^{\beta_1} s_3^{\beta_3} s_2^{\gamma} s_1^{\delta_1} s_3^{\delta_3} s_2^{\eps} = s_2^{\alpha} s_3^{\beta_3} (s_1^{\beta_1}  s_2^{\gamma} s_1^{-\beta_1}) s_3^{\delta_3} s_2^{\eps}
=s_2^{\alpha} s_3^{\beta_3} s_2^{-\beta_1}  s_1^{\gamma} s_2^{\beta_1} s_3^{\delta_3} s_2^{\eps}.$$
If $\alpha = \beta_1$ or $\eps = -\beta_1$, such a term belongs to $A_4^{[2]}$ by
$s_2^{\alpha} s_3^{\beta_3} s_2^{-\alpha} = s_3^{-\alpha} s_2^{\beta_3} s_3^{\alpha}$ or $s_2^{-\eps} s_3^{\delta} s_2^{\eps} =s_3^{\eps} s_2^{\delta} s_3^{-\eps}$
and lemma \ref{lemB1}. We thus only need to consider $X = s_2^{-\beta_1} s_3^{\beta_3} s_2^{-\beta_1}  s_1^{\gamma} s_2^{\beta_1} s_3^{\delta_3} s_2^{\beta_1}$.
Since $s_2^{-\beta_1} s_3^{-\beta_1} s_2^{-\beta_1} = s_3^{-\beta_1} s_2^{-\beta_1} s_3^{-\beta_1}$ and $s_2^{\beta_1}s_3^{\beta_1}s_2^{\beta_1}=s_3^{\beta_1}s_2^{\beta_1}s_3^{\beta_1}$,
by lemma \ref{lemB1} we can assume $\beta_3 = \beta_1$ and $\delta = - \beta_1$, that is $X = s_2^{-\beta_1} s_3^{\beta_1} s_2^{-\beta_1}  s_1^{\gamma} s_2^{\beta_1} s_3^{-\beta_1} s_2^{\beta_1}$.
By applying $\Phi$ and $\Psi$ we can assume $X = s_2 s_3^{-1} s_2 s_1 s_2^{-1} s_3 s_2^{-1}$. By lemma \ref{leminverse} we have $s_2^{-1} s_3 s_2^{-1} \in s_2 s_3^{-1} s_2 u_3^{\times} + u_3 u_2 u_3$
hence $s_2 s_3^{-1} s_2 s_1 (s_2^{-1} s_3 s_2^{-1}) \in s_2 s_3^{-1} s_2 s_1 s_2 s_3^{-1} s_2 u_3^{\times} + s_2 s_3^{-1} s_2 s_1 u_3 u_2 u_3$.
We have $ s_2 s_3^{-1} s_2 s_1 u_3 u_2 u_3 \subset A_4^{[2]}$ by lemma \ref{lemB1} and $s_2 s_3^{-1} s_2 s_1 s_2 s_3^{-1} s_2$ belongs to the RHS, which concludes this case.

The case $\beta_3 = -\delta_3$ is a consequence of the previous case by applying $\Ad \Delta$. We thus only need
to consider $X = s_2^{\alpha} s_3^{\beta_3} s_1^{\beta_1} s_2^{\gamma} s_1^{\beta_1} s_3^{\beta_3} s_2^{\eps}$. First
assume $\gamma = \beta_1$. Then $X = s_2^{\alpha} s_3^{\beta_3} (s_1^{\gamma} s_2^{\gamma} s_1^{\gamma}) s_3^{\beta_3} s_2^{\eps}
=s_2^{\alpha} s_3^{\beta_3} s_2^{\gamma} s_1^{\gamma} s_2^{\gamma} s_3^{\beta_3} s_2^{\eps}$ belongs as before to $A_4^{[2]}$
by lemma \ref{lemB1} and elementary transformations, unless $\eps = \gamma$, $\alpha = \gamma$, and then $\beta_3 = - \gamma$. In
that case $X = s_2^{\alpha} s_3^{-\alpha} (s_2^{\alpha} s_1^{\alpha} s_2^{\alpha}) s_3^{-\alpha} s_2^{\alpha}
= s_2^{\alpha} s_3^{-\alpha} s_1^{\alpha} s_2^{\alpha} s_1^{\alpha} s_3^{-\alpha} s_2^{\alpha}$
belongs to the RHS. We can thus assume $\gamma \neq \beta_1$ and, applying $\Ad \Delta$, $\gamma \neq \beta_3$,
hence we can assume $\beta_1 = \beta_3 = - \gamma$. Then $X = s_2^{\alpha} s_3^{\gamma} s_1^{\gamma} s_2^{-\gamma} s_1^{\gamma} s_3^{\gamma} s_2^{\eps}$,
which belongs to the RHS, and this concludes the proof.

\end{proof}
\begin{lemma} \label{lemB3}Let $\alpha, \beta, \eps \in \{ -1, 1 \}$. Then
$s_2^{\alpha} (s_1 s_3)^{\beta} s_2^{-\beta} (s_1 s_3)^{\beta} s_2^{\eps} $ belongs to
$$ A_4^{[2]} +
\sum_{\delta \in \{-1,1 \}} B s_2^{\delta} (s_1 s_3)^{\delta} s_2^{-\delta} (s_1 s_3)^{\delta} s_2^{\delta} B
+
\sum_{\delta \in \{-1,1 \}} B s_2^{-\delta} (s_1 s_3)^{\delta} s_2^{-\delta} (s_1 s_3)^{\delta} s_2^{-\delta} B
$$

\end{lemma}
\begin{proof}
First assume $\alpha = \beta$. Then $X = s_2^{\beta} s_1^{\beta} s_3^{\beta} s_2^{-\beta} s_3^{\beta} (s_1^{\beta} s_2^{\eps} s_1^{-\beta}) s_1^{\beta}
= s_2^{\beta} s_1^{\beta} (s_3^{\beta} s_2^{-\beta} s_3^{\beta} s_2^{-\beta}) s_1^{\eps} s_2^{\beta} s_1^{\beta}
\in s_2^{\beta} s_1^{\beta}  s_2^{-\beta} s_3^{\beta} s_2^{-\beta}s_3^{\beta} s_1^{\eps} s_2^{\beta} s_1^{\beta}
+
s_2^{\beta} s_1^{\beta} u_2 u_3 s_1^{\eps} s_2^{\beta} s_1^{\beta}
+
s_2^{\beta} s_1^{\beta} u_3 u_2 s_1^{\eps} s_2^{\beta} s_1^{\beta}
$
by lemma \ref{lemdecomp1212}. Now
$s_2^{\beta} s_1^{\beta} u_2 u_3 s_1^{\eps} s_2^{\beta} s_1^{\beta} \subset A_4^{[2]}$
and 
$s_2^{\beta} s_1^{\beta} u_3 u_2 s_1^{\eps} s_2^{\beta} s_1^{\beta}  \subset A_4^{[2]}$
by lemma \ref{lemB1}. We thus only need to consider
$$X = 
(s_2^{\beta} s_1^{\beta}  s_2^{-\beta}) s_3^{\beta} s_2^{-\beta}s_3^{\beta} s_1^{\eps} s_2^{\beta} 
= s_1^{-\beta} s_2^{\beta}  s_1^{\beta} s_3^{\beta} s_2^{-\beta}s_3^{\beta} s_1^{\eps} s_2^{\beta} $$
hence, if $\eps = -\beta$, we get
$$X = s_1^{-\beta}. s_2^{\beta}   s_3^{\beta} (s_1^{\beta}s_2^{-\beta} s_1^{-\beta})s_3^{\beta} s_2^{\beta}
=  s_1^{-\beta}. s_2^{\beta}   s_3^{\beta} s_2^{-\beta}s_1^{-\beta} (s_2^{\beta}s_3^{\beta} s_2^{\beta})
=  s_1^{-\beta}. s_2^{\beta}   s_3^{\beta} s_2^{-\beta}s_1^{-\beta} s_3^{\beta}s_2^{\beta} s_3^{\beta}
\in A_4^{[2]}$$ by lemma \ref{lemB1}. We can thus assume $\eps = \beta$, in which case
$X$ clearly belongs to the space we want.
 
This concludes the case $\alpha = \beta$, hence also the case $\eps = \beta$ by application of
$\Phi$ and $\Psi$. Thus we can assume $\alpha = -\beta = \eps$, and the conclusion is clear in
this case.

\end{proof}
\begin{lemma} For $\alpha, \beta \in \{ -1,1 \}$, we have
$$
s_2^{\alpha} (s_1 s_3^{-1})^{\beta} s_2^{\alpha} (s_1 s_3^{-1})^{\beta} s_2^{\alpha} 
\in A_4^{[2]} + \sum_{\delta \in \{ -1,1 \}} B s_2^{\delta} (s_1 s_3)^{\delta} s_2^{-\delta} (s_1 s_3)^{\delta} s_2^{\delta} B
$$
\end{lemma}
\begin{proof}
The RHS is stable under $\Ad \Delta$ and $\Phi$, hence we can assume
$\alpha = \beta = 1$, and thus we only need to consider $X = s_2 s_1 s_3^{-1} s_2 s_1 s_3^{-1} s_2 = 
s_2 s_1 (s_3^{-1} s_2 s_3^{-1}) s_1  s_2 \in
s_2 s_1 u_2 s_3 s_2^{-1} s_3  s_1  s_2 + s_2 s_1 u_2 u_3 u_2 s_1 s_2$
by lemmas \ref{leminverse} and \ref{lemquasicom}.  We have 
$$
s_2 s_1 u_2 u_3 u_2 s_1 s_2 \subset \sum_{a \in \{0,1,-1\}} 
s_2 s_1 s_2^a u_3 u_2 s_1 s_2
$$
and, 
\begin{itemize}
\item if $a=0$ we have 
$s_2 s_1  u_3 u_2 s_1 s_2 \subset A_4^{[2]}$ by lemma \ref{lemB1} ;
\item if $a=1$ we have 
$(s_2 s_1 s_2) u_3 u_2 s_1 s_2 = s_1 s_2 s_1 u_3 u_2 s_1 s_2 \subset A_4^{[2]}$ by lemma \ref{lemB1} ;

\item if $a=-1$ we have 
$(s_2 s_1 s_2^{-1}) u_3 u_2 s_1 s_2  = s_1^{-1} s_2 s_1 u_3 u_2 s_1 s_2  \subset A_4^{[2]}$ by lemma \ref{lemB1}
\end{itemize}
hence $X \in s_2 s_1 u_2 s_3 s_2^{-1} s_3 s_1 s_2 + A_4^{[2]}$. The
module $s_2 s_1 u_2 s_3 s_2^{-1} s_3 s_1 s_2$
is $R$-spanned by the $Y(a) = s_2 s_1 s_2^a s_3 s_2^{-1} s_3 s_1 s_2$
for $a \in \{ -1,0,1 \}$. We have $Y(0) = s_2 s_1 s_3 s_2^{-1} s_3 s_1 s_2 \in RHS$,
$Y(1) = (s_2 s_1 s_2) s_3 s_2^{-1} s_3 s_1 s_2
=s_1 s_2 s_1 s_3 s_2^{-1} s_3 s_1 s_2 \in RHS$
and $$Y(-1) = (s_2 s_1 s_2^{-1}) s_3 s_2^{-1} s_3 s_1 s_2
= s_1^{-1} s_2 s_1 s_3 s_2^{-1} s_3 s_1 s_2 \in RHS,$$ and this concludes the proof.
\end{proof}

In the braid group on $4$ strands , we have
$$
\Delta = (s_1 s_2 s_3)(s_1 s_2) s_1 = (s_1 s_3) (s_2 s_1 s_3 s_2) = (s_2 s_1 s_3 s_2) (s_1 s_3)
$$
hence the same equalities hold in $A_4$. In the remaining part of this section,
we let $s = s_2$, $p = s_1 s_3 = s_3 s_1$, hence $\Delta = spsp = psps$. Note
that $\Delta p  = p \Delta$, $\Delta s = s \Delta$.
It follows that $\Delta^2 = p. sps. \Delta = p. \Delta. sps = p(psps) sps = p^2 . sps^2 ps$,
$\Delta^3 = p^2.sps^2ps . \Delta = p^2 . \Delta. sps^2 ps = p^3 . sps^2 p s^2 p$,
and $\Delta^4 = p^4 . sps^2ps^2 p s^2 ps$.

We thus have $\Delta^2 = p^2.sps^2ps$ hence
$p^{-2} \Delta^2 \in R^{\times} sps^{-1}ps + R spsps + R sp^2 s$ and
we known $sp^2 s \in A_4^{[2]}$, $(spsp)s = psps^2 \in A_4^{[2]}$
by lemma \ref{lemB1}. It follows that
$$
\begin{array}{lccl}
(*) & p^{-2} \Delta^2 &\in& R^{\times} sps^{-1} ps + R psps^2 + R sp^2 s \\
 & p^{-2} \Delta^2& \in & R^{\times} sps^{-1} ps + A_4^{[2]}
\end{array}
$$
Applying $\Phi$, we have $\Phi(\Delta) = \Phi(s_1 s_2 s_3 s_1 s_2 s_1) = 
s_1^{-1} s_2^{-1} s_3^{-1} s_1^{-1} s_2^{-1} s_1^{-1} = (s_1 s_2 s_1 s_3 s_2 s_1)^{-1} = \Delta^{-1}$,
hence, since $\Phi(p) =p^{-1}$,
$$
\begin{array}{lccl}
(*)  & p^{-2} \Delta^2& \in & R^{\times} s^{-1}p^{-1}s p^{-1}s^{-1} + A_4^{[2]}
\end{array}
$$

\begin{lemma} {\ } \label{lemB5}
\begin{enumerate}
\item $s_2^{-1} p s_2^{-1} p s_2^{-1} . s_1^{-1} \in u_1^{\times} s_2 p^{-1} s_2 p^{-1} s_2 + A_4^{[2]} $
\item $s_2^{-1} p s_2^{-1} p s_2^{-1} B  \subset B s_2 p^{-1} s_2 p^{-1} s_2 + B s_2^{-1} p s_2^{-1} p s_2^{-1} + A_4^{[2]} $
\item $s_2 p^{-1} s_2 p^{-1} s_2 B  \subset B s_2 p^{-1} s_2 p^{-1} s_2 + B s_2^{-1} p s_2^{-1} p s_2^{-1} + A_4^{[2]} $
\end{enumerate}

\end{lemma}
\begin{proof}
$X= s_2^{-1} p s_2^{-1} p s_2^{-1} . s_1^{-1}  = s_2^{-1} p s_2^{-1} s_3 (s_1 s_2^{-1} . s_1^{-1} )
= s_2^{-1} p s_2^{-1} s_3 s_2^{-1} s_1^{-1}  s_2
= s_2^{-1} s_1 (s_3s_2^{-1} s_3 s_2^{-1}) s_1^{-1}  s_2
\in  s_2^{-1} s_1 s_2^{-1} s_3 s_2^{-1}s_3 s_1^{-1}  s_2 +  s_2^{-1} s_1u_2 u_3 s_1^{-1}  s_2+ s_2^{-1} s_1 u_3 u_2 s_1^{-1}  s_2$
by lemma \ref{lemdecomp1212}. We have $ s_2^{-1} s_1u_2 u_3 s_1^{-1}  s_2 \subset A_4^{[2]}$ and 
$s_2^{-1} s_1 u_3 u_2 s_1^{-1}  s_2\subset A_4^{[2]}$ by lemma \ref{lemB1}, hence
$$
\begin{array}{lcl}
X &\in& (s_2^{-1} s_1 s_2^{-1}) s_3 s_2^{-1}s_3 s_1^{-1}  s_2 + A_4^{[2]} \\
&\subset &u_1^{\times} s_2 s_1^{-1} (s_2 s_3 s_2^{-1}s_3 s_1^{-1}  s_2 +  u_1 u_2 u_1  s_3 s_2^{-1}s_3 s_1^{-1}  s_2 + A_4^{[2]}\\
&\subset &u_1^{\times} s_2 s_1^{-1} s_3^{-1} s_2 s_3^2 s_1^{-1}  s_2 + u_1 u_2 u_1  s_3 s_2^{-1}s_3 s_1^{-1}  s_2 + A_4^{[2]}\\
&\subset &u_1^{\times} s_2 s_1^{-1} s_3^{-1} s_2 p^{-1}  s_2 + u_1 s_2 s_1^{-1} s_3^{-1} s_2 s_3 s_1^{-1}  s_2+ u_1 s_2 s_1^{-1} s_3^{-1} s_2  s_1^{-1}  s_2 + u_1 u_2 u_1  s_3 s_2^{-1}s_3 s_1^{-1}  s_2 + A_4^{[2]}\\
\end{array}
$$
We know $s_2 s_1^{-1} s_3^{-1} s_2  s_1^{-1}  s_2 \in A_4^{[2]}$ by lemma \ref{lemB1}, $s_2 s_1^{-1} (s_3^{-1} s_2 s_3) s_1^{-1}  s_2 
= s_2 s_1^{-1} s_2 s_3 (s_2^{-1} s_1^{-1}  s_2) = 
= s_2 s_1^{-1} s_2 s_3 s_1 s_2^{-1}  s_1^{-1} \in A_4^{[2]}$ by lemma \ref{lemB1}, and  
$u_2 u_1  s_3 s_2^{-1}s_3 s_1^{-1}  s_2 = u_2   s_3 u_1 s_2^{-1}s_1^{-1} s_3   s_2 $ is the
sum of $u_2   s_3  s_2^{-1}s_1^{-1} s_3   s_2 \subset A_4^{[2]}$ (by lemma \ref{lemB1})
and of the $u_2   s_3 s_1^{\alpha} s_2^{-1}s_1^{-1} s_3   s_2 $ for $\alpha \in \{ -1,1 \}$. Since
$u_2   s_3 (s_1^{\alpha} s_2^{-1}s_1^{-1}) s_3   s_2 = u_2   s_3 s_2^{-1} s_1^{-1}(s_2^{\alpha} s_3   s_2)
= u_2   s_3 s_2^{-1} s_1^{-1}s_3 s_2   s_3^{\alpha} \subset A_4^{[2]}$ by lemma \ref{lemB1}, and this
proves (1). To get (2) from (1), we use $s_2^{-1} p s_2^{-1} p s_2^{-1} . s_3^{-1} \in u_3^{\times} s_2 p^{-1} s_2 p^{-1} s_2^{-1} + A_4^{[2]}$,
that we get from (1) by applying $\Ad \Delta$, and the fact that $B$ is generated as a unital $R$-algebra by $s_1^{-1} $ and $s_3^{-1}$. This
proves (2), and then (3) follows from (2) by a direct application of $\Phi$. 
\end{proof}

From all this we deduce the following.
\begin{lemma} {\ }
\begin{enumerate}
\item $A_4^{[3]} = A_4^{[2]} + \sum_{\delta \in \{ -1,1 \}} B s^{\delta} p^{\delta} s^{-\delta} p^{\delta} s^{\delta} +  \sum_{\delta \in \{ -1,1 \}} B s^{-\delta} p^{\delta} s^{-\delta} p^{\delta} s^{-\delta} $
\item $A_4 = A_4^{[3]}$
\end{enumerate}
\end{lemma}
\begin{proof}
As a consequence of lemmas \ref{lemB2} and \ref{lemB3}, we get
$$
A_4^{[3]} = A_4^{[2]} + \sum_{\delta \in \{ -1,1 \}} B s^{\delta} p^{\delta} s^{-\delta} p^{\delta} s^{\delta}B +  \sum_{\delta \in \{ -1,1 \}} B s^{-\delta} p^{\delta} s^{-\delta} p^{\delta} s^{-\delta} B.
$$
We know  $s^{-\delta} p^{\delta} s^{-\delta} p^{\delta} s^{-\delta} B \subset A_4^{[2]} +  \sum_{\eps \in \{-1,1 \}} B s^{-\eps} p^{\eps}s^{-\eps} p^{\eps}s^{-\eps} $
by lemma \ref{lemB5} hence
$$
A_4^{[3]} = A_4^{[2]} +  \sum_{\delta \in \{ -1,1 \}} B s^{-\delta} p^{\delta} s^{-\delta} p^{\delta} s^{-\delta}  + \sum_{\delta \in \{ -1,1 \}} B s^{\delta} p^{\delta} s^{-\delta} p^{\delta} s^{\delta}B  
$$
and finally $s^{\delta} p^{\delta} s^{-\delta} p^{\delta} s^{\delta} \in R^{\times} p^{-\delta} \Delta^{2 \delta} + A_4^{[2]}$ by (*),
hence $s^{\delta} p^{\delta} s^{-\delta} p^{\delta} s^{\delta}B \subset p^{-\delta} \Delta^{2 \delta}B + A_4^{[2]}
=  p^{-\delta}B \Delta^{2 \delta} + A_4^{[2]}= B \Delta^{2 \delta} + A_4^{[2]} \subset Bs^{\delta} p^{\delta} s^{-\delta} p^{\delta} s^{\delta} + A_4^{[2]}$
and this concludes the proof of (1). Now $A_4^{[3]}$ is a $R$-submodule of $A_4$ which contains 1, which is stable under right-multiplication
by $s_1$ and $s_3$ by definition. Moreover, in view of (1), we have
$$
A_4^{[3]} s_2 \subset A_4^{[2]}s + \sum_{\delta \in \{ -1,1 \}} B s^{\delta} p^{\delta} s^{-\delta} p^{\delta} s^{\delta}s +  \sum_{\delta \in \{ -1,1 \}} B s^{-\delta} p^{\delta} s^{-\delta} p^{\delta} s^{-\delta}s \subset
A_4^{[3]}
$$
hence $A_4^{[3]}$ is also stable under right multiplication by $s_2$, hence it is a right-ideal of $A_4$ containing $1$, hence (2).
\end{proof}

We let $x_{\pm} = s^{\pm} p^{\pm} s^{\mp} p^{\pm} s^{\pm}$ and 
$y_{\pm} = s^{\pm} p^{\mp} s^{\pm} p^{\mp} s^{\pm}$.

\begin{lemma}{\ } \label{lemB7}
\begin{enumerate}
\item $sBsps \subset A_4^{[2]}$
\item $sBs^{-1}ps \subset R sps^{-1} ps A_4^{[2]}$
\end{enumerate}
\end{lemma}
\begin{proof}
The $R$-module $sBsps$ is spanned by $s^2 ps \in A_4^{[2]}$, the $s s_i sps \in A_4^{[2]}$ for $i \in \{1,3 \}$ by lemma \ref{lemB1},
$s(psps) = s(spsp)= s^2 psp \in A_4^{[2]}$, 
$s_2 s_1 (s_3^{-1} s_2 s_3) s_1 s_2
=s_2 s_1 s_2 s_3 (s_2^{-1} s_1 s_2 )
=s_2 s_1 s_2 s_3 s_1 s_2 s_1^{-1}  \in A_4^{[2]}$ by lemma \ref{lemB1},
the image of this latest one by $\Ad \Delta$, and by
$$s_2 s_1^{-1} s_3^{-1} s_2 s_3 s_1 s_2
=
s_2 s_1^{-1} (s_3^{-1} s_2 s_3) s_1 s_2
=
s_2 s_1^{-1} s_2 s_3 (s_2^{-1} s_1 s_2)
=
s_2 s_1^{-1} s_2 s_3 s_1 s_2 s_1^{-1} \in A_4^{[2]}$$ by lemma \ref{lemB1},
and this proves (1).

Now $sBs^{-1} p s$ is $R$-spanned by $sps^{-1}ps$ and
\begin{itemize}
\item the $s s^{-1} ps  = ps \in A_4^{[2]}$
\item the $s s_is^{-1} ps   \in A_4^{[2]}$ for $i \in \{1,3 \}$  by lemma \ref{lemB1}
\item $s_2 s_1 (s_3^{-1} s_2^{-1} s_3) s_1 s_2
= s_2 s_1 s_2 s_3^{-1} (s_2^{-1} s_1 s_2)
= s_2 s_1 s_2 s_3^{-1} s_1 s_2 s_1^{-1}
   \in A_4^{[2]}$ for $i \in \{1,3 \}$  by lemma \ref{lemB1}
\item $ \Delta. s_2 s_1 s_3^{-1} s_2^{-1} s_3 s_1 s_2 \Delta^{-1} \in A_4^{[2]}$ 
\item $s_2 s_1^{-1} (s_3^{-1} s_2^{-1} s_3) s_1 s_2
= s_2 s_1^{-1} s_2 s_3^{-1} (s_2^{-1} s_1 s_2)
= s_2 s_1^{-1} s_2 s_3^{-1} s_1 s_2 s_1^{-1}
   \in A_4^{[2]}$ for $i \in \{1,3 \}$  by lemma \ref{lemB1}
\end{itemize}
and this proves (2).
\end{proof}

We want to express $\Delta^3$ in terms of the $x_{\pm}$ and $y_{\pm}$.
We recall that $\Delta^2 \in R^{\times} p^2 sps^{-1} ps + R p^3 sps^2 + R p^2 sp^2 s$
hence $\Delta^3 \in R^{\times} p^2 sps^{-1} ps\Delta + R p^3 sps^2\Delta + R p^2 sp^2 s\Delta$.
We have
\begin{itemize}
\item $sps^{-1} ps\Delta \in R sp \Delta + R sps \Delta + R sps^{-1} \Delta$ and 
\begin{enumerate}
\item $sps^{-1} \Delta = sps^{-1} (spsp) = sp^2 sp \in A_4^{[2]}$,
\item $sps \Delta = sps (spsp) = sp s^2psp  \in  R^{\times} sps^{-1} psp + R spspsp + R sp^2sp$,
and we have $ (spsp)sp =  psps^2p \in A_4^{[2]}$, 
$sp^2sp \in A_4^{[2]}$, hence $sps \Delta \in  R^{\times} sps^{-1} psp + A_4^{[2]}$.
\item $sps^{-1} \Delta = sps^{-1} spsp = sp^2sp \in A_4^{[2]}$
\end{enumerate}
hence $sps^2 \Delta \in  R^{\times} sps^{-1} psp + A_4^{[2]}$.
\item $sp^2 s \Delta = sp^2 s^2 psp \in R^{\times} sp^2 s^{-1} psp + R sp^2 spsp + R sp^2 psp$, and
$sp^2 (spsp) = sp^2 psps = sp^3 sps$, $sp^2 psp \in A_4^{[2]}$.
\end{itemize}
It follows that
$$
\Delta^3 \in R^{\times} p^2 s p^2 s p^2 s + R p^3 sps^{-1} psp + R p^2 s p^2 s^{-1} psp +
R p^2 sp^3 sps + A_4^{[2]}
$$
From (*) we have $p^2 sps^{-1} ps \in \Delta^2 + A_4^{[2]}$ hence
$p^3 sps^{-1} psp \in p \Delta^2 p + A_4^{[2]} = p^2 \Delta^2 + A_4^{[2]}$
hence 
$p^3 sps^{-1} psp \in R^{\times} p^4. sps^{-1}ps + A_4^{[2]}$. By lemma \ref{lemB7},
we have $s p^2 s^{-1} p s \in R x_+ + A_4^{[2]}$ hence
$p^2 sp^2 s^{-1} psp \in R p^2 x_+ p + A_4^{[2]} \subset B x_+ + A_4^{[2]}$. Since $sp^3 sps \in A_4^{[2]}$
this leads to
$$
 \Delta^3 \in R^{\times} p^2 sp^2 s p^2 s + B x_+ + A_4^{[2]}.
$$
Since $s_i^2 = a s_i + b + c s_i^{-1}$ we have
$p^2 = s_1^2 s_3^2 = (a s_1 + b + c s_1^{-1})
(a s_3 + b + c s_3^{-1}) \in a^2 p + c^2 p^{-1} + W$
with $W$ the $R$-span of $s_1 s_3^{-1}, s_3 s_1^{-1}, s_1, s_3, s_1^{-1}, s_3^{-1}, 1$.
After easy applications of lemma \ref{lemB1} it follows that
$sp^2 sp^2 s \in c^4 sp^{-1} s p^{-1} s + R spsBs + R sBsps + A_4^{[2]}$. Since 
$spsBs + sBsps  \subset A_4^{[2]}$ by lemma \ref{lemB7} we get
$$
\Delta^3 \in c^4 p^2 y_+ + B x_+ + A_4^{[2]} 
$$
and
$$
\Delta^{-3} = \Phi(\Delta^3)  \in c^{-4} p^{-2} y_- + B x_+ + A_4^{[2]} 
$$
Now we have $\Delta^3 s_1 = s_3 \Delta^3 \in c^4 s_3 p^2 y_+ + B x_+ + A_4^{[2]}$
and $\Delta^3 s_1 \in c^4 p^2 y_+ s_1 + B x_+ B + A_4^{[2]}$, 
$\Delta^3 s_1 \in c^4 p^2 u_1^{\times} y_-  + B x_+ B + A_4^{[2]}$ by lemma \ref{lemB5} (1),
$\Delta^3 s_1 \in c^4 p^2 u_1^{\times} y_-  + B x_+  + A_4^{[2]}$ by
using $p^{-2} \Delta^2 \in R^{\times} x_+ + A_4^{[2]}$.
It follows that
$c^4 s_3 p^2 y_+ \in c^4 p^2 u_1^{\times} y_- + B x_+ + A_4^{[2]}$
hence
$$
\left\lbrace
\begin{array}{lcl} 
y_+ & \in & B y_- + B x_+ + A_4^{[2]} \\
y_- & \in & B y_+ + B x_+ + A_4^{[2]}
\end{array} \right.
$$
As a consequence we get the following.

\begin{proposition}
$$
A_4 = A_4^{[3]} = A_4^{[2]} + B x_+ + B x_- + B y_+
= A_4^{[2]} + B x_+ + B x_- + B y_-
$$
\end{proposition}

For $x \in A_4^{\times}$ , we let $[x]$ denote its class in
$B^{\times} \backslash A_4^{\times} / B^{\times}$, and
we write $x \sim y$ for $[x] = [y]$.

\begin{lemma} Let $E_2 = \{ s_2^{\alpha} s_1^{\beta} s_3^{\gamma} s_2^{\delta} \ | \ \alpha,\beta,\gamma,\delta \in \{ 0,1,-1 \} \} \subset A_4^{\times}$.
The image of $[E_2]$ of $E_2$ in
%and $\overline{E}_2$ its image in 
$B^{\times} \backslash A_4^{\times} / B^{\times}$ has cardinality at most 13, and
is equal to $\mathcal{S}_2$, with
%$|\overline{E}_2| \leq 13$, and
$$
\begin{array}{lcl}
\mathcal{S}_2  &=& \{ [1], [s_2], [s_2^{-1}], [s_2 s_1^{-1} s_2], [s_2 s_3^{-1} s_2], [s_2 s_1 s_3 s_2], [s_2 s_1^{-1} s_3 s_2],
[s_2 s_1^{-1} s_3^{-1} s_2], [s_2 s_1 s_3^{-1} s_2^{-1}], \\ & &  [s_2 s_1^{-1} s_3 s_2^{-1}],
[s_2 s_1 s_3^{-1} s_2^{-1}], [s_2^{-1} s_1 s_3 s_2^{-1}], [s_2^{-1} s_1^{-1} s_3^{-1} s_2^{-1}] \} \\
\end{array}
$$
\end{lemma}
\begin{proof} Clearly $\mathcal{S}_2 \subset [E_2]$, hence we only need to
prove $[E_2] \subset \mathcal{S}_2$.
In view of $s_2^{\alpha}s_i^{\alpha}s_2^{\alpha} = s_i^{\alpha}s_2^{\alpha}s_i^{\alpha}$,
$s_2^{-1} s_i^{\alpha} s_2 = s_i s_2^{\alpha} s_i^{-1}$,
$s_2 s_i^{\alpha} s_2^{-1} = s_i^{-1} s_2^{\alpha} s_i$ for $\alpha \in \{ -1,1 \}$ and $i \in \{1,3 \}$,
we have $[s_2^{\alpha} s_i^{\beta} s_2^{\gamma}] \in \mathcal{S}_2$ for all $\alpha,\beta,\gamma$.
Among the $s_2 s_1^{\alpha} s_3^{\beta} s_2$ for $\alpha,\beta \in \{ -1,1 \}$, we
have $[s_2 s_1^{\alpha} s_3^{\beta} s_2] \in \mathcal{S}_2$ because
$s_2 s_1^{\alpha} s_3^{\beta} s_2\sim s_2 s_1^{-1} s_3 s_2$ : indeed,
we have the identity
$s_2 s_1^{-1} s_3 s_2 = s_1^{-1} s_3 (s_2 s_1 s_3^{-1} s_2) s_1^{-1} s_3$ in
the braid group on $4$ strands (because $ s_1^{-1} s_3 s_2 s_1 s_3^{-1} s_2 s_1^{-1} s_3
=s_1^{-1} (s_3 s_2 s_3^{-1})s_1  s_2 s_1^{-1} s_3
=s_1^{-1} s_2^{-1} s_3 (s_2s_1  s_2) s_1^{-1} s_3
=s_1^{-1} s_2^{-1} s_3 s_1s_2  s_1 s_1^{-1} s_3
=s_1^{-1} s_2^{-1}  s_1(s_3s_2  s_3)
=(s_1^{-1} s_2^{-1}  s_1)s_2s_3  s_2
=s_2 s_1^{-1}  s_2^{-1}s_2s_3  s_2
=s_2 s_1^{-1}  s_3  s_2$).
Among the $s_2 s_1^{\alpha} s_3^{\beta} s_2^{-1}$ for $\alpha,\beta \in \{ -1,1 \}$,
we have $[s_2 s_1^{\alpha} s_3^{\beta} s_2^{-1}] \in \mathcal{S}_2$
because $s_2 s_1 s_3 s_2^{-1} \sim s_2 s_1 s_3^{-1} s_2$ : indeed,
we have $s_1(s_2 s_1 s_3 s_2^{-1}) s_3^{-1} = 
(s_1s_2 s_1) s_3 s_2^{-1} s_3^{-1} = 
s_2s_1 (s_2 s_3 s_2^{-1}) s_3^{-1} = 
s_2s_1 s_3^{-1} s_2 s_3 s_3^{-1} = 
s_2s_1 s_3^{-1} s_2$. 

Again for $\alpha,\beta \in \{-1,1 \}$, we have
$[s_2^{-1} s_1^{\alpha} s_3^{\beta} s_2] \in \mathcal{S}_2$ because of the following identities
\begin{enumerate}
\item $s_2^{-1} s_1^{-1} s_3 s_2 \sim s_2 s_1 s_3^{-1} s_2^{-1} $ 
\item $s_2^{-1} s_1 s_3 s_2 \sim s_2 s_1 s_3^{-1} s_2 $ 
\item $s_2^{-1} s_1 s_3^{-1} s_2 \sim s_2 s_1^{-1} s_3 s_2^{-1} $ 
\item $s_2^{-1} s_1^{-1} s_3^{-1} s_2 \sim s_2 s_1^{-1} s_3^{-1} s_2^{-1} $ 
\end{enumerate}
We prove these identities now. We have $s_3(s_2 s_3^{-1} s_1 s_2^{-1})s_1^{-1}
=(s_3s_2 s_3^{-1}) s_1 s_2^{-1}s_1^{-1}
=s_2^{-1}s_3 (s_2 s_1 s_2^{-1})s_1^{-1}
=s_2^{-1}s_3 s_1^{-1} s_2 s_1s_1^{-1}
=s_2^{-1}s_3 s_1^{-1} s_2$ hence $s_2 s_3^{-1} s_1 s_2^{-1} \sim s_2^{-1}s_3 s_1^{-1} s_2$
that is (1). By applying $\Ad \Delta$ this implies $s_2 s_1^{-1} s_3 s_2^{-1} \sim s_2^{-1}s_1 s_3^{-1} s_2$
that is (3). We have
$s_3^{-1}(s_2^{-1} s_1 s_3 s_2) s_1 = 
(s_3^{-1}s_2^{-1} s_3)s_1  s_2 s_1 = 
s_2s_3^{-1} (s_2^{-1}s_1  s_2) s_1 = 
s_2s_3^{-1} s_1s_2  s_1^{-1} s_1 = 
s_2s_3^{-1} s_1s_2 $ hence $s_2^{-1} s_1 s_3 s_2 \sim s_2s_3^{-1} s_1s_2$ that is (2).

We have $s_1(s_2^{-1} s_1^{-1} s_3 s_2^{-1})s_3^{-1} = 
(s_1s_2^{-1} s_1^{-1}) s_3 s_2^{-1}s_3^{-1} = 
s_2^{-1}s_1^{-1} (s_2  s_3 s_2^{-1})s_3^{-1} = 
s_2^{-1}s_1^{-1} s_3^{-1}  s_2 s_3s_3^{-1} = 
s_2^{-1}s_1^{-1} s_3^{-1}  s_2$ hence $s_2^{-1} s_1^{-1} s_3 s_2^{-1} \sim
s_2^{-1}s_1^{-1} s_3^{-1}  s_2$.
Moreover, we have $s_1(s_2 s_1^{-1} s_3^{-1} s_2^{-1}) s_3^{-1} 
=  s_1s_2 s_1^{-1} (s_3^{-1} s_2^{-1} s_3^{-1})
=  s_1(s_2 s_1^{-1} s_2^{-1}) s_3^{-1} s_2^{-1}
=  s_1s_1^{-1} s_2^{-1} s_1 s_3^{-1} s_2^{-1}
=  s_2^{-1} s_1 s_3^{-1} s_2^{-1}$ hence $s_2 s_1^{-1} s_3^{-1} s_2^{-1} \sim
s_2^{-1} s_1 s_3^{-1} s_2^{-1}$. Applying $\Delta$ we get
$s_2 s_1^{-1} s_3^{-1} s_2^{-1} \sim
s_2^{-1} s_3 s_1^{-1} s_2^{-1}$, hence
 $s_2 s_1^{-1} s_3^{-1} s_2^{-1} \sim
s_2^{-1} s_3 s_1^{-1} s_2^{-1} \sim 
s_2^{-1}s_1^{-1} s_3^{-1}  s_2$ hence (4).

Now, for $\alpha,\beta \in \{ -1,1 \}$, we have $[s_2^{-1} s_1^{\alpha} s_3^{\beta} s_2^{-1}] \in \mathcal{S}_2$
because $s_2^{-1} s_1 s_3^{-1} s_2^{-1} \sim s_2 s_1^{-1} s_3^{-1} s_2^{-1}$ 
and $s_2^{-1} s_1^{-1} s_3 s_2^{-1} \sim s_2 s_1^{-1} s_3^{-1} s_2^{-1}$ as we proved above,
and this concludes the proof.

\end{proof}

From this we get
$$
\begin{array}{lcl}
A_4 &=& \sum_{\sigma \in [E_2] }B \sigma B + B x_+ + B x_- + B y_- \\
&=& \sum_{\sigma \in \mathcal{S}_2} B \sigma B + B x_+ + B x_- + B y_- \\
\end{array}
$$

We write $\mathcal{S}_2 = \mathcal{S}_2^1 \cup \mathcal{S}_2^{\Delta} \cup \mathcal{S}_2^{\alpha} \cup \mathcal{S}_2^{\beta} \cup \mathcal{S}_2^0$
with
$$
\begin{array}{lcl}
\mathcal{S}_2^1&=& \{ [1],[s_2],[s_2^{-1}] \} \\
 \mathcal{S}_2^{\Delta} &=&\{ [s_2 s_1 s_3 s_2],[s_2^{-1} s_1^{-1} s_3^{-1} s_2^{-1}] \} \\ 
 \mathcal{S}_2^{\alpha} &=& \{ [s_2 s_1^{-1} s_2], [s_2 s_3^{-1} s_2 ] \} \\ 
 \mathcal{S}_2^{\beta}&=& \{ [s_2 s_1 s_3^{-1} s_2^{-1} ], [s_2 s_1^{-1} s_3 s_2^{-1} ] \} \\
  \mathcal{S}_2^0&=& \{ [s_2 s_1^{-1} s_3 s_2], [s_2 s_1^{-1} s_3^{-1} s_2], [s_2 s_1^{-1} s_3^{-1} s_2^{-1} ], [s_2^{-1} s_1 s_3 s_2^{-1}] \} \\
\end{array}
$$

Recall that $B = u_1 u_3 = u_3 u_1$, with $u_i$ the unital subalgebra generated by $s_i$.
We prove the following.
\begin{lemma} {\ } \label{lemB9}
\begin{enumerate}
\item $s_2 s_1 s_3 s_2 B \subset B s_2 s_1 s_3 s_2$, $s_2^{-1} s_1^{-1} s_3^{-1} s_2^{-1} B \subset B s_2^{-1} s_1^{-1} s_3^{-1} s_2^{-1}$
\item $s_2 s_1^{-1} s_2 B \subset B s_2 s_1^{-1} s_2 u_3 + A_4^{[1]}$, $s_2 s_3^{-1} s_2 B \subset B s_2 s_3^{-1} s_2 u_1+ A_4^{[1]}$
\item $s_2 s_1 s_3^{-1} s_2^{-1} B \subset B s_2 s_1 s_3^{-1} s_2^{-1} u_1$, $s_2 s_1^{-1} s_3 s_2^{-1} B \subset B s_2 s_1^{-1} s_3 s_2^{-1} u_3$
\end{enumerate}
\end{lemma}
\begin{proof} We have $\Delta = s_1 s_3 (s_2 s_1 s_3 s_2) = (s_2 s_1 s_3 s_2) s_1 s_3$ and $\Delta B = B \Delta$
hence $s_2 s_1 s_3 s_2 B = (s_1 s_3)^{-1} \Delta B =  B (s_1 s_3)^{-1} \Delta= B (s_2 s_1 s_3 s_2)$. Applying
$\Phi$ (or considering $\Delta^{-1}$) we get $s_2^{-1} s_1^{-1} s_3^{-1} s_2^{-1} B = B s_2^{-1} s_1^{-1} s_3^{-1} s_2^{-1}$
hence (1).

By lemma \ref{lemdecomp1212} we have $(s_2 s_1^{-1} s_2) s_1^{-1} \in s_1^{-1}(s_2 s_1^{-1} s_2) + A_4^{[1]}$ hence
$(s_2 s_1^{-1} s_2) u_1 \subset u_1 (s_2 s_1^{-1} s_2) + A_4^{[1]}$ and
$(s_2 s_1^{-1} s_2) u_1 \subset B (s_2 s_1^{-1} s_2) + A_4^{[1]}$.
Since $B = u_1 u_3$ this yields
$(s_2 s_1^{-1} s_2) B = (s_2 s_1^{-1} s_2) u_1 u_3 \subset
B (s_2 s_1^{-1} s_2) u_3 + A_4^{[1]}$.
Using $\Ad \Delta$ this implies $(s_2 s_3^{-1} s_2) B \subset B (s_2 s_3^{-1} s_2) u_1 + A_4^{[1]}$,
hence (2). Finally, $(s_2 s_1 s_3^{-1} s_2^{-1})s_3 =
s_2 s_1 (s_3^{-1} s_2^{-1}s_3) = 
(s_2 s_1 s_2) s_3^{-1}s_2^{-1} = 
s_1 (s_2 s_1 s_3^{-1}s_2^{-1})$ hence
$(s_2 s_1 s_3^{-1} s_2^{-1}) u_3 \subset B (s_2 s_1 s_3^{-1} s_2^{-1})$
whence, using $B = u_3 u_1$, $(s_2 s_1 s_3^{-1} s_2^{-1}) B \subset B(s_2 s_1 s_3^{-1} s_2^{-1}) u_1$
and, applying $\Ad \Delta$, $(s_2 s_1^{-1} s_3 s_2^{-1})B \subset B(s_2 s_1^{-1}s_3 s_2^{-1}) u_3$, which proves (3). 
\end{proof}

\begin{proposition} {\ } \label{propA4Bmodule}
\begin{enumerate}
\item $A_4^{[1]} = B + B s_2 B + Bs_2^{-1} B$ is equal to
$$
B + \sum_{a,b \in \{0,1,-1\}} B s_2 s_1^a s_3^b +  \sum_{a,b \in \{0,1,-1\}} B s_2^{-1} s_1^a s_3^b 
$$
\item $A_4^{[2]} = B u_2 A_4^{[1]} = A_4^{[1]} u_2 B$ is equal to
$$
A_4^{[1]} + \sum_{x \in \mathcal{S}_2^{\Delta}} B x + \sum_{a \in \{0,1,-1 \}} B s_2 s_1^{-1} s_2 s_3^a 
+\sum_{a \in \{0,1,-1 \}} B s_2 s_3^{-1} s_2 s_1^a 
+ \sum_{a \in \{0,1,-1 \}} B s_2 s_1s_3^{-1} s_2^{-1} s_1^a $$ {} $$
+\sum_{a \in \{0,1,-1 \}} B s_2 s_1^{-1}s_3 s_2^{-1} s_3^a 
+ \sum_{\stackrel{x \in \mathcal{S}_2^0}{a,b \in \{0,1,-1\}}} B x s_1^a s_3^b 
$$
\item $A_4 = A_4^{[3]} = A_4^{[2]} + B x_+ + B x_- + B y_-$
\end{enumerate}

\end{proposition}
\begin{proof}
(1) is clear, (3) has been proved before, and (2) is an immediate consequence of $A_4^{[2]} = A_4^{[1]} +
\sum_{x \in \mathcal{S}_2} B xB$ and of lemma \ref{lemB9}.

\end{proof}
\begin{corollary} \label{corA4B72}
As a $B$-module, $A_4$ is generated by $72$ elements, which are images of elements
of the braid group on $4$ strands.
\end{corollary}
\begin{proof}
By proposition \ref{propA4Bmodule}, $A_4^{[1]}$ is generated by $1+9+9 = 19$
elements, $A_4^{[2]}$ by $A_4^{[1]}$ and $|\mathcal{S}_2^{\Delta}| + 4 \times 3 + 9 \times |\mathcal{S}_2^0| \leq 2 + 12 + 9 \times 4 = 50$
elements, and $A_4^{[3]}$ by $A_4^{[2]}$ and $3$ elements. Thus $A_4 = A_4^{[3]}$
is generated by $72$ elements, all originating from the braid group. 
\end{proof}

\section{The algebra $A_5$}
\label{sectA5}

Recall $w^+ = s_3 s_2^{-1} s_1 s_2^{-1} s_3$,
$w^- = s_3^{-1} s_2 s_1^{-1} s_2 s_3^{-1} \in A_4$. Our first goal in this section is to prove
the following theorem.

\begin{theorem} \label{theodecA5}
$$
\begin{array}{lcl}
A_5 &=& A_4 + A_4 s_4 A_4 + A_4 s_4^{-1} A_4 + A_4 s_4 s_3^{-1} s_4 A_4 +
A_4 s_4^{-1} s_3 s_2^{-1} s_3 s_4^{-1} A_4 +
A_4 s_4 s_3^{-1} s_2 s_3^{-1} s_4 A_4 \\
& & + 
A_4 s_4^{-1} w^+ s_4^{-1} A_4 +
A_4 s_4 w^-s_4 A_4 + A_4 s_4^{-1} w^- s_4^{-1} A_4
+ A_4 s_4 w^+ s_4 A_4 + A_4 s_4 w^- s_4 w^- s_4 A_4\\
& &  +
A_4 s_4 w^+ s_4^{-1} w^+ s_4 A_4 +
A_4 s_4^{-1} w^- s_4 w^- s_4^{-1} A_4
\end{array}
$$
\end{theorem}

We denote again $U$ the right-hand side. We let $A_5^{(0)} = A_4$ and
$A_5^{(n+1)} = A_5^{(n)} u_4 A_4$. This defines an increasing sequence
of $A_4$ sub-bimodules of $A_5$. An immediate consequence
of theorem \ref{theodecA4} is $sh(A_4) \subset U$. Also, we have $u_4 \subset U$
hence $A_5^{(1)} = A_4 u_4 A_4 \subset U$.

\begin{lemma} \label{lemuAuDansU} $u_4 A_4 u_4 \subset U$, hence $A_5^{(2)} = A_4 u_4 A_4 u_4 A_4 \subset U$.
\end{lemma}
\begin{proof}
According to theorem \ref{theodecA4}, we have $A_4 = A_3 + A_3 s_3 A_3+
A_3 s_3^{-1} A_3 + A_3 s_3 s_2^{-1} s_3 A_3 + A_3 w^- + A_3 w^+$,
hence $u_4 A_4 u_4 \subset A_3 u_4 + A_4 u_4 u_3 u_4 A_4 + 
A_4 u_4 s_3 s_2^{-1} s_3 u_4 A_3 + A_3 u_4 w^- u_4 + A_3 u_4 w^+ u_4.$
We have $A_3 u_4 + A_4 u_4 u_3 u_4 A_4 + 
A_4 u_4 s_3 s_2^{-1} s_3 u_4 A_3 \subset A_4 sh(A_4) A_4 \subset A_4 U A_4 \subset U$.
Moreover, since by definition $s_4^{\alpha} w^{\beta} s_4^{\alpha} \in U$
for all $\alpha,\beta \in \{ -1,1 \}$,
we have $s_4 A_4 s_4 \subset U$, $s_4^{-1} A_4 s_4^{-1} \subset U$,
and we only need to prove $s_4 w^{\pm} s_4^{-1} \in U$
and $s_4^{-1} w^{\pm} s_4 \in U$. We have $w^{\pm} \in s_3^{\alpha} A_3 s_3^{\alpha}$ for some $\alpha \in \{ -1, 1 \}$,
hence $s_4^{\alpha} w^{\pm} s_4^{-\alpha} \in s_4^{\alpha} s_3^{\alpha} A_3 s_3^{\alpha} s_4^{-\alpha} 
= s_3^{-\alpha} (s_3^{\alpha} s_4^{\alpha} s_3^{\alpha}) A_3 s_3^{\alpha} s_4^{-\alpha} 
= s_3^{-\alpha} s_4^{\alpha} s_3^{\alpha} s_4^{\alpha} A_3 s_3^{\alpha} s_4^{-\alpha} 
\subset A_4 s_4^{\alpha} s_3^{\alpha}  A_3 (s_4^{\alpha} s_3^{\alpha} s_4^{-\alpha}) 
\subset A_4 s_4^{\alpha} s_3^{\alpha}  A_3 s_3^{-\alpha} s_4^{\alpha} s_3^{\alpha} $
by lemma \ref{lemsplusmoins}. Now 
$$A_4 s_4^{\alpha} s_3^{\alpha}  A_3 s_3^{-\alpha} s_4^{\alpha} s_3^{\alpha}  \subset
A_4 (s_4^{\alpha} A_4 s_4^{\alpha}) A_4 \subset A_4 U A_4 \subset U,$$
as we already proved. 
\end{proof}

\subsection{The $A_4$-bimodule $A_5^{(3)}/A_5^{(2)}$ : first reduction.}

\begin{proposition}\label{propmoinsde55} If $p \leq 5$, $q \leq 5$ and $(p,q) \neq (5,5)$, then
for all $x \in u_4 u_{i_1}\dots u_{i_p} u_4 u_{j_1}\dots u_{j_q} u_4$ we have $x \in  A_4 u_4 A_4 u_4 A_4$,
for all choices of $i_1,\dots,i_p,j_1,\dots,j_q \in \{ 1, 2,3 \}$,
unless $(p,q) \in \{ (5,4), (4,5) \}$ and $x \in s_4 u_3 u_2 u_1 u_3 u_2 s_4 u_1 u_3 u_2 u_3 s_4
\cup s_4^{-1} u_3 u_2 u_1 u_3 u_2 s_4^{-1} u_1 u_3 u_2 u_3 s_4^{-1}$.
\end{proposition}
\begin{proof}
Note that $sh(A_4) \subset  A_4 u_4 A_4 u_4 A_4$.
By application of $\Psi$ %the anti-automorphism of $A_5$ induced by $s_i \mapsto s_i^{-1}$,
we may assume $p \geq q \geq 1$. We prove the statement by induction on $(p,q)$,
using lexicographic ordering. By commutation relations we can assume
$i_1 \not\in \{1,2 \}$ hence $i_1 = 3$, and similarly $j_{q} = 3$. In case $(p,q) = (1,1)$
we have then $u_4 u_3 u_4 u_3 u_4 \subset sh(A_4) \subset A_4 u_4 A_4 u_4 A_4$.
More generally, in the cases $(1,1)$, $(2,1)$, $(2,2)$, $(3,1)$, using only
commutation relations we check that the corresponding
algebras are necessarily included in $A_4 sh(A_4) A_4 \subset A_4 u_4 A_4 u_4 A_4$.

\begin{comment}
In case $(p,q) = (1,1)$
we have then $u_4 u_3 u_4 u_3 u_4 \subset sh(A_4) \subset A_4 u_4 A_4 u_4 A_4$.
In case $(p,q) = (2,1)$ then only nontrivial case is
$u_4 u_3 u_2 u_4 u_3 u_4 \subset sh(A_4) \subset A_4 u_4 A_4 u_4 A_4$.
By commutation relations, the case $(p,q) = (2,2)$ can be reduced to $(2,1)$.
\end{comment}

If $(p,q) = (3,2)$, the only case which is not clearly included in $A_4 sh(A_4) A_4$ is
$u_4 u_3 u_2 u_1 u_4 u_2 u_3 u_4 = u_4 u_3 u_4 u_2 u_1  u_2 u_3 u_4$,
and we have $u_4 u_3 u_4 \subset u_3 s_4 s_3^{-1} s_4 + u_3 u_4 u_3$
by theorem \ref{theodecA3}
hence $u_4 u_3 u_4 u_2 u_1  u_2 u_3 u_4 \subset 
u_3 s_4 s_3^{-1} s_4 u_2 u_1  u_2 u_3 u_4 + u_3 u_4 u_3 u_2 u_1  u_2 u_3 u_4
\subset u_3 s_4 s_3^{-1}  u_2 u_1  u_2 s_4 u_3 u_4 +A_4 u_4 A_4 u_4 A_4$.
Again $s_4 u_3 u_4 \subset s_4^{-1} s_3 s_4^{-1} u_3 + u_3 u_4 u_3$
by theorem \ref{theodecA3} hence
$u_3 s_4 s_3^{-1}  u_2 u_1  u_2 (s_4 u_3 u_4)
\subset 
u_3 s_4 s_3^{-1}  u_2 u_1  u_2 s_4^{-1} s_3 s_4^{-1} + A_4 u_4 A_4 u_4 A_4$
and 
$u_3 s_4 s_3^{-1}  u_2 u_1  u_2 s_4^{-1} s_3 s_4^{-1}
= u_3 (s_4 s_3^{-1}s_4^{-1})  u_2 u_1  u_2  s_3 s_4^{-1} = 
u_3 s_3^{-1} s_4^{-1}s_3  u_2 u_1  u_2  s_3 s_4^{-1}  \subset A_4 u_4 A_4 u_4 A_4$.

If $(p,q) = (3,3)$, the corresponding algebra is either included in $A_4 sh(A_4) A_4 \subset U$,
or can be reduced using commutation relations to the case $(3,2)$, or we
are dealing with the remaining case $u_4 u_3 u_2 u_1 u_4 u_3 u_2 u_3 u_4$
(or its image under the natural anti-isomorphism $u_4 u_3 u_2 u_3 u_4 u_1 u_2 u_3 u_4$).
We want to prove $u_4 u_3 u_2 u_1 u_4 u_3 u_2 u_3 u_4  \subset A_4 u_4 A_4 u_4 A_4$.
Up to using the natural isomorphism induced by $s_i \mapsto s_i^{-1}$, we only
need to prove $u_4 u_3 u_2 u_1 u_4 u_3 u_2 u_3 s_4  \subset A_4 u_4 A_4 u_4 A_4$.
We have $u_4 u_3 u_2 u_1 u_4 u_3 u_2 u_3 s_4 = 
 u_4 u_3 u_4 u_2 u_1  u_3 u_2 u_3 s_4$ and we know from theorem \ref{theodecA3} that
 $u_4 u_3 u_4 \subset A_4 s_4^{-1} s_3 s_4^{-1} + u_3 u_4 u_3$,
 hence $u_4 u_3 u_4 u_2 u_1  u_3 u_2 u_3 s_4
 \subset
 A_4s_4^{-1} s_3 s_4^{-1} u_2 u_1  u_3 u_2 u_3 s_4  +
 A_4 u_4 A_4 u_4 A_4$.
 Now, using $u_3 u_2 u_3 \subset s_3 s_2^{-1} s_3 u_2 + u_2 u_3 u_3$
 we have $$s_4^{-1} s_3 s_4^{-1} u_2 u_1  u_3 u_2 u_3 s_4 \subset
 s_4^{-1} s_3 s_4^{-1} u_2 u_1   s_3 s_2^{-1} s_3 u_2 s_4 + s_4^{-1} s_3 s_4^{-1} u_2 u_1  u_2 u_3 u_2 s_4.$$
 But  $s_4^{-1} s_3 s_4^{-1} u_2 u_1  u_2 u_3 u_2 s_4 = s_4^{-1} s_3 u_2 u_1 s_4^{-1}   u_2 u_3 s_4 u_2 \subset
 A_4 u_4 A_4 u_4 A_4$ by the induction assumption, hence  
 $s_4^{-1} s_3 s_4^{-1} u_2 u_1  u_3 u_2 u_3 s_4 \subset s_4^{-1} s_3 s_4^{-1} u_2 u_1   s_3 s_2^{-1} s_3 s_4u_2  + 
 A_4 u_4 A_4 u_4 A_4$. Now we need to prove 
 $s_4^{-1} s_3 s_4^{-1} s_2^{\alpha} s_1^{\beta}   s_3 s_2^{-1} s_3 s_4 \in A_4 u_4 A_4 u_4 A_4$ for $\alpha,\beta \in \{ -1,1 \}$.
 If $\alpha = 1$, this holds true because
$$
\begin{array}{lcll}
 s_4^{-1} s_3 s_4^{-1} s_2 u_1   s_3 s_2^{-1} s_3 s_4 
 &=&   s_4^{-1} s_3 s_4^{-1} s_2 u_1   s_3 s_2^{-1} (s_3 s_4 s_3) s_3^{-1} \\
 &=&   s_4^{-1} s_3 s_4^{-1} s_2 u_1   s_3 s_2^{-1} s_4 s_3 s_4 s_3^{-1} \\
 &=&   s_4^{-1} s_3  s_2 u_1 (s_4^{-1}  s_3  s_4) s_2^{-1} s_3 s_4 s_3^{-1} \\
 &=&   s_4^{-1} s_3  s_2 u_1 s_3  s_4  s_3^{-1} s_2^{-1} s_3 s_4 s_3^{-1} \\
 &=&   s_4^{-1} (s_3  s_2 s_3) u_1   s_4  (s_3^{-1} s_2^{-1} s_3) s_4 s_3^{-1} \\
 &=&   s_4^{-1} s_2  s_3 s_2 u_1   s_4  s_2 s_3^{-1} s_2^{-1} s_4 s_3^{-1} \\
 &=&   s_2 s_4^{-1}   s_3 s_2 u_1   s_4 s_2 s_3^{-1}  s_4 s_2^{-1} s_3^{-1} \\
 &\subset & A_4 u_4 A_4 u_4 A_4 \\
 \end{array}$$
 by the induction assuption. We thus assume $\alpha = -1$. 
If $\beta = 1$,
 then 
$$
\begin{array}{lcll}
 s_4^{-1} s_3 s_4^{-1} s_2^{-1} s_1  ( s_3 s_2^{-1} s_3) s_4
 &\subset &   s_4^{-1} s_3 s_4^{-1} s_2^{-1} s_1   s_3^{-1} s_2 s_3^{-1} s_4u_2  & + A_4 u_4 A_4 u_4 A_4 \ \mbox{(lemmas \ref{leminverse} + \ref{lemquasicom})} \\
 &\subset &  s_1 s_1^{-1} s_4^{-1} s_3 s_4^{-1} s_2^{-1} s_1   s_3^{-1} s_2 s_3^{-1} s_4u_2  & + A_4 u_4 A_4 u_4 A_4  \\
 &\subset &  s_1  s_4^{-1} s_3 s_4^{-1}(s_1^{-1} s_2^{-1} s_1)   s_3^{-1} s_2 s_3^{-1} s_4u_2  & + A_4 u_4 A_4 u_4 A_4  \\
 &\subset &  s_1  s_4^{-1} s_3 s_4^{-1}s_2 s_1^{-1} s_2^{-1}   s_3^{-1} s_2 s_3^{-1} s_4u_2  & + A_4 u_4 A_4 u_4 A_4  \\
 &\subset &  s_1  s_4^{-1} s_3 s_4^{-1}s_2 s_1^{-1} sh(A_3) s_4u_2  & + A_4 u_4 A_4 u_4 A_4  \\
 &\subset &  s_1  s_4^{-1} s_3 s_4^{-1}s_2 s_1^{-1} s_3 s_2^{-1} s_3 u_2 s_4u_2  & + A_4 u_4 A_4 u_4 A_4  \\
 \end{array}$$
 by theorem \ref{theodecA3} and the induction assumtion, and we already proved
 $s_4^{-1} s_3 s_4^{-1}s_2 s_1^{-1} s_3 s_2^{-1} s_3 u_2 s_4 = s_4^{-1} s_3 s_4^{-1}s_2 s_1^{-1} s_3 s_2^{-1} s_3  s_4 u_2 \subset  A_4 u_4 A_4 u_4 A_4  $.
 The remaining case is then $(\alpha,\beta) = (-1,-1)$,
 for which we have
$$
\begin{array}{lcll}
 s_4^{-1} s_3 s_4^{-1} s_2^{-1}s_1^{-1}  ( s_3 s_2^{-1} s_3) s_4
% &\subset &  (s_4^{-1} s_3 s_4^{-1}) s_2^{-1} s_1^{-1}  s_3^{-1} s_2 s_3^{-1} s_4 A_4  & + A_4 u_4 A_4 u_4 A_4  \\
%&\subset &  s_3 (s_3^{-1} s_4^{-1} s_3 s_4^{-1}) s_2^{-1} u_1  s_3^{-1} s_2 s_3^{-1} s_4 A_4  & + A_4 u_4 A_4 u_4 A_4  \\
%&\subset &  u_3 s_4 s_3^{-1} s_4 s_2^{-1} s_1^{-1}  s_3^{-1} s_2 s_3^{-1} s_4 A_4  & + A_4 u_4 A_4 u_4 A_4  \mbox{lemma \ref{leminverse})}\\
&\subset &    s_4^{-1} s_3s_4^{-1}  s_2^{-1}  s_1^{-1} s_3^{-1}  s_2 s_3^{-1} s_4 A_4  & + A_4 u_4 A_4 u_4 A_4 \  \mbox{(lemma \ref{leminverse})} \\
&\subset &  s_3 s_3^{-1} ( s_4^{-1} s_3s_4^{-1} ) s_2^{-1} s_3^{-1} s_1^{-1}   s_2 s_3^{-1} s_4 A_4  & + A_4 u_4 A_4 u_4 A_4  \\
&\subset &  s_3 (s_4^{-1} s_3s_4^{-1} )s_3^{-1}  s_2^{-1} s_3^{-1} s_1^{-1}  s_2 s_3^{-1} s_4 A_4  & + A_4 u_4 A_4 u_4 A_4 \  \mbox{(lemma \ref{lemquasicom})} \\
&\subset &  A_4  s_4^{-1} s_3s_4^{-1}(s_3^{-1}  s_2^{-1} s_3^{-1}) s_1^{-1}  s_2 s_3^{-1} s_4 A_4  & + A_4 u_4 A_4 u_4 A_4   \\
&\subset &  A_4  s_4^{-1} s_3s_4^{-1} s_2^{-1}  s_3^{-1} (s_2^{-1} s_1^{-1}  s_2) s_3^{-1} s_4 A_4  & + A_4 u_4 A_4 u_4 A_4   \\
&\subset &  A_4  s_4^{-1} s_3s_4^{-1}s_2^{-1}  s_3^{-1} s_1 s_2^{-1}  s_1^{-1} s_3^{-1} s_4 A_4  & + A_4 u_4 A_4 u_4 A_4   \\
&\subset &  A_4  s_4^{-1} s_3s_4^{-1} s_2^{-1}  s_1(s_3^{-1}  s_2^{-1}   s_3^{-1}) s_4 s_1^{-1} A_4  & + A_4 u_4 A_4 u_4 A_4   \\
&\subset &  A_4  s_4^{-1} s_3s_4^{-1}s_2^{-1}  s_1s_2^{-1}  s_3^{-1}   s_2^{-1} s_4 s_1^{-1} A_4  & + A_4 u_4 A_4 u_4 A_4   \\
&\subset &  A_4 s_4^{-1} s_3s_4^{-1}s_2^{-1}  s_1s_2^{-1}  s_3^{-1}    s_4 s_2^{-1} A_4  & + A_4 u_4 A_4 u_4 A_4   \\
&\subset &  A_4 s_4^{-1} s_3s_2^{-1} s_1 s_4^{-1}  s_2^{-1}  s_3^{-1}    s_4  A_4  & + A_4 u_4 A_4 u_4 A_4   \\
&\subset & A_4 u_4 A_4 u_4 A_4 & \mbox{(induction assumption)} \\ \end{array}$$
and this concludes the case $(p,q) = (3,3)$.

All cases $(4,q)$ for $q = 1,2,4$ can be easily reduced to smaller cases by using commutation relations and relations
$u_i u_{j} u_i u_j = u_j u_i u_j u_i$. Most cases for $(4,3)$ can also be
reduced this way, except for one remaining case $u_4 u_3 u_2 u_3 u_1 u_4 u_3 u_2 u_3 u_4$. Using $\Phi$,
%the automorphism
%induced by $s_i \mapsto s_i^{-1}$, 
we only need to prove $u_4 u_3 u_2 u_3 u_1 s_4 u_3 u_2 u_3 u_4 \subset A_4 u_4 A_4 u_4 A_4$.
Using the induction assumption and theorem \ref{theodecA3} on $sh(A_3)$, we get
$$
\begin{array}{lcll}
u_4 (u_3 u_2 u_3) u_1 s_4 (u_3 u_2 u_3) u_4 
& \subset & u_4 u_2 s_3 s_2^{-1} s_3 u_1 s_4  s_3 s_2^{-1} s_3 u_2  u_4 & + A_4 u_4 A_4 u_4 A_4 \\
& \subset & A_3 u_4 s_3 s_2^{-1}  u_1 (s_3s_4  s_3) s_2^{-1} s_3  u_4A_3  & + A_4 u_4 A_4 u_4 A_4 \\
& \subset & A_3u_4 s_3 s_2^{-1}  u_1 s_4s_3  s_4 s_2^{-1} s_3  u_4 A_3& + A_4 u_4 A_4 u_4 A_4 \\
& \subset & A_3(u_4 s_3 s_4) s_2^{-1}  u_1 s_3  s_2^{-1}  (s_4 s_3  u_4)A_3 & + A_4 u_4 A_4 u_4 A_4 \\
& \subset & A_3u_3 u_4 u_3 s_2^{-1}  u_1 s_3  s_2^{-1}  u_3 u_4  u_3A_3 & + A_4 u_4 A_4 u_4 A_4 \\
\end{array} $$
by lemma \ref{lemsplusmoins}, which proves the claim.

We now deal with the cases $(5,q)$ with $1 \leq q < 5$. We can
assume that $u_{i_1}\dots u_{i_p} = u_3 u_2 u_1 u_2 u_3$ or
$u_{i_1}\dots u_{i_p} = u_3 u_2 u_1 u_3 u_2$, because otherwise
we can reduce to smaller cases by using commutation relations and the relation
$u_a u_b u_a u_b = u_b u_a u_b u_a $. From this remark one easily
checks that the cases $(5,1)$ are readily reduced to smaller cases, and also the cases
$(5,2)$ except for the case $u_4 u_3 u_2 u_1 u_2 u_3  u_4 u_2 u_3 u_4 = u_4 u_3 u_2 u_1 u_2 u_3  u_2 u_4  u_3 u_4$ that we
tackle now :
we have $u_3 u_2 u_1 u_2 u_3  u_2 \subset A_4 = A_3 u_3 A_3 + A_3 u_3u_2u_3 A_3 + A_3 u_3u_2u_1u_2u_3$
by theorem \ref{theodecA4},
hence 
$$
\begin{array}{lcl}u_4 u_3 u_2 u_1 u_2 u_3  u_2 u_4  u_3 u_4 &\subset& 
u_4 A_3 u_3 A_3 u_4  u_3 u_4 + 
u_4A_3 u_3u_2u_3 A_3  u_4  u_3 u_4 +
u_4 A_3 u_3u_2u_1u_2u_3 u_4  u_3 u_4 \\
&\subset& A_3 u_4  u_3  u_4 A_3 u_3 u_4 + 
A_3 u_4 u_3u_2u_3   u_4 A_3 u_3 u_4 +
A_3 u_4  u_3u_2u_1u_2u_3 u_4  u_3 u_4 \\
&\subset& A_3 u_4  u_3  u_4 u_2 u_1 u_2 u_1 u_3 u_4 + 
A_3 u_4 u_3u_2u_3   u_4 u_2 u_1 u_2 u_1 u_3 u_4 \\ & & + 
A_3 u_4  u_3u_2u_1u_2u_3 u_4  u_3 u_4 \\
&\subset& A_3 u_4  u_3  u_4 u_2 u_1 u_2  u_3 u_4 A_2 + 
A_3 u_4 u_3u_2u_3   u_4 u_2 u_1 u_2  u_3 u_4 A_2 \\ & & +
A_3 u_4  u_3u_2u_1u_2u_3 u_4  u_3 u_4
\end{array}
$$
using $A_3 = u_2 u_1 u_2 u_1$, and we are thus reduced to smaller cases.
%, since $A_3 \subset u_2 u_1 u_2 + u_1 u_2 u_1$.

When $(p,q) = (5,3)$, the only nontrivial case (up to commutation and $u_au_bu_au_b = u_b u_a u_b u_a$ relations)
is $u_4 u_3 u_2 u_1 u_2 u_3 u_4 u_3 u_2 u_3 u_4$. We have $u_2 u_3 u_4 u_3 u_2 u_3 u_4 \subset sh(A_4)
\subset A_4 u_4 A_4 + sh(A_3) u_4 u_3 u_4 A_4 + u_4 u_3 u_2 u_3 u_4 A_4$ by theorem \ref{theodecA4},
hence 
$$
u_4 u_3 u_2 u_1 u_2 u_3 u_4 u_3 u_2 u_3 u_4 \subset A_4 u_4 A_4 u_4 A_4 + 
u_4 u_3 u_2 u_1 sh(A_3) u_4 u_3 u_4 A_4 +
u_4 u_3 u_2 u_1 u_4 u_3 u_2 u_3 u_4 A_4
$$
and we have $u_4 u_3 u_2 u_1 u_4 u_3 u_2 u_3 u_4  \subset A_4 u_4 A_4 u_4 A_4$
by the induction assumption,
and, since $sh(A_3)= u_2 u_3 u_2 u_3$ by theorem \ref{theodecA3},
$$
\begin{array}{lclcl}
u_4 u_3 u_2 u_1 sh(A_3) u_4 u_3 u_4& \subset&
u_4 u_3 u_2 u_1 u_2 u_3 u_2 (u_3 u_4 u_3 u_4)& =&
u_4 u_3 u_2 u_1 u_2 u_3 u_2 u_4 u_3 u_4 u_3 \\ & & &=&
u_4 u_3 u_2 u_1 u_2 u_3  u_4u_2 u_3 u_4 u_3 \\
\end{array}
$$
and we are reduced to case $(5,2)$.
%and, since $sh(A_3) \subset u_2 u_3 u_2 + u_3 u_2 u_3$ by theorem \ref{theodecA3},
\begin{comment}
$$u_4 u_3 u_2 u_1 sh(A_3) u_4 u_3 u_4 \subset
u_4 u_3 u_2 u_1 u_3 u_2 u_3 u_4 u_3 u_4 +
u_4 u_3 u_2 u_1 u_2 u_3 u_2 u_4 u_3 u_4.
$$
and we are reduced to case $(5,2)$.
We have $u_4 u_3 u_2 u_1 u_2 u_3 u_2 u_4 u_3 u_4 = u_4 u_3 u_2 u_1 u_2 u_3  u_4 u_2 u_3 u_4 \subset A_4 u_4 A_4 u_4 A_4$
by the induction assumption (case $(5,2)$)
and $u_4 u_3 u_2 u_1 u_3 u_2 (u_3 u_4 u_3 u_4) \subset 
 u_4 u_3 u_2 u_1 u_3 u_2 u_3 u_4 u_3 +u_4 u_3 u_2 u_1 u_3 u_2  u_4 u_3 u_4
 \subset A_4 + u_4 A_4 u_4 A_4 + u_4 u_3 u_2 u_1 u_3 u_2  u_4 u_3 u_4
 \subset A_4 u_4 A_4 u_4 A_4$
by the induction assumption.
\end{comment}

When $(p,q) = (5,4)$, the only nontrivial cases
are
%$u_4 u_3 u_2 u_1 u_2 u_3 u_4    u_3 u_1 u_2 u_3 u_4$  se reduit en $u_4 u_3 (u_2 u_1 u_2 u_1)u_3 u_4    u_3  u_2 u_3 u_4$
$$u_4 u_3 u_2 u_1 u_2 u_3 u_4    u_2 u_1 u_2 u_3 u_4
\mbox{ and }
%$u_4 u_3 u_2 u_1 u_3 u_2 u_4         u_2    u_1 u_2  u_3 u_4$ -> $u_4 u_3 u_2 u_1 u_3 u_2 u_4           u_1 u_2  u_3 u_4$
u_4 u_3 u_2 u_1 u_3 u_2 u_4   u_3  u_1 u_2  u_3 u_4.$$
In the first case,
$u_4 u_3 u_2 u_1 u_2 u_3 u_4    u_2 u_1 u_2 u_3 u_4 = u_4 u_3 u_2 u_1 u_2 u_3 u_2 u_1 u_2 u_4     u_3 u_4 \subset u_4 A_4 u_4 u_3 u_4$.
By theorem \ref{theodecA4}, we have $A_4 = A_3u_3 A_3 + A_3 u_3 u_2 u_3 A_3 + A_3 u_3 u_2 u_1 u_2 u_3$
hence
$$
\begin{array}{lcl}
u_4 A_4 u_4 u_3 u_4 &\subset& u_4 A_3u_3 A_3 u_4 u_3 u_4 + u_4 A_3 u_3 u_2 u_3 A_3 u_4 u_3 u_4 + u_4 A_3 u_3 u_2 u_1 u_2 u_3 u_4 u_3 u_4 \\
& \subset&  A_3 u_4 u_3 A_3 u_4 u_3 u_4 + A_3 u_4  u_3 u_2 u_3  u_4 A_3 u_3 u_4 + A_3 u_4  u_3 u_2 u_1 u_2 u_3 u_4 u_3 u_4 \\
& \subset & A_4 u_4 A_4 u_4 A_4 \\
\end{array}
$$
by the induction assumption and $A_3 = u_2 u_1 u_2 u_1$.

In the second case, we need to consider the sets $s_4^{\alpha} u_3 u_2 u_1 u_3 u_2 s_4^{\beta}           u_3  u_1 u_2  u_3 s_4^{\gamma}$
with $\alpha,\beta, \gamma \in \{ -1, 1 \}$, and we can assume that two of them have distinct signs, otherwise we are in the exceptional
case of the statement. Up to using $\Phi$ and $\Psi$,
%the usual anti-automorphism and automorphism, 
we can assume $\gamma = 1$ and $\beta = -1$.
We are thus considering expressions of the type  $u_4 u_3 u_2 u_1 u_3 u_2 s_4^{-1}           u_3  u_1 u_2  u_3 s_4 = u_4 u_3 u_2  u_3 u_1 u_2 u_1 s_4^{-1}           u_3   u_2  u_3 s_4.$
Notice that $$
\begin{array}{clcl}
& u_4 u_3 u_2  u_3 (u_2 u_1 u_2) s_4^{-1}           u_3   u_2  u_3 s_4 &=& u_4 (u_3 u_2  u_3 u_2) u_1 u_2 s_4^{-1}           u_3   u_2  u_3 s_4 \\
= &u_4 u_2 u_3  u_2 u_3 u_1 u_2 s_4^{-1}           u_3   u_2  u_3 s_4 
&=& u_2  u_4 u_3  u_2 u_3 u_1 u_2 s_4^{-1}           u_3   u_2  u_3 s_4 \\
\end{array}$$ hence reduces to smaller cases. As a consequence, among the natural spanning set of $u_1 u_2 u_1$,
only the $s_1^{\alpha} s_2^{-\alpha} s_1^{\alpha}$ do not reduce to smaller cases, and so we may restrict ourselves to these.
Moreover, using $u_3 u_2 u_3 \subset u_2 s_3^{-1} s_2 s_3^{-1} + u_2 u_3 u_2$ and
$u_3 u_2 u_3 \subset  s_3 s_2^{-1} s_3 u_2 + u_2 u_3 u_2$ we are reduced to expressions of
the form $u_4 s_3^{-1} s_2 s_3^{-1}s_1^{\alpha} s_2^{-\alpha} s_1^{\alpha} s_4^{-1} s_3 s_2^{-1} s_3  s_4$.
We then have
$$
\begin{array}{lcl}
u_4 s_3^{-1} s_2 s_3^{-1}s_1^{\alpha} s_2^{-\alpha} s_1^{\alpha} s_4^{-1} s_3 s_2^{-1} s_3  s_4 
& =& u_4 s_3^{-1} s_2 s_3^{-1}s_1^{\alpha} s_2^{-\alpha} s_1^{\alpha} s_4^{-1} s_3 s_2^{-1} (s_3  s_4 s_3) s_3^{-1} \\ 
& =& u_4 s_3^{-1} s_2 s_3^{-1}s_1^{\alpha} s_2^{-\alpha} s_1^{\alpha} s_4^{-1} s_3 s_2^{-1} s_4  s_3 s_4 s_3^{-1} \\ 
& =& u_4 s_3^{-1} s_2 s_3^{-1}s_1^{\alpha} s_2^{-\alpha} s_1^{\alpha} (s_4^{-1} s_3 s_4) s_2^{-1}   s_3 s_4 s_3^{-1} \\ 
& =& u_4 s_3^{-1} s_2 s_3^{-1}s_1^{\alpha} s_2^{-\alpha} s_1^{\alpha} s_3 s_4 s_3^{-1} s_2^{-1}   s_3 s_4 s_3^{-1} \\ 
& =& u_4 s_3^{-1} s_2 s_3^{-1}s_1^{\alpha} s_2^{-\alpha} s_1^{\alpha} s_3 s_4 (s_3^{-1} s_2^{-1}   s_3) s_4 s_3^{-1} \\ 
& =& u_4 s_3^{-1} s_2 s_3^{-1}s_1^{\alpha} s_2^{-\alpha} s_1^{\alpha} s_3 s_4 s_2 s_3^{-1}   s_2^{-1} s_4 s_3^{-1} \\ 
& =& u_4 s_3^{-1} s_2 s_3^{-1}s_1^{\alpha} s_2^{-\alpha} s_1^{\alpha} s_3 s_4 s_2 s_3^{-1}    s_4 s_2^{-1} s_3^{-1} \\ 
\end{array}
$$
so we now  need to prove that $u_4 s_3^{-1} s_2 s_3^{-1}s_1^{\alpha} s_2^{-\alpha} s_1^{\alpha} s_3 s_4 s_2 s_3^{-1}    s_4 
\subset A_4 u_4 A_4 u_4 A_4$.
When $\alpha = 1$ we get
$$
\begin{array}{lcl}
u_4 s_3^{-1} s_2 s_3^{-1}s_1 s_2^{-1} s_1 s_3 s_4 s_2 s_3^{-1}    s_4 
&=& u_4 s_3^{-1} s_2 s_1 (s_3^{-1} s_2^{-1} s_3) s_1  s_4 s_2 s_3^{-1}    s_4 \\
&=& u_4 s_3^{-1} (s_2 s_1 s_2) s_3^{-1} s_2^{-1} s_1  s_4 s_2 s_3^{-1}    s_4 \\
&=& u_4 s_3^{-1} s_1 s_2 s_1 s_3^{-1} s_2^{-1} s_1  s_4 s_2 s_3^{-1}    s_4 \\
&=& s_1 u_4 s_3^{-1}  s_2 s_1 s_3^{-1} s_2^{-1} s_1  s_4 s_2 s_3^{-1}    s_4 \\
&=& s_1 u_4 s_3^{-1}  s_2 s_1 s_3^{-1} s_2^{-1}   s_4 s_1 s_2 s_3^{-1}    s_4 \\
&\subset& A_4 u_4 A_4 u_4 A_4\\
\end{array}
$$
by the induction assumption.
When $\alpha = -1$ we get
$$
\begin{array}{lcl}
u_4 s_3^{-1} s_2 s_3^{-1}s_1^{-1} s_2 s_1^{-1} s_3 s_4 s_2 s_3^{-1}    s_4 
&=& u_4 s_3^{-1} s_2 s_1^{-1}(s_3^{-1} s_2 s_3)  s_1^{-1} s_4 s_2 s_3^{-1}    s_4  \\
&=& u_4 s_3^{-1} s_2 s_1^{-1}s_2 s_3 s_2^{-1}  s_1^{-1} s_4 s_2 s_3^{-1}    s_4  \\
&=& u_4 s_3^{-1} s_2 s_1^{-1}s_2 s_3 (s_2^{-1}  s_1^{-1} s_2) s_4  s_3^{-1}    s_4  \\
&=& u_4 s_3^{-1} s_2 s_1^{-1}s_2 s_3 s_1  s_2^{-1} s_1^{-1} s_4  s_3^{-1}    s_4  \\
&=& u_4 s_3^{-1} s_2 s_1^{-1}s_2 s_3 s_1  s_2^{-1}  s_4  s_3^{-1}    s_4s_1^{-1}  \\
&=& u_4 s_3^{-1} s_2 s_1^{-1}s_2 s_3 s_4 s_1  s_2^{-1}    s_3^{-1}    s_4s_1^{-1}  \\
&\subset& A_4 u_4 A_4 u_4 A_4\\
\end{array}
$$
by the induction assumption.

\begin{comment}
% u_4 u_3 u_2 u_3 u_4  u_2 u_1 u_2 u_3u_2u_3  u_4 + u_4 u_3 u_2 u_3 u_4  u_1 u_2 u_1 u_3u_2u_3  u_4$$
and $u_4 u_3 u_2 u_3 u_4  u_2 u_1 (u_2 u_3u_2u_3)  u_4 \subset  u_4 u_3 u_2 u_3 u_4  u_2 u_1 u_2 u_3u_2  u_4 + u_4 u_3 u_2 u_3 u_4  u_2 u_1 u_3u_2u_3  u_4 \subset A_4 u_4 A_4 u_4 A_4$
by the induction assumption ;
$u_4 u_3 u_2 u_3 u_4  u_3u_2u_1u_2u_3  u_4
\subset A_4 u_4 A_4 u_4 A_4$
by the induction assumption.
\end{comment}
This concludes the case $(5,4)$ and the proof of the proposition.
%This concludes the proof of the lemma.

\end{proof}

\begin{lemma}  \label{lemaux2AuAuA}
$$u_4 u_3 u_2 u_3 u_1 u_2 u_1 u_4 u_3 u_2 u_3 u_4 \subset A_4 (u_4 u_3 u_2 u_1 u_2 u_3u_4 u_3 u_2 u_1 u_2 u_3 u_4)u_2 + A_4 u_4 A_4 u_4 A_4$$
\end{lemma}
\begin{proof}
By proposition \ref{propmoinsde55} it is enough to prove
$$s_4 u_3 u_2 u_3 u_1 u_2 u_1 s_4 u_3 u_2 u_3 s_4 \subset A_4 u_4 u_3 u_2 u_1 u_2 u_3u_4 u_3 u_2 u_1 u_2 u_3 u_4u_2 + A_4 u_4 A_4 u_4 A_4$$
and, as noted in the proof of proposition \ref{propmoinsde55}, we can restrict to the forms
$s_4 u_3 u_2 u_3 s_1^{\alpha} s_2^{-\alpha} s_1^{\alpha} s_4 u_3 u_2 u_3 s_4$. Moreover,
since $s_1 s_2^{-1} s_1 = s_2^{-1} s_1^{-1} s_2 s_1^2 \in s_2^{-1} s_1^{-1} s_2 u_1$,
and $s_1^{-1} s_2 u_1 \subset R s_1^{-1} s_2 s_1^{-1} + u_2 u_1 u_2$,
we get 
$$
\begin{array}{lcl}
s_4 u_3 u_2 u_3 s_1 s_2^{-1} s_1 s_4 u_3 u_2 u_3 s_4 &\subset&
s_4 u_3 u_2 u_3 s_2^{-1} s_1^{-1} s_2 u_1s_4 u_3 u_2 u_3 s_4 \\ &\subset& 
s_4 (u_3 u_2 u_3 u_2) s_1^{-1} s_2 u_1s_4 u_3 u_2 u_3 s_4 \\
&\subset& 
s_4 (u_2 u_3 u_2 u_3) s_1^{-1} s_2 u_1s_4 u_3 u_2 u_3 s_4 \\
&\subset&
u_2 s_4  u_3 u_2 u_3 s_1^{-1} s_2 u_1s_4 u_3 u_2 u_3 s_4 \\
&\subset&
u_2 s_4  u_3 u_2 u_3 s_1^{-1} s_2 s_1^{-1}s_4 u_3 u_2 u_3 s_4  + 
u_2 s_4  (u_3 u_2 u_3 u_2) u_1 u_2 s_4 u_3 u_2 u_3 s_4 \\
&\subset&
u_2 s_4  u_3 u_2 u_3 s_1^{-1} s_2 s_1^{-1}s_4 u_3 u_2 u_3 s_4  + 
u_2 s_4  u_2 u_3 u_2 u_3 u_1 u_2 s_4 u_3 u_2 u_3 s_4 \\
&\subset&
u_2 s_4  u_3 u_2 u_3 s_1^{-1} s_2 s_1^{-1}s_4 u_3 u_2 u_3 s_4  + 
u_2 s_4   u_3 u_2 u_3 u_1 u_2 s_4 u_3 u_2 u_3 s_4 \\
&\subset&
u_2 s_4  u_3 u_2 u_3 s_1^{-1} s_2 s_1^{-1}s_4 u_3 u_2 u_3 s_4  + 
A_4 u_4 A_4 u_4 A_4 \\
\end{array}
$$
by proposition \ref{propmoinsde55}. We can thus restrict
to $ s_4  u_3 u_2 u_3 s_1^{-1} s_2 s_1^{-1}s_4 u_3 u_2 u_3 s_4$.
Moreover, %as we already noticed in the proof of proposition \ref{propmoinsde55},
using that $u_3 u_2 u_3 \subset u_2 s_3 s_2^{-1} s_3 + u_2 u_3 u_2$
and  $u_3 u_2 u_3 \subset s_3^{-1} s_2 s_3^{-1}u_2  + u_2 u_3 u_2$
leads to
$s_4  u_3 u_2 u_3 s_1^{-1} s_2 s_1^{-1}s_4 u_3 u_2 u_3 s_4 \subset
u_2 s_4  s_3 s_2^{-1} s_3  s_1^{-1} s_2 s_1^{-1}s_4 s_3^{-1} s_2 s_3^{-1} s_4 u_2 + 
A_4 u_4 A_4 u_4 A_4 $ by proposition \ref{propmoinsde55}.
Now 
$$
\begin{array}{lcl}
s_4  s_3 s_2^{-1} s_3  s_1^{-1} s_2 s_1^{-1}s_4 s_3^{-1} s_2 s_3^{-1} s_4
&=& s_1^{-1} s_1 s_4  s_3 s_2^{-1} s_3  s_1^{-1} s_2 s_1^{-1}s_4 s_3^{-1} s_2 s_3^{-1} s_4\\
&=& s_1^{-1}  s_4  s_3 (s_1 s_2^{-1} s_1^{-1}) s_3   s_2 s_1^{-1}s_4 s_3^{-1} s_2 s_3^{-1} s_4\\
&=& s_1^{-1}  s_4  s_3 s_2^{-1} s_1^{-1} (s_2 s_3   s_2) s_1^{-1}s_4 s_3^{-1} s_2 s_3^{-1} s_4\\
&=& s_1^{-1}  s_4  s_3 s_2^{-1} s_1^{-1} s_3 s_2   s_3 s_1^{-1}s_4 s_3^{-1} s_2 s_3^{-1} s_4\\
&=& s_1^{-1}  s_4  s_3 s_2^{-1} s_1^{-1} s_3 s_2   s_1^{-1}(s_3 s_4 s_3^{-1}) s_2 s_3^{-1} s_4\\
&=& s_1^{-1}  s_4  s_3 s_2^{-1} s_1^{-1} s_3 s_2   s_1^{-1}s_4^{-1} s_3 s_4 s_2 s_3^{-1} s_4\\
&=& s_1^{-1}s_3^{-1} (s_3  s_4  s_3) s_2^{-1} s_1^{-1} s_3 s_2   s_1^{-1}s_4^{-1} s_3 s_4 s_2 s_3^{-1} s_4\\
&=& s_1^{-1}s_3^{-1} s_4  s_3  s_4 s_2^{-1} s_1^{-1} s_3 s_2   s_1^{-1}s_4^{-1} s_3 s_4 s_2 s_3^{-1} s_4\\
&=& s_1^{-1}s_3^{-1} s_4  s_3  s_2^{-1} s_1^{-1} (s_4 s_3 s_4^{-1})s_2   s_1^{-1} s_3 s_4 s_2 s_3^{-1} s_4\\
&=& s_1^{-1}s_3^{-1} s_4  s_3  s_2^{-1} s_1^{-1} s_3^{-1} s_4 s_3s_2   s_1^{-1} s_3 s_4 s_2 s_3^{-1} s_4\\
&=& s_1^{-1}s_3^{-1} s_4  (s_3  s_2^{-1} s_3^{-1}) s_1^{-1}  s_4 s_3s_2   s_1^{-1} s_3 s_4 s_2 s_3^{-1} s_4\\
&=& s_1^{-1}s_3^{-1} s_4  s_2 ^{-1} s_3^{-1} s_2 s_1^{-1}  s_4 s_3s_2   s_1^{-1} s_3 s_4 s_2 s_3^{-1} s_4\\
&=& s_1^{-1}s_3^{-1} s_2 ^{-1} s_4   s_3^{-1} s_2 s_1^{-1}  s_4 s_3s_2   s_1^{-1} s_3 s_4 s_2 s_3^{-1} s_4\\
&=& s_1^{-1}s_3^{-1} s_2 ^{-1} s_4   s_3^{-1} s_2 s_1^{-1}  s_4 (s_3s_2   s_3)s_1^{-1}  s_4 s_2 s_3^{-1} s_4\\
&=& s_1^{-1}s_3^{-1} s_2 ^{-1} s_4   s_3^{-1} s_2 s_1^{-1}  s_4 s_2s_3   s_2s_1^{-1}  s_4 s_2 s_3^{-1} s_4\\
&=& s_1^{-1}s_3^{-1} s_2 ^{-1} s_4   s_3^{-1} s_2 s_1^{-1}  s_2 (s_4s_3  s_4) s_2s_1^{-1}   s_2 s_3^{-1} s_4\\
&=& s_1^{-1}s_3^{-1} s_2 ^{-1} s_4   s_3^{-1} s_2 s_1^{-1}  s_2 s_3s_4  s_3 s_2s_1^{-1}   s_2 s_3^{-1} s_4\\
&\subset & A_4 s_4   s_3^{-1} s_2 s_1^{-1}  s_2 s_3s_4  s_3 s_2s_1^{-1}   s_2 s_3^{-1} s_4\\
\end{array}
$$
and this proves the claim.
\end{proof}

\begin{lemma}  \label{lemaux3AuAuA}
$$u_4 u_3 u_2 u_1 u_2 u_3  u_4 u_2 u_3 u_1 u_2 u_3 u_4 \subset A_4 (u_4 u_3 u_2 u_1 u_2 u_3u_4 u_3 u_2 u_1 u_2 u_3 u_4)A_4 + A_4 u_4 A_4 u_4 A_4$$
\end{lemma}
\begin{proof}
We consider the expression
$u_4 s_3^{\alpha} u_2 u_1 u_2 s_3^{\beta}  u_4 u_2 u_3 u_1 u_2 u_3 u_4$ and we first assume $\alpha = \beta$ ; by applying if necessary $\Phi$,
%the automorphism
%induced by $s_i \mapsto s_i^{-1}$, 
we can then assume $\alpha = \beta = -1$. Since $u_2 u_1 u_2 \subset u_1 s_2 s_1^{-1} s_2 + u_1 u_2 u_1$
we have 
$$
\begin{array}{lcl}
u_4 s_3^{-1} u_2 u_1 u_2 s_3^{-1}  u_4 u_2 u_3 u_1 u_2 u_3 u_4
&\subset& u_4 s_3^{-1} u_1 s_2 s_1^{-1} s_2 s_3^{-1}  u_4 u_2 u_3 u_1 u_2 u_3 u_4 + u_4 s_3^{-1} u_1 u_2 u_1 s_3^{-1}  u_4 u_2 u_3 u_1 u_2 u_3 u_4 \\
&\subset& u_1u_4 s_3^{-1}  s_2 s_1^{-1} s_2 s_3^{-1}  u_4 u_2 u_3 u_1 u_2 u_3 u_4 + u_1 u_4 s_3^{-1}  u_2 u_1 s_3^{-1}  u_4 u_2 u_3 u_1 u_2 u_3 u_4 \\
\end{array}
$$ and we are reduced to $u_4 s_3^{-1}  s_2 s_1^{-1} s_2 s_3^{-1}  u_4 u_2 u_3 u_1 u_2 u_3 u_4$ by lemmas \ref{propmoinsde55} and \ref{lemaux2AuAuA}.
Now 
$$u_4 s_3^{-1}  s_2 s_1^{-1} s_2 s_3^{-1}  u_4 u_2 u_3 u_1 u_2 u_3 u_4= u_4 (s_3^{-1}  s_2 s_1^{-1} s_2 s_3^{-1})  u_2u_1 u_4  u_3  u_2 u_3 u_4$$
and $(s_3^{-1}  s_2 s_1^{-1} s_2 s_3^{-1})  u_2u_1 \subset A_3(s_3^{-1}  s_2 s_1^{-1} s_2 s_3^{-1})   + A_3 u_3 A_3 + A_3 s_3 s_2^{-1} s_3 A_3$
by lemma \ref{lemA4droitegauche}.
We then have
$$
\begin{array}{lcl}
u_4  (A_3 u_3 A_3 + A_3 s_3 s_2^{-1} s_3 A_3)u_4 u_3 u_2 u_3 u_4
&\subset& A_3 u_4 u_3 A_3u_4 u_3 u_2 u_3 u_4 + A_3 u_4s_3 s_2^{-1} s_3 A_3 u_4 u_3 u_2 u_3 u_4 \\
&\subset & A_3 u_4 u_3 (u_1 u_2 u_1 u_2) u_4 u_3 u_2 u_3 u_4 \\ & & + A_3 u_4s_3 s_2^{-1} s_3 (u_1 u_2 u_1 u_2) u_4 u_3 u_2 u_3 u_4 \\
&\subset & A_3 u_4 u_3 u_1 u_2 u_1 u_2u_4 u_3 u_2 u_3 u_4 \\ & & + A_3 u_4s_3 s_2^{-1} s_3 u_1 u_2 u_1  u_4 (u_2 u_3 u_2 u_3) u_4 \\
&\subset & A_3 u_4 u_3 u_1 u_2 u_1 u_2u_4 u_3 u_2 u_3 u_4 \\ & & + A_3 u_4s_3 s_2^{-1} s_3 u_1 u_2 u_1  u_4 u_3 u_2 u_3 u_2 u_4 \\
&\subset & A_3 u_4 u_3 u_1 u_2 u_1 u_2u_4 u_3 u_2 u_3 u_4 \\ & & + A_3 u_4s_3 s_2^{-1} s_3 u_1 u_2 u_1  u_4 u_3 u_2 u_3  u_4u_2 \\
&\subset& A_4 (u_4 u_3 u_2 u_1 u_2 u_3u_4 u_3 u_2 u_1 u_2 u_3 u_4)A_4 \\ & &  + A_4 u_4 A_4 u_4 A_4 \\
\end{array}$$
by lemmas \ref{propmoinsde55} and \ref{lemaux2AuAuA},
and $u_4  A_3(s_3^{-1}  s_2 s_1^{-1} s_2 s_3^{-1})u_4 u_3 u_2 u_3 u_4 = A_3 u_4  s_3^{-1}  s_2 s_1^{-1} s_2 s_3^{-1}u_4 u_3 u_2 u_3 u_4 \subset A_4 u_4 A_4 u_4 A_4$
by proposition \ref{propmoinsde55}, so this solves the case $\alpha = \beta$.

We can thus assume $\alpha = - \beta$, that is
we consider the expression $u_4 s_3^{\beta} u_2 u_1 u_2 s_3^{-\beta} u_2 u_4 u_1 u_3 u_2 u_3 u_4$,
that we split in two cases $u_4 s_3^{\beta} u_2 u_1 u_2 s_3^{-\beta} s_2^{\gamma} u_4 u_1 u_3 u_2 u_3 u_4$
for $\gamma \in \{ -1,1 \}$. Up to applying $\Phi$,
%the automorphism induced by $s_i \mapsto s_i^{-1}$,
we can restrict to $u_4 s_3^{\beta} u_2 u_1 u_2 s_3^{-\beta} s_2^{\gamma} s_4 u_1 u_3 u_2 u_3 s_4^{\alpha}$
for some $\alpha \in \{ -1,1 \}$, and using $u_3 u_2 u_3 \subset s_3^{\alpha} s_2^{-\alpha} s_3^{\alpha} u_2 + u_2 u_3 u_2$
we can restrict to $u_4 s_3^{\beta} u_2 u_1 u_2 s_3^{-\beta} s_2^{\gamma} s_4 u_1 s_3^{\alpha} s_2^{-\alpha} s_3^{\alpha} s_4^{\alpha}$
by proposition \ref{propmoinsde55}.

First assume $\gamma = -1$.
Using again $u_2 u_1 u_2 \subset u_1 s_2 s_1^{-1} s_2 + u_1 u_2 u_1$
we can restrict to
$$u_4 s_3^{\beta} s_2 s_1^{-1} s_2 s_3^{-\beta} s_2^{-1} s_4 u_1 s_3^{\alpha} s_2^{-\alpha} s_3^{\alpha} s_4^{\alpha}.$$
If $\beta = 1$, then we get 
$$
\begin{array}{lcl}
u_4 s_3 s_2 s_1^{-1} s_2 s_3^{-1} s_2^{-1} s_4 u_1 s_3^{\alpha} s_2^{-\alpha} s_3^{\alpha} s_4^{\alpha}
&\subset& u_4 s_3 s_2 s_1^{-1} (s_2 s_3^{-1} s_2^{-1}) s_4 u_1 s_3^{\alpha} s_2^{-\alpha} s_3^{\alpha} s_4^{\alpha} \\
&\subset& u_4 s_3 s_2 s_1^{-1} s_3^{-1} s_2^{-1} s_3 s_4 u_1 s_3^{\alpha} s_2^{-\alpha} s_3^{\alpha} s_4^{\alpha} \\
&\subset& u_4 (s_3 s_2 s_3^{-1})s_1^{-1}  s_2^{-1} s_3 s_4 u_1 s_3^{\alpha} s_2^{-\alpha} s_3^{\alpha} s_4^{\alpha} \\
&\subset& u_4 s_2^{-1} s_3 s_2s_1^{-1}  s_2^{-1} s_3 s_4 u_1 s_3^{\alpha} s_2^{-\alpha} s_3^{\alpha} s_4^{\alpha} \\
&\subset& s_2^{-1} u_4  s_3 s_2s_1^{-1}  s_2^{-1} s_3 s_4 u_1 s_3^{\alpha} s_2^{-\alpha} s_3^{\alpha} s_4^{\alpha} \\
&\subset& A_4 u_4 A_4 u_4 A_4
\end{array}
$$ 
by proposition \ref{propmoinsde55}. For the case $\beta = -1$, we can restrict
to an expression of the form 
$$u_4 s_3^{-1} s_2 s_1^{-1} s_2 s_3 s_2^{-1}s_4 u_1 s_3^{\alpha} s_2^{-\alpha} s_3^{\alpha} s_4^{\alpha},$$
and
we get
$$
\begin{array}{lcl}
u_4 s_3^{-1} s_2 s_1^{-1} (s_2 s_3 s_2^{-1}) s_4 u_1 s_3^{\alpha} s_2^{-\alpha} s_3^{\alpha} s_4^{\alpha}
&\subset& u_4 s_3^{-1} s_2 s_1^{-1} s_3^{-1} s_2 s_3 s_4 u_1 s_3^{\alpha} s_2^{-\alpha} s_3^{\alpha} s_4^{\alpha}\\
&\subset& u_4 s_3^{-1} s_2 s_1^{-1} s_3^{-1} s_2 s_3 s_4 u_1 s_3^{\alpha} s_2^{-\alpha} s_3^{\alpha} s_4^{\alpha}\\
&\subset& u_4 s_3^{-1} s_2 s_1^{-1} s_3^{-1} s_2 (s_3 s_4 s_3^{\alpha})u_1  s_2^{-\alpha} s_3^{\alpha} s_4^{\alpha}\\
&\subset& u_4 s_3^{-1} s_2 s_1^{-1} s_3^{-1} s_2 s_4^{\alpha} s_3 s_4 u_1  s_2^{-\alpha} s_3^{\alpha} s_4^{\alpha}\\
&\subset& u_4 s_3^{-1} s_2 s_1^{-1} s_3^{-1} s_2 s_4^{\alpha} s_3  u_1  s_2^{-\alpha}s_4 s_3^{\alpha} s_4^{\alpha})\\
&\subset& u_4 s_3^{-1} s_2 s_1^{-1} s_3^{-1} s_2 s_4^{\alpha} s_3  u_1  s_2^{-\alpha}s_3^{\alpha} s_4^{\alpha} s_3\\
&\subset&A_4 (u_4 u_3 u_2 u_1 u_2 u_3u_4 u_3 u_2 u_1 u_2 u_3 u_4)A_4 + A_4 u_4 A_4 u_4 A_4 \\
\end{array}
$$ 
by proposition \ref{propmoinsde55} and lemma \ref{lemaux2AuAuA}.

Now assume $\gamma = 1$. Using again $u_2 u_1 u_2 \subset u_1 s_2^{-1} s_1 s_2^{-1} + u_1 u_2 u_1$ we
can restrict to the form $u_4 s_3^{\beta} s_2^{-1} s_1 s_2^{-1} s_3^{-\beta} s_2 u_4 u_1 u_3 u_2 u_3 u_4$.
If $\beta = 1$ we get
$$
\begin{array}{lcl}
u_4 s_3 s_2^{-1} s_1 (s_2^{-1} s_3^{-1} s_2) u_4 u_1 u_3 u_2 u_3 u_4  
& \subset & u_4 s_3 s_2^{-1} s_1 s_3 s_2^{-1} s_3^{-1} u_4 u_1 u_3 u_2 u_3 u_4  \\
& \subset & u_4 (s_3 s_2^{-1} s_3) s_1  s_2^{-1} s_3^{-1} u_4 u_1 u_3 u_2 u_3 u_4  \\
& \subset & u_4 u_2(s_3^{-1} s_2 s_3^{-1} ) s_1  s_2^{-1} s_3^{-1} u_4 u_1 u_3 u_2 u_3 u_4  \\ & & +  u_4 u_2 u_3 u_2 s_1  s_2^{-1} s_3^{-1} u_4 u_1 u_3 u_2 u_3 u_4  \\
& \subset & u_2 u_4 s_3^{-1} s_2 s_3^{-1}  s_1  s_2^{-1} s_3^{-1} u_4 u_1 u_3 u_2 u_3 u_4  + \\ & &  u_2 u_4  u_3 u_2 s_1  s_2^{-1} s_3^{-1} u_4 u_1 u_3 u_2 u_3 u_4  \\
& \subset & u_2 u_4 s_3^{-1} s_2  s_1 (s_3^{-1}   s_2^{-1} s_3^{-1}) u_4 u_1 u_3 u_2 u_3 u_4  +   A_4 u_4 A_4 u_4 A_4\\ 
& \subset & u_2 u_4 s_3^{-1} (s_2  s_1 s_2^{-1})   s_3^{-1} s_2^{-1} u_4 u_1 u_3 u_2 u_3 u_4  +   A_4 u_4 A_4 u_4 A_4\\ 
& \subset & u_2 u_4 s_3^{-1} s_1^{-1}  s_2 s_1   s_3^{-1} s_2^{-1} u_4 u_1 u_3 u_2 u_3 u_4  +   A_4 u_4 A_4 u_4 A_4\\ 
& \subset & u_2 s_1^{-1} u_4 s_3^{-1}   s_2 s_1   s_3^{-1} s_2^{-1} u_4 u_1 u_3 u_2 u_3 u_4  +   A_4 u_4 A_4 u_4 A_4\\ 
&\subset&A_4 (u_4 u_3 u_2 u_1 u_2 u_3u_4 u_3 u_2 u_1 u_2 u_3 u_4)A_4 + A_4 u_4 A_4 u_4 A_4 \\
\end{array}
$$ 
by proposition \ref{propmoinsde55} and lemma \ref{lemaux2AuAuA}.

If $\beta = -1$ we get 
$$
\begin{array}{lcl}
u_4 s_3^{-1} s_2^{-1} s_1 (s_2^{-1} s_3 s_2) u_4 u_1 u_3 u_2 u_3 u_4
&\subset & u_4 s_3^{-1} s_2^{-1} s_1 s_3 s_2 s_3^{-1} u_4 u_1 u_3 u_2 u_3 u_4 \\
&\subset & u_4 (s_3^{-1} s_2^{-1} s_3) s_1  s_2 s_3^{-1} u_4 u_1 u_3 u_2 u_3 u_4 \\
&\subset & u_4 s_2 s_3^{-1} s_2^{-1} s_1  s_2 s_3^{-1} u_4 u_1 u_3 u_2 u_3 u_4 \\
&\subset & s_2 u_4  s_3^{-1} s_2^{-1} s_1  s_2 s_3^{-1} u_4 u_1 u_3 u_2 u_3 u_4 \\
&\subset & A_4 u_4 A_4 u_4 A_4 \\
\end{array}
$$ 
by proposition \ref{propmoinsde55}. This concludes the proof of the lemma.
\end{proof}

\begin{proposition} \label{propreducuuu}
$
u_4 A_4 u_4 A_4 u_4  \subset A_4 u_4 u_3 u_2 u_1 u_2 u_3u_4 u_3 u_2 u_1 u_2 u_3 u_4A_4 + A_4 u_4 A_4 u_4 A_4 
$, and thus $A_5^{(3)} \subset A_4 u_4 u_3 u_2 u_1 u_2 u_3u_4 u_3 u_2 u_1 u_2 u_3 u_4A_4 + A_4 u_4 A_4 u_4 A_4$.
\end{proposition}

%\begin{lemma}
\begin{comment}
$$
\begin{array}{lcl}
u_4 A_4 u_4 A_4 u_4 &\subset& A_4 u_4 u_3 u_2 u_1 u_2 u_3u_4 u_3 u_2 u_1 u_2 u_3 u_4A_3 + A_4 u_4 A_4 u_4 A_4 \\
& & + A_3u_4   u_3 u_2 u_3 u_1 u_2  u_4 u_3 u_2u_1 u_2  u_3 u_4A_3 +  A_3u_4  u_3  u_2 u_1 u_2 u_3 u_4  u_2 u_1 u_3 u_2 u_3   u_4A_3
\end{array}
$$
\end{comment}
%\end{lemma}
\begin{proof}
We will actually prove
$$
\begin{array}{lcl}
u_4 A_4 u_4 A_4 u_4 &\subset& A_4 u_4 u_3 u_2 u_1 u_2 u_3u_4 u_3 u_2 u_1 u_2 u_3 u_4A_3 + A_4 u_4 A_4 u_4 A_4 \\
& & + A_3u_4   u_3 u_2 u_3 u_1 u_2  u_4 u_3 u_2u_1 u_2  u_3 u_4A_3 +  A_3u_4  u_3  u_2 u_1 u_2 u_3 u_4  u_2 u_1 u_3 u_2 u_3   u_4A_3
\end{array}
$$
and the statement will then follow by lemmas \ref{lemaux2AuAuA} and \ref{lemaux3AuAuA}.

By theorem \ref{theodecA4} we have
$A_4 = A_3 u_3 A_3 + A_3 u_3 u_2 u_3 A_3 + A_3 u_3 u_2 u_1 u_2 u_3$
and $A_4 = A_3 u_3 A_3 + A_3 u_3 u_2 u_3 A_3 + u_3 u_2 u_1 u_2 u_3 A_3$,
whence
$$
\begin{array}{lcl}
u_4 A_4 u_4 A_4 u_4 &\subset&
u_4 A_3 u_3 A_3 u_4 A_3 u_3 A_3 u_4 + u_4 A_3 u_3 A_3 u_4 A_3 u_3 u_2 u_3 A_3 u_4 \\ & & + u_4 A_3 u_3 A_3 u_4 u_3 u_2 u_1 u_2 u_3 A_3u_4 +
u_4 A_3 u_3 u_2 u_3 A_3 u_4 A_3 u_3 A_3 u_4 \\ & & + u_4 A_3 u_3 u_2 u_3 A_3 u_4 A_3 u_3 u_2 u_3 A_3 u_4 + u_4 A_3 u_3 u_2 u_3 A_3 u_4 u_3 u_2 u_1 u_2 u_3 A_3u_4 \\ & & +
u_4 A_3 u_3 u_2 u_1 u_2 u_3 u_4 A_3 u_3 A_3 u_4 + u_4 A_3 u_3 u_2 u_1 u_2 u_3 u_4 A_3 u_3 u_2 u_3 A_3 u_4 \\ & & + u_4 A_3 u_3 u_2 u_1 u_2 u_3 u_4 u_3 u_2 u_1 u_2 u_3 A_3u_4 \\
&\subset&
A_3u_4  u_3 A_3 u_4  u_3  u_4A_3 + A_3u_4  u_3 A_3 u_4  u_3 u_2 u_3  u_4A_3 \\ & & + A_3u_4  u_3 A_3 u_4 u_3 u_2 u_1 u_2 u_3 u_4A_3 +
A_3u_4  u_3 u_2 u_3  u_4 A_3 u_3  u_4A_3 \\ & & + A_3u_4  u_3 u_2 u_3 A_3 u_4  u_3 u_2 u_3  u_4A_3 + A_3u_4  u_3 u_2 u_3 A_3 u_4 u_3 u_2 u_1 u_2 u_3 u_4A_3 \\ & & +
A_3u_4  u_3 u_2 u_1 u_2 u_3 u_4 A_3 u_3  u_4A_3 + A_3u_4  u_3 u_2 u_1 u_2 u_3 u_4 A_3 u_3 u_2 u_3  u_4A_3 \\ & & + A_3u_4  u_3 u_2 u_1 u_2 u_3 u_4 u_3 u_2 u_1 u_2 u_3 u_4 A_3\\
%&\subset&
%A_3u_4  u_3 (u_2 u_1 u_2 + u_1 u_2 u_1) u_4  u_3  u_4A_3 + A_3u_4  u_3 (u_2 u_1 u_2 + u_1 u_2 u_1) u_4  u_3 u_2 u_3  u_4A_3 \\ & & + A_3u_4  u_3 (u_2 u_1 u_2 + u_1 u_2 u_1) u_4 u_3 u_2 u_1 u_2 u_3 u_4A_3 +
%A_3u_4  u_3 u_2 u_3  u_4 (u_2 u_1 u_2 + u_1 u_2 u_1) u_3  u_4A_3 \\ & & + A_3u_4  u_3 u_2 u_3 A_3 u_4  u_3 u_2 u_3  u_4A_3 + A_3u_4  u_3 u_2 u_3 A_3 u_4 u_3 u_2 u_1 u_2 u_3 u_4A_3 \\ & & +
%A_3u_4  u_3 u_2 u_1 u_2 u_3 u_4 A_3 u_3  u_4A_3 + A_3u_4  u_3 u_2 u_1 u_2 u_3 u_4 A_3 u_3 u_2 u_3  u_4A_3 \\ & & + A_3u_4  u_3 u_2 u_1 u_2 u_3 u_4 u_3 u_2 u_1 u_2 u_3 u_4 A_3\\
%&\subset&
%A_3u_4  u_3 (u_2 u_1 u_2 + u_1 u_2 u_1) u_4  u_3  u_4A_3 + A_3u_4  u_3 (u_2 u_1 u_2 + u_1 u_2 u_1) u_4  u_3 u_2 u_3  u_4A_3 \\ & & + A_3u_4  u_3 (u_2 u_1 u_2 + u_1 u_2 u_1) u_4 u_3 u_2 u_1 u_2 u_3 u_4A_3 +
%A_3u_4  u_3 u_2 u_3  u_4 (u_2 u_1 u_2 + u_1 u_2 u_1) u_3  u_4A_3 \\ 
%& \subset & A_4 u_4 A_4 u_4 A_4 + A_3u_4  u_3 u_2 u_3 A_3 u_4  u_3 u_2 u_3  u_4A_3 + A_3u_4  u_3 u_2 u_3 A_3 u_4 u_3 u_2 u_1 u_2 u_3 u_4A_3 \\ & & +
%A_3u_4  u_3 u_2 u_1 u_2 u_3 u_4 A_3 u_3  u_4A_3 + A_3u_4  u_3 u_2 u_1 u_2 u_3 u_4 A_3 u_3 u_2 u_3  u_4A_3 \\ & & + A_3u_4  u_3 u_2 u_1 u_2 u_3 u_4 u_3 u_2 u_1 u_2 u_3 u_4 A_3\\
&\subset&
A_3u_4  u_3 A_3 u_4  u_3  u_4A_3 + A_3u_4  u_3 A_3 u_4  u_3 u_2 u_3  u_4A_3 \\ & & + A_3u_4  u_3 A_3 u_4 u_3 u_2 u_1 u_2 u_3 u_4A_3 +
A_3u_4  u_3 u_2 u_3  u_4 A_3 u_3  u_4A_3 \\ & & + A_3u_4  u_3 u_2 u_3 A_3 u_4  u_3 u_2 u_3  u_4A_3 + A_3u_4  u_3 u_2 u_3 A_3 u_4 u_3 u_2 u_1 u_2 u_3 u_4A_3 \\ & & +
A_3u_4  u_3 u_2 u_1 u_2 u_3 u_4 A_3 u_3  u_4A_3 + A_3u_4  u_3 u_2 u_1 u_2 u_3 u_4 A_3 u_3 u_2 u_3  u_4A_3 \\ & & + A_3u_4  u_3 u_2 u_1 u_2 u_3 u_4 u_3 u_2 u_1 u_2 u_3 u_4 A_3\\
\end{array}
$$
We have
\begin{enumerate}
\item $A_3u_4  u_3 A_3 u_4  u_3  u_4A_3 \subset A_3u_4  u_3 (u_2 u_1 u_2 u_1) u_4  u_3  u_4A_3 \subset A_4 u_4 A_4 u_4 A_4$ by proposition \ref{propmoinsde55}.
\item $A_3u_4  u_3 A_3 u_4  u_3 u_2 u_3  u_4A_3 \subset A_3u_4  u_3 u_2u_1u_2 u_1 u_4  u_3 u_2 u_3  u_4A_3\subset A_4 u_4 A_4 u_4 A_4$ by proposition \ref{propmoinsde55}.
\item We have
$$
\begin{array}{lcl}
A_3u_4  u_3 A_3 u_4 u_3 u_2 u_1 u_2 u_3 u_4A_3 &\subset&  A_3u_4  u_3 u_2 u_1 u_2 u_1 u_4 u_3 u_2 u_1 u_2 u_3 u_4A_3 \\
&\subset& A_3u_4  u_3 u_2 u_1 u_2  u_4 u_3 (u_1 u_2 u_1 u_2) u_3 u_4A_3 \\
&\subset& A_3u_4  u_3 u_2 u_1 u_2  u_4 u_3 u_2 u_1 u_2 u_1 u_3 u_4A_3 \\
&\subset& A_3u_4  u_3 u_2 u_1 u_2  u_4 u_3 u_2 u_1 u_2  u_3 u_4A_3 \\
&\subset & A_4 u_4 A_4 u_4 A_4 
\end{array}
$$
by proposition \ref{propmoinsde55}.
\item $A_3u_4  u_3 u_2 u_3  u_4 A_3 u_3  u_4A_3 \subset A_3u_4  u_3 u_2 u_3  u_4 u_2 u_1 u_2 u_1 u_3  u_4A_3 \subset A_4 u_4 A_4 u_4 A_4 $
by proposition \ref{propmoinsde55}. 
\item Using $A_3 = u_2 u_1 u_2 u_1$ we get
$$
\begin{array}{lcl}
A_3u_4  u_3 u_2 u_3 A_3 u_4  u_3 u_2 u_3  u_4A_3 &\subset& A_3u_4 ( u_3 u_2 u_3 u_2)u_1u_2u_1 u_4  u_3 u_2 u_3  u_4A_3 \\
&\subset& A_3u_4 u_2 u_3 u_2 u_3u_1u_2u_1 u_4  u_3 u_2 u_3  u_4A_3\\
&\subset& A_3u_4  u_3 u_2 u_3u_1u_2u_1 u_4  u_3 u_2 u_3  u_4A_3 \\ &\subset&  A_4 (u_4 u_3 u_2 u_1 u_2 u_3u_4 u_3 u_2 u_1 u_2 u_3 u_4)A_3 + A_4 u_4 A_4 u_4 A_4\\
\end{array}
$$
by lemma \ref{lemaux2AuAuA}.
\item Using $A_3 = u_2 u_1 u_2 u_1$ we get
$$
\begin{array}{lcl}
A_3u_4  u_3 u_2 u_3 A_3 u_4 u_3 u_2 u_1 u_2 u_3 u_4A_3
&\subset& A_3u_4  u_3 u_2 u_3 u_2 u_1 u_2 u_1 u_4 u_3 u_2 u_1 u_2 u_3 u_4A_3 \\
&\subset& A_3u_4  (u_3 u_2 u_3 u_2) u_1 u_2 u_1 u_4 u_3 u_2 u_1 u_2 u_3 u_4A_3 \\
&\subset& A_3u_4  u_2 u_3 u_2 u_3 u_1 u_2 u_1 u_4 u_3 u_2 u_1 u_2 u_3 u_4A_3 \\
&\subset& A_3u_4   u_3 u_2 u_3 u_1 u_2 u_1 u_4 u_3 u_2 u_1 u_2 u_3 u_4A_3 \\
&\subset& A_3u_4   u_3 u_2 u_3 u_1 u_2  u_4 u_3 u_1u_2 u_1 u_2 u_3 u_4A_3 \\
&\subset& A_3u_4   u_3 u_2 u_3 u_1 u_2  u_4 u_3 (u_1u_2 u_1 u_2) u_3 u_4A_3 \\
&\subset& A_3u_4   u_3 u_2 u_3 u_1 u_2  u_4 u_3 u_2u_1 u_2 u_1 u_3 u_4A_3 \\
&\subset& A_3u_4   u_3 u_2 u_3 u_1 u_2  u_4 u_3 u_2u_1 u_2  u_3 u_4A_3 \\
\end{array}
$$
\item Using $A_3 = u_1 u_2 u_1 u_2$ we get
$$
\begin{array}{lcl}
A_3u_4  u_3 u_2 u_1 u_2 u_3 u_4 A_3 u_3  u_4A_3 &\subset& A_3u_4  u_3 u_2 u_1 u_2 u_3 u_4 u_1 u_2 u_1 u_2 u_3  u_4A_3 \\
&\subset& A_3u_4  u_3 u_2 u_1 u_2 u_1u_3 u_4  u_2 u_1 u_2 u_3  u_4A_3  \\
&\subset& A_3u_4  u_3 (u_2 u_1 u_2 u_1)u_3 u_4  u_2 u_1 u_2 u_3  u_4A_3 \\
&\subset& A_3u_4  u_3 u_1 u_2 u_1 u_2 u_3 u_4  u_2 u_1 u_2 u_3  u_4A_3 \\
&\subset& A_3u_4  u_3  u_2 u_1 u_2 u_3 u_4  u_2 u_1 u_2 u_3  u_4A_3 \\
&\subset & A_4 u_4 A_4 u_4 A_4 \\
\end{array}
$$
by proposition \ref{propmoinsde55}.
\item Using $A_3 = u_1 u_2 u_1 u_2$ we get
$$
\begin{array}{lcl}
A_3u_4  u_3 u_2 u_1 u_2 u_3 u_4 A_3 u_3 u_2 u_3  u_4A_3
&\subset & A_3u_4  u_3 u_2 u_1 u_2 u_3 u_4 u_1 u_2 u_1 u_2 u_3 u_2 u_3  u_4A_3 \\
&\subset & A_3u_4  u_3 u_2 u_1 u_2 u_3 u_4 u_1 u_2 u_1 (u_2 u_3 u_2 u_3)  u_4A_3 \\
&\subset & A_3u_4  u_3 u_2 u_1 u_2 u_3 u_4 u_1 u_2 u_1 u_3 u_2 u_3 u_2  u_4A_3 \\
&\subset & A_3u_4  u_3 u_2 u_1 u_2 u_3 u_4 u_1 u_2 u_1 u_3 u_2 u_3   u_4A_3 \\
&\subset & A_3u_4  u_3 u_2 u_1 u_2 u_1u_3 u_4  u_2 u_1 u_3 u_2 u_3   u_4A_3 \\
&\subset & A_3u_4  u_3 (u_2 u_1 u_2 u_1)u_3 u_4  u_2 u_1 u_3 u_2 u_3   u_4A_3 \\
&\subset & A_3u_4  u_3 u_1 u_2 u_1 u_2 u_3 u_4  u_2 u_1 u_3 u_2 u_3   u_4A_3 \\
&\subset & A_3u_4  u_3  u_2 u_1 u_2 u_3 u_4  u_2 u_1 u_3 u_2 u_3   u_4A_3 \\
\end{array}
$$

\item the case $ A_3u_4  u_3 u_2 u_1 u_2 u_3 u_4 u_3 u_2 u_1 u_2 u_3 u_4 A_3$ is clear.
\end{enumerate}

\end{proof}

%\newpage

\subsection{The $A_4$-bimodule $A_5^{(3)}/A_5^{(2)}$ : a smaller set of generators.}

\begin{lemma} \label{lemsimplif212} For all $\alpha,\beta,\gamma,\dots \in \{ -1,1 \}$,
$$
\begin{array}{lcl}
s_4^{\alpha} s_3^{\beta} A_3 s_3^{\gamma} s_4^{\delta} s_3^{\eps} A_3 s_3^{\zeta} s_4^{\eta}  &\subset&
u_1 s_4^{\alpha} s_3^{\beta} (s_2 s_1^{-1} s_2) s_3^{\gamma} s_4^{\delta} s_3^{\eps} (s_2 s_1^{-1} s_2) s_3^{\zeta} s_4^{\eta} u_1 + A_5^{(2)} \\
s_4^{\alpha} s_3^{\beta} A_3 s_3^{\gamma} s_4^{\delta} s_3^{\eps} A_3 s_3^{\zeta} s_4^{\eta}  &\subset&
u_1 s_4^{\alpha} s_3^{\beta} (s_2^{-1} s_1 s_2^{-1}) s_3^{\gamma} s_4^{\delta} s_3^{\eps} (s_2 s_1^{-1} s_2) s_3^{\zeta} s_4^{\eta} u_1 + A_5^{(2)} \\
s_4^{\alpha} s_3^{\beta} A_3 s_3^{\gamma} s_4^{\delta} s_3^{\eps} A_3 s_3^{\zeta} s_4^{\eta}  &\subset&
u_1 s_4^{\alpha} s_3^{\beta} (s_2 s_1^{-1} s_2) s_3^{\gamma} s_4^{\delta} s_3^{\eps} (s_2^{-1} s_1 s_2^{-1}) s_3^{\zeta} s_4^{\eta} u_1 + A_5^{(2)} \\
s_4^{\alpha} s_3^{\beta} A_3 s_3^{\gamma} s_4^{\delta} s_3^{\eps} A_3 s_3^{\zeta} s_4^{\eta}  &\subset&
u_1 s_4^{\alpha} s_3^{\beta} (s_2^{-1} s_1 s_2^{-1}) s_3^{\gamma} s_4^{\delta} s_3^{\eps} (s_2^{-1} s_1 s_2^{-1}) s_3^{\zeta} s_4^{\eta} u_1 + A_5^{(2)} \\
\end{array}
$$
%u_1^{\times} s_4 s_3^{\beta} A_3 s_3^{\gamma} s_4^{\delta} s_3^{\eps} A_3 s_3^{\zeta} s_4^{\eta} u_1^{\times} + A_5^{(2)}

\end{lemma}
\begin{proof}
This is an easy consequence of the decompositions $A_3 = u_1 u_2 u_1 + u_1 s_2 s_1^{-1} s_2 = u_1 u_2 u_1 +  s_2 s_1^{-1} s_2u_1
=  u_1 u_2 u_1 + u_1 s_2^{-1} s_1 s_2^{-1}=  u_1 u_2 u_1 +  s_2^{-1} s_1 s_2^{-1}u_1$ of
theorem \ref{theodecA3} and of proposition \ref{propmoinsde55}.
\end{proof}

\begin{lemma} {\ } \label{lemreducspec} For $i,j,k,\alpha,\beta,\gamma \in \{ -1,1 \}$,
\begin{enumerate} 
\item $s_4^i s_3^{\alpha} A_3 s_3^{-\alpha} s_4^j s_3^{\beta} A_3 s_3^{\gamma} s_4^k \subset A_5^{(2)}$ unless $i=j=k$
\item $s_4^i s_3^{\alpha} A_3 s_3^{\beta} s_4^j s_3^{\gamma} A_3 s_3^{-\gamma} s_4^k \subset A_5^{(2)}$ unless $i=j=k$
\item $s_4^i s_3^{\alpha} A_3 s_3^{-\alpha} s_4^j s_3^{\beta} A_3 s_3^{\beta} s_4^k \subset A_5^{(2)}$
\item $s_4^i s_3^{\alpha} A_3 s_3^{\alpha} s_4^j s_3^{\beta} A_3 s_3^{-\beta} s_4^k \subset A_5^{(2)}$
\item $s_4^i s_3^{\alpha} A_3 s_3^{-\alpha} s_4^j s_3^{\alpha} A_3 s_3^{-\alpha} s_4^k\subset A_5^{(2)}$
 \end{enumerate}
\end{lemma}
\begin{proof}
We use the formulas $s_3^{-1} (s_2 s_1^{-1} s_2) s_3 = s_2 s_1 (s_3  s_2^{-1} s_3) s_1^{-1}  s_2^{-1}$
and $s_3 (s_2 s_1^{-1} s_2) s_3^{-1} = s_2^{-1} s_1^{-1} (s_3  s_2^{-1} s_3) s_1  s_2$
which are easy to prove and which already hold in the braid group $B_4$, and can be summarized
as $s_3^{-\alpha} (s_2 s_1^{-1} s_2) s_3^{\alpha} = s_2^{\alpha} s_1^{\alpha} (s_3  s_2^{-1} s_3) s_1^{-\alpha}  s_2^{-\alpha}$
for $\alpha \in \{ -1 ,1 \}$. We also use the fact that
$s_2$ (and thus $s_2^{-1}$) commutes with $s_3 s_2 s_1^{-1} s_2 s_3$ (already in the braid group $B_4$),
and similarly $s_2^{-1}$ (and thus $s_2$) commutes with $s_3^{-1} s_2^{-1} s_1 s_2^{-1} s_3^{-1}$.
Together with lemma \ref{lemsimplif212}, this yields
$$
\begin{array}{lcl}
s_4^i s_3^{\alpha} A_3 s_3^{-\alpha} s_4^j s_3^{\beta} A_3 s_3^{\beta} s_4^k 
&\subset & s_4^i s_3^{\alpha} (s_2 s_1^{-1} s_2) s_3^{-\alpha} s_4^j s_3^{\beta} (s_2^{\beta} s_1^{-\beta} s_2^{\beta}) s_3^{\beta} s_4^k  + A_5^{(2)} \\
&\subset & s_4^i (s_3^{\alpha} s_2 s_1^{-1} s_2 s_3^{-\alpha}) s_4^j (s_3^{\beta} s_2^{\beta} s_1^{-\beta} s_2^{\beta} s_3^{\beta}) s_4^k  + A_5^{(2)} \\
&\subset & s_4^i s_2^{-\alpha} s_1^{-\alpha} (s_3  s_2^{-1} s_3) s_1^{\alpha}  s_2^{\alpha} s_4^j (s_3^{\beta} s_2^{\beta} s_1^{-\beta} s_2^{\beta} s_3^{\beta}) s_4^k  + A_5^{(2)} \\
&\subset & s_2^{-\alpha} s_1^{-\alpha} s_4^i  s_3  s_2^{-1} s_3 s_1^{\alpha}   s_4^j s_2^{\alpha} (s_3^{\beta} s_2^{\beta} s_1^{-\beta} s_2^{\beta} s_3^{\beta}) s_4^k  + A_5^{(2)} \\
&\subset & A_3 s_4^i  s_3  s_2^{-1} s_3 s_1^{\alpha}   s_4^j (s_3^{\beta} s_2^{\beta} s_1^{-\beta} s_2^{\beta} s_3^{\beta}) s_2^{\alpha} s_4^k  + A_5^{(2)} \\
&\subset & A_3 s_4^i  s_3  s_2^{-1} s_3 s_1^{\alpha}   s_4^j (s_3^{\beta} s_2^{\beta} s_1^{-\beta} s_2^{\beta} s_3^{\beta}) s_4^k A_3  + A_5^{(2)} \\
&\subset & A_5^{(2)} \\
\end{array}
$$
by proposition \ref{propmoinsde55}, and this proves (3), as well as the symmetric case (4). This also proves (1) in case $\beta = \gamma$. We thus deal with
$$
\begin{array}{lcl}
s_4^i s_3^{\alpha} A_3 s_3^{-\alpha} s_4^j s_3^{\beta} A_3 s_3^{-\beta} s_4^k 
&\subset & s_4^i s_3^{\alpha} (s_2 s_1^{-1} s_2) s_3^{-\alpha} s_4^j s_3^{\beta} (s_2 s_1^{-1} s_2) s_3^{-\beta} s_4^k  + A_5^{(2)} \\
&\subset & s_4^i (s_3^{\alpha} s_2 s_1^{-1} s_2 s_3^{-\alpha}) s_4^j (s_3^{\beta} s_2 s_1^{-1} s_2 s_3^{-\beta}) s_4^k  + A_5^{(2)} \\
&\subset & s_4^i s_2^{-\alpha} s_1^{-\alpha} (s_3  s_2^{-1} s_3) s_1^{\alpha}  s_2^{\alpha} s_4^j s_2^{-\beta} s_1^{-\beta} (s_3  s_2^{-1} s_3) s_1^{\beta}  s_2^{\beta} s_4^k  + A_5^{(2)} \\
&\subset & s_2^{-\alpha} s_1^{-\alpha} s_4^i  s_3  s_2^{-1} s_3 s_1^{\alpha}  s_2^{\alpha}s_2^{-\beta} s_4^j  s_1^{-\beta} (s_3  s_2^{-1} s_3)  s_4^ks_1^{\beta}  s_2^{\beta}  + A_5^{(2)} \\
&\subset & A_5^{(2)} \\
\end{array}
$$
if $\alpha = \beta$ by proposition \ref{propmoinsde55}, and we get (5). Otherwise, $\alpha = - \beta$,
 and
$$
\begin{array}{lcl}
s_4^i s_3^{\alpha} A_3 s_3^{-\alpha} s_4^j s_3^{\beta} A_3 s_3^{-\beta} s_4^k 
&\subset & s_2^{-\alpha} s_1^{-\alpha} s_4^i  s_3  s_2^{-1} s_3 s_1^{\alpha}  s_2^{2 \alpha} s_1^{\alpha} s_4^j   (s_3  s_2^{-1} s_3)  s_4^ks_1^{-\alpha}  s_2^{-\alpha}  + A_5^{(2)} \\
&\subset & A_5^{(2)} \\
\end{array}
$$
unless $i=j=k$ by proposition \ref{propmoinsde55}, and we get (1). (2) is proved symmetrically.

\end{proof}

\begin{corollary} \label{correducspec} {\ }
\begin{enumerate}
\item $s_4 s_3^{-1} (s_2 s_1^{-1} s_2) s_3 s_4^{-1} s_3 (s_2 s_1^{-1} s_2) s_3^{-1} s_4 \in A_4 u_4 A_4 u_4 A_4$
\item $s_4^{-1} s_3 (s_2 s_1^{-1} s_2) s_3 s_4^{-1} s_3 (s_2 s_1^{-1} s_2) s_3^{-1}s_4 \in A_4 u_4 A_4 u_4 A_4$
\end{enumerate}
\end{corollary}

\begin{lemma} \label{lemsymswsws}
$s_4^{-1} w^+ s_4^{-1} w^+ s_4^{-1} \in A_4 s_4 w^- s_4 w^- s_4 A_4 + A_4 u_4 A_4 u_4 A_4$
\end{lemma}
\begin{proof}
We first use $s_2^{-1} s_1 s_2^{-1} \in u_1 s_2 s_1^{-1} s_2 + u_1 u_2 u_1$ and $s_2^{-1}s_1 s_2^{-1} \in s_2 s_1^{-1} s_2 u_1 + u_1 u_2 u_1$
together with proposition \ref{propmoinsde55}
to get
$$
\begin{array}{lcl}
s_4^{-1} w^+ s_4^{-1} w^+ s_4^{-1} 
&=& s_4^{-1} s_3 s_2^{-1} s_1 s_2^{-1} s_3 s_4^{-1} s_3 s_2^{-1} s_1 s_2^{-1} s_3 s_4^{-1} \\
&\subset & u_1 s_4^{-1} s_3 s_2 s_1^{-1} s_2 s_3 s_4^{-1} s_3 s_2^{-1} s_1 s_2^{-1} s_3 s_4^{-1} + u_1 s_4^{-1} s_3  u_2 u_1 s_3 s_4^{-1} s_3 s_2^{-1} s_1 s_2^{-1} s_3 s_4^{-1}\\
&\subset & u_1 s_4^{-1} s_3 s_2 s_1^{-1} s_2 s_3 s_4^{-1} s_3 (s_2^{-1} s_1 s_2^{-1}) s_3 s_4^{-1} +  A_4 u_4 A_4 u_4 A_4 \\
&\subset & u_1 s_4^{-1} s_3 s_2 s_1^{-1} s_2 s_3 s_4^{-1} s_3 s_2 s_1^{-1} s_2 s_3 s_4^{-1} u_1 +  A_4 u_4 A_4 u_4 A_4 \\
&\subset & A_4 (s_3^{-1} s_4^{-1} s_3) s_2 s_1^{-1} s_2 s_3 s_4^{-1} s_3 s_2 s_1^{-1} s_2 (s_3 s_4^{-1} s_3^{-1}) A_4 +  A_4 u_4 A_4 u_4 A_4 \\
&\subset & A_4 s_4 s_3^{-1} s_4^{-1} s_2 s_1^{-1} s_2 s_3 s_4^{-1} s_3 s_2 s_1^{-1} s_2 s_4^{-1} s_3^{-1} s_4 A_4 +  A_4 u_4 A_4 u_4 A_4 \\
&\subset & A_4 s_4 s_3^{-1}  s_2 s_1^{-1} s_2 (s_4^{-1} s_3 s_4^{-1} s_3 s_4^{-1})s_2 s_1^{-1} s_2  s_3^{-1} s_4 A_4 +  A_4 u_4 A_4 u_4 A_4 \\
\end{array}
$$
By lemma \ref{lemdecomp212121} $s_4^{-1} s_3 s_4^{-1} s_3 s_4^{-1}$ is a linear combination of terms of several kinds
\begin{enumerate}
\item elements $x$ of $u_3 u_4$ or $u_4 u_3$, for which we get $s_4 s_3^{-1}  s_2 s_1^{-1} s_2 x s_2 s_1^{-1} s_2  s_3^{-1} s_4 \subset A_4 u_4 A_4 u_4 A_4$
by a direct application of proposition \ref{propmoinsde55}.
\item elements $x$ that can be put in the the form $s_4^{\alpha} s_3^{\beta} s_4^{\gamma}$ with $\alpha = -1$ or $\gamma = -1$,
in which case we get  $s_4 s_3^{-1}  s_2 s_1^{-1} s_2 x s_2 s_1^{-1} s_2  s_3^{-1} s_4 \subset A_4 u_4 A_4 u_4 A_4$ through one application of 
the equation $s_4 s_3^{\dots} s_4^{-1} \in u_3 u_4 u_3$ or $s_4^{-1} s_3^{\dots} s_4 \in u_3 u_4 u_3$, and proposition \ref{propmoinsde55}.
\item the element $s_3^{-1} s_4 s_3^{-1}$, which provides  $s_4 s_3^{-1}  s_2 s_1^{-1} s_2  s_3^{-1} s_4 s_3^{-1} s_2 s_1^{-1} s_2  s_3^{-1} s_4 = s_4 w^-  s_4 w^- s_4 w^- $.
\item the element $x = s_3 s_4^{-1} s_3$, for
which we get $s_4 s_3^{-1}  s_2 s_1^{-1} s_2 x s_2 s_1^{-1} s_2  s_3^{-1} s_4 \subset A_4 u_4 A_4 u_4 A_4$
by  corollary \ref{correducspec} (1).
\item the element $x = s_4^{-1} s_3 s_4^{-1} s_3$, for
which we get 
$$
\begin{array}{lcl}
s_4 s_3^{-1}  s_2 s_1^{-1} s_2 x s_2 s_1^{-1} s_2  s_3^{-1} s_4 
&=& s_4 s_3^{-1}  s_2 s_1^{-1} s_2 s_4^{-1} s_3 s_4^{-1} s_3 s_2 s_1^{-1} s_2  s_3^{-1} s_4 \\
&=& (s_4 s_3^{-1}  s_4^{-1})s_2 s_1^{-1} s_2  s_3  s_4^{-1} s_3 s_2 s_1^{-1} s_2  s_3^{-1} s_4  \\
&=& s_3^{-1} s_4^{-1}  s_3s_2 s_1^{-1} s_2  s_3 s_4^{-1}  s_3 s_2 s_1^{-1} s_2 s_3^{-1} s_4 \\
&\subset& A_4 s_4^{-1}  s_3s_2 s_1^{-1} s_2  s_3 s_4^{-1}  s_3 s_2 s_1^{-1} s_2 s_3^{-1} s_4 \\
&\subset& A_4 u_4 A_4 u_4 A_4\\
\end{array}
$$
by  corollary \ref{correducspec} (2).
\end{enumerate}
This proves the inclusion.

\end{proof}

\begin{lemma} \label{lemreduc2} {\ }
\begin{enumerate}
\item $u_4 A_4 u_4 u_3 u_4 \subset A_4 u_4 A_4 u_4 A_4$
\item $u_4 u_3 u_4 A_4 u_4 \subset A_4 u_4 A_4 u_4 A_4$
\item $s_4^{\beta} u_3 u_2 u_1 u_2 s_3^{\alpha} s_4^{\gamma} s_3^{-\alpha} u_2 u_1 u_2 u_3 s_4^{\beta} \subset A_4 u_4 A_4 u_4 A_4$
\item $s_4^{\alpha} s_3^{\alpha} u_2 u_1 u_2 s_3^{\alpha} s_4^{\gamma} s_3^{-\alpha} u_2 u_1 u_2 u_3 s_4^{\beta} \subset A_4 u_4 A_4 u_4 A_4$
\item $s_4^{\beta} u_3 u_2 u_1 u_2 s_3^{\alpha} s_4^{\gamma} s_3^{-\alpha} u_2 u_1 u_2 s_3^{-\alpha} s_4^{-\alpha} \subset A_4 u_4 A_4 u_4 A_4$
\item $u_4 s_3^{\alpha} u_2 u_1 u_2 s_3^{\alpha} s_4^{\alpha} s_3^{\alpha} u_2 u_1 u_2 s_3^{\alpha} u_4 \subset A_4 u_4 A_4 u_4 A_4$
\item $s_4 w^+ s_4^{-1} w^+ s_4^{-1} \in A_4 u_4 A_4 u_4 A_4$
\item $s_4 w^- s_4 w^- s_4^{-1} \in A_4 u_4 A_4 u_4 A_4$
\end{enumerate}

\end{lemma}
\begin{proof}
Since $A_4 = A_3 u_3 A_3 + A_3 u_3 u_2 u_3 A_3 + A_3 u_3 u_2 u_1 u_2 u_3$
we have $u_4 A_4 u_4 u_3 u_4 \subset
A_3u_4  u_3 A_3 u_4 u_3 u_4 +
A_3u_4  u_3 u_2 u_3 A_3 u_4 u_3 u_4 +
A_3u_4  u_3 u_2 u_1 u_2 u_3 u_4 u_3 u_4 $.
We have $u_4  u_3 A_3 u_4 u_3 u_4 \subset u_4  u_3 u_1 u_2 u_1 u_2  u_4 u_3 u_4 \subset  A_4 u_4 A_4 u_4 A_4$
by proposition \ref{propmoinsde55},  
$$
\begin{array}{clclcl}
& u_4  u_3 u_2 u_3 A_3 u_4 u_3 u_4 &\subset& u_4  (u_3 u_2 u_3 u_2) u_1 u_2 u_1 u_4 u_3 u_4 \\
= &u_4  u_2 u_3 u_2 u_3 u_1 u_2 u_1 u_4 u_3 u_4
&=& u_2 u_4   u_3 u_2 u_3 u_1 u_4 u_2 u_1  u_3 u_4 
\subset &A_4 u_4 A_4 u_4 A_4\\
\end{array}
$$
by proposition \ref{propmoinsde55}, and $u_4  u_3 u_2 u_1 u_2 u_3 u_4 u_3 u_4  \subset A_4 u_4 A_4 u_4 A_4$
by proposition \ref{propmoinsde55}. This proves (1). (2) is deduced from (1) by applying $\Psi$.
%the natural anti-automorphism
%of $A_5$. 
We turn to (3).
Since $s_4^{\beta} u_3 u_2 u_1 u_2 (s_3^{\alpha} s_4^{\gamma} s_3^{-\alpha}) u_2 u_1 u_2 u_3 s_4^{\beta} 
=
s_4^{\beta} u_3 u_2 u_1 u_2 s_4^{-\alpha} s_3^{\gamma} s_4^{\alpha} u_2 u_1 u_2 u_3 s_4^{\beta}
=
s_4^{\beta} u_3 s_4^{-\alpha} u_2 u_1 u_2  s_3^{\gamma}  u_2 u_1 u_2 s_4^{\alpha} u_3 s_4^{\beta}$
and either  $s_4^{\beta} u_3 s_4^{-\alpha} \subset u_3 u_4 u_3$ or $s_4^{\alpha} u_3 s_4^{\beta} \subset u_3 u_4 u_3$. In both
cases we get an element of $A_4 u_4 A_4 u_4 u_3 u_4 A_4 \subset A_4 u_4 A_4 u_4 A_4$ or  $A_4 u_4 u_3 u_4 A_4 u_4A_4 \subset A_4 u_4 A_4 u_4 A_4$ by (1) or (2), and this proves (3).
(4) and (5) are similar and left to the reader.

\begin{comment}
Since $s_4^{\beta} u_3 u_2 u_1 u_2 (s_3^{\alpha} s_4^{\alpha} s_3^{-\alpha}) u_2 u_1 u_2 u_3 s_4^{\beta} 
=
s_4^{\beta} u_3 u_2 u_1 u_2 s_4^{-\alpha} s_3^{\alpha} s_4^{\alpha} u_2 u_1 u_2 u_3 s_4^{\beta}
=
s_4^{\beta} u_3 s_4^{-\alpha} u_2 u_1 u_2  s_3^{\alpha}  u_2 u_1 u_2 s_4^{\alpha} u_3 s_4^{\beta}$
and either  $s_4^{\beta} u_3 s_4^{-\alpha} \subset u_3 u_4 u_3$ or $s_4^{\alpha} u_3 s_4^{\beta} \subset u_3 u_4 u_3$. In both
bases we get an element of $A_4 u_4 A_4 u_4 u_3 u_4 A_4 \subset A_4 u_4 A_4 u_4 A_4$ or  $A_4 u_4 u_3 u_4 A_4 u_4A_4 \subset A_4 u_4 A_4 u_4 A_4$ by (1) or (2), and this proves (3).
\end{comment}
Now $$
\begin{array}{lclcl}
u_4 s_3^{\alpha} u_2 u_1 u_2 s_3^{\alpha} s_4^{\alpha} s_3^{\alpha} u_2 u_1 u_2 s_3^{\alpha} u_4 
&=& u_4 s_3^{\alpha} u_2 u_1 u_2 s_4^{\alpha} s_3^{\alpha} s_4^{\alpha} u_2 u_1 u_2 s_3^{\alpha} u_4 \\
&=& (u_4 s_3^{\alpha} s_4^{\alpha})u_2 u_1 u_2  s_3^{\alpha}  u_2 u_1 u_2 (s_4^{\alpha} s_3^{\alpha} u_4) \\
&\subset& u_3 u_4 u_3 u_2 u_1 u_2 u_3 u_2 u_1 u_2 u_3 u_4 u_3
&\subset &A_4 u_4 A_4 u_4 A_4\\
\end{array}$$ and this proves (6).

To prove (7), we compute, using $b,b'$ for elements in $u_2u_1 u_2$,
$$
\begin{array}{lcl}
s_4 w^+ s_4^{-1} w^+ s_4^{-1} 
&\subset & s_4 s_3 b s_3 s_4^{-1} s_3 b' s_3 s_4^{-1} \\
&\subset & s_3^{-1} (s_3 s_4 s_3) b s_3 s_4^{-1} s_3 b' s_3 s_4^{-1} \\
&\subset & s_3^{-1} s_4 s_3 s_4 b s_3 s_4^{-1} s_3 b' s_3 s_4^{-1} \\
&\subset & s_3^{-1} s_4 s_3  b (s_4 s_3 s_4^{-1}) s_3 b' s_3 s_4^{-1} \\
&\subset & s_3^{-1} s_4 s_3  b s_3^{-1} s_4 s_3^2 b' s_3 s_4^{-1} \\
\end{array}
$$
Now $s_3^2 \in R + R s_3 + R s_3^{-1}$, and 
$s_3^{-1} s_4 s_3  b s_3^{-1} s_4  b' s_3 s_4^{-1}  \in A_4 u_4 A_4 u_4 A_4$ by proposition \ref{propmoinsde55},
$$s_3^{-1} s_4 s_3  b s_3^{-1} s_4 s_3 b' s_3 s_4^{-1} \in A_4 u_4 A_4 u_4 A_4$$ by lemma \ref{lemreducspec} (3),
and $s_3^{-1} s_4 s_3  b s_3^{-1} s_4 s_3^{-1} b' s_3 s_4^{-1} \subset A_5^{(2)}$ by lemma \ref{lemreducspec} (1).
\begin{comment}
$$
\begin{array}{lcl}
s_3^{-1} s_4 s_3  b s_3^{-1} s_4 s_3^{-1} b' s_3 s_4^{-1} 
&\subset & s_3^{-1} s_4 s_3  s_2^{-1} s_1 s_2^{-1} s_3^{-1} s_4 s_3^{-1} s_2^{-1} s_1 s_2^{-1} s_3 s_4^{-1}  + A_4 u_4 A_4 u_4 A_4 \\
&\subset & s_3^{-1} s_4 (s_3  s_2^{-1} s_1 s_2^{-1} s_3^{-1}) s_4  (s_3^{-1} s_2^{-1} s_1 s_2^{-1} s_3)  s_4^{-1}  + A_4 u_4 A_4 u_4 A_4 \\
&\subset & s_3^{-1} s_4 u_2 u_1 u_3 u_2 u_3 u_1 u_2s_4 u_2 u_1 u_3 u_2 u_3 u_1 u_2  s_4^{-1}  + A_4 u_4 A_4 u_4 A_4 \\
&\subset & s_3^{-1} u_2 u_1 s_4  u_3 u_2 u_3 u_1 u_2s_4 u_1 u_3 u_2 u_3   s_4^{-1}u_1 u_2  + A_4 u_4 A_4 u_4 A_4 \\
&\subset&  A_4 u_4 A_4 u_4 A_4\\
\end{array}
$$
by proposition \ref{propmoinsde55}.
(\fbox{idem, les formules utilisées sont mentionnées après\dots})
\end{comment}
The proof of (8) is similar :
$$
\begin{array}{lcl}
s_4 w^- s_4 w^- s_4^{-1} 
&\subset & s_4 s_3^{-1} b s_3^{-1} s_4 s_3^{-1} b' s_3^{-1} s_4^{-1} \\
&\subset & s_4 s_3^{-1} b s_3^{-1} s_4 s_3^{-1} b' (s_3^{-1} s_4^{-1} s_3^{-1})\\
&\subset & s_4 s_3^{-1} b s_3^{-1} s_4 s_3^{-1} b' s_4^{-1} s_3^{-1} s_4^{-1}\\
&\subset & s_4 s_3^{-1} b s_3^{-1} (s_4 s_3^{-1} s_4^{-1}) b'  s_3^{-1} s_4^{-1}\\
&\subset & s_4 s_3^{-1} b s_3^{-1} s_3^{-1} s_4^{-1} s_3 b'  s_3^{-1} s_4^{-1}\\
&\subset & s_4 s_3^{-1} b s_3^{-2} s_4^{-1} s_3 b'  s_3^{-1} s_4^{-1}\\
\end{array}
$$
Now $s_3^{-2} \in R + R s_3 + R s_3^{-1}$, and we conclude similarly.

\end{proof}

\begin{lemma} {\ } \label{lemtransf1}
\begin{enumerate}
\item $s_4 s_3^{-1} A_3 s_3 s_4 s_3 A_3 s_3^{-1} s_4 \subset u_3 s_4^- w^+ s_4 w^- s_4 + A_5^{(2)}$
\item $(s_4 w^+ s_4^{-1} w^+ s_4)s_3^{-1} \in s_3^{-1} (s_4 w^+ s_4^{-1} w^+ s_4) + A_5^{(2)}$
\item $s_4 w^+ s_4^{-1} w^+ s_4 \in A_3^{\times} s_4 (s_3 s_2^{-1} s_3)(s_1 s_2^{-1} s_1) s_4 (s_3 s_2^{-1} s_3) s_4 A_4^{\times} + A_5^{(2)}$
\item $s_4 w^- s_4 w^+ s_4^{-1} \in A_4^{\times} s_4^{-1} w^+ s_4 w^- s_4 A_4^{\times} + A_5^{(2)} $
\item $s_4 w^- s_4^{-1} w^+ s_4^{-1} \in A_4^{\times} s_4^{-1} w^+ s_4^{-1} w^- s_4A_4^{\times} + A_5^{(2)} $ 
\item $s_4 s_3 A_3 s_3^{-1} s_4 s_3^{-1} A_3 s_3 s_4 \subset u_3 s_4 w^+ s_4^{-1} w^+ s_4 u_3 + A_5^{(2)}$
\end{enumerate}
\end{lemma}
\begin{proof}
We first prove (1). By lemma \ref{lemsimplif212} we need to prove
$s_4 s_3^{-1} s_2 s_1^{-1} s_2 s_3 s_4 s_3 s_2 s_1^{-1} s_2  s_3^{-1} s_4 \subset u_3 s_4^- w^+ s_4 w^- s_4 + A_5^{(2)}$,
and we get, using proposition \ref{propmoinsde55}

$$
\begin{array}{lcl}
s_4 s_3^{-1} s_2 s_1^{-1} s_2 (s_3 s_4 s_3) s_2 s_1^{-1} s_2  s_3^{-1} s_4 
& = & s_4 s_3^{-1} s_2 s_1^{-1} s_2 s_4 s_3 s_4 s_2 s_1^{-1} s_2  s_3^{-1} s_4  \\
& = & (s_4 s_3^{-1} s_4) s_2 s_1^{-1} s_2  s_3 s_4 s_2 s_1^{-1} s_2  s_3^{-1} s_4  \\
& \subset & u_3s_4^{-1} s_3 s_4^{-1} s_2 s_1^{-1} s_2  s_3 s_4 s_2 s_1^{-1} s_2  s_3^{-1} s_4 
\\ & & +
u_3u_4 u_3 s_2 s_1^{-1} s_2  s_3 s_4 s_2 s_1^{-1} s_2  s_3^{-1} s_4 \\
& \subset & u_3s_4^{-1} s_3 s_4^{-1} s_2 s_1^{-1} s_2  s_3 s_4 s_2 s_1^{-1} s_2  s_3^{-1} s_4 + A_5^{(2)} \\
& \subset & u_3s_4^{-1} s_3  s_2 s_1^{-1} s_2 (s_4^{-1} s_3 s_4) s_2 s_1^{-1} s_2  s_3^{-1} s_4 + A_5^{(2)} \\
& \subset & u_3s_4^{-1} s_3  (s_2 s_1^{-1} s_2) s_3 s_4 s_3^{-1} s_2 s_1^{-1} s_2  s_3^{-1} s_4 + A_5^{(2)} \\
& \subset & u_3s_4^{-1} s_3  s_2^{-1} s_1 s_2^{-1} s_3 s_4 s_3^{-1} s_2 s_1^{-1} s_2  s_3^{-1} s_4 + A_5^{(2)} \\
& \subset & u_3s_4^{-1} w^+ s_4 w^- s_4 + A_5^{(2)}. \\ 
\end{array}
$$
We now prove (2). We have, using $s_3 s_4^{-1} s_3 s_4^{-1} \in R s_4 s_3^{-1} s_4 s_3^{-1}  + u_3 u_4 u_3 + u_4 u_3 u_4$
and $s_3 s_4^{-1} s_3 s_4^{-1} - s_4^{-1} s_3 s_4^{-1} s_3 \in u_3 u_4 + u_4 u_3$ by lemma \ref{lemdecomp1212}, we get
$$
\begin{array}{lcl}
s_4 w^+ s_4^{-1} w^+ s_4.s_3^{-1} 
&=&  s_4 w^+ s_4^{-1} s_3 s_2^{-1} s_1 s_2^{-1} (s_3 s_4s_3^{-1}) \\ 
&=&  s_4 w^+ s_4^{-1} s_3 s_2^{-1} s_1 s_2^{-1} s_4^{-1} s_3s_4 \\ 
&=&  s_4 w^+ s_4^{-1} s_3 s_4^{-1} s_2^{-1} s_1 s_2^{-1}  s_3s_4 \\ 
&=&  s_4 s_3 s_2^{-1} s_1 s_2^{-1} (s_3 s_4^{-1} s_3 s_4^{-1}) s_2^{-1} s_1 s_2^{-1}  s_3s_4 \\ 
&\in &  s_4 s_3 s_2^{-1} s_1 s_2^{-1} s_4^{-1} s_3 s_4^{-1} s_3 s_2^{-1} s_1 s_2^{-1}  s_3s_4 + A_5^{(2)}\\ 
&\subset&  (s_4 s_3 s_4^{-1}) s_2^{-1} s_1 s_2^{-1}  s_3 s_4^{-1} s_3 s_2^{-1} s_1 s_2^{-1}  s_3s_4+ A_5^{(2)} \\ 
&\subset&  s_3^{-1} s_4 s_3 s_2^{-1} s_1 s_2^{-1}  s_3 s_4^{-1} s_3 s_2^{-1} s_1 s_2^{-1}  s_3s_4 + A_5^{(2)}\\ 
&\subset&  s_3^{-1}. s_4 w^+ s_4^{-1} w^+ s_4 + A_5^{(2)} \\ 
\end{array}
$$
We now prove (3). We have
$$
\begin{array}{lcl}
s_4 w^+ s_4^{-1} w^+ s_4 &=& s_4 s_3 (s_2^{-1} s_1  s_2^{-1}) s_3 s_4^{-1} s_3 (s_2^{-1} s_1  s_2^{-1}) s_3 s_4 \\
&\in & s_4 s_3 s_2 s_1^{-1}  s_2 s_3 s_4^{-1} s_3 s_2 s_1^{-1}  s_2 s_3 s_4 + A_5^{(2)} \\
&\subset & s_3^{-1} (s_3 s_4 s_3) s_2 s_1^{-1}  s_2 s_3 s_4^{-1} s_3 s_2 s_1^{-1}  s_2 (s_3 s_4s_3 )s_3^{-1} + A_5^{(2)} \\
&\subset & s_3^{-1} s_4 s_3 s_4 s_2 s_1^{-1}  s_2 s_3 s_4^{-1} s_3 s_2 s_1^{-1}  s_2 s_4 s_3 s_4 s_3^{-1} + A_5^{(2)} \\
&\subset & s_3^{-1} s_4 s_3  s_2 s_1^{-1}  s_2 s_4 s_3 (s_4^{-1} s_3 s_4)s_2 s_1^{-1}  s_2  s_3 s_4 s_3^{-1} + A_5^{(2)} \\
&\subset & s_3^{-1} s_4 s_3  s_2 s_1^{-1}  s_2 s_4 (s_3^2) s_4 s_3^{-1})s_2 s_1^{-1}  s_2  s_3 s_4 s_3^{-1} + A_5^{(2)} \\
&\subset & Rs_3^{-1} s_4 s_3  s_2 s_1^{-1}  s_2 s_4^2 s_3^{-1}s_2 s_1^{-1}  s_2  s_3 s_4 s_3^{-1}  \\
& & +Rs_3^{-1} s_4 s_3  s_2 s_1^{-1}  s_2 (s_4 s_3 s_4) s_3^{-1}s_2 s_1^{-1}  s_2  s_3 s_4 s_3^{-1}  \\
& & +R^{\times} s_3^{-1} s_4 s_3  s_2 s_1^{-1}  s_2 s_4 s_3^{-1} s_4 s_3^{-1}s_2 s_1^{-1}  s_2  s_3 s_4 s_3^{-1} + A_5^{(2)} \\
&\subset &    Rs_3^{-1} s_4 s_3  s_2 s_1^{-1}  s_2 (s_3 s_4 s_3) s_3^{-1}s_2 s_1^{-1}  s_2  s_3 s_4 s_3^{-1}  \\
& & +R^{\times} s_3^{-1} s_4 s_3  s_2 s_1^{-1}  s_2 s_4 s_3^{-1} s_4 s_3^{-1}s_2 s_1^{-1}  s_2  s_3 s_4 s_3^{-1} + A_5^{(2)} \\
&\subset &    Rs_3^{-1} s_4 s_3  s_2 s_1^{-1}  s_2 s_3 s_4 s_2 s_1^{-1}  s_2  s_3 s_4 s_3^{-1}  \\
& & +R^{\times} s_3^{-1} s_4 s_3  s_2 s_1^{-1}  s_2 s_4 s_3^{-1} s_4 s_3^{-1}s_2 s_1^{-1}  s_2  s_3 s_4 s_3^{-1} + A_5^{(2)} \\
&\subset & R^{\times} s_3^{-1} s_4 s_3  s_2 s_1^{-1}  s_2 s_4 s_3^{-1} s_4 s_3^{-1}s_2 s_1^{-1}  s_2  s_3 s_4 s_3^{-1} + A_5^{(2)} \\
&\subset & R^{\times} s_3^{-1} (s_4 s_3  s_4)s_2 s_1^{-1}  s_2  s_3^{-1} s_4 s_3^{-1}s_2 s_1^{-1}  s_2  s_3 s_4 s_3^{-1} + A_5^{(2)} \\
&\subset & R^{\times} s_3^{-1} s_3 s_4  s_3s_2 s_1^{-1}  s_2  s_3^{-1} s_4 s_3^{-1}s_2 s_1^{-1}  s_2  s_3 s_4 s_3^{-1} + A_5^{(2)} \\
&\subset & R^{\times}  s_4  (s_3s_2 s_1^{-1}  s_2  s_3^{-1}) s_4 (s_3^{-1}s_2 s_1^{-1}  s_2  s_3) s_4 s_3^{-1} + A_5^{(2)} \\
&\subset & R^{\times}  s_4 s_2^{-1} s_1^{-1} (s_3 s_2^{-1} s_3) s_1 s_2  s_4 s_2 s_1 (s_3 s_2^{-1} s_3) s_1^{-1} s_2^{-1}  s_4 s_3^{-1} + A_5^{(2)} \\
&\subset & A_3^{\times}  s_4  (s_3 s_2^{-1} s_3) s_1 s_2 s_2 s_1 s_4  (s_3 s_2^{-1} s_3)   s_4 s_1^{-1} s_2^{-1} s_3^{-1} + A_5^{(2)} \\
\end{array}
$$
and then $s_1 s_2 s_2 s_1 =  s_1 s_2^2 s_1 \in R^{\times} s_1 s_2^{-1} s_1 + R s_1 s_2 s_1 + R s_1^2$. Since
$s_4  (s_3 s_2^{-1} s_3) s_1 s_2  s_1 s_4  (s_3 s_2^{-1} s_3)   s_4 =
s_4  (s_3 s_2^{-1} s_3) s_2 s_1  s_2 s_4  (s_3 s_2^{-1} s_3)   s_4 \subset s_4 (u_3 u_2 u_3 u_2) s_1 s_2 s_4  u_3 u_2 u_3 s_4=
s_4 u_2 u_3 u_2 u_3 s_1 s_2 s_4  u_3 u_2 u_3 s_4
= u_2s_4  u_3 u_2 u_3 s_1 s_2 s_4  u_3 u_2 u_3 s_4 \subset A_5^{(2)}$ by proposition \ref{propmoinsde55} and similarly
$s_4  (s_3 s_2^{-1} s_3) s_1 s_1 s_4  (s_3 s_2^{-1} s_3)   s_4 \in  s_4  (s_3 s_2^{-1} s_3) u_1 s_4  (s_3 s_2^{-1} s_3)   s_4 \subset A_5^{(2)}$, this proves
$$s_4 w^+ s_4^{-1} w^+ s_4 \in A_3^{\times} s_4  (s_3 s_2^{-1} s_3) (s_1 s_2^{-1} s_1) s_4  (s_3 s_2^{-1} s_3)   s_4 A_4^{\times} + A_5^{(2)}.$$
We now prove (4).
$$
\begin{array}{lcl}
s_4 w^- s_4 w^+ s_4^{-1} 
&=& s_4 s_3^{-1} (s_2 s_1^{-1} s_2) s_3^{-1} s_4 s_3 (s_2^{-1} s_1 s_2^{-1} s_3 s_4^{-1} \\
&=& s_3^{-1} (s_3 s_4 s_3^{-1}) (s_2 s_1^{-1} s_2) s_3^{-1} s_4 s_3 (s_2^{-1} s_1 s_2^{-1} (s_3 s_4^{-1} s_3^{-1} )s_3 \\
&=& s_3^{-1} (s_4^{-1} s_3 s_4) (s_2 s_1^{-1} s_2) s_3^{-1} s_4 s_3 (s_2^{-1} s_1 s_2^{-1} (s_4^{-1} s_3^{-1} s_4 )s_3 \\
&=& s_3^{-1} s_4^{-1} s_3  (s_2 s_1^{-1} s_2) s_4 s_3^{-1} s_4 s_3s_4^{-1} (s_2^{-1} s_1 s_2^{-1}  s_3^{-1} s_4 s_3 \\
\end{array}
$$
We now prove (4).
$$
\begin{array}{lcl}
s_4 w^- s_4 w^+ s_4^{-1} 
&=& s_4 s_3^{-1} (s_2 s_1^{-1} s_2) s_3^{-1} s_4 s_3 (s_2^{-1} s_1 s_2^{-1})s_3 s_4^{-1} \\
&=& s_4 s_3^{-1} (s_2 s_1^{-1} s_2) s_3^{-1} s_4 s_3 (s_2^{-1} s_1 s_2^{-1})s_3 s_4^{-1} \\
&=& s_3^{-1} (s_3 s_4 s_3^{-1}) (s_2 s_1^{-1} s_2) s_3^{-1} s_4 s_3 (s_2^{-1} s_1 s_2^{-1})(s_3 s_4^{-1} s_3^{-1}) s_3 \\
&=& s_3^{-1} s_4^{-1} s_3 s_4 (s_2 s_1^{-1} s_2) s_3^{-1} s_4 s_3 (s_2^{-1} s_1 s_2^{-1})s_4^{-1} s_3^{-1} s_4 s_3 \\
&=& s_3^{-1} s_4^{-1} s_3  (s_2 s_1^{-1} s_2) s_4 s_3^{-1} s_4 s_3 s_4^{-1}(s_2^{-1} s_1 s_2^{-1}) s_3^{-1} s_4 s_3 \\
\end{array}
$$
We have $s_4 s_3^{-1} (s_4 s_3 s_4^{-1})= 
s_4 s_3^{-1} s_3^{-1} s_4 s_3)
=
s_4 s_3^{-2}  s_4 s_3 \in R^{\times} s_4 s_3 s_4 s_3 + u_4 s_3 + R s_4 s_3^{-1} s_4 s_3
= R^{\times} (s_4 s_3 s_4) s_3 + u_4 s_3 + R s_4 s_3^{-1} s_4 s_3
= R^{\times} s_3 s_4 s_3^2 + u_4 s_3 + R s_4 s_3^{-1} s_4 s_3
\subset R^{\times} s_3 s_4 s_3^{-1} + R s_3 s_4 s_3 + R s_3 s_4 + u_4 s_3 + R s_4 s_3^{-1} s_4 s_3 $.
We have 
$$
s_4^{-1} s_3  (s_2 s_1^{-1} s_2) (u_4 s_3)(s_2^{-1} s_1 s_2^{-1}) s_3^{-1} s_4 \subset A_5^{(2)}
$$
by lemma \ref{lemreduc2} (2) ;
$$
s_4^{-1} s_3  (s_2 s_1^{-1} s_2) (s_3 s_4)(s_2^{-1} s_1 s_2^{-1}) s_3^{-1} s_4 \subset A_5^{(2)}
$$
by lemma \ref{lemreduc2} (1) ;
$$
s_4^{-1} s_3  (s_2 s_1^{-1} s_2) (s_3 s_4 s_3)(s_2^{-1} s_1 s_2^{-1}) s_3^{-1} s_4 \subset A_5^{(2)}
$$
by lemma \ref{lemreducspec} (4) ; 
$$
\begin{array}{lcl}
s_4^{-1} s_3  (s_2 s_1^{-1} s_2) (s_4 s_3^{-1} s_4 s_3)(s_2^{-1} s_1 s_2^{-1}) s_3^{-1} s_4 &=& 
s_4^{-1} s_3  (s_2 s_1^{-1} s_2) s_4 (s_3^{-1} s_4 s_3)(s_2^{-1} s_1 s_2^{-1}) s_3^{-1} s_4 \\
&=& s_4^{-1} s_3  (s_2 s_1^{-1} s_2) s_4 (s_4 s_3 s_4^{-1})(s_2^{-1} s_1 s_2^{-1}) s_3^{-1} s_4 \\
&=& s_4^{-1} s_3  (s_2 s_1^{-1} s_2) s_4^2 s_3 s_4^{-1}(s_2^{-1} s_1 s_2^{-1}) s_3^{-1} s_4 \\
&=& s_4^{-1} s_3  s_4^2 (s_2 s_1^{-1} s_2)  s_3 (s_2^{-1} s_1 s_2^{-1})(s_4^{-1}  s_3^{-1} s_4) \\
&=& s_4^{-1} s_3  s_4^2 (s_2 s_1^{-1} s_2)  s_3 (s_2^{-1} s_1 s_2^{-1})s_3  s_4^{-1} s_3^{-1} \\
& \in & A_5^{(2)} \\
\end{array}
$$
by lemma \ref{lemreduc2}. It follows that 
$s_4 w^- s_4 w^+ s_4^{-1}  \in u_3^{\times} s_4^{-1} s_3  (s_2 s_1^{-1} s_2)  s_3 s_4 s_3^{-1} (s_2^{-1} s_1 s_2^{-1}) s_3^{-1} s_4 u_3^{\times} + A_5^{(2)}
$
hence $s_4 w^- s_4 w^+ s_4^{-1} \in u_3^{\times} s_4^{-1} w^+ s_4 w^- s_4 u_3^{\times} + A_5^{(2)}$ 

The proof of (5) is similar : one first gets
$$s_4 w^- s_4^{-1} w^+ s_4^- = s_3^{-1} s_4^{-1} s_3 (s_2 s_1^{-1} s_2) s_4 s_3^{-1} s_4^{-1} s_3 s_4^{-1} (s_2^{-1} s_1 s_2^{-1}) s_3^{-1} s_4 s_3$$
and then writes down  $(s_4 s_3^{-1} s_4^{-1}) s_3 s_4^{-1} = s_3^{-1} s_4^{-1} s_3^2 s_4^{-1} \in R^{\times} s_3^{-1} (s_4^{-1} s_3^{-1} s_4^{-1}) +
 R (s_3^{-1} s_4^{-1} s_3) s_4^{-1} + R s_3^{-1} s_4^{-2} =
 R^{\times} s_3^{-2}  s_4^{-1} s_3^{-1} +
 R s_4 s_3^{-1} s_4^{-2} + R s_3^{-1} s_4^{-2} 
\subset R^{\times} s_3  s_4^{-1} s_3^{-1} +R s_3^{-1}  s_4^{-1} s_3^{-1}  + R  s_4^{-1} s_3^{-1}
+  R s_4 s_3^{-1} u_4 + R s_3^{-1} s_4^{-2} $ ; one then shows using the same arguments as before that all terms but $R^{\times} s_3  s_4^{-1} s_3^{-1}$ provide
an element of $A_5^{(2)}$, thus 
$$
\begin{array}{lcl}
s_4 w^- s_4^{-1} w^+ s_4^- &\in& u_3^{\times}  s_4^{-1} s_3 (s_2 s_1^{-1} s_2)s_3  s_4^{-1} s_3^{-1}  (s_2^{-1} s_1 s_2^{-1}) s_3^{-1} s_4 u_3^{\times} + A_5^{(2)} \\
& \in & u_3^{\times} s_4^{-1} w^+ s_4^{-1} w^-  s_4 u_3^{\times} + A_5^{(2)} \\
\end{array}
$$

We prove (6).
$$
\begin{array}{lclr}
s_4 s_3 A_3 s_3^{-1} s_4 s_3^{-1} A_3  s_3 s_4
& \subset &s_4 s_3 A_3 s_3^{-1} s_4 s_3^{-1} A_3  (s_3 s_4 s_3 )s_3^{-1}\\
& \subset &s_4 s_3 A_3 s_3^{-1} s_4 s_3^{-1} A_3  s_4 s_3 s_4 s_3^{-1}\\
& \subset &s_4 s_3 A_3 (s_3^{-1} s_4 s_3^{-1} s_4) A_3   s_3 s_4 s_3^{-1}\\
& \subset &s_4 s_3 A_3 s_4^{-1} s_3 s_4^{-1} s_3 A_3   s_3 s_4 s_3^{-1}+A_5^{(2)}& \mbox{(lemmas \ref{lemdecomp1212} and \ref{lemreduc2} (1))}\\ 
& \subset &(s_4 s_3 s_4^{-1}) A_3  s_3 s_4^{-1} s_3 A_3   s_3 s_4 s_3^{-1}+A_5^{(2)}\\ 
& \subset & s_3^{-1} s_4 s_3 A_3  s_3 s_4^{-1} s_3 A_3   s_3 s_4 s_3^{-1}+A_5^{(2)}\\ 
& \subset & u_3 s_4 w^+ s_4^{-1} w^+ s_4 u_3+A_5^{(2)} & \mbox{(lemma \ref{lemsimplif212})}\\ 
\end{array}
$$ 
%by lemmas \ref{lemdecomp1212} and \ref{lemreduc2} (1).
\end{proof}

\begin{proposition} \label{propdecomp1bimoduuu}
$$ 
\begin{array}{lcl}
u_4 u_3 u_2 u_1 u_2 u_3 u_4u_3 u_2 u_1 u_2 u_3 u_4 &\subset&   A_4 s_4 w^- s_4 w^- s_4 A_4
+
A_4 s_4 w^+ s_4^{-1} w^+ s_4 A_4  
%\\ & & 
+
A_4 s_4^{-1} w^- s_4 w^- s_4^{-1} A_4 \\
& & +
A_4 s_4 w^- s_4 w^+ s_4^{-1} A_4 
%\\ & & 
+
A_4 s_4 w^- s_4^{-1} w^+ s_4^{-1} A_4 +A_5^{(2)}\\
\end{array}$$
\end{proposition}

\begin{proof}

We first note that, by lemma \ref{lemsymswsws} and lemma \ref{lemtransf1} (4) and (5),
the right-hand side (RHS) of the statement is invariant under $\Phi$ and $\Psi$.
%the usual automorphism and anti-automorphism.
We now consider an expression of the form $s_4^{?} s_3^{\alpha} u_2 u_1 u_2 s_3^{\beta} s_4^{?} s_3^{\gamma}
u_2 u_1 u_2 s_3^{\delta} s_4^{?}$ with $\alpha,\beta,\gamma \in \{ -1 1 \}$. By lemma \ref{lemreducspec}
we can assume $\alpha = \beta$ and$\gamma = \delta$, except for the
expression $s_4^{\eps} s_3^{\alpha}u_2 u_1 u_2 s_3^{-\alpha} s_4^{\eps} s_3^{-\alpha}
u_2 u_1 u_2 s_3^{\alpha} s_4^{\eps}$.
%which is easily checked to lie
%inside $A_3 s_4^{\eps} s_3 s_2^{-1} s_3  s_1 s_2  s_4^{\eps} s_2 s_1  s_3^{-\alpha}
%s_2 s_3 s_4^{\eps}A_3 + A_5^{(2)}$ by using the identities
%$s_3^{-\alpha} (s_2 s_1^{-1} s_2) s_3^{\alpha} = s_2^{\alpha} s_1^{\alpha} (s_3 s_2^{-1} s_3) s_1^{-\alpha} s_2^{-\alpha}$
%and lemma \ref{lemsimplif212}. 
Up to applying $\Phi$,
%the usual automorphism, 
we can moreover assume $\eps = 1$, and we get the conclusion
by lemma \ref{lemtransf1} (6) for $\alpha = 1$, by lemma \ref{lemtransf1} (1) and (4) for $\alpha = -1$.

We can now assume $\alpha = \beta$, $\gamma = \delta$, and still $\eps = 1$. By lemma \ref{lemsimplif212}
this reduces our examination to expressions $x = s_3 w^{\alpha} s_4^{\eps} w^{\beta} s_4^{\eta}$
for new parameters $\alpha,\eps,\beta,\eta \in \{ -1,1 \}$. If $\alpha = \beta = \eps$,
we have $x \in A_5^{(2)}$ by lemma \ref{lemreduc2} (6) ; if $\alpha= \beta = -\eps$,
we get $x \in A_5^{(2)}$ if in addition $\eta = -1$, by lemma \ref{lemreduc2} (7) and (8),
and $x = s_4 w^{\alpha} s_4^{-\alpha} w^{\alpha} s_4 \in RHS$ otherwise.
As a consequence, we can reduce to the case $\alpha = -\beta$,
that is $x = s_4 w^{\alpha} s_4^{\eps} w^{-\alpha} s_4^{\eta}$. If $\alpha = 1$, $x \in A_5^{(2)}$ by lemma
\ref{lemreduc2} (4). If $\alpha = -1$, all the possibilities for $x$ clearly lie in the RHS, except for
$s_4 w^- s_4^{-1} w^+ s_4$, which belongs to $A_5^{(2)}$ by lemma \ref{lemreduc2} (3).
This concludes the proof.
\end{proof}

\subsection{Image of the center of the braid group in $A_5^{(3)}/ A_5^{(2)}$}

Recall that the center of the braid group $B_n$ is infinite cyclic, generated for $n \geq 3$  by
$c_n = (s_1 \dots s_{n-1})^n$, and that this generator can be written as
$c_n = c_{n-1} y_n = y_n c_{n-1} = y_n y_{n-1} \dots y_3 y_2$
where the $y_n \in B_n \setminus B_{n-1}$ under the usual inclusions $B_2 \subset B_3 \subset \dots \subset B_{n-1}$
form another family of commuting elements defined by $y_2 = s_1^2$ and $y_{n+1} = s_n y_n s_n = s_n s_{n-1} \dots s_2 s_1^2 s_2 \dots s_{n-1} s_n$.

We let $c = c_5 = (s_1 s_2 s_3 s_4)^5 = (s_4 s_3 s_2 s_1)^5$. The center of $G_{32}$ is cyclic
of order $6$ and is generated by the image of $c$. 
We let $w_0 = y_4 = s_3 s_2 s_1^2 s_2 s_3 = c_4 c_3^{-1}$, which by definition commutes with $B_3$,
and $\delta = y_5 = s_4 s_3 s_2 s_1^2 s_2 s_3 s_4=c_5 c_4^{-1}$ which commutes with $B_4$.

\begin{comment}
It is classical that $\delta = c_5 c_4^{-1} = c_4^{-1} c_5$, where $c_n = (s_1\dots s_{n-1})^n$
denotes the usual generator of the center of $B_n$. In particular, $\delta$ centralizes $B_4$ in
$B_5$, and $c \in B_4 \delta = \delta B_4$. 

Thus $c^k \in B_4 \delta^k$ for all $k$.
\end{comment}

We first need a preparatory lemma.

\begin{lemma} {\ } \label{lemsimplifw0}
\begin{enumerate}
\item In $A_4$, $s_4^{\alpha} w_0^2 s_4^{\beta} w_0 s_4^{\gamma} \in A_3^{\times}s_4^{\alpha} w_0^{-1} s_4^{\beta} w_0 s_4^{\gamma} 
+ A_3 s_4^{\alpha} w_0 s_4^{\beta} w_0 s_4^{\gamma} + A_5^{(2)}$  
\item For all $\alpha,\beta,\gamma,\delta,\eps \in \{ -1 ,1 \} = \{ -,+ \}$,
$$
\begin{array}{lcl}
s_4^{\alpha} w^{\beta} s_4^{\gamma} w^{\delta} s_4^{\eps}
&\in& 
 s_4^{\alpha} w^{\beta} s_4^{\gamma} w_0^{\delta}  s_4^{\eps}A_3^{\times} +
 s_4^{\alpha} w^{\beta} s_4^{\gamma}u_1 u_3 u_2 u_3  s_4^{\eps}A_3 \\
&\subset& 
 s_4^{\alpha} w^{\beta} s_4^{\gamma} w_0^{\delta}  s_4^{\eps}A_3^{\times}  + A_5^{(2)} \\
s_4^{\alpha} w^{\beta} s_4^{\gamma} w^{\delta} s_4^{\eps}
&\in&   A_3^{\times}s_4^{\alpha} w_0^{\beta} s_4^{\gamma} w^{\delta} s_4^{\eps} +  A_3 s_4^{\alpha} u_3 u_2 u_3 u_1  s_4^{\gamma} w^{\delta} s_4^{\eps}  \\
&\subset &   A_3^{\times}s_4^{\alpha} w_0^{\beta} s_4^{\gamma} w^{\delta} s_4^{\eps} + A_5^{(2)} \\
s_4^{\alpha} w^{\beta} s_4^{\gamma} w^{\delta} s_4^{\eps}
&\in&  A_3^{\times} s_4^{\alpha} w_0^{\beta} s_4^{\gamma} w_0^{\delta} s_4^{\eps} + A_3 s_4^{\alpha} u_3 u_2 u_3  u_1 s_4^{\gamma} w^{\delta} s_4^{\eps} +
A_3 s_4^{\alpha} w_0^{\alpha} s_4^{\gamma} u_3 u_2 u_3 u_1 s_4^{\eps}  \\
&\subset&  A_3^{\times} s_4^{\alpha} w_0^{\beta} s_4^{\gamma} w_0^{\delta} s_4^{\eps} + A_5^{(2)} \\
\end{array}
$$
\end{enumerate}
\end{lemma}
\begin{proof} (1) is a straightforward consequence of lemma \ref{lemw0carre} and of the fact that 
$s_4^{\alpha} U_0 s_4^{\beta} w_0 s_4^{\gamma} \subset
s_4^{\alpha} A_3 u_3 A_3 s_4^{\beta} w_0 s_4^{\gamma}
+ s_4^{\alpha} A_3 u_3 u_2 u_3 s_4^{\beta} w_0 s_4^{\gamma}
= A_3 s_4^{\alpha}  u_3  s_4^{\beta} w_0 s_4^{\gamma}A_3
+ A_3 s_4^{\alpha}  u_3 u_2 u_3 s_4^{\beta} w_0 s_4^{\gamma} A_3
\subset A_5^{(2)}$ by proposition \ref{propmoinsde55}.
(2) follows from
an easy variation in the proof of lemma \ref{lemsimplif212} and from lemma \ref{lemauxA4w0}.

\end{proof}
\begin{comment}
\begin{lemma} {\ } \label{lemsimplifw0}
\begin{enumerate}
\item In $A_4$, $w_0^2 \in A_3^{\times} w_0^{-1} + A_3 w_0 + A_3 =  w_0^{-1}A_3^{\times} +  w_0 A_3+ A_3$  
\item For all $\alpha,\beta,\gamma,\delta,\eps \in \{ -1 ,1 \} = \{ -,+ \}$,
$$
s_4^{\alpha} w^{\beta} s_4^{\gamma} w^{\delta} s_4^{\eps}
\equiv s_4^{\alpha} w^{\beta} s_4^{\gamma} w_0^{\delta} s_4^{\eps}
\equiv s_4^{\alpha} w_0^{\beta} s_4^{\gamma} w^{\delta} s_4^{\eps}
\equiv s_4^{\alpha} w_0^{\beta} s_4^{\gamma} w_0^{\delta} s_4^{\eps} \mod A_5^{(2)}
$$
\end{enumerate}
\end{lemma}
\begin{proof}
We have $w_0 = c_4 c_3^{-1}$ hence, by lemma \ref{lemcentrecubiqueA4},
 $w_0^2 = c_4^2 c_3^{-2} \in R^{\times} c_4^{-1} c_3^{-2} + R c_4 c_3^{-2} + R c_3^{-2} 
=  R^{\times} w_0^{-1} c_3^{-1} c_3^{-2} + R w_0 c_3  c_3^{-2} + R c_3^{-2} 
\subset w_0^{-1} A_3^{\times} + w_0 A_3 + A_3 = A_3^{\times} w_0^{-1}  + A_3w_0  + A_3 $
since $A_3$ commutes with $w_0$. This proves (1). (2) follows from
an easy variation in the proof of lemma \ref{lemsimplif212}.

\end{proof}
\end{comment}

We are then in position to prove the following.
\begin{lemma} {\ } \label{lem2lignesout}
\begin{enumerate}
\item $s_4 w^- s_4 w^+ s_4^{-1} \in A_4^{\times} s_4 w^+ s_4^{-1} w^+ s_4 + A_5^{(2)}$ 
\item $s_4 w^- s_4^{-1} w^+ s_4^{-1} \in s_4^{-1} w^- s_4 w^- s_4^{-1} A_4^{\times} + A_5^{(2)}$ 
\end{enumerate}
\end{lemma}
\begin{proof}
We have $s_4w^- s_4 w^+ s_4^{-1} \in A_3^{\times} s_4 w_0^{-1} s_4 w_0 s_4^{-1} + A_5^{(2)}$ by lemma \ref{lemsimplifw0} (2).
Since $s_4 w_0^{-1} s_4 w_0 s_4 \in A_3^{\times} s_4 w^- s_4 w^+ s_4 \subset A_5^{(2)}$ by lemma \ref{lemreduc2}  (5)
and since $s_4^{-1}\in R^{\times} s_4^2 + R s_4 + R$, we have
$s_4 w_0^{-1} s_4 w_0 s_4^{-1} \equiv s_4 w_0^{-1} s_4 w_0 s_4^{2}  \mod A_5^{(2)}$.
Then $s_4 w_0^{-1} s_4 w_0 s_4^{2} \in A_3^{\times}  s_4 w_0^2 s_4 w_0 s_4^{2} 
+ A_3  s_4 w_0 s_4 w_0 s_4^{2} + A_5^{(2)}$ by lemma \ref{lemsimplifw0} (1), and
$s_4 w_0 s_4 w_0 s_4^{2} \in 
R s_4 w_0 s_4 w_0 s_4 + Rs_4 w_0 s_4 w_0 s_4^{-1} + A_5^{(2)}$. Then
$$
\begin{array}{llclcl}
 & s_4 w_0 s_4 w_0 s_4^{-1} &\in& A_3^{\times} s_4 w^+ s_4 w^+ s_4^{-1} &\subset& A_5^{(2)}\\ 
\mbox{and} & s_4 w_0 s_4 w_0 s_4 &\in& A_3^{\times}  s_4 w^+ s_4 w^+ s_4  &\subset& A_5^{(2)}\\
\end{array}$$ by lemmas \ref{lemsimplifw0} (2) and \ref{lemreduc2} (6).
It follows that $s_4 w_0^{-1} s_4 w_0 s_4^2 \in A_3^{\times} s_4 w_0^2 s_4 w_0 s_4^2 + A_5^{(2)}$.

Now $s_4 w_0^2 s_4 w_0 s_4^2 = s_4 w_0 (w_0 (s_4 w_0 s_4)) s_4$ and
$w_0 (s_4 w_0 s_4) = c_2^{-1} c_4 \in A_3^{\times} c_4$ commutes with
$w_0$ and $s_4$. Thus   $s_4 w_0^2 s_4 w_0 s_4^2 =  (w_0 (s_4 w_0 s_4)) s_4 w_0 s_4 \in A_4^{\times} s_4 w_0 s_4^2 w_0 s_4$.
Now $s_4 w_0 s_4^2 w_0 s_4 \in R^{\times} s_4 w_0 s_4^{-1} w_0 s_4  + R s_4 w_0 s_4 w_0 s_4  + A_5^{(2)}$ ;
moreover we already noticed $ s_4 w_0 s_4 w_0 s_4   \in A_5^{(2)}$, hence
$s_4 w_0 s_4^2 w_0 s_4 \in R^{\times} s_4 w_0 s_4^{-1} w_0 s_4 + A_5^{(2)} \subset A_3^{\times} s_4 w^+ s_4^{-1} w^+ s_4+ A_5^{(2)}$
by lemma \ref{lemsimplifw0} (2) , and this proves (1).

Now we have $s_4 w^- s_4^{-1} w^+ s_4^{-1} = \Psi(   s_4 w^- s_4 w^+ s_4^{-1}) \in \Psi(A_4^{\times} s_4 w^+ s_4^{-1} w^+ s_4) + \Psi(A_5^{(2)})
=  s_4^{-1} w^- s_4 w^- s_4^{-1} A_4^{\times} + A_5^{(2)}$, and this proves (2).
\end{proof}

By a direct computation, we will prove the following lemma, which will turn
out to be crucial in the proof of the main theorem. We postpone this
(lengthy) calculation to section \ref{sectcrucial}.

\begin{lemma} {\ } \label{lemcrucial}
In $A_5$, $\delta^3$
belongs to 
$$
A_4^{\times} s_4 w^- s_4 w^- s_4 A_3^{\times}+
A_4 s_4 w^+ s_4^{-1} w^+ s_4 A_4+
A_4 s_4^{-1} w^- s_4 w^- s_4^{-1} A_4+
A_5^{(2)}
$$
\end{lemma}

\begin{comment}
\begin{enumerate}
\item In $A_5^{(3)}/A_5^{(2)}$, $\delta^2 \in A_4^{\times}s_4 w^+ s_4^{-1} w^+ s_4 A_4^{\times} $.%, with $w_0 = s_3 s_2 s_1^{-1} s_2 s_3$. 
\item In $A_5$, $c^3$ is mapped to $A_4^{\times} s_4 w^+ s_4^{-1}  w^+ s_4^{-1} w^+ s_4 A_4 + A_5^{(2)}$ 
\item In $A_5$, $c^3$ is mapped into $A_5^{(3)}$.
\item In $A_4$, $w_0^2 \in R^{\times} w_0^{-1} + R w_0 + R$
\item In $A_5^{(3)}/A_5^{(2)}$, $c^3$ is mapped to $A_4^{\times} s_4 w^- s_4 w^- s_4 A_3^{\times}$.
\end{enumerate}
\end{lemma}
\begin{proof}
(1) is clear \footnote{détailler à l'occasion}. Le reste est à montrer !\dots
\end{comment}

\subsection{Right actions are left actions}

\begin{lemma} {\ } \label{lemdroitegaucheA5}
\begin{enumerate}
\item 
For all $\alpha,\beta,\gamma,\dots \in \{ -1,1 \}$, $x,y \in A_3$,
$$
s_1(s_4^{\alpha} s_3^{\beta} x s_3^{\gamma} s_4^{\delta} s_3^{\eps} y  s_3^{\zeta} s_4^{\eta}) \equiv (s_4^{\alpha} s_3^{\beta} x s_3^{\gamma} s_4^{\delta} s_3^{\eps} y  s_3^{\zeta} s_4^{\eta})s_1
\mod A_5^{(2)}
$$
\item For all $x \in A_4$, $ (s_4 w^+ s_4^{-1} w^+ s_4) x \in A_4 (s_4 w^+ s_4^{-1} w^+ s_4) \mod A_5^{(2)} $
\item For all $x \in A_4$, $ (s_4^{-1} w^- s_4  w^- s_4^{-1}) x \in A_4 (s_4^{-1} w^- s_4  w^- s_4^{-1}) \mod A_5^{(2)} $
\item $(s_4 w^- s_4 w^+ s_4^{-1} )s_3^{-1} \in s_3^{-1} (s_4 w^- s_4 w^+ s_4^{-1} ) + A_5^{(2)}$
\item $(s_4 w^- s_4^{-1} w^+ s_4^{-1}) s_3^{-1} \in u_3 s_4^{-1} w^+ s_4^{-1} w^- s_4 + A_5^{(2)}$. 
\end{enumerate}

\end{lemma}
\begin{proof}
We first prove (1). By lemma \ref{lemsimplif212} and because $s_1$ commutes with $u_1$
we can assume $x =y = s_2^{-1} s_1 s_2^{-1}$. Since $(s_2^{-1} s_1 s_2^{-1}) s_1 \in s_1(s_2^{-1} s_1 s_2^{-1})  
+ u_1 u_2 u_1$ by lemma \ref{lemquasicom} (and even $(s_2^{-1} s_1 s_2^{-1}) s_1 \in s_1(s_2^{-1} s_1 s_2^{-1})  
+ u_1 u_2 + u_2u_1$, see lemma \ref{lemdecomp1212}), by proposition \ref{propmoinsde55} we get the conclusion.

We then prove (2). Because of (1), and because
we have the result for $x = s_3^{-1}$ by lemma \ref{lemtransf1} (2), we
need only consider $x = s_2$.
%For $x = s_2$ this is a consequence of
%$s_4 w^+ s_4^{-1} w^+ s_4 \equiv 
\begin{comment}
For $x = s_1$ we use that
$s_3 s_2^{-1} s_1 s_2^{-1} s_3 s_1 = s_3 (s_2^{-1} s_1 s_2^{-1})  s_1 s_3 $ belongs to
 $
 s_3 s_1(s_2^{-1} s_1 s_2^{-1})   s_3 + s_3 u_1 u_2 u_1 s_3 $ (lemma \ref{lemquasicom})
 and thus $s_4 w^+ s_4^{-1} w^+ s_4 s_1 \in s_4 w^+ s_4^{-1}s_1 w^+ s_4^{-1} + A_5^{(2)}$
 by proposition \ref{propmoinsde55}, and by the same argument 
$s_4 w^+ s_4^{-1}s_1 w^+ s_4^{-1} \in s_1 s_4 w^+ s_4^{-1} w^+ s_4^{-1} + A_5^{(2)}$.
\end{comment}
For $x = s_2$, we first use that
$$s_4 w^+ s_4^{-1} w^+ s_4 \in u_1^{\times} s_4 s_3 s_2 s_1^{-1} s_2 s_3 s_4^{-1} s_3 s_2 s_1^{-1} s_2 s_3 s_4 u_1^{\times} +
A_5^{(2)}$$ by lemma \ref{lemsimplif212} ; then, because of (1) we get that
$$u_1^{\times} s_4 s_3 s_2 s_1^{-1} s_2 s_3 s_4^{-1} s_3 s_2 s_1^{-1} s_2 s_3 s_4 u_1^{\times}
\subset
u_1^{\times} s_4 s_3 s_2 s_1^{-1} s_2 s_3 s_4^{-1} s_3 s_2 s_1^{-1} s_2 s_3 s_4  + A_5^{(2)}.$$
Then 
$$
\begin{array}{lcl}
s_4 w^+ s_4^{-1} w^+ s_4.s_2& \in& 
u_1^{\times} s_4 (s_3 s_2 s_1^{-1} s_2 s_3) s_4^{-1} (s_3 s_2 s_1^{-1} s_2 s_3 )s_4 s_2  + A_5^{(2)} \\
&\subset & u_1^{\times} s_2 s_4 (s_3 s_2 s_1^{-1} s_2 s_3) s_4^{-1} (s_3 s_2 s_1^{-1} s_2 s_3 )s_4   + A_5^{(2)} \\
\end{array}
$$
because $s_2$ commutes with both $s_4$ and $s_3 s_2 s_1^{-1} s_2 s_3$ and this proves (2). One
gets (3) by applying $\Phi$
to (2).

We prove (4).  One easily gets $(s_4 w^- s_4 w^+ s_4^{-1} )s_3^{-1} =s_4 s_3^{-1} s_2 s_1^{-1} s_2 s_4 s_3 s_4^{-2} s_2 s_1^{-1} s_2 s_3^{-1} s_4$.
Now $s_4^{-2} \in R^{\times} s_4 + R s_4^{-1} + R$, and it is easily checked that
the terms originating from  $ R s_4^{-1}$ and $R$ belong to $A_5^{(2)}$. We thus get
$s_4 s_3^{-1} s_2 s_1^{-1} s_2 s_4 s_3 s_4^{-2} s_2 s_1^{-1} s_2 s_3^{-1} s_4 \in 
s_4 s_3^{-1} s_2 s_1^{-1} s_2 (s_4 s_3 s_4) s_2 s_1^{-1} s_2 s_3^{-1} s_4 +A_5^{(2)}
\subset
s_4 s_3^{-1} s_2 s_1^{-1} s_2 s_3 s_4 s_3 s_2 s_1^{-1} s_2 s_3^{-1} s_4+ A_5^{(2)}$,
and 
$$s_4 s_3^{-1} s_2 s_1^{-1} s_2 s_3 s_4 s_3 s_2 s_1^{-1} s_2 s_3^{-1} s_4
\in u_3 s_4^- w^+ s_4 w^- s_4 + A_5^{(2)}$$ by lemma \ref{lemtransf1} (1). 
We prove (5). 
One easily gets 
$$(s_4 w^- s_4^{-1} w^+ s_4^{-1}) s_3^{-1} = s_4 s_3^{-1} s_2 s_1^{-1} s_2 s_4 s_3^{-1} s_4^{-2} s_2 s_1^{-1} s_2 s_3^{-1} s_4,$$
and $s_4 s_3^{-1} s_2 s_1^{-1} s_2 s_4 s_3^{-1} x s_2 s_1^{-1} s_2 s_3^{-1} s_4 \in A_5^{(2)}$ for $x \in 1, s_4^{-1}$,
hence $(s_4 w^- s_4^{-1} w^+ s_4^{-1}) s_3^{-1}$ belongs to
$$
\begin{array}{clcl}
& s_4 s_3^{-1} s_2 s_1^{-1} s_2 s_4 s_3^{-1} s_4s_2 s_1^{-1} s_2 s_3^{-1} s_4
A_5^{(2)} &
=& s_4 s_3^{-1}s_4  s_2 s_1^{-1} s_2 s_3^{-1} s_4s_2 s_1^{-1} s_2 s_3^{-1} s_4
A_5^{(2)} \\
\subset & u_3s_4^{-1} s_3s_4^{-1}  s_2 s_1^{-1} s_2 s_3^{-1} s_4s_2 s_1^{-1} s_2 s_3^{-1} s_4
A_5^{(2)} &
= &u_3s_4^{-1} s_3  s_2 s_1^{-1} s_2 (s_4^{-1}s_3^{-1} s_4)s_2 s_1^{-1} s_2 s_3^{-1} s_4
A_5^{(2)} \\
=& u_3s_4^{-1} s_3  s_2 s_1^{-1} s_2 s_3s_4^{-1} s_3^{-1}s_2 s_1^{-1} s_2 s_3^{-1} s_4
A_5^{(2)}& \subset & u_3 s_4^{-1} w^+ s_4^{-1} w^- s_4 + A_5^{(2)}\\
\end{array}$$ 
\end{proof}

\begin{remarq} Another proof of item (2). It is easily checked
that  $s_4 w^+ s_4^{-1} w^+ s_4 \equiv (s_4 s_3 s_2 s_1^{2} s_2 s_3 s_4)^2 \mod A_5^{(2)}$,
and the element $s_4 s_3 s_2 s_1^{2} s_2 s_3 s_4$ of the braid group $B_5$ is well-known
to centralize $B_4$.
\end{remarq}

\begin{proposition} \label{propdroitegaucheA53}
$$ 
\begin{array}{lcl}
A_5^{(3)} & = &   A_4 s_4 w^- s_4 w^- s_4
+
A_4 s_4 w^+ s_4^{-1} w^+ s_4  
+
A_4 s_4^{-1} w^- s_4 w^- s_4^{-1} +A_5^{(2)}\\
A_5^{(3)} & = &    s_4 w^- s_4 w^- s_4 A_4
+
 s_4 w^+ s_4^{-1} w^+ s_4  A_4
+
 s_4^{-1} w^- s_4 w^- s_4^{-1}A_4 +A_5^{(2)} \\

\end{array}$$
\end{proposition}
\begin{proof}
Clearly the RHS are included in $A_5^{(3)}$.
By propositions \ref{propreducuuu} and  \ref{propdecomp1bimoduuu} we have 
$A_5^{(3)} \subset  A_4 s_4 w^- s_4 w^- s_4 A_4
+
A_4 s_4 w^+ s_4^{-1} w^+ s_4 A_4  
+
A_4 s_4^{-1} w^- s_4 w^- s_4^{-1} A_4 
 +
A_4 s_4 w^- s_4 w^+ s_4^{-1} A_4 
+
A_4 s_4 w^- s_4^{-1} w^+ s_4^{-1} A_4 +A_5^{(2)}$.
Lemma \ref{lem2lignesout} then implies
$A_5^{(3)} \subset  A_4 s_4 w^- s_4 w^- s_4 A_4
+
A_4 s_4 w^+ s_4^{-1} w^+ s_4 A_4  
+
A_4 s_4^{-1} w^- s_4 w^- s_4^{-1} A_4 
+ A_5^{(2)}$. By lemma \ref{lemdroitegaucheA5}
this implies $A_5^{(3)} \subset  A_4 s_4 w^- s_4 w^- s_4 A_4
+
A_4 s_4 w^+ s_4^{-1} w^+ s_4   
+
A_4 s_4^{-1} w^- s_4 w^- s_4^{-1} 
+ A_5^{(2)}$
and 
$A_5^{(3)} \subset  A_4 s_4 w^- s_4 w^- s_4 A_4
+
 s_4 w^+ s_4^{-1} w^+ s_4 A_4  
+
 s_4^{-1} w^- s_4 w^- s_4^{-1} A_4 
+ A_5^{(2)}$. Now, by
lemma \ref{lemcrucial}, 
$s_4 w^- s_4 w^- s_4$ belongs
to $A_4^{\times} \delta^3 A_3^{\times} 
+  A_4 s_4 w^+ s_4^{-1} w^+ s_4 A_4  
+
A_4  s_4^{-1} w^- s_4 w^- s_4^{-1} A_4 
+ A_5^{(2)}$,
hence $A_4 s_4 w^- s_4 w^- s_4 A_4
\subset A_4  \delta^3 A_4 
+  A_4 s_4 w^+ s_4^{-1} w^+ s_4 A_4  
+
 A_4 s_4^{-1} w^- s_4 w^- s_4^{-1} A_4 
+ A_5^{(2)}
= A_4 c^3 A_4 +  A_4 s_4 w^+ s_4^{-1} w^+ s_4 A_4  
+
A_4 s_4^{-1} w^- s_4 w^- s_4^{-1} A_4 
+ A_5^{(2)}
= A_4 c^3 A_4 +  A_4 s_4 w^+ s_4^{-1} w^+ s_4 A_4  
+
 s_4^{-1} w^- s_4 w^- s_4^{-1} A_4 
+ A_5^{(2)}
$
and, since $c^3$ is central and by lemma \ref{lemdroitegaucheA5},
this latter expression can be written as
$A_4 c^3  +  A_4 s_4 w^+ s_4^{-1} w^+ s_4   
+
 s_4^{-1} w^- s_4 w^- s_4^{-1}  
+ A_5^{(2)} = 
A_4 \delta^3  +  A_4 s_4 w^+ s_4^{-1} w^+ s_4   
+
 s_4^{-1} w^- s_4 w^- s_4^{-1}  
+ A_5^{(2)}
=A_4 s_4 w^- s_4 w^- s_4 A_3^{\times}  +  A_4 s_4 w^+ s_4^{-1} w^+ s_4   
+
 s_4^{-1} w^- s_4 w^- s_4^{-1}  
+ A_5^{(2)}
=A_4 s_4 w^- s_4 w^- s_4   +  A_4 s_4 w^+ s_4^{-1} w^+ s_4   
+
 s_4^{-1} w^- s_4 w^- s_4^{-1}  
+ A_5^{(2)}$. The other expression is deduced from this
one by application of $\Phi \circ \Psi$.

\end{proof}

\begin{proposition} \label{propA5A53} $A_5 = A_5^{(3)}$.
\end{proposition} 
\begin{proof}
One only needs to prove $A_5^{(4)} \subset A_5^{(3)}$,
that is $u_4 a u_4 b u_4 c u_4  \subset A_5^{(3)}$
for all $a,b,c \in A_4$. We have
$ u_4 a u_4 b u_4 c \in A_5^{(3)} = A_5^{(2)}
A_4 s_4 w^- s_4 w^- s_4
+
A_4 s_4 w^+ s_4^{-1} w^+ s_4  
+
A_4 s_4^{-1} w^- s_4 w^- s_4^{-1}$
hence 
$$
u_4 a u_4 b u_4 c u_4
\subset 
A_5^{(2)}u_4
A_4 s_4 w^- s_4 w^- s_4 u_4
+
A_4 s_4 w^+ s_4^{-1} w^+ s_4  u_4
+
A_4 s_4^{-1} w^- s_4 w^- s_4^{-1} u_4
\subset A_5^{(3)}.
$$
This proves the claim.
\end{proof}

This proves theorem \ref{theodecA5}, and actually the following refinement :
\begin{theorem} \label{theodecA5raf}
$$
\begin{array}{lcl}
A_5 &=& A_4 + A_4 s_4 A_4 + A_4 s_4^{-1} A_4 + A_4 s_4 s_3^{-1} s_4 A_4 +
A_4 s_4^{-1} s_3 s_2^{-1} s_3 s_4^{-1} A_4 +
A_4 s_4 s_3^{-1} s_2 s_3^{-1} s_4 A_4 \\
& & + 
A_4 s_4^{-1} w^+ s_4^{-1} A_4 +
A_4 s_4 w^-s_4 A_4 + A_4 s_4^{-1} w^- s_4^{-1} A_4
+ A_4 s_4 w^+ s_4 A_4 +  s_4 w^- s_4 w^- s_4 A_4\\
& &  +
 s_4 w^+ s_4^{-1} w^+ s_4 A_4 +
 s_4^{-1} w^- s_4 w^- s_4^{-1} A_4
\end{array}
$$
\end{theorem}

\subsection{$A_5$ as a $A_4$-module}

We need the following lemma on $A_3$ :
\begin{lemma} {\ } \label{lemAM1} 
\begin{enumerate}
\item $u_2 u_1 u_2 \subset u_1 s_2 s_1^2 s_2 + u_1 u_2 u_1$
\item $u_2 u_1 u_2 \subset u_1 s_2^{-1} s_1^{-2} s_2^{-1} + u_1 u_2 u_1$
\end{enumerate}
\end{lemma}
\begin{proof}
(2) is a consequence of (1) by using $\Phi$, so it is enough to prove (1). 
We have $u_2 u_1 u_2 \subset u_1 u_2 u_1 + \sum_{\alpha \in \{ -1, 1 \}} R s_2^{\alpha} s_1^{-\alpha} s_2^{\alpha}$
because $u_i$ is $R$-spanned by $1,s_i,s_i^{-1}$ and because of lemma \ref{lemspmgross}. Moreover
$s_2^{-1} s_1 s_2^{-1} \in u_1 s_2 s_1^{-1} s_2 + u_1 u_2 u_1$ by lemmas \ref{leminverse} and \ref{lemquasicom}, hence
$u_2 u_1 u_2 \subset u_1 s_2 s_1^{-1} s_2 + u_1 u_2 u_1$. Since $s_1^{-1} \in R s_1^2 + R s_1 + R$
we get $s_2 s_1^{-1} s_2 \in R s_2 s_1^2 s_2 + R s_2 s_1 s_2 + R s_2^2 \subset R s_2 s_1^2 s_2 + u_1 u_2 u_1$
hence $u_2 u_1 u_2 \subset u_1 s_2 s_1^2 s_2 + u_1 u_2 u_1$.
\end{proof}

We introduce or re-introduce the following submodules of $A_5$ : 
$$
\begin{array}{lcl}
A_5^{(1)} &=& A_4 u_4 A_4 = A_4 + A_4 s_4 A_4 + A_4 s_4^{-1} A_4 \\
A_5^{(1 \frac{1}{4})} &=& A_5^{(1)} + A_4 s_4 s_3^{-1} s_4 A_4 ( = A_5^{(1)} + A_4 sh^2(A_3) A_4) \\
A_5^{(1 \frac{1}{2} ) } &=& A_5^{(1 \frac{1}{4})} + A_4 u_4 u_3 u_2 u_3 u_4 A_4 \\
 &=& A_5^{(1 \frac{1}{4})} + A_4 s_4 s_3^{-1} s_2 s_3^{-1} s_4 A_4 + A_4 s_4^{-1} s_3 s_2^{-1} s_3 s_4^{-1} A_4 \\
 A_5^{(2)} &=& A_4 u_4 A_4 u_4 A_4 = A_5^{(1 \frac{1}{2} ) } + \sum_{\alpha,\beta \in \{1,-1 \}} A_4 s_4^{\alpha} w^{\beta} s_4^{\alpha} A_4 
\end{array}
$$
and
$$A_5 = A_5^{(3)} = A_4 u_4A_4 u_4A_4 u_4 A_4 = A_5^{(2)} + A_4 s_4 w^- s_4 w^- s_4
+ A_4 s_4 w^+ s_4^{-1} w^+ s_4 + A_4 s_4^{-1} w^- s_4 w^- s_4^{-1}
$$

We let $\mathcal{B}$ denote the family of elements defined in corollary \ref{corA4B72},
which span $A_4$ as a left $B$-module, $\mathcal{A}$ the family spanning
$A_4$ as a $A_3$-module defined in proposition \ref{propA4A327}, and $\mathcal{A'}$
its image under the automorphism $\Ad \Delta$ of $A_4$ (that is $s_1 \leftrightarrow s_3$, $s_2 \leftrightarrow s_2$).
We prove the following.

\begin{lemma}{\ } \label{lemAM2}
\begin{enumerate}
\item $A_5^{(1)} = A_4 + \sum_{x \in \mathcal{A}} A_4 s_4 x + \sum_{x \in \mathcal{A}} A_4 s_4^{-1} x$
\item $A_5^{(1 \frac{1}{4})} = A_5^{(1)} + \sum_{x \in \mathcal{B}} A_4 s_4 s_3^{-1} s_4 x $
\end{enumerate}

\end{lemma}
\begin{proof} (1) is a consequence of $A_3 s_4^{\pm 1 } = s_4^{\pm 1} A_3$, because
$$
A_5^{(1)} = A_4 + A_4 s_4 A_4 + A_4 s_4^{-1} A_4
= A_4 + A_4 s_4 \sum_{x \in \mathcal{A}} A_3 x
 + A_4 s_4^{-1} \sum_{x \in \mathcal{A}} A_3 x
= A_4 +  \sum_{x \in \mathcal{A}}A_4 s_4  x
 + \sum_{x \in \mathcal{A}} A_4 s_4^{-1}   x$$
We prove (2). We have $(s_4 s_3^{-1} s_4)s_1 = s_1 (s_4 s_3^{-1} s_4)$
and $(s_4 s_3^{-1} s_4) s_3^{-1} \in s_3^{-1} (s_4 s_3^{-1} s_4) + u_3 u_4 + u_4 u_3$ by lemma \ref{lemdecomp1212},
hence $(s_4 s_3^{-1} s_4) B \subset B (s_4 s_3^{-1} s_4)+ A_5^{(1)}$, where we recall $B = \langle s_1, s_3^{-1} \rangle = \langle s_1,s_3 \rangle$.
Thus
$$
A_5^{(1 \frac{1}{4})} = A_5^{(1)} + A_4 s_4 s_3^{-1} s_4 \sum_{x \in \mathcal{B}} B x
= A_5^{(1)} +  \sum_{x \in \mathcal{B}} A_4 s_4 s_3^{-1} s_4 B x
= A_5^{(1)} +  \sum_{x \in \mathcal{B}} A_4 s_4 s_3^{-1} s_4  x
$$
\end{proof}

\begin{lemma}{\ } \label{lemAM3}
\begin{enumerate}
\item $A_5^{(1 \frac{1}{2})} \subset A_5^{(1 \frac{1}{4})} + A_4 s_4 s_3 s_2^2 s_3 s_4 A_4 + A_4 s_4^{-1} s_3^{-1} s_2^{-2} s_3^{-1} s_4^{-1} A_4$
\item $A_5^{(1 \frac{1}{2})} \subset A_5^{(1 \frac{1}{4})} + \sum_{x \in \mathcal{A}'} A_4 s_4 s_3 s_2^2 s_3 s_4 x + \sum_{x \in \mathcal{A}'}A_4 s_4^{-1} s_3^{-1} s_2^{-2} s_3^{-1} s_4^{-1} x$
\end{enumerate}

\end{lemma}
\begin{proof}
We have $s_3^{-1} s_2 s_3^{-1} \subset u_2 s_3 s_2^2 s_3 + u_2 u_3 u_2$ by lemma \ref{lemAM1} hence
$s_4 s_3^{-1} s_2 s_3^{-1} s_4\subset s_4u_2 s_3 s_2^2 s_3s_4 + s_4u_2 u_3 u_2s_4
\subset  u_2 s_4s_3 s_2^2 s_3s_4 + u_2 s_4u_3 s_4u_2
\subset  u_2 s_4s_3 s_2^2 s_3s_4 + A_5^{(1 \frac{1}{4})}$. Applying $\Phi$ this
implies 
$s_4^{-1} s_3 s_2^{-1} s_3 s_4^{-1}\subset  u_2 s_4^{-1}s_3^{-1} s_2^{-2} s_3^{-1}s_4^{-1} + A_5^{(1 \frac{1}{4})}$ which
proves (1). Let $A'_3 = \langle s_2, s_3 \rangle$. We have $A_4 = \sum_{x \in \mathcal{A}'} x$. Since
$s_4 s_3 s_2^2 s_3 s_4$ commutes with $s_2$ and $s_3$ hence to $A'_3$, we get
$$
A_4 s_4s_3 s_2^2 s_3s_4 A_4 
\subset \sum_{x \in \mathcal{A}'} A_4 s_4s_3 s_2^2 s_3s_4 A'_3 x
\subset \sum_{x \in \mathcal{A}'} A_4 s_4s_3 s_2^2 s_3s_4  x
$$
and similarly $s_4^{-1}s_3^{-1} s_2^{-2} s_3^{-1}s_4^{-1}= (s_4s_3 s_2^2 s_3s_4)^{-1}$ commutes with $s_2$ and $s_3$ hence
$$A_4 s_4^{-1}s_3^{-1} s_2^{-2} s_3^{-1}s_4^{-1} A_4 \subset \sum_{x \in \mathcal{A}'} A_4 s_4^{-1}s_3^{-1} s_2^{-2} s_3^{-1}s_4^{-1} x$$
which proves (2).
\end{proof}

\begin{lemma}{\ }\label{lemAM4} 
$$
A_5^{(2)} = A_5^{(1 \frac{1}{2})} + \sum_{\alpha \in \{ -1,1 \} } A_4 s_4^{\alpha} w^{\alpha} s_4^{\alpha} + \sum_{\alpha \in \{ -1,1 \}} \sum_{x \in \mathcal{A}} A_4 s_4^{\alpha} w^{-\alpha} s_4^{\alpha} x
$$
\end{lemma}
\begin{proof}
By lemma \ref{lemauxA4w0} (1), we have $w_0 \in A_3^{\times} w^+ + U_0$,
$w_0^{-1} \in A_3^{\times} w^- + U_0$,
hence $w^+ \in A_3^{\times} w_0 + U_0$,
$w^- \in A_3^{\times} w_0^{-1} + U_0$,
 with
$U_0 = A_3 u_3 A_3 + A_3 u_3 u_2 u_3 A_3 \subset A_4$.
As a consequence, for $\alpha,\beta \in \{ -1 , 1 \}$, we have $A_4 s_4^{\alpha} w^{\beta} s_4^{\alpha} A_4 \subset A_4 s_4^{\alpha} A_3^{\times} w_0^{\beta} s_4^{\alpha} A_4
+ A_4 s_4^{\alpha} U_0 s_4^{\alpha} A_4$
Moreover, $s_4^{\alpha} U_0 s_4^{\alpha} = s_4^{\alpha} A_3 u_3 A_3 s_4^{\alpha} + s_4^{\alpha} A_3 u_3 u_2 u_3 A_3 s_4^{\alpha}
= A_3s_4^{\alpha}  u_3  s_4^{\alpha}A_3 + A_3s_4^{\alpha}  u_3 u_2 u_3  s_4^{\alpha}A_3 \subset A_5^{(1 \frac{1}{4})}+A_5^{(1 \frac{1}{2})}=A_5^{(1 \frac{1}{2})}$,
hence $A_4 s_4^{\alpha} w^{\beta} s_4^{\alpha} A_4 \subset A_4 s_4^{\alpha}  w_0^{\beta} s_4^{\alpha} A_4 + A_5^{(1 \frac{1}{2})}$.
Since $w_0$ and $s_4$ commute with $s_1$ and $s_2$, we have
$$
A_4 s_4^{\alpha}  w_0^{\beta} s_4^{\alpha} A_4 \subset \sum_{x \in \mathcal{A}} A_4 s_4^{\alpha}  w_0^{\beta} s_4^{\alpha} A_3 x
\subset \sum_{x \in \mathcal{A}} A_4 s_4^{\alpha}  w_0^{\beta} s_4^{\alpha}  x.
$$
If moreover $\alpha = \beta$, $s_4^{\alpha} w_0^{\alpha} s_4^{\alpha} = (s_4 s_3 s_2 s_1^2 s_2 s_3 s_4)^{\alpha}$ commutes
with $\langle s_1,s_2,s_3 \rangle = A_4$, hence $A_4 s_4^{\alpha} w_0^{\alpha} s_4^{\alpha} A_4
= A_4 s_4^{\alpha} w_0^{\alpha} s_4^{\alpha} $, and this concludes the proof.
\end{proof}

From this one can conlude the following.

\begin{theorem}{\ } \label{theofinal}
\begin{enumerate}
\item $A_5 = A_5^{(3)}$ is generated as a $A_4$-module by $240$ elements.
\item $A_5 = A_5^{(3)}$ is generated as a $R$-module by $155,520$ elements.
\end{enumerate}
\end{theorem}
\begin{proof}
By lemma \ref{lemAM1}, $A_5^{(1)}$ is generated as an $A_4$-module by $1+2 \times 27 = 55$ elements,
$A_5^{(1 \frac{1}{4})}$ by $A_5^{(1)}$ and $|\mathcal{B}| = 72$ elements,
$A_5^{(1 \frac{1}{2})}$ after lemma \ref{lemAM3} by $A_5^{(1 \frac{1}{4})}$ and $2 \times | \mathcal{A}'| = 2 \times 27 = 54$
elements, $A_5^{(2)}$ by $A_5^{(1 \frac{1}{2})}$ and $2 + 2 \times | \mathcal{A}| = 56$ elements (lemma \ref{lemAM4}), and
$A_5^{(3)}$ by $A_5^{(2)}$ and $3$ elements. It follows that $A_5$ is $A_4$-generated
by $55+72+54+56+3 = 240$ elements, which proves (1). Since $A_4$ is $R$-generated by $648$ elements,
we get that $A_5$ is $R$-generated by $240 \times 648 = 155,520$ elements, which proves (2).

\end{proof}
\begin{comment}
\begin{lemma}{\ }

\end{lemma}
\begin{proof}

\end{proof}
\begin{lemma}{\ }

\end{lemma}
\begin{proof}

\end{proof}

\end{comment}

\section{Proof of lemma \ref{lemcrucial}}

\label{sectcrucial}

For the sake of concision we denote $V_0 = A_5^{(2)}$ and $V^+ = A_5^{(2)} + A_4 s_4 w^+ s_4^{-1} w^+ s_4 A_4 = V_0 + A_4 s_4 w^+ s_4^{-1} w^+ s_4 A_4$.
We will prove that $X \in A_4^{\times} s_4 w^- s_4 w^- s_4 A_4^{\times} + V^+$, starting from $X = \delta^3 = s_4 w_0 s_4^2 w_0 s_4^{-1} w_0 s_4$
to $X = \delta^3 = s_4 w^- s_4 w^- s_4$ (for which the statement is trivial) through a sequence of reductions of the type $X \to X'$ where
$X' \in A_4^{\times} X A_3^{\times} +V^+$

\subsection{Reduction to $s_4 w_0 s_4^2 w_0 s_4^{-1} w_0 s_4$} 
Using $s_4^2 \in R^{\times} s_4^{-1} + R s_4 + R$ we get
$s_4 w_0 s_4^2 w_0 s_4^2 w_0 s_4 \in R^{\times} s_4 w_0 s_4^2 w_0 s_4^{-1} w_0 s_4 + R s_4 w_0 s_4^2 w_0 s_4 w_0 s_4 + R
s_4 w_0 s_4^2 w_0 ^2 s_4$. The fact that
$s_4 w_0 s_4^2 w_0 s_4 w_0 s_4, s_4 w_0 s_4^2 w_0 ^2 s_4 $ belongs to  $V^+$ is proved in the following lemma \ref{lemN1}

\begin{lemma}{\ } \label{lemN1}
\begin{enumerate}
\item $s_4 w_0 s_4^2 w_0^2 s_4 \in V^+$.
\item $s_4 w_0 s_4^2 w_0 s_4 w_0 s_4 \in V^+$
\item $s_4 w_0^2 s_4^{-1} w_0 s_4  \in V^+$
\item $s_4 w_0 s_4 w_0 s_4^{-1} w_0 s_4 \in V_0$
\end{enumerate}
\end{lemma}
\begin{proof}
We prove (1). By lemma \ref{lemw0carre} we have $s_4 w_0 s_4^2 w_0^2 s_4 \in
A_3^{\times} s_4 w_0 s_4^2 w_0^{-1}s_4 + 
A_3 s_4 w_0 s_4^2 w_0 s_4 + A_3 s_4 w_0   s_4^3$. Clearly $s_4 w_0   s_4^3 \in V_0$,
hence, expanding $s_4^2$,
$$s_4 w_0   s_4^2 w_0^2 s_4 \in A_3 s_4 w_0 s_4^{-1} w_0^{-1} s_4 +A_3 s_4 w_0 s_4 w_0^{-1} s_4
+ A_3 s_4 w_0 s_4^{-1} w_0 s_4 + A_3 s_4 w_0 s_4 w_0 s_4 + V_0.$$
We already know $s_4 w_0 s_4 w_0 s_4 \in V_0$ by lemma \ref{lemreduc2} (6). Moreover
$s_4 w_0s_4^{-1} w_0^{-1} s_4 \in A_3 s_4 w^+ s_4^{-1} w^- s_4+ V_0 \subset V_0$ 
and
$s_4 w_0s_4 w_0^{-1} s_4 \in A_3 s_4 w^+ s_4  w^- s_4+ V_0 \subset V_0$ 
by lemma \ref{lemtransf1} (4), and finally 
$$s_4 w_0s_4^{-1} w_0^{-1} s_4 \in A_3 s_4 w^+ s_4^{-1} w^+ s_4+ V_0 \subset V_+.$$ 
We now prove (2). We have 
$$s_4 w_0 s_4^2 (w_0 s_4 w_0 s_4) = 
(w_0 s_4 w_0 s_4) s_4 w_0 s_4^2 \subset A_4 s_4 w_0 s_4^2 w_0 s_4^2,$$
as $w_0 s_4 w_0 s_4 = c_5 c_3^{-1}$ commutes with $s_4$ and $w_0$. The
term $s_4 w_0 s_4^2 w_0 s_4^2$ is a linear combinations
of terms of the form $s_4 w_0 s_4^{\alpha} w_0 s_4^{\beta}$ for $\alpha,\beta \in \{ 0,-1,1 \}$,
and we have (lemma \ref{lemreduc2} (6) and (7))
$s_4 w_0 s_4^{\alpha} w_0 s_4^{\beta} \in A_3 s_4 w^+ s_4^{\alpha} w^+ s_4^{\beta}
+ V_0 \subset V_0$, unless $(\alpha,\beta) = (-1,1)$, in which
case $s_4 w_0 s_4^{-1} w_0 s_4 \in A_3 s_4 w^+ s_4^{-1} w^+ s_4 + V_0 \subset V^+$.
For proving (3) we use that $s_4 w_0^2 s_4^{-1} w_0 s_3 \in V_0 + A_3 \sum_{\alpha \in \{ \pm \} } s_4 w^{\alpha} s_4^-
w^+ s_4$, and that $s_4 w^- s_4^{-1} w^+ s_4 \in V_0$ by lemma \ref{lemreduc2} (5).
We prove (4). We have $(s_4 w_0 s_4 w_0) s_4^{-1} w_0 s_4 = 
 s_4^{-1} w_0 s_4(s_4 w_0 s_4 w_0) \in s_4^{-1} w_0 s_4^2 w_0 s_4A_4$
 and  $s_4^{-1} w_0 s_4^2 w_0 s_4 \in R s_4^{-1} w_0^2 s_4 + R s_4^{-1} w_0 s_4^{-1} w_0 s_4 + R s_4^{-1} w_0 s_4 w_0 s_4$.
 Now, $s_4^{-1} w_0^2 s_4 \in V_0$, $s_4^{-1} (w_0 s_4 w_0 s_4) =  (w_0 s_4 w_0 s_4)s_4^{-1} \in V_0$,
 and $s_4^{-1} w_0 s_4^{-1} w_0 s_4 \in A_3 s_4^{-1} w^+ s_4^{-1} w^+ s_4 + V_0$. Since
 $\Phi(s_4^{-1} w^+ s_4^{-1} w^+ s_4) \in V_0$ by lemma \ref{lemreduc2} (8) we get (4).
\end{proof}

\subsection{Reduction to $s_4 s_3 (s_2 s_1^2 s_2) s_3^{-1} s_4 s_3^2 (s_2 s_1^2 s_2) s_3 s_4^{-1} w_0 s_4$} 
Using $s_4^2 \in R^{\times} s_4^{-1} + R s_4 + R$ we get
$s_4 w_0 s_4^2 w_0 s_4^{-1} w_0 s_4 \in R^{\times} s_4 w_0 s_4^{-1} w_0 s_4^{-1} w_0 s_4 + R s_4 w_0^2 s_4^{-1} w_0 s_4+ R
s_4 w_0 s_4 w_0 s_4^{-1} w_0 s_4$. The fact that $ R s_4 w_0^2 s_4^{-1} w_0 s_4+ R
s_4 w_0 s_4 w_0 s_4^{-1} w_0 s_4 \subset V^+$ has been proved in lemma \ref{lemN1}. Finally,
$s_4 w_0 s_4^{-1} w_0 s_4^{-1} w_0 s_4 = s_3^{-1} . s_4 s_3 (s_2 s_1^2 s_2) s_3^{-1} s_4 s_3^2 (s_2 s_1^2 s_2) s_3 s_4^{-1} w_0 s_4$
is easily checked to hold in the braid group $B_5$.

Before going further, we first need to establish several lemmas.

\begin{lemma} {\ }  \label{lemN2}
\begin{enumerate}
\item For all $\alpha,\beta \in \Z$, $s_4 w_0 s_4^{\alpha} s_3^{\beta} w_0 s_4 \in V^+$.
\item For all $\alpha,\beta \in \Z$, $s_4 w_0  s_3^{\beta} s_4^{\alpha} w_0 s_4 \in V^+$.
\end{enumerate}
\end{lemma}
\begin{proof}
Since $s_4 w_0 s_4^{\alpha} s_3^{\beta} w_0 s_4 = s_4 w_0 s_4^{\alpha} s_3^{\beta+1} s_2 s_1^2 s_2 s_3 s_4$,
we need to consider $s_4 w_0 s_4^{\alpha} s_3^{\beta} s_2 s_1^2 s_2s_3 s_4$ , where we can assume $\alpha, \beta \in \{ 1, -1\}$,
the cases $\alpha = 0$ and $\beta = 0$ being obvious by proposition \ref{propmoinsde55}. If $\beta = 1$
we have $s_4 w_0 s_4^{\alpha} w_0 s_4 \in A_3s_4 w^+ s_4^{\alpha} w^+ s_4  +V_0 \subset V^+$,
so we can assume $\beta = -1$. Expanding $s_1^2$ we get a linear combinations of 
$s_4 w_0 s_4^{\alpha} s_3^{-1} s_2 s_1^{-1} s_2s_3 s_4$,
$s_4 w_0 s_4^{\alpha} s_3^{-1} s_2 s_1 s_2s_3 s_4$
and $s_4 w_0 s_4^{\alpha} s_3^{-1} s_2^2 s_3 s_4$.  We have $s_4 w_0 s_4^{\alpha} s_3^{-1} s_2^2 s_3 s_4 \in V_0$
by \ref{propmoinsde55} and
$s_4 w_0 s_4^{\alpha} s_3^{-1}( s_2 s_1 s_2)s_3 s_4 = 
s_4 w_0 s_4^{\alpha} s_3^{-1} s_1 s_2 s_1s_3 s_4=
s_4 w_0 s_4^{\alpha} s_1 s_3^{-1}  s_2 s_3 s_4s_1 
% \in
%s_4 w_0 s_4^{\alpha} s_1 u_2 u_3 u_2 s_4s_1
  \in V_0$ by proposition \ref{propmoinsde55}.
%s_4 w_0 s_4^{\alpha} s_1 u_2 u_3  s_4 u_2 s_1 
There remains to consider $s_4 w_0 s_4^{\alpha} s_3^{-1} (s_2 s_1^{-1} s_2)s_3 s_4 =  
s_4 w_0 s_4^{\alpha} s_2 s_1 (s_3 s_2^{-1} s_3) s_1^{-1} s_2^{-1}  s_4 = 
s_2 s_1 s_4 w_0 s_4^{\alpha}  (s_3 s_2^{-1} s_3)   s_4 s_1^{-1} s_2^{-1}  \in V_0$
by proposition \ref{propmoinsde55}. This concludes the proof of (1). Then (2) is an immediate
consequence of (1) by application of $\Phi \circ \Psi$.

 \end{proof}
\begin{comment}
\begin{lemma}{\ } (pour remplacer dans le comptage le lemme N3) 
\begin{enumerate}\item
\end{enumerate}
\end{lemma}
\end{comment}

\begin{lemma}{\ } \label{lemN4}
\begin{enumerate}
\item $s_4 w_0 s_4^{-1} s_3 s_4^{-1} w_0 s_4 \in V^+$
\item $s_4 s_3(s_2 s_1^2 s_2) s_3^{-1} s_4 w_0 s_4^{-1} w_0 s_4 \in V^+$
\end{enumerate}
\end{lemma}
\begin{proof}
By using braid relations one gets
$s_4 s_3(s_2 s_1^2 s_2) s_3^{-1} s_4 w_0 s_4^{-1} w_0 s_4 = s_3 (s_2 s_1^2 s_2) s_4 w_0 s_4^{-1} s_3 s_4^{-1} w_0 s_4$,
hence (2) reduces to (1).  We now prove (1). Expanding $s_4^{-1}$ as a linear combination of $s_4^2$, $s_4$ and $1$,
we get that $s_4 w_0 (s_4^{-1}) s_3 s_4^{-1} w_0 s_4$
is a linear combination of $s_4 w_0 s_3 s_4^{-1} w_0 s_4$,
$s_4 w_0 s_4 s_3 s_4^{-1} w_0 s_4$ and $s_4 w_0 s_4^2 s_3 s_4^{-1} w_0 s_4$.
We have $s_4 w_0 s_3 s_4^{-1} w_0 s_4 \in V_0$ by lemma \ref{lemN2} (2),
$(s_4 w_0 s_4) s_3 s_4^{-1} w_0 s_4 = 
 s_3 (s_4 w_0 s_4) s_4^{-1} w_0 s_4
 =s_3 s_4 w_0^2  s_4 \in V_0$ because $s_4 w_0 s_4 = c_5 c_4^{-1}$ commutes with $B_4$ and in particular $s_3$,
 and similarly $s_4 w_0 s_4^2 s_3 s_4^{-1} w_0 s_4 = s_4 w_0 s_4 (s_4 s_3 s_4^{-1}) w_0 s_4
=  (s_4 w_0 s_4) s_3^{-1}  s_4 s_3 w_0 s_4
=  s_3^{-1} (s_4 w_0 s_4)  s_4 s_3 w_0 s_4 \in A_4 s_4 w_0 s_4^2 s_3 w_0 s_4 \in V^+$ by lemma \ref{lemN2} (2).
\end{proof}

\begin{lemma}{\ }  \label{lemN5}
\begin{enumerate}
\item For all $\beta$, $s_4 A_4 s_4 s_3^{\beta} w_0 s_4 \subset V^+$
\item $s_4 s_3 s_2^{-1} s_3 u_1 u_2 s_4 u_3 w_0 s_4 \subset V^+$
\item $s_4 s_3 (s_2 s_1^2 s_2)s_3^{-1} s_4 (s_2 s_1^2 s_2) s_3 s_4^{-1} w_0 s_4 \in V^+$ % ex N3 (2)
\item $u_4 u_3 u_2 u_3 u_4 A_4 s_4 \subset V^+$ ; moreover $s_4^{\alpha}u_3 u_2 u_3 s_4^{\beta} A_4 s_4 \subset V_0$ when $\alpha,\beta \in \{ -1,1 \}$ with $(\alpha,\beta) \neq (1,1)$  
\item $u_4 U_0 u_4 A_4 s_4 \subset V^+$ 
\item $u_4 s_3^{\pm} A_3  s_3^{\mp} u_4 A_4 s_4 \subset V^+$ 
\end{enumerate}
\end{lemma}
\begin{proof}
From theorem \ref{theodecA4} one easily deduces $A_4 = A_3 u_3 A_3 + A_3 s_3 s_2^{-1} s_3 A_3 +  A_3 w_0
+ A_3 w_0^{-1}$. Then $s_4 (A_3 u_3 A_3)  s_4 s_3^{\beta} w_0 s_4 =
A_3 s_4  u_3  s_4 A_3 s_3^{\beta} w_0 s_4 \subset V_0$ by lemma \ref{lemreduc2} (2) ;
$s_4 (A_3 w_0)  s_4 s_3^{\beta} w_0 s_4 =
A_3 s_4 w_0   s_4 s_3^{\beta} w_0 s_4 \subset V^+$ by lemma \ref{lemN2} (2) ;
$s_4 (A_3 w_0^{-1})  s_4 s_3^{\beta} w_0 s_4 =
A_3 s_4 w_0 ^{-1}  s_4 s_3^{\beta} w_0 s_4 $ and
$s_4 w_0 ^{-1}  s_4 s_3^{\beta} w_0 s_4$ is
a linear combination of 
$s_4 w_0 ^{-1}  s_4 s_3^{\beta'} s_2 s_1^2 s_2  s_3 s_4$ for $\beta' \in \{ 0,1,-1 \}$.
For $\beta' = 1$,
$s_4 w_0 ^{-1}  s_4 s_3 s_2 s_1^2 s_2 s_3  s_4 = s_4 w_0^{-1} s_4 w_0 s_4
\in A_3  s_4 w^- s_4 w^+ s_4 + V_0 \subset V_0$ by lemma \ref{lemreduc2} (5).
For $\beta' = -1$, $s_4 w_0 ^{-1}  s_4 s_3^{-1} s_2 s_1^2 s_2 s_3  s_4 \in V_0$
by lemma \ref{lemreducspec} (4), and the case $\beta = 0$ also lies in $V_0$
by proposition \ref{propmoinsde55}.
It then remains to prove $s_4 A_3 s_3 s_2^{-1}s_3 A_3 s_4 s_3^{\beta} w_0 s_4 \subset V^+$,
that is $s_4 s_3 s_2^{-1} s_3 A_3 s_4 s_3^{\beta} w_0 s_4 \subset V^+$.
We use $A_3 \subset s_2^{-1} s_1 s_2^{-1} u_1 + u_1 u_2 u_1$ to get 
$s_4 s_3 s_2^{-1} s_3 A_3 s_4 s_3^{\beta} w_0 s_4\subset
s_4 s_3 s_2^{-1} s_3 s_2^{-1} s_1 s_2^{-1} u_1 s_4 s_3^{\beta} w_0 s_4+ s_4 s_3 s_2^{-1} s_3 u_1 u_2 u_1 s_4 s_3^{\beta} w_0s_4$.
Now, for $s_4 s_3 s_2^{-1} s_3 u_1 u_2 u_1 s_4 s_3^{\beta} w_0 s_4= s_4 s_3 s_2^{-1} s_3 u_1 u_2  s_4 s_3^{\beta} w_0s_4u_1$
we are reduced to proving (2), while 
the expression $s_4 (s_3 s_2^{-1} s_3 s_2^{-1}) s_1 s_2^{-1} u_1 s_4 s_3^{\beta} w_0 s_4 =
s_4 (s_3 s_2^{-1} s_3 s_2^{-1}) s_1 s_2^{-1} s_4 s_3^{\beta} w_0 s_4  u_1$ is 
$$
s_4  s_2^{-1} s_3 s_2^{-1}s_3 s_1 s_2^{-1}  s_4 s_3^{\beta} w_0 s_4 u_1+ s_4  u_2 u_3  s_1 s_2^{-1} u_1 s_4 s_3^{\beta} w_0 s_4
+ s_4  u_3 u_2 s_1 s_2^{-1} s_4  s_3^{\beta} w_0 s_4 u_1$$
 by lemma \ref{lemdecomp1212}. We have $ s_4  u_2 u_3  s_1 s_2^{-1} s_4 s_3^{\beta} w_0 s_4
+ s_4  u_3 u_2 s_1 s_2^{-1} s_4  s_3^{\beta} w_0 s_4 \subset V_0$ by proposition \ref{propmoinsde55},
while $s_4  s_2^{-1} s_3 s_2^{-1}s_3 s_1 s_2^{-1}  s_4 s_3^{\beta} w_0 s_4 =
 s_2^{-1} s_4  s_3 s_2^{-1}s_3 s_1 s_2^{-1}  s_4 s_3^{\beta} w_0 s_4  $
$s_4  s_3 s_2^{-1}s_3 s_1 s_2^{-1}  s_4 s_3^{\beta} w_0 s_4$
is a linear combination of $s_4  s_3 s_2^{-1}s_3 s_1 s_2^{-1}  s_4 s_3^{\beta'} s_2 s_1^2 s_2 s_3 s_4$ for $\beta' \in \{ 0,1,-1 \}$.
Now 
$$
\begin{array}{llclc}
& s_4  s_3 s_2^{-1}s_3 s_1 s_2^{-1}  s_4 s_3^{\beta'} s_2 s_1^2 s_2 s_3 s_4 &=&
s_4  s_3 s_2^{-1}s_3 s_1   s_4 (s_2^{-1}s_3^{\beta'} s_2) s_1^2 s_2 s_3 s_4 \\
=&s_4  s_3 s_2^{-1}s_3 s_1   s_4 s_3s_2^{\beta'} s_3^{-1} s_1^2 s_2 s_3 s_4
&=& s_4  s_3 s_2^{-1}s_3 s_1   s_4 s_3s_2^{\beta'}  s_1^2 (s_3^{-1}s_2 s_3) s_4 \\
=&s_4  s_3 s_2^{-1}s_3 s_1   s_4 s_3s_2^{\beta'}  s_1^2 s_2s_3 s_2^{-1} s_4
&=&s_4  s_3 s_2^{-1}s_3 s_1   s_4 s_3s_2^{\beta'}  s_1^2 s_2s_3  s_4s_2^{-1}  \in V_0
\end{array}$$ by proposition \ref{propmoinsde55}. This proves (1).
For proving (2), we can reduce to an expression of the form
$s_4 s_3 s_2^{-1} s_3 s_1^{\alpha} s_2^{\beta} s_4 s_3^{\gamma} s_2 s_1^2 s_2 s_3 s_4$ (with $\alpha,\beta,\gamma \in \{0,1,-1 \}$).
Using only braid relations, one gets
$$
\begin{array}{lccl} 
& s_4 s_3 s_2^{-1} s_3 s_1^{\alpha} s_2^{\beta} s_4 s_3^{\gamma} s_2 s_1^2 s_2 s_3 s_4  
&= & s_4 s_3 s_2^{-1} s_3 s_1^{\alpha} s_2^{\beta} s_4 s_3^{\gamma} s_2 s_1^2 s_2 (s_3 s_4 s_3^{-1}) s_3 \\
= & s_4 s_3 s_2^{-1} s_3 s_1^{\alpha} s_2^{\beta} s_4 s_3^{\gamma} s_2 s_1^2 s_2 s_4^{-1} s_3 s_4 s_3 &
= & s_4 s_3 s_2^{-1} s_3 s_1^{\alpha} s_2^{\beta} (s_4 s_3^{\gamma} s_4^{-1}) s_2 s_1^2 s_2  s_3 s_4 s_3 \\
= & s_4 s_3 s_2^{-1} s_1^{\alpha}(s_3  s_2^{\beta} s_3^{-1}) s_4^{\gamma} s_3 s_2 s_1^2 s_2  s_3 s_4 s_3 &
= & s_4 s_3 s_2^{-1} s_1^{\alpha}s_2^{-1}  s_3^{\beta} s_2 s_4^{\gamma} s_3 s_2 s_1^2 s_2  s_3 s_4 s_3 \\
= & s_4 s_3 s_2^{-1} s_1^{\alpha}s_2^{-1}  s_3^{\beta}  s_4^{\gamma} (s_2 s_3 s_2) s_1^2 s_2  s_3 s_4 s_3 &
= & s_4 s_3 s_2^{-1} s_1^{\alpha}s_2^{-1}  s_3^{\beta}  s_4^{\gamma} s_3 s_2 s_3 s_1^2 s_2  s_3 s_4 s_3 \\
= & s_4 s_3 s_2^{-1} s_1^{\alpha}s_2^{-1}  s_3^{\beta}  s_4^{\gamma} s_3 s_2  s_1^2 (s_3 s_2  s_3) s_4 s_3 &
= & s_4 s_3 s_2^{-1} s_1^{\alpha}s_2^{-1}  s_3^{\beta}  s_4^{\gamma} s_3 s_2  s_1^2 s_2 s_3  s_2 s_4 s_3 \\
= & s_4 s_3 s_2^{-1} s_1^{\alpha}s_2^{-1}  s_3^{\beta}  s_4^{\gamma} s_3 s_2  s_1^2 s_2 s_3   s_4 .s_2 s_3 \\
\end{array}
$$
%is equal to
%$s_4 s_3 s_2^{-1} s_1^{\alpha} s_2^{-1} s_3^{\beta} s_4^{\gamma} s_3  s_2 s_1^2 s_2 s_3 s_4. s_2 s_3$,
which belongs to $V_0$ by lemma \ref{lemreducspec} (3) as soon as $\beta = -1$. If $\beta = 0$, it
is equal to 
$$s_4 s_3 s_2^{-1} s_1^{\alpha} s_2^{-1}  s_4^{\gamma} s_3  s_2 s_1^2 s_2 s_3 s_4. s_2 s_3 =
s_4 s_3 s_4^{\gamma}  s_2^{-1} s_1^{\alpha} s_2^{-1}  s_3  s_2 s_1^2 s_2 s_3 s_4. s_2 s_3 \in V_0$$
by lemma \ref{lemreduc2} (2).
Otherwise, considering all possibilities for $\alpha$ and applying lemmas \ref{lemsimplif212} and \ref{lemreducspec}
it lies inside
$V_0 + s_4 s_3 s_2^{-1} s_1 s_2^{-1}  s_3 s_4^{\gamma} s_3  s_2 s_1^2 s_2 s_3 s_4 A_4 
\subset s_4 w^+ s_4^{\gamma} w^+ s_4 A_4 + V_0$. In cases $\gamma \in \{ -1,0 \}$ this clearly belongs to $V^+$,
while $s_4 w^+ s_4 w^+ s_4 \in V_0$ by lemma \ref{lemreduc2} (6).

We now prove (3). We have
$$
\begin{array}{llcl}
&s_4 s_3 (s_2 s_1^2 s_2)s_3^{-1} s_4 (s_2 s_1^2 s_2) s_3 s_4^{-1} w_0 s_4 &=& 
s_4 s_3 (s_2 s_1^2 s_2)s_3^{-1}  (s_2 s_1^2 s_2) (s_4s_3 s_4^{-1}) w_0 s_4\\
=&
s_4 s_3 (s_2 s_1^2 s_2)s_3^{-1}  (s_2 s_1^2 s_2) s_3^{-1}s_4 s_3 w_0 s_4
&\in&   s_4 A_4s_4 s_3 w_0 s_4 \subset V^+\end{array}$$ by (1).

We prove (4), considering an expression of the form 
$s_4^{\alpha} u_3 u_2 u_3 s_4^{\beta} A_4 s_4$ for $\alpha , \beta \in \{ -1, 1 \}$,
the case $\alpha = 0$ or $\beta = 0$ being obvious. We use the decomposition
$A_4 = A_3 u_3 A_3 + A_3 u_3 u_2 u_3 A_3 + u_3 u_2 u_1 u_2 u_3 A_3$.
We have $$ \begin{array}{llcl} & s_4^{\alpha} u_3 u_2 u_3 s_4^{\beta} A_3 u_3 A_3 s_4
&=& s_4^{\alpha} u_3 u_2 u_3 s_4^{\beta} A_3 u_3 s_4 A_3\\
= &s_4^{\alpha} u_3 u_2 u_3 s_4^{\beta} u_2 u_1 u_2 u_1 u_3 s_4 A_3
&=& s_4^{\alpha} u_3 u_2 u_3 s_4^{\beta} u_2 u_1 u_2  u_3 s_4 u_1 A_3 \subset V_0\end{array}$$
by proposition \ref{propmoinsde55},
and 
$s_4^{\alpha} u_3 u_2 u_3 s_4^{\beta} u_3 u_2 u_1 u_2 u_3  A_3 s_4 =
s_4^{\alpha} u_3 u_2 u_3 s_4^{\beta} u_3 u_2 u_1 u_2 u_3   s_4 A_3  \subset V_0$
by proposition \ref{propmoinsde55}.
There remains to consider
$$  \begin{array}{llcl} & s_4^{\alpha} u_3 u_2 u_3 s_4^{\beta} A_3 u_3 u_2 u_3  A_3 s_4 &=&
s_4^{\alpha} u_3 u_2 u_3 s_4^{\beta} A_3 u_3 u_2 u_3 s_4A_3\\
=&
s_4^{\alpha} u_3 u_2 u_3 s_4^{\beta} u_1 u_2 u_1 (u_2  u_3 u_2 u_3) s_4A_3&
=&
s_4^{\alpha} u_3 u_2 u_3 s_4^{\beta} u_1 u_2 u_1 u_3  u_2 u_3 u_2 s_4A_3\\
=&
s_4^{\alpha} u_3 u_2 u_3 s_4^{\beta} u_1 u_2 u_1 u_3  u_2 u_3  s_4 u_2A_3 &\subset&
V_0 \end{array}$$ by proposition \ref{propmoinsde55}, unless $\alpha = \beta = 1$.
In that case the proof of lemma \ref{lemaux2AuAuA}, lemma \ref{lemtransf1} (1) and lemma
\ref{lem2lignesout} (2) together yield $s_4 u_3 u_2 u_3 s_4 u_1 u_2 u_1 u_3  u_2 u_3  s_4
\in V^+$.

Now (5) is a consequence of (4) and proposition \ref{propmoinsde55},
 as $U_0 = A_3 u_3 A_3 + A_3 u_3 u_2 u_3 A_3$ and 
 $$
 \begin{array}{lcl}
 u_4 U_0 u_4 A_4 s_4 &\subset& u_4 A_3 u_3 A_3 u_4 A_4 s_4
 + u_4 A_3 u_3 u_2 u_3 A_3 u_4 A_4 s_4 \\
& = & A_3 u_4  u_3  u_4 A_4 s_4
 + u_4 A_3 u_3 u_2 u_3  u_4 A_4 s_4 \\
 \end{array}
 $$
 and both terms belong to $V^+$, by lemma \ref{lemreduc2} (2) and by (4).
Then (6) is an immediate consequence of (5) and lemma \ref{lemauxA4w0} (3).

\end{proof}

\subsection{Reduction to $s_4 s_3 (s_2 s_1^2 s_2) s_3^{-1} s_4 s_3^{-1} (s_2 s_1^2 s_2) s_3 s_4^{-1} w_0 s_4$} 
Expanding $s_3^2$, we get %that
$$
\begin{array}{lcl}
s_4 s_3 (s_2 s_1^2 s_2) s_3^{-1} s_4 s_3^2 (s_2 s_1^2 s_2) s_3 s_4^{-1} w_0 s_4 & \in
& R^{\times} 
s_4 s_3 (s_2 s_1^2 s_2) s_3^{-1} s_4 s_3^{-1} (s_2 s_1^2 s_2) s_3 s_4^{-1} w_0 s_4 \\
&& + R s_4 s_3 (s_2 s_1^2 s_2) s_3^{-1} s_4 s_3 (s_2 s_1^2 s_2) s_3 s_4^{-1} w_0 s_4\\
&&+ R s_4 s_3 (s_2 s_1^2 s_2) s_3^{-1} s_4  (s_2 s_1^2 s_2) s_3 s_4^{-1} w_0 s_4.
\end{array}$$
We have 
$$
\begin{array}{lr}
s_4 s_3 (s_2 s_1^2 s_2) s_3^{-1} s_4  (s_2 s_1^2 s_2) s_3 s_4^{-1} w_0 s_4 \in V^+&
\mbox{by lemma \ref{lemN5} (3)} \\
s_4 s_3 (s_2 s_1^2 s_2) s_3^{-1} s_4 s_3 (s_2 s_1^2 s_2) s_3 s_4^{-1} w_0 s_4
\in V^+& \mbox{by lemma \ref{lemN4} (2).}
\end{array}
$$

\begin{lemma}{\ } \label{lemN6}
\begin{enumerate}
\item For all $\alpha \in \Z$ $s_4 s_3^{-1} s_2^{\alpha} s_3^{-1} s_4^{-1} s_3 w_0 s_4 \in V^+$
\item $s_4 s_3^{-1} s_2^2 s_3^{-1} s_4^2 s_3 w_0 s_4 \in V^+$
\item $s_4 s_3^{-1} s_2 s_4 s_3^{-1} s_2^2 s_3 s_4^{-1} w_0 s_4 \in V^+$
\item $s_4 s_3^{-1} s_2 s_4 s_3^{-1} s_2 s_1 s_2 s_3 s_4^{-1} w_0 s_4 \in V^+$ 
\end{enumerate}
\end{lemma}
\begin{proof}
We prove (1). We get
$$ \begin{array}{llcl}
& s_4 s_3^{-1} s_2^{\alpha} (s_3^{-1} s_4^{-1} s_3) w_0 s_4
 &=& 
s_4 s_3^{-1} s_2^{\alpha} s_4 (s_3^{-1} s_4^{-1} s_3) s_2 s_1^2 s_2 s_3 s_4 \\
=&
s_4 s_3^{-1} s_2^{\alpha} s_4 s_4 s_3^{-1} s_4^{-1} s_2 s_1^2 s_2 s_3 s_4&
=&
s_4 s_3^{-1} s_2^{\alpha} s_4^2 s_3^{-1}  s_2 s_1^2 s_2 (s_4^{-1}s_3 s_4) \\
=&
s_4 s_3^{-1} s_4^2 s_2^{\alpha}  s_3^{-1}  s_2 s_1^2 s_2 s_3s_4 s_3^{-1}
&\in& V_0\end{array}$$ by lemma \ref{lemreduc2} (2). 
Part (2) is obtained by expanding $s_4^2$ and using (1) and lemma
\ref{lemN5} (1).
For (3), we use
$$
\begin{array}{ llcl}
& s_4 s_3^{-1} s_2 s_4 (s_3^{-1} s_2^2 s_3) s_4^{-1} w_0 s_4 & = & 
s_4 s_3^{-1} s_2 s_4 s_2 s_3^2 s_2^{-1} s_4^{-1} w_0 s_4\\ =
& s_4 s_3^{-1} s_2^2 (s_4  s_3^2  s_4^{-1}) w_0 s_4s_2^{-1}& =& 
s_4 s_3^{-1} s_2^2 s_3^{-1}  s_4^2  s_3 w_0 s_4s_2^{-1} \in V^+
\end{array} $$ because of (2).
For (4) we use 
$s_4 s_3^{-1} s_2 s_4 s_3^{-1} (s_2 s_1 s_2) s_3 s_4^{-1} w_0 s_4
= s_4 s_3^{-1} s_2 s_4 s_3^{-1} s_1 s_2 s_1 s_3 s_4^{-1} w_0 s_4
= s_4 s_3^{-1} s_2 s_4  s_1 (s_3^{-1}s_2  s_3) s_4^{-1} w_0 s_4s_1
= s_4 s_3^{-1} s_2 s_4  s_1 s_2s_3  s_2^{-1} s_4^{-1} w_0 s_4s_1
= s_4 s_3^{-1} s_2   s_1 s_2 s_4 s_3   s_4^{-1} w_0 s_4s_2^{-1}s_1
= s_4 s_3^{-1} s_2   s_1 s_2 s_3^{-1} s_4   s_3 w_0 s_4s_2^{-1}s_1 \in V^+$
by lemma \ref{lemN5} (1).

%Parts (3) and (4) are obvious consequences of lemma \ref{lemreduc2} (2).
%Expanding $s_4^2$, we get a linear combination of 
%$s_4 s_3^{-1} s_2^{\alpha}  s_3^{-1}  s_2 s_1^2 s_2 s_3s_4 s_3^{-1} \in V_0$,
%$s_4 s_3^{-1} s_2^{\alpha} s_4 s_3^{-1}  s_2 s_1^2 s_2 s_3s_4 s_3^{-1} =
%s_4 s_3^{-1} s_2^{\alpha} s_4^2 s_3^{-1}  s_2 s_1^2 s_2 s_3s_4 s_3^{-1}$.

\end{proof}

\begin{lemma}{\ }  \label{lemN7}
\begin{enumerate}
\item $s_4 s_3^{-1} A_3 s_4 u_3 s_4^{-1} w_0 s_4 \subset V^+$
\item$s_4 s_3^{-1} s_2 s_4 s_3^{-1} u_1 u_2 u_1 s_3 s_4^{-1} w_0 s_4 \subset V^+$
\item $s_4 s_3^{-1} s_2 s_4 s_3^{-1} (s_2 s_1^{-1} s_2) s_3 s_4^{-1} w_0 s_4 \in V^+$
%\item $s_4^{-1} s_3 u_2 s_3 s_1 s_4^{-1} s_3^{-1} s_2 s_3^{-1} s_2 u_1 s_2 s_3 s_4 \subset V_0$
%\item $u_4 u_3 s_2^2 s_1 s_3^{-1} s_2 s_4 s_3^{-1} (s_2 s_1^2 s_2) s_3 s_4 \subset V^+$
\end{enumerate}
\end{lemma}
\begin{proof}
We consider $s_4 s_3^{-1} A_3 s_4 s_3^{\alpha} s_4^{-1} w_0 s_4$
for $\alpha \in \{ -1,0,1 \}$. When $\alpha = 0$ this
expression clearly belongs to $V_0$, when $\alpha = 1$
we get $s_4 s_3^{-1} A_3 (s_4 s_3 s_4^{-1}) w_0 s_4
=s_4 s_3^{-1} A_3 s_3^{-1} s_4 s_3 w_0 s_4 \subset V^+$
by lemma \ref{lemN5} (1). When $\alpha = -1$,
we get
 $
 s_4 s_3^{-1} A_3 s_4 (s_3^{-1} s_4^{-1} s_3) s_2 s_1^2 s_2 s_3 s_4
=
s_4 s_3^{-1} A_3 s_4 s_4 s_3^{-1} s_4^{-1} s_2 s_1^2 s_2 s_3 s_4
=
s_4 s_3^{-1} A_3 s_4^2 s_3^{-1}  s_2 s_1^2 s_2 (s_4^{-1}s_3 s_4)
=
s_4 s_3^{-1}  s_4^2 A_3 s_3^{-1}  s_2 s_1^2 s_2 s_3s_4 s_3^{-1}
\subset V_0$ by lemma \ref{lemreduc2} (1). This proves (1).
We consider now
$s_4 s_3^{-1} s_2 s_4 s_3^{-1} u_1 u_2 u_1 s_3 s_4^{-1} w_0 s_4
=
s_4 s_3^{-1} s_2 u_1 s_4 (s_3^{-1}  u_2  s_3) s_4^{-1} w_0 s_4 u_1
=
s_4 s_3^{-1} s_2 u_1 s_4 s_2  u_3  s_2^{-1} s_4^{-1} w_0 s_4 u_1
=
s_4 s_3^{-1} s_2 u_1 s_2 (s_4   u_3   s_4^{-1}) w_0 s_4 s_2^{-1}u_1
=
s_4 s_3^{-1} s_2 u_1 s_2 s_3^{-1}   u_4   s_3 w_0 s_4 s_2^{-1}u_1$
which is a linear combination of the
$Y = s_4 s_3^{-1} s_2 u_1 s_2 s_3^{-1}   s_4^{\alpha}   s_3 w_0 s_4 s_2^{-1}u_1$
for $\alpha \in \{ -1,0,1 \}$. When $\alpha = 0$ clearly $Y \subset V_0$,
when $\alpha = 1$ we have $Y \subset V^+$ by lemma \ref{lemN5} (1),
and when $\alpha = -1$ we get
$s_4 s_3^{-1} s_2 u_1 s_2 (s_3^{-1}   s_4^{-1}   s_3) w_0 s_4 s_2^{-1}u_1
=s_4 s_3^{-1} s_2 u_1 s_2 s_4   s_3^{-1}   s_4^{-1} w_0 s_4 s_2^{-1}u_1 \subset V^+$
by (1). This proves (2).
We consider now $s_4 s_3^{-1} s_2 s_4 s_3^{-1} (s_2 s_1^{-1} s_2) s_3 s_4^{-1} w_0 s_4
= s_4 s_3^{-1} s_2 s_4 s_2 s_1 (s_3 s_2^{-1} s_3) s_1^{-1} s_2^{-1} s_4^{-1} w_0 s_4
=s_4 s_3^{-1} s_2^2 s_1 s_4  (s_3 s_2^{-1} s_3)  s_4^{-1} w_0 s_4s_1^{-1} s_2^{-1}
$. By lemma \ref{leminverse} it belongs to
$$s_4 s_3^{-1} s_2^2 s_1 s_4  s_3^{-1} s_2 s_3^{-1}  u_2 s_4^{-1} w_0 s_4s_1^{-1} s_2^{-1}
+ s_4 s_3^{-1} s_2^2 s_1 s_4 u_2 u_3 u_2  s_4^{-1} w_0 s_4s_1^{-1} s_2^{-1}.$$
We have 
$$s_4 s_3^{-1} s_2^2 s_1 s_4 u_2 u_3 u_2  s_4^{-1} w_0 s_4s_1^{-1} s_2^{-1} = 
s_4 s_3^{-1} s_2^2 s_1 u_2  s_4 u_3   s_4^{-1} w_0 s_4 u_2s_1^{-1} s_2^{-1} \subset V^+$$
by (1), and 
$$
\begin{array}{lcl}
 s_4 s_3^{-1} s_2^2 s_1 s_4  s_3^{-1} s_2 s_3^{-1}  u_2 s_4^{-1} w_0 s_4s_1^{-1} s_2^{-1}
&=& s_4 s_3^{-1} s_2^2 s_1 s_4  s_3^{-1} s_2 (s_3^{-1}   s_4^{-1} s_3) s_2 s_1^2 s_2 s_3  s_4u_2 s_1^{-1} s_2^{-1}\\
&= &s_4 s_3^{-1} s_2^2 s_1 s_4  s_3^{-1} s_2 s_4   s_3^{-1} s_4^{-1} s_2 s_1^2 s_2 s_3  s_4u_2 s_1^{-1} s_2^{-1}\\
&=& s_4 s_3^{-1} s_2^2 s_1 s_4  s_3^{-1} s_2 s_4   s_3^{-1}  s_2 s_1^2 s_2 (s_4^{-1} s_3  s_4)u_2 s_1^{-1} s_2^{-1}\\
&=& (s_4 s_3^{-1} s_4)s_2^2 s_1   s_3^{-1} s_2 s_4   s_3^{-1}  s_2 s_1^2 s_2 s_3 s_4  s_3^{-1}u_2 s_1^{-1} s_2^{-1}\\
\end{array}$$
belongs to 
$$u_3 s_4^{-1} s_3 s_4^{-1}s_2^2 s_1   s_3^{-1} s_2 s_4   s_3^{-1}  s_2 s_1^2 s_2 s_3 s_4  s_3^{-1}u_2 s_1^{-1} s_2^{-1}+
u_3 u_4 u_3s_2^2 s_1   s_3^{-1} s_2 s_4   s_3^{-1}  s_2 s_1^2 s_2 s_3 s_4  s_3^{-1}u_2 s_1^{-1} s_2^{-1}.$$
We have 

$$
\begin{array}{llcl}
 &s_4^{-1} s_3 s_4^{-1}s_2^2 s_1   s_3^{-1} s_2 s_4   s_3^{-1}  s_2 s_1^2 s_2 s_3 s_4&=&
s_4^{-1} s_3 s_2^2 s_1 (s_4^{-1}  s_3^{-1} s_4 )s_2    s_3^{-1}  s_2 s_1^2 s_2 s_3 s_4 \\ =
& s_4^{-1} s_3 s_2^2 s_1 s_3  s_4^{-1} s_3^{-1} s_2    s_3^{-1}  s_2 s_1^2 s_2 s_3 s_4 &=& 
s_4^{-1} s_3 s_2^2 s_3  s_4^{-1} s_1   s_3^{-1} s_2    s_3^{-1}  s_2 s_1^2 s_2 s_3 s_4   \in u_4 u_3 u_2 u_3 u_4 A_4 s_4  \\ 
\end{array}
$$
and also
$u_4 u_3s_2^2 s_1   s_3^{-1} s_2 s_4   s_3^{-1}  s_2 s_1^2 s_2 s_3 s_4  s_3^{-1}u_2 s_1^{-1} s_2^{-1}
\subset u_4 u_3 u_2 u_3 u_4 A_4 s_4 A_4$. The conclusion follows from \ref{lemN5} (4).

%(4) is a consequence of lemma \ref{lemN5} (4)

\end{proof}

\subsection{Reduction to $s_4 s_3 (s_2 s_1^{-1} s_2) s_3^{-1} s_4 s_3^{-1} (s_2 s_1^2 s_2) s_3 s_4^{-1} w_0 s_4$} 
Expanding $s_1^2$, we get
$$
\begin{array}{lcl}
s_4 s_3 (s_2 s_1^2 s_2) s_3^{-1} s_4 s_3^{-1} (s_2 s_1^2 s_2) s_3 s_4^{-1} w_0 s_4 & \in & 
R^{\times} s_4 s_3 (s_2 s_1^{-1} s_2) s_3^{-1} s_4 s_3^{-1} (s_2 s_1^2 s_2) s_3 s_4^{-1} w_0 s_4 \\
 & & + R s_4 s_3 (s_2 s_1 s_2) s_3^{-1} s_4 s_3^{-1} (s_2 s_1^2 s_2) s_3 s_4^{-1} w_0 s_4 \\
 & & + R s_4 s_3 s_2^2 s_3^{-1} s_4 s_3^{-1} (s_2 s_1^2 s_2) s_3 s_4^{-1} w_0 s_4 \\
 \end{array}
 $$
 Since $$
 \begin{array}{lcl}
 s_4 s_3 (s_2 s_1 s_2) s_3^{-1} s_4 s_3^{-1} (s_2 s_1^2 s_2) s_3 s_4^{-1} w_0 s_4 &=& 
 s_4 s_3 s_1 s_2 s_1 s_3^{-1} s_4 s_3^{-1} (s_2 s_1^2 s_2) s_3 s_4^{-1} w_0 s_4\\
&=&
 s_1 s_4 s_3  s_2  s_3^{-1} s_4 s_3^{-1} (s_2 s_1^2 s_2) s_3 s_4^{-1} w_0 s_4s_1\\
 \end{array}$$ (as $s_1$ commutes with
 $s_2 s_1^2 s_2 = c_3 c_2^{-1}$), the latter two terms belong to
 $$
 \begin{array}{lcl}
 A_2 s_4 (s_3 u_2 s_3^{-1}) s_4 s_3^{-1} (s_2 s_1^2 s_2) s_3 s_4^{-1} w_0 s_4 A_2
& =& A_2 s_4 s_2^{-1} u_3 s_2 s_4 s_3^{-1} (s_2 s_1^2 s_2) s_3 s_4^{-1} w_0 s_4 A_2 \\
& \subset& A_3 s_4  u_3 s_4 s_2  s_3^{-1} (s_2 s_1^2 s_2) s_3 s_4^{-1} w_0 s_4 A_2\end{array}.$$
 We thus only need to prove that the
 $s_4  s_3^{\alpha} s_4 s_2  s_3^{-1} (s_2 s_1^2 s_2) s_3 s_4^{-1} w_0 s_4$
 belong to $V^+$ for $\alpha \in \{ -1,0,1 \}$. When $\alpha = 0$ this is a consequence of lemma \ref{lemN5} (6) ;
 when $\alpha = 1$ we get 
$$
 \begin{array}{lcl}
(s_4  s_3 s_4) s_2  s_3^{-1} (s_2 s_1^2 s_2) s_3 s_4^{-1} w_0 s_4
&=&
s_3  s_4 (s_3 s_2  s_3^{-1}) (s_2 s_1^2 s_2) s_3 s_4^{-1} w_0 s_4\\
&=&
s_3  s_4 s_2^{-1} s_3  (s_2^2 s_1^2 s_2) s_3 s_4^{-1} w_0 s_4.
\end{array}$$ Since 
$s_2^2 s_1^2 s_2 \in u_1 s_2^{-1} s_1 s_2^{-1} + u_1 u_2 u_1$
the conclusion follows from proposition \ref{propmoinsde55}.
When $\alpha = -1$, expanding $s_1^2$ we only need to consider the 
$s_4  s_3^{-1} s_4 s_2  s_3^{-1} s_2 s_1^{\beta} s_2 s_3 s_4^{-1} w_0 s_4$
for $\beta \in \{-1,0,1 \}$. The case $\beta = -1$ is a consequence of lemma \ref{lemN7} (3), while
the other two cases follow from lemma \ref{lemN6} (3) and (4).

\begin{lemma}{\ } \label{lemN8} 
\begin{enumerate}
\item $s_4 u_3 s_2 u_1 s_2 s_3^{-1} u_4 s_3 w_0 s_4 \subset V^+$
\item $s_4 s_3 s_2 s_1^{-1} u_3 u_2 u_4 s_3 w_0 s_4 \subset V_0$
\item $s_4 s_3 s_2 s_1^{-1} u_2 u_3 u_4 s_3 w_0 s_4 \subset V^+$
\end{enumerate}
\end{lemma}
\begin{proof}
For proving (1), we consider
the expression $s_4 s_3^{\alpha} s_2 s_1^{\beta} s_2 s_3^{\gamma} u_4 s_3 w_0 s_4$
for $\alpha,\beta,\gamma \in \{ -1,0,1 \}$. If one of these is $0$, it lies
in $V_0$ by proposition \ref{propmoinsde55}. If $\beta = 1$,
using $s_2 s_1 s_2 = s_1 s_2 s_1$ we get the same conclusion, so we
can assume $\beta = -1$.  By expanding if necessary $s_3^2$, we are then reduced to considering expressions of
the form $s_4 s_3^{\alpha} s_2 s_1^{\beta} s_2 s_3^{\gamma} u_4 s_3^{\delta} s_2 s_1^2 s_2 s_3 s_4$
for $\delta \in \{0,1,-1 \}$, the case $\delta = 0$ being again trivial. We then get the conclusion
from lemmas \ref{lemreducspec} and \ref{lemtransf1} (6).

We now prove (2).  We have
$$
\begin{array}{clclcl}
& s_4 s_3 s_2 s_1^{-1} u_3 u_2 u_4 s_3 w_0 s_4 &=& 
s_4 (s_3 s_2 u_3 )s_1^{-1}  u_2 u_4 s_3 w_0 s_4\\
\subset & 
s_4 u_2 s_3 s_2 s_1^{-1}  u_2 u_4 s_3 w_0 s_4&=&
u_2s_4  s_3 s_2 s_1^{-1}  u_2 u_4 s_3 w_0 s_4& \subset & V_0 \\
\end{array}$$
by proposition \ref{propmoinsde55}. Finally, (3) is similar to (1) : considering
$s_4 s_3 s_2 s_1^{-1} s_2^{\alpha} s_3^{\beta} s_4^{\gamma} s_3^{\delta} s_2 s_1^2 s_2 s_3 s_4$,
if one of the exponents is zero we get trivially the conclusion by proposition \ref{propmoinsde55} ; if $\alpha = -1$
it lies inside $V_0$ by $s_2 s_1^{-1} s_2^{-1} = s_1^{-1} s_2^{-1} s_1$ and proposition \ref{propmoinsde55}, so
we can assume $\alpha = 1$. By studying separately the cases $\beta = -1$ and $\beta = 1$ one easily
gets the conclusion from lemmas \ref{lemreducspec}, \ref{lemreduc2} and \ref{lemtransf1}.

%If $\gamma = -\alpha$, this is a consequence of lemma \ref{lemN5} (6),
%so we can assume $\alpha = \gamma$.
\end{proof}
\begin{comment}
\begin{lemma}{\ } (pour remplacer dans le comptage le lemme N9) 
\end{lemma}
\begin{lemma}{\ } (pour remplacer dans le comptage le lemme N10) 
\end{lemma}
\end{comment}

\subsection{Reduction to $s_4 s_3 (s_2 s_1^{-1} s_2) s_3^{-1} s_4 s_3^{-1} (s_2^{-1} s_1 s_2^{-1}) s_3 s_4^{-1} w_0 s_4$} 

Using $s_2 s_1^2 s_2 \in s_2^{-1} s_1 s_2^{-1} u_1^{\times} + u_1 u_2 u_1$ we
get
$$
\begin{array}{lcl}
s_4 s_3 (s_2 s_1^{-1} s_2) s_3^{-1} s_4 s_3^{-1} (s_2 s_1^2 s_2) s_3 s_4^{-1} w_0 s_4 & \in &
R^{\times} s_4 s_3 (s_2 s_1^{-1} s_2) s_3^{-1} s_4 s_3^{-1} s_2^{-1} s_1 s_2^{-1} s_3 s_4^{-1} w_0 s_4 u_1^{\times} \\
& & + R s_4 s_3 (s_2 s_1^{-1} s_2) s_3^{-1} s_4 s_3^{-1} u_1 u_2 u_1 s_3 s_4^{-1} w_0 s_4. \\
\end{array}
$$

We have 
$$
 \begin{array}{llcl}
& s_4 s_3 (s_2 s_1^{-1} s_2) s_3^{-1} s_4 s_3^{-1} u_1 u_2 u_1 s_3 s_4^{-1} w_0 s_4 &=&
s_4 s_3 (s_2 s_1^{-1} s_2) u_1 s_3^{-1} s_4 (s_3^{-1}  u_2  s_3) s_4^{-1} w_0 s_4u_1\\
=&
s_4 s_3 (s_2 s_1^{-1} s_2) u_1 s_3^{-1} s_4 s_2  u_3  s_2^{-1} s_4^{-1} w_0 s_4u_1
&=&
s_4 s_3 (s_2 s_1^{-1} s_2) u_1 s_3^{-1} s_2 ( s_4  u_3   s_4^{-1}) w_0 s_4 s_2^{-1}u_1\\
=&
s_4 s_3 (s_2 s_1^{-1} s_2) u_1 s_3^{-1} s_2  s_3^{-1}  u_4   s_3 w_0 s_4 s_2^{-1}u_1.
\end{array}$$
Using $(s_2 s_1^{-1} s_2) u_1  \in u_1 s_2 s_1^{-1} s_2 + u_1 u_2 u_1$ we get that
$s_4 s_3 (s_2 s_1^{-1} s_2) u_1 s_3^{-1} s_2  s_3^{-1}  u_4   s_3 w_0 s_4$ belongs to
$$ 
u_1 s_4 s_3 (s_2 s_1^{-1} s_2) s_3^{-1} s_2  s_3^{-1}  u_4   s_3 w_0 s_4 +
s_4 s_3 u_1 u_2 u_1 s_3^{-1} s_2  s_3^{-1}  u_4   s_3 w_0 s_4.
$$
Now
$$
 \begin{array}{llcl}
&s_4 s_3 u_1 u_2 u_1 s_3^{-1} s_2  s_3^{-1}  u_4   s_3 w_0 s_4 
&=&
u_1 s_4 (s_3  u_2  s_3^{-1}) u_1s_2  s_3^{-1}  u_4   s_3 w_0 s_4\\
=&
u_1 s_4 s_2^{-1}  u_3  s_2 u_1s_2  s_3^{-1}  u_4   s_3 w_0 s_4
&=&
u_1 s_2^{-1} s_4   u_3  s_2 u_1s_2  s_3^{-1}  u_4   s_3 w_0 s_4 \subset V^+
\end{array}$$
by lemma \ref{lemN8} (1),
and
$s_4 s_3 s_2 s_1^{-1} (s_2 s_3^{-1} s_2  s_3^{-1})  u_4   s_3 w_0 s_4$ belong to
$$ s_4 s_3 s_2 s_1^{-1}  s_3^{-1} s_2  s_3^{-1}s_2 u_4   s_3 w_0 s_4 + 
s_4 s_3 s_2 s_1^{-1} u_2 u_3 u_4   s_3 w_0 s_4
+
s_4 s_3 s_2 s_1^{-1} u_3 u_2 u_4   s_3 w_0 s_4$$ by lemma \ref{lemdecomp1212}.
The latter two terms belong to $V+$ by lemma \ref{lemN8} (1) and (2),
and 
$$
 \begin{array}{llcl}
& s_4 s_3 s_2 s_1^{-1}  s_3^{-1} s_2  s_3^{-1}s_2 u_4   s_3 w_0 s_4 &=& 
s_4 (s_3 s_2 s_3^{-1} )s_1^{-1}   s_2  s_3^{-1}s_2 u_4   s_3 w_0 s_4\\
=&
s_4 s_2^{-1} s_3 s_2s_1^{-1}   s_2  s_3^{-1}s_2 u_4   s_3 w_0 s_4
&=&
s_2^{-1} s_4  s_3 s_2 s_1^{-1}   s_2  s_3^{-1} u_4 s_2   s_3 w_0 s_4 \subset V^+
\end{array}$$
by lemma \ref{lemN5} (6).

\begin{lemma}{\ } \label{lemN11}
%\begin{enumerate}
%\item 
$s_4 (s_3 s_2^{-1} s_3) u_1 s_4 (s_3^{-1} s_2 s_3^{-1}) s_4^{-1} w_0 s_4 \subset V^+$
%\item $s_4 s_3 u_2 u_1 s_3 s_2 s_4^{-1} s_3 w_0 s_4 \subset V_0$
%\end{enumerate}
\end{lemma}
\begin{proof}
%We prove (1). 
Using braid relations one gets
$$
\begin{array}{clcl}
& s_4 (s_3 s_2^{-1} s_3) u_1 s_4 (s_3^{-1} s_2 s_3^{-1}) s_4^{-1} w_0 s_4 &
= & s_4 s_3 s_2^{-1}  u_1 (s_3s_4 s_3^{-1}) s_2 s_3^{-1} s_4^{-1} w_0 s_4 \\
= & s_4 s_3 s_2^{-1}  u_1 s_4^{-1}s_3 s_4 s_2 s_3^{-1} s_4^{-1} w_0 s_4 &
= & (s_4 s_3 s_4^{-1}) s_2^{-1}  u_1 s_3  s_2 (s_4s_3^{-1} s_4^{-1}) w_0 s_4 \\
= & s_3^{-1} s_4 s_3 s_2^{-1}  u_1 (s_3  s_2 s_3^{-1})s_4^{-1} s_3 w_0 s_4 &
= & s_3^{-1} s_4 s_3 s_2^{-1}  u_1 s_2^{-1}  s_3 s_2s_4^{-1} s_3 w_0 s_4 \\
\end{array}
%= s_3^{-1} s_4 s_3 s_2^{-1} u_1 s_2^{-1} s_3 s_2 s_4^{-1} s_3 w_0 s_4,
$$
which is a linear combination of 
$s_3^{-1} s_4 s_3 s_2^{-1} s_1^{\alpha} s_2^{-1} s_3 s_2 s_4^{-1} s_3 w_0 s_4$
for $\alpha \in \{ -1,0,1 \}$. For $\alpha = -1$ we get 
$$
 \begin{array}{lcl}
s_3^{-1} s_4 s_3 (s_2^{-1} s_1^{-1} s_2^{-1}) s_3 s_2 s_4^{-1} s_3 w_0 s_4
&=&
s_3^{-1} s_4 s_3 s_1^{-1} s_2^{-1} s_1^{-1} s_3 s_2 s_4^{-1} s_3 w_0 s_4\\
&=&
s_3^{-1}s_1^{-1} s_4 s_3  s_2^{-1}  s_3 s_4^{-1} s_1^{-1}s_2  s_3 w_0 s_4 \in V^+
\end{array}$$
by lemma \ref{lemN5} (4) ; for $\alpha  =0$ we get
$s_3^{-1} s_4 s_3 s_2^{-2} s_3 s_2 s_4^{-1} s_3 w_0 s_4 \in V_0$
by proposition \ref{propmoinsde55} ; for $\alpha = 1$
it remains to consider $s_3^{-1} s_4 s_3 (s_2^{-1} s_1 s_2^{-1}) s_3 s_2 s_4^{-1} s_3 w_0 s_4$.
Using $s_2^{-1} s_1 s_2^{-1} \in u_1 s_2 s_1^{-1} s_2 + u_1 u_2 u_1$ we get
$s_4 s_3 (s_2^{-1} s_1 s_2^{-1}) s_3 s_2 s_4^{-1} s_3 w_0 s_4 \in u_1 
s_4 s_3 s_2 s_1^{-1} s_2 s_3 s_2 s_4^{-1} s_3 w_0 s_4 +
u_1 s_4 s_3  u_2 u_1 s_3 s_2 s_4^{-1} s_3 w_0 s_4$. Now
$s_4 s_3  u_2 u_1 s_3 s_2 s_4^{-1} s_3 w_0 s_4 = s_4 s_3  u_2 s_3 s_4^{-1} u_1  s_2  s_3 w_0 s_4\subset V^+ $ by lemma \ref{lemN5} (4)
while 
$$
 \begin{array}{llcl}
& s_4 s_3 s_2 s_1^{-1} (s_2 s_3 s_2) s_4^{-1} s_3 w_0 s_4
&=&
s_4 s_3 s_2 s_1^{-1} s_3 s_2 s_3 s_4^{-1} s_3 w_0 s_4 \\
=&
s_4 (s_3 s_2  s_3) s_1^{-1} s_2 s_3 s_4^{-1} s_3 w_0 s_4
&=&
s_4 s_2 s_3  s_2 s_1^{-1} s_2 s_3 s_4^{-1} s_3 w_0 s_4\\
=&
s_2 s_4  s_3  s_2 s_1^{-1} s_2 s_3 s_4^{-1} s_3 w_0 s_4\end{array}$$
lies in $V^+$ by lemma \ref{lemreducspec}. This concludes the proof.
%We now prove (2). 
\end{proof}

\begin{lemma}{\ } \label{lemN13}
%\begin{enumerate}\item
%\end{enumerate}
$s_4 (s_3 s_2^{-1} s_3) s_2 s_1 s_4 u_2 u_3 s_4^{-1} w_0 s_4 \subset V^+$
\end{lemma}
\begin{proof}
We have 
$$
 \begin{array}{lcl}
 s_4 (s_3 s_2^{-1} s_3) s_2 s_1 s_4 u_2 u_3 s_4^{-1} w_0 s_4
&=&
s_4 (s_3 s_2^{-1} s_3) s_2 s_1 u_2 (s_4  u_3 s_4^{-1}) w_0 s_4\\
&\subset&
s_4 s_3 (s_2^{-1} s_3 s_2) s_1 u_2 s_3^{-1}  u_4 s_3 w_0 s_4 \\
&=&
s_4 s_3^2 s_2 s_3^{-1} s_1 u_2 s_3^{-1}  u_4 s_3 w_0 s_4,
\end{array}$$
whose elements are linear combinations of the
$s_4 s_3^2 s_2 s_3^{-1} s_1 s_2^{\alpha} s_3^{-1}  u_4 s_3 w_0 s_4$
for $\alpha \in \{0,1,-1 \}$. When $\alpha = 0$,
such an element belongs to $V_0$ by proposition \ref{propmoinsde55} ;
when $\alpha = -1$, we have 
$s_4 s_3^2 s_2 s_3^{-1} s_1 s_2^{-1} s_3^{-1}  u_4 s_3 w_0 s_4=
s_4 s_3^2 s_2  s_1 (s_3^{-1}s_2^{-1} s_3^{-1})  u_4 s_3 w_0 s_4=
s_4 s_3^2( s_2  s_1 s_2^{-1})s_3^{-1} s_2^{-1}  u_4 s_3 w_0 s_4=
s_4 s_3^2 s_1^{-1}  s_2 s_1s_3^{-1} s_2^{-1}  u_4 s_3 w_0 s_4=
s_1^{-1}s_4 s_3^2   s_2 s_3^{-1} u_4 s_1 s_2^{-1}   s_3 w_0 s_4 \in V^+$
by lemma \ref{lemN5} (4) ; when $\alpha = 1$,
expanding $s_3^2$ we get a linear combination of 
$s_4 s_3^{\beta} s_2 s_3^{-1} s_1 s_2 s_3^{-1}  u_4 s_3 w_0 s_4$ for $\beta \in \{ 0,1,-1 \}$.
When $\beta =0$ such an element lies in $V_0$ by commutation relations and proposition \ref{propmoinsde55} ;
when $\beta = 1$ we get
$s_4 (s_3 s_2 s_3^{-1}) s_1 s_2 s_3^{-1}  u_4 s_3 w_0 s_4=
s_4 s_2^{-1} s_3 s_2 s_1 s_2 s_3^{-1}  u_4 s_3 w_0 s_4=
s_2^{-1} s_4  s_3 (s_2 s_1 s_2) s_3^{-1}  u_4 s_3 w_0 s_4=
s_2^{-1} s_4  s_3 s_1 s_2 s_1 s_3^{-1}  u_4 s_3 w_0 s_4=
s_2^{-1}s_1 s_4  s_3  s_2 s_1 s_3^{-1}  u_4 s_3 w_0 s_4 \in V_0$
by proposition \ref{propmoinsde55} ;
when $\beta = -1$ we get that
$s_4 s_3^{-1} s_2 s_3^{-1} s_1 s_2 s_3^{-1}  u_4 s_3 w_0 s_4 
=s_4 s_3^{-1} s_2 s_1 (s_3^{-1}  s_2 s_3^{-1})  u_4 s_3 w_0 s_4$ belongs to
$$s_4 s_3^{-1} s_2 s_1 u_2 s_3  s_2^{-1} s_3   u_4 s_3 w_0 s_4 +
s_4 s_3^{-1} s_2 s_1 u_2 u_3 u_2 u_4 s_3 w_0 s_4 $$ by lemma \ref{leminverse}.
Now 
$$
 \begin{array}{lcl}
 s_4 s_3^{-1} (s_2 s_1 u_2) u_3 u_2 u_4 s_3 w_0 s_4&
\subset &s_4 s_3^{-1} u_1 s_2 s_1 u_3 u_2 u_4 s_3 w_0 s_4\\
&\subset & u_1 s_4 s_3^{-1}  s_2 u_3 u_4 s_1  u_2  s_3 w_0 s_4 \subset V^+
\end{array}$$ by lemma \ref{lemN5} (4) and
$$
 \begin{array}{llcl}
& s_4 s_3^{-1} (s_2 s_1 u_2) s_3  s_2^{-1} s_3   u_4 s_3 w_0 s_4 &=&
s_4 s_3^{-1} u_1 s_2 s_1  s_3  s_2^{-1} s_3   u_4 s_3 w_0 s_4 \\ =&
u_1 s_4 s_3^{-1}  s_2 s_1  s_3  s_2^{-1} s_3   u_4 s_3 w_0 s_4 &=&
u_1 s_4 (s_3^{-1}  s_2 s_3)s_1    s_2^{-1} s_3   u_4 s_3 w_0 s_4\\ =&
u_1 s_4 s_2  s_3 s_2^{-1}s_1    s_2^{-1} s_3   u_4 s_3 w_0 s_4 &=&
u_1  s_2  s_4 s_3 s_2^{-1}s_1    s_2^{-1} s_3   u_4 s_3 w_0 s_4.\\
\end{array}$$
Since $s_4 s_3 s_2^{-1}s_1    s_2^{-1} s_3   u_4 s_3 w_0 s_4$
is spanned by the  $s_4 w^+   s_4^{\alpha} s_3^{\beta} s_2 s_1^2 s_2 s_3 s_4$
one readily gets 
$$s_4 s_3 s_2^{-1}s_1    s_2^{-1} s_3   u_4 s_3 w_0 s_4 \subset V^+$$ and the conclusion.

\end{proof}

\begin{lemma}{\ } \label{lemN14}
\begin{enumerate}
\item $s_4 (s_3 s_2^{-1} s_3) s_2 s_1 s_4 u_3 u_2 s_4^{-1} w_0 s_4 \subset V_0$
\item $u_4 u_3 u_2 u_1 s_3^{-1} s_2 s_3^{-1} s_4^{-1} w_0 s_4 \subset V^+$
\end{enumerate}
\end{lemma}
\begin{proof}
We prove (1). 
$$
 \begin{array}{llcl}
& s_4 (s_3 s_2^{-1} s_3) s_2 s_1 s_4 u_3 u_2 s_4^{-1} w_0 s_4 &=&
s_4 (s_3 s_2^{-1} s_3) s_2 s_1 (s_4 u_3  s_4^{-1}) w_0 s_4u_2  \\ \subset
& s_4 (s_3 s_2^{-1} s_3) s_2 s_1 s_3^{-1} u_4  s_3 w_0 s_4u_2 &=&
s_4 s_3 s_2^{-1} (s_3 s_2 s_3^{-1})s_1  u_4  s_3 w_0 s_4u_2 \\=&
s_4 s_3 s_2^{-2}  s_3 s_2 s_1  u_4  s_3 w_0 s_4u_2 & \subset&
s_4 u_3 u_2 u_3 u_2    s_1u_4   s_3 w_0 s_4u_2 \\
=&
s_4 u_2 u_3 u_2 u_3    s_1 u_4  s_3 w_0 s_4u_2 
&=&
 u_2s_4 u_3 u_2 u_3    s_1 u_4  s_3 w_0 s_4u_2 \subset V_0\end{array}$$
 by proposition \ref{propmoinsde55}.
We now prove (2). We have
$$
 \begin{array}{lcl}
 u_4 u_3 u_2 u_1 s_3^{-1} s_2 s_3^{-1} s_4^{-1} w_0 s_4 &=& 
u_4 (u_3 u_2  s_3^{-1}) u_1 s_2 s_3^{-1} s_4^{-1} w_0 s_4 \\
&\subset &u_4 u_2 s_3 s_2^{-1} s_3 
u_1 s_2 s_3^{-1} s_4^{-1} w_0 s_4
+u_4 u_2 u_3 u_2 u_1 s_2 s_3^{-1} s_4^{-1} w_0 s_4,
\end{array}$$
and 
$u_4 u_2 u_3 u_2 u_1 s_2 s_3^{-1} s_4^{-1} w_0 s_4
=
 u_2 u_4  u_3 u_2 u_1 s_2 s_3^{-1} s_4^{-1} w_0 s_4$
with  $u_4  u_3 u_2 u_1 s_2 s_3^{-1} s_4^{-1} w_0 s_4 \subset
V_0 + A_4 s_4 w^- s_4^{-1} w^+ s_4 A_4
+ A_4 s_4^{-1} w^- s_4^{-1} w^+ s_4 A_4 \subset V_0$
by lemma \ref{lemreduc2} (3) and (4).
Moreover
$$
 \begin{array}{llcl}
& u_4 u_2 s_3 s_2^{-1} s_3 
u_1 s_2 s_3^{-1} s_4^{-1} w_0 s_4 &=& 
u_2 u_4  s_3 s_2^{-1}  u_1 (s_3  s_2 s_3^{-1}) s_4^{-1} w_0 s_4 \\ = &
u_2 u_4  s_3 s_2^{-1}  u_1 s_2^{-1}  s_3 s_2 s_4^{-1} w_0 s_4 &=& 
u_2 u_4  s_3 s_2^{-1}  u_1 s_2^{-1}  s_3  s_4^{-1} w_0 s_4s_2 ,
\end{array}$$
so we are reduced to studying
$$
\begin{array}{lcl}
u_4  s_3 (s_2^{-1}  u_1 s_2^{-1})  s_3  s_4^{-1} w_0 s_4
&\subset&
u_4  s_3 u_1 s_2  s_1^{-1} s_2  s_3  s_4^{-1} w_0 s_4
+
u_4  s_3 u_1 u_2 u_1  s_4^{-1} w_0 s_4 \\
&\subset&u_1 u_4  s_3  s_2  s_1^{-1} s_2  s_3  s_4^{-1} w_0 s_4
+
u_1 u_4  s_3  u_2 u_1  s_4^{-1} w_0 s_4\\
&\subset& V_0 + A_4 s_4 w^+ s_4^{-1} w^+ s_4 A_4 + A_4 s_4^{-1} w^+ s_4^{-1} w^+ s_4 A_4.
\end{array}$$
Since $s_4^{-1} w^+ s_4^{-1} w^+ s_4 \in V_0$ by lemma \ref{lemreduc2} (8) (after applying $\Phi$),
this concludes the proof of (2).
\end{proof}

\begin{lemma}{\ } \label{lemN12}
$s_4 (s_3 s_2^{-1} s_3) s_1 s_2 s_1 s_4 (s_3^{-1} s_2 s_3^{-1}) s_4^{-1} w_0 s_4 \in V^+$
\end{lemma}
We have
$s_4 (s_3 s_2^{-1} s_3) (s_1 s_2 s_1) s_4 (s_3^{-1} s_2 s_3^{-1}) s_4^{-1} w_0 s_4  
= s_4 (s_3 s_2^{-1} s_3) s_2 s_1 s_2 s_4 (s_3^{-1} s_2 s_3^{-1}) s_4^{-1} w_0 s_4 
= s_4 (s_3 s_2^{-1} s_3) s_2 s_1  s_4 (s_2 s_3^{-1} s_2 s_3^{-1}) s_4^{-1} w_0 s_4 $
which belongs to $s_4 (s_3 s_2^{-1} s_3) s_2 s_1  s_4  s_3^{-1} s_2 s_3^{-1}s_2 s_4^{-1} w_0 s_4  +
s_4 (s_3 s_2^{-1} s_3) s_2 s_1  s_4 u_2 u_3 s_4^{-1} w_0 s_4 +
s_4 (s_3 s_2^{-1} s_3) s_2 s_1  s_4 u_3 u_2 s_4^{-1} w_0 s_4 $ by lemma \ref{lemdecomp1212}.
Now $$s_4 (s_3 s_2^{-1} s_3) s_2 s_1  s_4 u_2 u_3 s_4^{-1} w_0 s_4 \subset V^+$$
by lemma \ref{lemN13} %(2)
, while 
$s_4 (s_3 s_2^{-1} s_3) s_2 s_1  s_4 u_3 u_2 s_4^{-1} w_0 s_4 \subset V^+$ by lemma
\ref{lemN14} (1). We are thus reduced to considering
$$
 \begin{array}{lcl}
 s_4 (s_3 s_2^{-1} s_3) s_2 s_1  s_4  s_3^{-1} s_2 s_3^{-1}s_2 s_4^{-1} w_0 s_4 &\in&
s_4 s_2(s_3 s_2^{-1} s_3)  s_1  s_4  s_3^{-1} s_2 s_3^{-1}s_2 s_4^{-1} w_0 s_4\\ &&+
s_4u_2 u_3 u_2 s_1  s_4  s_3^{-1} s_2 s_3^{-1}s_2 s_4^{-1} w_0 s_4\end{array}$$ by lemma \ref{lemquasicom}.
We have $$s_4 s_2(s_3 s_2^{-1} s_3)  s_1  s_4  s_3^{-1} s_2 s_3^{-1}s_2 s_4^{-1} w_0 s_4=
s_2 s_4 (s_3 s_2^{-1} s_3)  s_1  s_4  s_3^{-1} s_2 s_3^{-1} s_4^{-1} w_0 s_4s_2 \in V^+$$
by lemma \ref{lemN11}. We have
$s_4u_2 u_3 u_2 s_1  s_4  s_3^{-1} s_2 s_3^{-1}s_2 s_4^{-1} w_0 s_4=
u_2 s_4 u_3 u_2 s_1  s_4  s_3^{-1} s_2 s_3^{-1} s_4^{-1} w_0 s_4s_2$
and $s_4 u_3 u_2 s_1  s_4  s_3^{-1} s_2 s_3^{-1} s_4^{-1} w_0 s_4$
is a linear combination of the
$s_4 s_3^{\alpha} u_2 s_1  s_4  s_3^{-1} s_2 s_3^{-1} s_4^{-1} w_0 s_4$
for $\alpha \in \{ 0,1,-1 \}$. When $\alpha = 0$ we get 
$s_4  u_2 s_1  s_4  s_3^{-1} s_2 s_3^{-1} s_4^{-1} w_0 s_4
= u_2 s_1  s_4^2  s_3^{-1} s_2 s_3^{-1} s_4^{-1} w_0 s_4 \subset V^+$ by proposition \ref{propmoinsde55} ;
%lemma \ref{lemN5} (4)
for $\alpha = 1$ we have
$$
 \begin{array}{llcl}
& s_4 s_3 u_2 s_1  s_4  s_3^{-1} s_2 s_3^{-1} s_4^{-1} w_0 s_4&=&
( s_4 s_3 s_4)u_2 s_1    s_3^{-1} s_2 s_3^{-1} s_4^{-1} w_0 s_4\\
=&
 s_3 s_4 s_3 u_2 s_1    s_3^{-1} s_2 s_3^{-1} s_4^{-1} w_0 s_4
&=&
 s_3 s_4 (s_3 u_2 s_3^{-1}) s_1     s_2 s_3^{-1} s_4^{-1} w_0 s_4\\
=&
 s_3 s_4 s_2^{-1} u_3 s_2 s_1     s_2 s_3^{-1} s_4^{-1} w_0 s_4
&=&
 s_3 s_2^{-1} s_4  u_3 (s_2 s_1     s_2) s_3^{-1} s_4^{-1} w_0 s_4 \\
=&
 s_3 s_2^{-1} s_4  u_3s_1 s_2     s_1 s_3^{-1} s_4^{-1} w_0 s_4 
&=&
 s_3 s_2^{-1} s_1 s_4  u_3 s_2      s_3^{-1} s_4^{-1} w_0 s_4s_1  
 \subset V^+\end{array}$$ by proposition \ref{propmoinsde55} ;
 for $\alpha = -1$ we have
$$
 \begin{array}{lcl}
 s_4 s_3^{-1} u_2 s_1  s_4  s_3^{-1} s_2 s_3^{-1} s_4^{-1} w_0 s_4
&=&(s_4 s_3^{-1} s_4) u_2 s_1    s_3^{-1} s_2 s_3^{-1} s_4^{-1} w_0 s_4\\
&\subset &u_3^{\times} s_4^{-1} s_3 s_4^{-1} u_2 s_1    s_3^{-1} s_2 s_3^{-1} s_4^{-1} w_0 s_4 \\ & & 
+u_3 u_4 u_3 u_2 s_1    s_3^{-1} s_2 s_3^{-1} s_4^{-1} w_0 s_4 
\end{array}$$
by lemma \ref{leminverse},
and $u_4 u_3 u_2 s_1    s_3^{-1} s_2 s_3^{-1} s_4^{-1} w_0 s_4 \subset V^+$
by lemma \ref{lemN14} (2). Finally,
$$
 \begin{array}{llcl}
& s_4^{-1} s_3 s_4^{-1} u_2 s_1    s_3^{-1} s_2 s_3^{-1} s_4^{-1} w_0 s_4
&=&
s_4^{-1} s_3  u_2 s_1  s_4^{-1}  s_3^{-1} s_2 (s_3^{-1} s_4^{-1}s_3) s_2 s_1^2 s_2 s_3 s_4\\
=&
s_4^{-1} s_3  u_2 s_1  s_4^{-1}  s_3^{-1} s_2 s_4 s_3^{-1}s_4^{-1} s_2 s_1^2 s_2 s_3 s_4
&=&
s_4^{-1} s_3  u_2 s_1  (s_4^{-1}  s_3^{-1} s_4) s_2  s_3^{-1} s_2 s_1^2 s_2(s_4^{-1} s_3 s_4)\\
=&
s_4^{-1} s_3  u_2 s_1  s_3  s_4^{-1} s_3^{-1} s_2  s_3^{-1} s_2 s_1^2 s_2s_3 s_4 s_3^{-1}
&=&
s_4^{-1} s_3  u_2 s_3  s_4^{-1} s_1   s_3^{-1} s_2  s_3^{-1} s_2 s_1^2 s_2s_3 s_4 s_3^{-1}
\subset V^+
\end{array}
$$ by lemma \ref{lemN5} (4).

\subsection{Reduction to $s_4 (s_3s_2^{-1} s_3) s_1 s_2^{-1} s_1 s_4 (s_3^{-1} s_2 s_3^{-1}) s_4^{-1} w_0 s_4$}
We apply the following relations of $B_4$  :
$$
\left\lbrace \begin{array}{lcl}
s_3(s_2 s_1^{-1} s_2)s_3^{-1} &=& s_2^{-1} s_1^{-1} (s_3 s_2^{-1} s_3) s_1 s_2 \\
s_3^{-1}(s_2^{-1} s_1 s_2^{-1})s_3 &=& s_2 s_1 (s_3^{-1} s_2 s_3^{-1}) s_1^{-1} s_2^{-1} \\
 \end{array} \right.$$
 This yields
$$
 \begin{array}{ll}
& s_4 s_3 (s_2 s_1^{-1} s_2) s_3^{-1} s_4 s_3^{-1} (s_2^{-1} s_1 s_2^{-1}) s_3 s_4^{-1} w_0 s_4\\ =&
s_4s_2^{-1} s_1^{-1} (s_3 s_2^{-1} s_3) s_1 s_2 s_4 s_2 s_1 (s_3^{-1} s_2 s_3^{-1}) s_1^{-1} s_2^{-1} s_4^{-1} w_0 s_4\\ =&
s_2^{-1} s_1^{-1}s_4 (s_3 s_2^{-1} s_3) s_1 s_2^2s_1s_4  (s_3^{-1} s_2 s_3^{-1}) s_4^{-1} w_0 s_4s_1^{-1} s_2^{-1}.
\end{array}$$
Expanding $s_2^2$ we get 
$$\begin{array}{lcl}
s_4 (s_3 s_2^{-1} s_3) s_1 s_2^2s_1s_4  (s_3^{-1} s_2 s_3^{-1}) s_4^{-1} w_0 s_4
& \in & R^{\times}
s_4 (s_3 s_2^{-1} s_3) s_1 s_2^{-1}s_1s_4  (s_3^{-1} s_2 s_3^{-1}) s_4^{-1} w_0 s_4 \\
&&+R s_4 (s_3 s_2^{-1} s_3) s_1 s_2s_1s_4  (s_3^{-1} s_2 s_3^{-1}) s_4^{-1} w_0 s_4 \\
&& + R s_4 (s_3 s_2^{-1} s_3) s_1^2s_4  (s_3^{-1} s_2 s_3^{-1}) s_4^{-1} w_0 s_4 
\end{array}
$$
We have 
$s_4 (s_3 s_2^{-1} s_3) s_1^2s_4  (s_3^{-1} s_2 s_3^{-1}) s_4^{-1} w_0 s_4 \in V^+$ by lemma \ref{lemN11},
and $$s_4 (s_3 s_2^{-1} s_3) s_1 s_2s_1s_4  (s_3^{-1} s_2 s_3^{-1}) s_4^{-1} w_0 s_4 \in V^+$$by
lemma \ref{lemN12}.

\begin{lemma}{\ } \label{lemN15}
\begin{enumerate}
\item $s_4 s_3 s_1^{-1} s_2 s_3^{-1} s_4^{-1} s_2 s_3 s_4^{-1} s_1 s_2 s_3 w_0 s_4 \in V^+$
\item$s_4 s_3 s_2^{-1} s_1^{-1} s_2 s_3^{-1} s_4^{-1} s_2 s_3 s_4^{-1} s_1 s_2 s_3 w_0 s_4 \in V_0$
\end{enumerate}
\end{lemma}
\begin{proof}
We prove (1). Using braid relations we get 
%$$s_4 s_3 s_1^{-1} s_2 s_3^{-1} s_4^{-1} s_2 s_3 s_4^{-1} s_1 s_2 s_3 w_0 s_4 = 
%s_1^{-1} s_2^{-1} s_3^{-1} s_4 s_3 s_2^2 s_3 s_4^{-1} s_1 s_2 s_3 w_0 s_4$$
%{}
$$
\begin{array}{clcl}
& s_4 s_3 s_1^{-1} s_2 s_3^{-1} s_4^{-1} s_2 s_3 s_4^{-1} s_1 s_2 s_3 w_0 s_4 &
=&  s_1^{-1} s_4 (s_3  s_2 s_3^{-1}) s_4^{-1} s_2 s_3 s_4^{-1} s_1 s_2 s_3 w_0 s_4 \\
=&  s_1^{-1} s_4 s_2^{-1}  s_3 s_2 s_4^{-1} s_2 s_3 s_4^{-1} s_1 s_2 s_3 w_0 s_4 &
=&  s_1^{-1} s_2^{-1} (s_4   s_3  s_4^{-1}) s_2^2 s_3 s_4^{-1} s_1 s_2 s_3 w_0 s_4 \\
=&  s_1^{-1} s_2^{-1} s_3^{-1}   s_4  s_3 s_2^2 s_3 s_4^{-1} s_1 s_2 s_3 w_0 s_4 \\
\end{array} 
$$
\begin{comment} % pour cause de bug arxiv
that is
%\begin{center}
$$\mbox{\rotatebox{90}{\resizebox{!}{5cm}{\begin{tikzpicture}
\braid[braid colour=black,strands=5,braid start={(0,0)}]%
{ \dsigma _4 \dsigma _3  \dsigma_1^{-1} \dsigma_2 \dsigma _3^{-1} \dsigma _4^{-1} \dsigma _2 \dsigma _3
\dsigma _4^{-1} \dsigma _1 \dsigma _2 \dsigma _3  \dsigma _3 \dsigma _2 \dsigma _1 \dsigma _1 \dsigma _2 \dsigma _3 \dsigma _4}
\end{tikzpicture} }}}
\raisebox{.8cm}{ =} 
\mbox{\rotatebox{90}{\resizebox{!}{5cm}{\begin{tikzpicture}
\braid[braid colour=black,strands=5,braid start={(0,0)}]%
{ \dsigma_1^{-1} \dsigma_2^{-1} \dsigma_3^{-1} \dsigma_4 \dsigma_3 \dsigma_2 \dsigma_2 \dsigma_3 \dsigma_4^{-1} \dsigma_1 \dsigma_2 \dsigma_3  \dsigma _3 \dsigma _2 \dsigma _1 \dsigma _1 \dsigma _2 \dsigma _3 \dsigma_4 }
\end{tikzpicture} }}}
$$
%\end{center}
\end{comment}
and this latter term $s_1^{-1} s_2^{-1} s_3^{-1} s_4 s_3 s_2^2 s_3 s_4^{-1} s_1 s_2 s_3 w_0 s_4$
belongs to $V^+$ by lemma \ref{lemN5} (4). We now prove (2).
By using braid relations we get 
%$s_4 s_3 s_2^{-1} s_1^{-1} s_2 s_3^{-1} s_4^{-1} s_2 s_3 s_4^{-1} s_1 s_2 s_3 w_0 s_4 
%= (s_1 s_2^{-1} s_3^{-1} ) s_4^{-1} s_3 s_2^2 s_1 s_2^{-1} s_3 s_2 s_4^{-1} s_3^2 s_2 s_1^2 s_2 s_3 s_4$,
$$
\begin{array}{clcl}
& s_4 s_3 (s_2^{-1} s_1^{-1} s_2) s_3^{-1} s_4^{-1} s_2 s_3 s_4^{-1} s_1 s_2 s_3 w_0 s_4  &=& 
s_4 s_3 s_1 s_2^{-1} s_1^{-1} s_3^{-1} s_4^{-1} s_2 s_3 s_4^{-1} s_1 s_2 s_3 w_0 s_4 \\
=& s_1 s_4 (s_3  s_2^{-1}  s_3^{-1}) s_4^{-1} s_1^{-1} s_2 s_3 s_4^{-1} s_1 s_2 s_3 w_0 s_4 &
=& s_1 s_4 s_2^{-1}  s_3^{-1}  s_2 s_4^{-1} s_1^{-1} s_2 s_3 s_4^{-1} s_1 s_2 s_3 w_0 s_4 \\
=& s_1 s_2^{-1} (s_4   s_3^{-1}   s_4^{-1}) s_2 s_1^{-1} s_2 s_3 s_4^{-1} s_1 s_2 s_3 w_0 s_4 &
=& s_1 s_2^{-1} s_3^{-1}   s_4^{-1}   s_3 s_2 s_1^{-1} s_2 s_3 s_4^{-1} s_1 s_2 s_3 w_0 s_4 \\
=& s_1 s_2^{-1} s_3^{-1}   s_4^{-1}   s_3 s_2 (s_1^{-1} s_2 s_1)s_3 s_4^{-1}  s_2 s_3 w_0 s_4 &
=& s_1 s_2^{-1} s_3^{-1}   s_4^{-1}   s_3 s_2 s_2 s_1 s_2^{-1}s_3 s_4^{-1}  s_2 s_3 w_0 s_4 \\
=& s_1 s_2^{-1} s_3^{-1}   s_4^{-1}   s_3 s_2 s_2 s_1 s_2^{-1}s_3 s_2s_4^{-1}   s_3 w_0 s_4 &
=& s_1 s_2^{-1} s_3^{-1}   s_4^{-1}   s_3 s_2 s_2 s_1 s_2^{-1}s_3 s_2s_4^{-1}   s_3^2s_2 s_1^2 s_2 s_3  s_4 \\
\end{array}
$$
and $s_4^{-1} s_3 s_2^2 s_1 s_2^{-1} s_3 s_2 s_4^{-1} s_3^2 s_2 s_1^2 s_2 s_3 s_4$
is a linear combination of
$s_4^{-1} s_3 s_2^2 s_1 s_2^{-1} s_3 s_2 s_4^{-1} s_3^{\alpha} s_2 s_1^2 s_2 s_3 s_4$
for $\alpha \in \{ -1,1,0 \}$. When $\alpha = 1$,
we get $s_4^{-1} s_3 s_2^2 s_1 s_2^{-1} s_3 s_2 s_4^{-1} s_3 s_2 s_1^2 s_2 s_3 s_4
= s_4^{-1} s_3 s_2^2 s_1 s_2^{-1} s_3 s_2 s_4^{-1} w_0 s_4 
= s_4^{-1} s_3 s_2^2 s_1 s_2^{-1} s_3  s_4^{-1} w_0 s_4s_2 $
which clearly (by expanding $s_2^2$)  belongs to $V_0 + A_4 s_4^{-1} w^+ s_4^{-1} w_0 s_4 A_4
= V_0 + A_4 s_4^{-1} w^+ s_4^{-1} w^+s_4 A_4
 \subset V^+$ by  lemma \ref{lemreduc2} (8) (take the image by $\Phi$ of the identity there).
 When $\alpha \in \{ 0,-1 \}$ we write 
$s_4^{-1} s_3 s_2^2 s_1 (s_2^{-1} s_3 s_2) s_4^{-1} s_3^{\alpha} s_2 s_1^2 s_2 s_3 s_4 =
s_4^{-1} s_3 s_2^2 s_1 s_3 s_2 s_3^{-1} s_4^{-1} s_3^{\alpha} s_2 s_1^2 s_2 s_3 s_4$.
When $\alpha = 0$, we get 
$s_4^{-1} s_3 s_2^2 s_1 s_3 s_2 s_3^{-1} s_4^{-1} (s_2 s_1^2 s_2) s_3 s_4=
s_4^{-1} s_3 s_2^2 s_1 s_3 s_2 s_3^{-1}  s_2 s_1^2 s_2(s_4^{-1} s_3 s_4) =
s_4^{-1} s_3 s_2^2 s_1 s_3 s_2 s_3^{-1}  s_2 s_1^2 s_2s_3 s_4 s_3^{-1} \in V_0$, so
we can assume $\alpha = -1$. Then
$$
 \begin{array}{lcl}
 s_4^{-1} s_3 s_2^2 s_1 s_3 s_2 (s_3^{-1} s_4^{-1} s_3^{-1}) s_2 s_1^2 s_2 s_3 s_4 &=&
s_4^{-1} s_3 s_2^2 s_1 s_3 s_2 s_4^{-1} s_3^{-1} s_4^{-1} s_2 s_1^2 s_2 s_3 s_4\\
&=& s_4^{-1} s_3 s_2^2 s_1 s_3 s_2 s_4^{-1} s_3^{-1}  s_2 s_1^2 s_2 (s_4^{-1} s_3 s_4)\\
&=& s_4^{-1} s_3 s_2^2 s_1 s_3 s_2 s_4^{-1} s_3^{-1}  s_2 s_1^2 s_2 s_3 s_4 s_3^{-1}\end{array}$$
is a linear combination of 
$s_4^{-1} s_3 s_2^2 s_1 s_3 s_2 s_4^{-1} s_3^{-1}  s_2 s_1^{\beta} s_2 s_3 s_4 s_3^{-1}$
for $\beta \in \{ -1,0,1 \}$. When $\beta = 0$ we get
$s_4^{-1} s_3 s_2^2 s_1 s_3 s_2 s_4^{-1} s_3^{-1}  s_2^2 s_3 s_4 s_3^{-1} \in V_0$
by lemma \ref{lemN5} (4) (taking the image by $\Psi$ of the second identity there).
When $\beta = 1$ we get 
$$
 \begin{array}{lcl}
 s_4^{-1} s_3 s_2^2 s_1 s_3 s_2 s_4^{-1} s_3^{-1}  (s_2 s_1 s_2) s_3 s_4 s_3^{-1}
&=&s_4^{-1} s_3 s_2^2 s_1 s_3 s_2 s_4^{-1} s_3^{-1}  s_1 s_2 s_1 s_3 s_4 s_3^{-1}\\
&=&s_4^{-1} s_3 s_2^2 s_1 s_3 s_2 s_1 s_4^{-1} s_3^{-1}  s_2  s_3 s_4 s_3^{-1}s_1
\in V_0\end{array}$$
by the same argument. When $\beta = -1$ we get
$$
 \begin{array}{lcl}
s_4^{-1} s_3 s_2^2 s_1 s_3 s_2 s_4^{-1} s_3^{-1}  (s_2 s_1^{-1} s_2) s_3 s_4 s_3^{-1}
&=&
s_4^{-1} s_3 s_2^2 s_1 s_3 s_2 s_4^{-1} (s_2 s_1) (s_3 s_2^{-1} s_3) (s_1^{-1} s_2^{-1}) s_4 s_3^{-1}\\
&=&
s_4^{-1} s_3 s_2^2 s_1 s_3 s_2 (s_2 s_1)  s_4^{-1} (s_3 s_2^{-1} s_3)  s_4 (s_1^{-1} s_2^{-1})s_3^{-1} \in V_0\end{array}$$
again by the same argument, and this concludes the proof.

\end{proof}

\subsection{Reduction to $s_4 s_3 s_2 s_1^{-1} s_2 s_3^{-1} s_4^{-1} s_2 s_3 s_4^{-1} s_1 s_2 s_3  w_0 s_4$}

We have
$$
\begin{array}{lcl}
s_4 (s_3s_2^{-1} s_3) s_1 s_2^{-1} s_1 s_4 (s_3^{-1} s_2 s_3^{-1}) s_4^{-1} w_0 s_4
 &=& s_4 (s_3s_2^{-1} s_3) s_1 s_2^{-1} s_1 s_4 s_3^{-1} s_2 (s_3^{-1} s_4^{-1} s_3^{-1}) s_3  w_0 s_4 \\
 &=& s_4 (s_3s_2^{-1} s_3) s_1 s_2^{-1} s_1 s_4 s_3^{-1} s_2 s_4^{-1} s_3^{-1} s_4^{-1} s_3  w_0 s_4 \\
 &=& s_4 (s_3s_2^{-1} s_3) s_1 s_2^{-1} s_1 (s_4 s_3^{-1} s_4^{-1}) s_2  s_3^{-1} s_4^{-1} s_3  w_0 s_4 \\
 &=& s_4 (s_3s_2^{-1} s_3) s_1 s_2^{-1} s_1 s_3^{-1} s_4^{-1} (s_3 s_2  s_3^{-1}) s_4^{-1} s_3  w_0 s_4 \\
 &=& s_4 (s_3s_2^{-1} s_3) s_1 s_2^{-1} s_1 s_3^{-1} s_4^{-1} s_2^{-1} s_3  s_2 s_4^{-1} s_3  w_0 s_4 \\
 &=& s_4 s_3s_2^{-1}  s_1 (s_3s_2^{-1}  s_3^{-1}) s_1 s_4^{-1} s_2^{-1} s_3   s_4^{-1} s_2 s_3  w_0 s_4 \\
 &=& s_4 s_3s_2^{-1}  s_1 s_2^{-1}s_3^{-1}  s_2 s_1 s_4^{-1} s_2^{-1} s_3   s_4^{-1} s_2 s_3  w_0 s_4 \\
 &=& s_4 s_3s_2^{-1}  s_1 s_2^{-1}s_3^{-1}  (s_2 s_1s_2^{-1} ) s_4^{-1} s_3   s_4^{-1} s_2 s_3  w_0 s_4 \\
 &=& s_4 s_3s_2^{-1}  s_1 s_2^{-1}s_3^{-1}  s_1^{-1} s_2s_1  s_4^{-1} s_3   s_4^{-1} s_2 s_3  w_0 s_4 \\
 &=& s_4 s_3s_2^{-1} ( s_1 s_2^{-1}s_1^{-1})s_3^{-1}   s_2 s_4^{-1} s_3   s_4^{-1} s_1 s_2 s_3  w_0 s_4 \\
 &=& s_4 s_3s_2^{-1} s_2^{-1} s_1^{-1}s_2 s_3^{-1}   s_2 s_4^{-1} s_3   s_4^{-1} s_1 s_2 s_3  w_0 s_4 \\
 &=& s_4 s_3s_2^{-2} s_1^{-1}s_2 s_3^{-1}    s_4^{-1}s_2 s_3   s_4^{-1} s_1 s_2 s_3  w_0 s_4 \\
 \end{array}
 $$
 and, expanding $s_2^{-2}$, we get
 $$
 \begin{array}{lcl} s_4 s_3s_2^{-2} s_1^{-1}s_2 s_3^{-1}    s_4^{-1}s_2 s_3   s_4^{-1} s_1 s_2 s_3  w_0 s_4 
& \in& R^{\times} s_4 s_3s_2 s_1^{-1}s_2 s_3^{-1}    s_4^{-1}s_2 s_3   s_4^{-1} s_1 s_2 s_3  w_0 s_4 \\
&&+ R s_4 s_3s_2^{-1} s_1^{-1}s_2 s_3^{-1}    s_4^{-1}s_2 s_3   s_4^{-1} s_1 s_2 s_3  w_0 s_4 \\
&&+ R s_4 s_3 s_1^{-1}s_2 s_3^{-1}    s_4^{-1}s_2 s_3   s_4^{-1} s_1 s_2 s_3  w_0 s_4 \end{array}$$
and the last two terms belong to $V_0$ by lemma \ref{lemN15} (1) and (2).

\begin{lemma}{\ } \label{lemN16}
\begin{enumerate}
\item $s_4 (s_3 s_2^{-1} s_3) s_1 s_4^{-1} s_3 s_4^{-1} s_1 s_2 s_3 w_0 s_4 \in V^+$
\item $s_4 (s_3 s_2^{-1} s_3) s_1 s_2 s_4^{-1} s_3 s_4^{-1} s_1 s_2 s_3 w_0 s_4 \in V^+$
\item $s_4 s_3 s_2^{-1} s_3 s_1 s_2 s_4^{-1} s_3 s_4^{-1} \in A_4^{\times} s_4 s_3 s_2^{-1} s_3 s_4^{-2} A_3^{\times}$
\item $s_4 u_2 u_3 s_1^{-1} s_2 s_4^{-1} s_3 s_4^{-1} s_1^2 s_2 s_3 w_0 s_4 \in V_0$
\end{enumerate}
\end{lemma}
\begin{proof} We prove (1).
$
s_4 s_3 s_2^{-1} s_3 s_1 s_4^{-1} s_3 s_4^{-1} s_1 s_2 s_3 w_0 s_4 
=s_4 s_3 s_2^{-1} s_1(s_3  s_4^{-1} s_3 s_4^{-1}) s_1 s_2 s_3 w_0 s_4 $ belongs to
$$s_4 s_3 s_2^{-1} s_1  s_4^{-1} s_3 s_4^{-1}s_3 s_1 s_2 s_3 w_0 s_4  +
s_4 s_3 s_2^{-1} s_1 u_3 u_4  s_1 s_2 s_3 w_0 s_4 
+ s_4 s_3 s_2^{-1} s_1 u_4 u_3  s_1 s_2 s_3 w_0 s_4 $$
by lemma \ref{lemdecomp1212}.
We have $s_4 s_3 s_2^{-1} s_1 u_3 u_4  s_1 s_2 s_3 w_0 s_4 =
s_4 s_3 s_2^{-1} u_3 u_4   s_1^2 s_2 s_3 w_0 s_4 \subset V^+$
by lemma \ref{lemN5} (4), and
$s_4 s_3 s_2^{-1} s_1 u_4 u_3  s_1 s_2 s_3 w_0 s_4=
s_4 s_3 u_4 s_2^{-1} s_1  u_3  s_1 s_2 s_3 w_0 s_4 \subset V_0$
by lemma \ref{lemreduc2} (2). Moreover 
$s_4 s_3 s_2^{-1} s_1  s_4^{-1} s_3 s_4^{-1}s_3 s_1 s_2 s_3 w_0 s_4 =
(s_4 s_3 s_4^{-1} )s_2^{-1} s_1   s_3 s_4^{-1}s_3 s_1 s_2 s_3 w_0 s_4 
=
s_3^{-1} s_4 s_3s_2^{-1}   s_3 s_4^{-1}s_3  s_1^2 s_2 s_3 w_0 s_4 
=
s_3^{-1} s_4 s_3s_2^{-1}   s_3 s_4^{-1}s_3  s_1^2 s_2 s_3 w_0 s_4 \in V^+$
by lemma \ref{lemN5} (4). This proves (1).
%We now prove (3).  
Using only braid relations we get 
%$ s_4 s_3 s_2^{-1} s_3 s_1 s_2 s_4^{-1} s_3 s_4^{-1} =( s_1 s_2 s_3 s_2^{-1} s_1^{-1} ) s_4 s_3 s_2^{-1} s_3 s_4^{-2} (s_1 s_2)$,
$$
\begin{array}{clcl}
&s_4 s_3 s_2^{-1} s_3 s_1 s_2 s_4^{-1} s_3 s_4^{-1}&
= & s_1 s_1^{-1} s_4 s_3 s_2^{-1} s_1 s_3  s_2 s_4^{-1} s_3 s_4^{-1} \\
= & s_1  s_4 s_3 (s_1^{-1} s_2^{-1} s_1) s_3  s_2 s_4^{-1} s_3 s_4^{-1} &
= & s_1  s_4 s_3 s_2 s_1^{-1} (s_2^{-1} s_3  s_2) s_4^{-1} s_3 s_4^{-1} \\
= & s_1  s_4 s_3 s_2 s_1^{-1} s_3 s_2  s_3^{-1} s_4^{-1} s_3 s_4^{-1} &
= & s_1  s_4 (s_3 s_2 s_3)s_1^{-1}  s_2  s_3^{-1} s_4^{-1} s_3 s_4^{-1} \\
= & s_1  s_4 s_2 s_3 s_2s_1^{-1}  s_2  s_3^{-1} s_4^{-1} s_3 s_4^{-1} &
= & s_1  s_2 s_4  s_3 s_2s_1^{-1}  s_2  (s_3^{-1} s_4^{-1} s_3) s_4^{-1} \\
= & s_1  s_2 s_4  s_3 s_2s_1^{-1}  s_2  s_4 s_3^{-1} s_4^{-1} s_4^{-1} &
= & s_1  s_2 (s_4  s_3 s_4 ) s_2s_1^{-1}  s_2  s_3^{-1} s_4^{-2}  \\
= & s_1  s_2 s_3  s_4 (s_3  (s_2s_1^{-1}  s_2)  s_3^{-1}) s_4^{-2} &
= & s_1  s_2 s_3  s_4 (s_2^{-1} s_1^{-1})(s_3 s_2^{-1} s_3) (s_1 s_2) s_4^{-2} \\
= & (s_1  s_2 s_3  s_2^{-1} s_1^{-1}) s_4 s_3 s_2^{-1} s_3  s_4^{-2}(s_1 s_2) \\
\end{array}
$$
\begin{comment} % pour cause de bug arxiv
that is 
$$\mbox{\rotatebox{90}{\resizebox{2cm}{!}{\begin{tikzpicture}
\braid[braid colour=black,strands=5,braid start={(0,0)}]%
{\dsigma_4 \dsigma_3 \dsigma_2^{-1} \dsigma_3 \dsigma_1 \dsigma_2 \dsigma_4^{-1} \dsigma_3 \dsigma_4^{-1}}
\end{tikzpicture} }}}
\raisebox{1.2cm}{ =} 
\mbox{\rotatebox{90}{\resizebox{2cm}{!}{\begin{tikzpicture}
\braid[braid colour=black,strands=5,braid start={(0,0)}]%
{ \dsigma_1 \dsigma_2 \dsigma_3 \dsigma_2^{-1} \dsigma_1^{-1}  \dsigma_4 \dsigma_3 \dsigma_2^{-1} \dsigma_3 \dsigma_4^{-1}\dsigma_4^{-1} \dsigma_1 \dsigma_2 }
\end{tikzpicture} }}}
$$
\end{comment}
\begin{comment}
$$\mbox{\rotatebox{90}{\resizebox{!}{5cm}{\begin{tikzpicture}
\braid[braid colour=black,strands=5,braid start={(0,0)}]%
{ }
\end{tikzpicture} }}}
\raisebox{.8cm}{ =} 
\mbox{\rotatebox{90}{\resizebox{!}{5cm}{\begin{tikzpicture}
\braid[braid colour=black,strands=5,braid start={(0,0)}]%
{  }
\end{tikzpicture} }}}
$$
\end{comment}
which proves (3). From this we deduce 
that $s_4 (s_3 s_2^{-1} s_3) s_1 s_2 s_4^{-1} s_3 s_4^{-1} s_1 s_2 s_3 w_0 s_4
\in A_4 s_4 u_3 u_2 u_3 u_4 A_4 s_4 \subset V^+$ by lemma \ref{lemN5} (4). This proves (2).
Finally 
$$
 \begin{array}{lcl}
 s_4 u_2 u_3 s_1^{-1} s_2 s_4^{-1} s_3 s_4^{-1} s_1^2 s_2 s_3 w_0 s_4
&=&u_2s_1^{-1}s_4  u_3  s_4^{-1}s_2  s_3 s_4^{-1} s_1^2 s_2 s_3 w_0 s_4 \\
&\subset &u_2s_1^{-1}s_3^{-1}  u_4  (s_3 s_2  s_3) s_4^{-1} s_1^2 s_2 s_3 w_0 s_4 \\
&\subset& u_2s_1^{-1}s_3^{-1}  u_4  s_2 s_3  s_2 s_4^{-1} s_1^2 s_2 s_3 w_0 s_4 \\
&\subset& u_2s_1^{-1}s_3^{-1}   s_2 u_4  s_3  s_4^{-1} s_2  s_1^2 s_2 s_3 w_0 s_4  \subset V_0
\end{array}$$
by lemma \ref{lemreduc2} (2). This proves (4).

\end{proof}

\begin{lemma}{\ }  \label{lemN17}
\begin{enumerate}
\item $u_4 A_4 s_4^{-1} s_3^{\beta} w_0 s_4 \subset V^+$
\item $s_4 u_3 u_2 s_1^{-1} s_2 s_4^{-1} s_3 s_4^{-1} s_1^2 s_2 s_3 w_0 s_4 \subset V^+$
\end{enumerate}
\end{lemma}
\begin{proof}
We prove (1). Using $A_4  = A_3 u_3 A_3 + A_3 u_3 u_2 u_3 + A_3 w^+ + A_3 w^-$
we get that $u_4 A_4 s_4^{-1} s_3^{\beta} w_0 s_4$ is the sum of
the following abelian groups :
\begin{itemize}
\item $u_4 A_3 u_3 A_3 s_4^{-1} s_3^{\beta} w_0 s_4= A_3 u_4  u_3  s_4^{-1}A_3 s_3^{\beta} w_0 s_4 \subset V_0$ by lemma \ref{lemreduc2} (2).
\item $u_4 A_3 w^+ s_4^{-1} s_3^{\beta} w_0 s_4= A_3 u_4  w^+ s_4^{-1} s_3^{\beta} w_0 s_4$, which is included
in $V_0 + A_4 u_4 w^+ s_4^{-1} w^+ s_4 A_4$ by lemma \ref{lemreducspec} (4) and proposition \ref{propmoinsde55}.
Now $u_4 w^+ s_4^{-1} w^+ s_4 \subset V_0 +R s_4 w^+ s_4^{-1} w^+ s_4+ R s_4^{-1}w^+ s_4^{-1} w^+ s_4$ and we have 
$s_4^{-1}w^+ s_4^{-1} w^+ s_4 \in V_0$ by lemma \ref{lemreduc2} (7) (apply $\Phi \circ \Psi$ to the identity there),
so $u_4 A_3 w^+ s_4^{-1} s_3^{\beta} w_0 s_4 \subset V^+$
\item $u_4 A_3 w^- s_4^{-1} s_3^{\beta} w_0 s_4= A_3 u_4  w^- s_4^{-1} s_3^{\beta} w_0 s_4$, which is included
in $V_0 + A_4 u_4 w^- s_4^{-1} w^+ s_4 A_4$ by lemma \ref{lemreducspec} (4) and proposition \ref{propmoinsde55}.
Since $u_4 w^- s_4^{-1} w^+ s_4 \subset V_0$ by lemma \ref{lemreduc2} (5) we get $u_4 A_3 w^- s_4^{-1} s_3^{\beta} w_0 s_4 \subset V^+$
\item $u_4 A_3 u_3 u_2 u_3  A_3 s_4^{-1} s_3^{\beta} w_0 s_4 = A_3 u_4  u_3 u_2 u_3  s_4^{-1} A_3 s_3^{\beta} w_0 s_4 \subset V^+$
by lemma \ref{lemN5} (4).
\end{itemize}
This proves (1).  Since $s_4 u_3 u_2 s_1^{-1} s_2 s_4^{-1} s_3 s_4^{-1} s_1^2 s_2 s_3 w_0 s_4 
= s_4 u_3 s_4^{-1}u_2 s_1^{-1} s_2  s_3 s_4^{-1} s_1^2 s_2 s_3 w_0 s_4 
= s_3^{-1} u_4 s_3 u_2 s_1^{-1} s_2  s_3 s_1^2 s_2 s_4^{-1}  s_3 w_0 s_4  \subset s_3^{-1} u_4 A_4 s_4^{-1} s_3 w_0 s_4 \subset V^+$
by (1), and this proves (2).
%\item $u_4 A_4 s_4^{-1} s_3^{\beta} w_0 s_4$
%\item $u_4 A_4 s_4^{-1} s_3^{\beta} w_0 s_4$

\end{proof}

\subsection{Reduction to $s_4   w^-   s_4    w_0^2 s_4$}
We have
$$
 \begin{array}{lcl}
 s_4 s_3 (s_2 s_1^{-1} s_2) s_3^{-1} s_4^{-1} s_2 s_3 s_4^{-1} s_1 s_2 s_3  w_0 s_4
&=& s_4 (s_2^{-1} s_1^{-1})(s_3 s_2^{-1} s_3)s_1 s_2 s_4^{-1} s_2 s_3 s_4^{-1} s_1 s_2 s_3  w_0 s_4\\
&=& (s_2^{-1} s_1^{-1}) s_4 (s_3 s_2^{-1} s_3)s_1 s_2^2 s_4^{-1}  s_3 s_4^{-1} s_1 s_2 s_3  w_0 s_4\\
\end{array}$$
and, expanding $s_2^2$, we get that this last element belongs
to $$\begin{array}{ll} & R^{\times} (s_2^{-1} s_1^{-1}) s_4 (s_3 s_2^{-1} s_3)s_1 s_2^{-1} s_4^{-1}  s_3 s_4^{-1} s_1 s_2 s_3  w_0 s_4 \\+&
R(s_2^{-1} s_1^{-1}) s_4 (s_3 s_2^{-1} s_3)s_1 s_2 s_4^{-1}  s_3 s_4^{-1} s_1 s_2 s_3  w_0 s_4 \\ +&
R(s_2^{-1} s_1^{-1}) s_4 (s_3 s_2^{-1} s_3)s_1  s_4^{-1}  s_3 s_4^{-1} s_1 s_2 s_3  w_0 s_4 . \end{array}$$
Now $s_4 (s_3 s_2^{-1} s_3)s_1  s_4^{-1}  s_3 s_4^{-1} s_1 s_2 s_3  w_0 s_4 \in V^+$ by lemma \ref{lemN16} (1)
and $$s_4 (s_3 s_2^{-1} s_3)s_1 s_2 s_4^{-1}  s_3 s_4^{-1} s_1 s_2 s_3  w_0 s_4  \in V^+$$ by lemma \ref{lemN16} (2).
We are thus reduced to 
$$
\begin{array}{lcl}
s_4 (s_3 s_2^{-1} s_3)s_1 s_2^{-1} s_4^{-1}  s_3 s_4^{-1} s_1 s_2 s_3  w_0 s_4
 &=& s_4 (s_3 s_2^{-1} s_3)s_1 s_2^{-1} s_4^{-1}  s_3 s_4^{-1} s_1^{-1} s_1^2 s_2 s_3  w_0 s_4\\
 &=& s_4 (s_3 s_2^{-1} s_3)(s_1 s_2^{-1} s_1^{-1}) s_4^{-1}  s_3 s_4^{-1}  s_1^2 s_2 s_3  w_0 s_4\\
 &=& s_4 (s_3 s_2^{-1} s_3s_2^{-1}) s_1^{-1} s_2 s_4^{-1}  s_3 s_4^{-1}  s_1^2 s_2 s_3  w_0 s_4\\
\end{array}
$$
which, by lemma \ref{lemdecomp1212}, lies in 
$$\begin{array}{ll} & R^{\times} s_4  s_2^{-1} s_3s_2^{-1}s_3 s_1^{-1} s_2 s_4^{-1}  s_3 s_4^{-1}  s_1^2 s_2 s_3  w_0 s_4\\
+ & s_4 u_2 u_3 s_1^{-1} s_2 s_4^{-1}  s_3 s_4^{-1}  s_1^2 s_2 s_3  w_0 s_4\\ +&s_4u_3u_2  s_1^{-1} s_2 s_4^{-1}  s_3 s_4^{-1}  s_1^2 s_2 s_3  w_0 s_4 .\end{array}$$
The two latter terms lie in $V^+$ by lemma \ref{lemN16} (4) and \ref{lemN17} (2), so we are reduced to
$$
\begin{array}{lcl}
s_4 s_2^{-1} s_3s_2^{-1}s_3 s_1^{-1} s_2 s_4^{-1}  s_3 s_4^{-1}  s_1^2 s_2 s_3  w_0 s_4
&=&s_2^{-1} s_4  s_3s_2^{-1}s_3 s_1^{-1} s_2 s_4^{-1}  s_3 s_4^{-1}  s_1^2 s_2 s_3  w_0 s_4\\
&=&s_2^{-1}s_3^{-1} (s_3 s_4  s_3)s_2^{-1}s_3 s_1^{-1} s_2 s_4^{-1}  s_3 s_4^{-1}  s_1^2 s_2 s_3  w_0 s_4\\
&=&s_2^{-1}s_3^{-1} s_4 s_3  s_4s_2^{-1}s_3 s_1^{-1} s_2 s_4^{-1}  s_3 s_4^{-1}  s_1^2 s_2 s_3  w_0 s_4\\
&=&s_2^{-1}s_3^{-1} s_4 s_3  s_2^{-1}(s_4s_3 s_4^{-1})s_1^{-1} s_2   s_3 s_4^{-1}  s_1^2 s_2 s_3  w_0 s_4\\
&=&s_2^{-1}s_3^{-1} s_4 (s_3  s_2^{-1}s_3^{-1})s_4 s_3 s_1^{-1} s_2   s_3 s_4^{-1}  s_1^2 s_2 s_3  w_0 s_4\\
&=&s_2^{-1}s_3^{-1} s_4 s_2^{-1}  s_3^{-1}s_2s_4 s_3 s_1^{-1} s_2   s_3 s_4^{-1}  s_1^2 s_2 s_3  w_0 s_4\\
&=&s_2^{-1}s_3^{-1} s_2^{-1}s_4   s_3^{-1}s_2s_1^{-1}s_4( s_3  s_2   s_3) s_4^{-1}  s_1^2 s_2 s_3  w_0 s_4\\
&=&s_2^{-1}s_3^{-1} s_2^{-1}s_4   s_3^{-1}s_2s_1^{-1}s_4 s_2  s_3   s_2 s_4^{-1}  s_1^2 s_2 s_3  w_0 s_4\\
&=&s_2^{-1}s_3^{-1} s_2^{-1}s_4   s_3^{-1}s_2s_1^{-1}s_2(s_4   s_3   s_4^{-1})  s_2 s_1^2 s_2 s_3  w_0 s_4\\
&=&s_2^{-1}s_3^{-1} s_2^{-1}s_4   s_3^{-1}s_2s_1^{-1}s_2s_3^{-1}   s_4   s_3  s_2 s_1^2 s_2 s_3  w_0 s_4\\
&=&s_2^{-1}s_3^{-1} s_2^{-1}s_4   w^-   s_4    w_0^2 s_4\\
\end{array}
$$
 hence to $s_4   w^-   s_4    w_0^2 s_4$.
 
 \subsection{Conclusion of the computation}
 
 By lemma \ref{lemw0carre}, we have $w_0^2 \in A_3^{\times} w_0^{-1} + U^+ = w_0^{-1} A_3^{\times} + U^+$,
 and $U^+ = A_3 w_0 + U_0 = w_0 A_3 + U_0$. We then have
 $s_4   w^-   s_4    w_0^2 s_4 \in s_4   w^-   s_4    w_0^{-1} s_4 A_3^{\times} + s_4   w^-   s_4    w_0 s_4 A_3
 + s_4   w^-   s_4    U_0s_4$. 
 
On the one hand, we know that $s_4   w^-   s_4    A_3 u_3 A_3 s_4 = s_4   w^-   s_4    A_3 u_3  s_4 A_3 =
 s_4   w^-   s_4  u_2  u_1 u_2 u_1 u_3  s_4 A_3 =  
s_4   w^-   s_4  u_2  u_1 u_2  u_3  s_4 u_1A_3  \subset V_0$ by proposition \ref{propmoinsde55},
and that  
$$
 \begin{array}{lclcr}
 s_4   w^-   s_4    A_3 u_3 u_2 u_3 A_3 s_4
&=& s_4   w^-   s_4    A_3 u_3 u_2 u_3  s_4 A_3\\
&=& s_4   w^-   s_4    u_1 u_2 u_1 (u_2 u_3 u_2 u_3)  s_4 A_3\\
&=& s_4   w^-   s_4    u_1 u_2 u_1 u_3 u_2 u_3 u_2  s_4 A_3\\
&=& s_4   w^-   s_4    u_1 u_2 u_1 u_3 u_2 u_3  s_4 u_2 A_3 &\subset& V^+
\end{array}$$
by lemma \ref{lemN5} (4) (apply $\Phi \circ \Psi$ to the identity there).
 From $U_0 = A_3 u_3 A_3 + A_3 u_3 u_2 u_3 A_3$ one thus gets
 $s_4   w^-   s_4    U_0s_4 \subset V^+$.
 
 On the other hand, we have $s_4   w^-   s_4    w_0 s_4 \in V_0 + s_4 w^- s_4 w^+ s_4 A_3 \subset V_0$
 by  lemma \ref{lemreduc2} (5). We are thus reduced to $s_4 w^- s_4 w_0^{-1} \in s_4 w^- s_4 w^- s_4 A_3^{\times} + V_0$,
 which concludes the proof.

%\begin{comment)
%\end{comment)

\tableofcontents

\end{document}